\documentclass[12pt]{amsart}

\usepackage[margin=75.0pt]{geometry}
\usepackage{amssymb,amsmath,bm,cancel, amsthm, comment}
\usepackage{xcolor}

\def\a{\alpha}

\newtheorem{theorem}{Theorem}[section]
\newtheorem{lemma}[theorem]{Lemma}
\newtheorem{proposition}[theorem]{Proposition}
\newtheorem{corollary}[theorem]{Corollary}
\newtheorem{definition}[theorem]{Definition}

\newtheorem{remark}[theorem]{Remark}

\def\i{\mathrm{i}}

\def\C{\mathbb{C}}
\def\cP{\mathcal{P}}
\def\geq{\geqslant}
\def\leq{\leqslant}
\def\R{\mathbb{R}}
\def\T{\mathbb{T}}

\def\Z{\mathbb{Z}}
\def\N{\mathbb{N}}

\def\Q{\mathbb{Q}}

\def\epsilon{\varepsilon}

\newcommand{\vep}{\varepsilon}

\newcommand{\bea}{\begin{eqnarray}}
  \newcommand{\eea}{\end{eqnarray}}
  \newcommand{\beab}{\begin{eqnarray*}}
  \newcommand{\eeab}{\end{eqnarray*}}
\renewcommand{\a}{\alpha}
  \newcommand{\be}{\begin{equation}}
  \newcommand{\ee}{\end{equation}}

\begin{document}
\title{Prime number theorem for analytic skew products}
\author{Adam Kanigowski}
\address{Department of Mathematics, University of Maryland at College Park, College Park,
MD 20740, USA}
\email{adkanigowski@gmail.com}
\author{Mariusz Lema{\'n}czyk}
\address{Faculty of Mathematics and Computer Science, Nicolaus Copernic University Ul. Chopina 12/18 87-100 Torun, Poland}
\email{mlem@mat.umk.pl}
\author{Maksym Radziwi{\l\l}}
\address{Department of Mathematics
Caltech, Linde Hall,
1200 E California Blvd,
Pasadena, CA, 91125, USA}
\email{maksym.radziwill@gmail.com}

\maketitle
\begin{abstract} We establish a prime number theorem for all uniquely ergodic, analytic skew products on the $2$-torus $\mathbb{T}^2$. More precisely, for every irrational $\alpha$ and every $1$-periodic real analytic $g:\mathbb{R}\to\mathbb{R}$ of zero mean, let $T_{\alpha,g} : \mathbb{T}^2 \rightarrow \mathbb{T}^2$ be defined by $(x,y) \mapsto (x+\alpha,y+g(x))$. We prove that if $T_{\alpha, g}$ is uniquely ergodic then, for every $(x,y) \in \mathbb{T}^2$, the sequence $\{T_{\alpha, g}^p(x,y)\}$ is equidistributed on $\T^2$ as $p$ traverses prime numbers. This is the first example of a class of natural, non-algebraic and smooth dynamical systems for which a prime number theorem holds. We also show that such a prime number theorem does not necessarily hold if $g$ is only continuous on $\mathbb{T}^2$.

\end{abstract}

\section*{Introduction}

Let $X$ be a compact metric space and $T : X \rightarrow X$ a continuous map so that $(X,T)$ is a topological dynamical system. Given a $T$-invariant Borel measure $\nu$, following the work of Bourgain \cite{Bo1} and Wierdl \cite{Wi}, we know that for $\nu$-almost all $x \in X$, the sequence
\begin{equation} \label{eq:limit}
\frac{1}{N} \sum_{p \leq N} f(T^p x) \log p
\end{equation}
converges, with $p$ traversing prime numbers. However, we are in general lacking a description of the limit. 
 More importantly, the problem of understanding when convergence in \eqref{eq:limit} holds for all $x \in X$ remains open.

Whenever \eqref{eq:limit} converges to a limit for all $x \in X$ for any given continuous $f : X \rightarrow \mathbb{R}$ we will say that \textit{a prime number theorem holds for $(X, T)$}.
There is at present no clear understanding for which dynamical systems a prime number theorem should hold. On the other hand, we have a very precise conjecture, due to Sarnak, for the seemingly related notion of \textit{M\"obius disjointness}. Sarnak's conjecture asserts that for any dynamical system of topological entropy zero,
$$
\frac{1}{N} \sum_{n \leq N} f(T^n x) \mu(n) \rightarrow 0
$$
for all $x \in X$ as $N \rightarrow \infty$. Sarnak's conjecture is verified for a vast array of dynamical systems (see \cite{SarnakSurvey}). Meanwhile prime number theorems are established only for a few special dynamical systems:
\begin{itemize}
\item cyclic rotations on $\Z / d \Z $ (i.e.\ the Prime Number Theorem in arithmetic progressions),
\item rotations on $\mathbb{T}$ (i.e.\ Vinogradov's \cite{Vinogradov} theorem),
\item nilsystems (i.e.\ the Green-Tao \cite{Gr-Ta} theorem),
\item Rudin-Shapiro sequences (Mauduit-Rivat \cite{Ma-Ri}),
\item enumeration systems (Bourgain \cite{Bo, Bo2}, Green \cite{Gr}),
\item certain finite rank symbolic systems (Bourgain \cite{Bour4}, Ferenczi-Maduit \cite{FerMa}),
\item automata (M{\"u}llner \cite{Mullner}).
\end{itemize}
%

One of the reasons for this discrepancy is that we have more tools to address Sarnak's conjecture.
The number-theoretic tools (e.g.\ \cite{BSZ}, \cite{Daboussi}, \cite{Katai},  \cite{MatoRad}, \cite{Tao}) that are used in the context of Sarnak's conjecture rely on the fact that most integers are composite and thus factor. In particular, these tools completely ignore the behavior on the subsequence of prime numbers and are therefore inapplicable in the context of establishing a prime number theorem.

All dynamical systems for which a prime number theorem is currently known are either algebraic or symbolic.
This is, as we will explain later, an important technical advantage.
In this paper we are interested in establishing a prime number theorem for a natural class of zero entropy \textit{smooth dynamical systems}
that are neither algebraic nor symbolic.
Specifically, we consider \textit{analytic skew products} (also known as \textit{Anzai skew products}\footnote{in honour of Anzai \cite{Anzai} who introduced them in the 1950's.}), that is, maps $T_{\alpha, g} : \mathbb{T}^2 \rightarrow \mathbb{T}^2$ defined by
$$T_{\alpha, g} (x, y) = (x + \alpha, y + g(x))$$
with $\alpha$ irrational and  $g : \mathbb{R} \rightarrow \mathbb{R}$ a $1$-periodic real-analytic function.
The behavior of these systems can be quite complex: Furstenberg \cite{Fu1} famously showed that $T_{\alpha, g}$ (with $g$ analytic)
can be minimal without being uniquely ergodic (i.e.\ the orbits $\{T_{\alpha, g}^{n}(x,y)\}$ can be dense without being equidistributed). Yet, analytic skew products are some of the simplest (non-algebraic) generalizations of irrational rotations and they can be viewed as {\em random rotations}: at the $n$th step $T_{\alpha,g}$ rotates the second coordinate of $(x,y)$ by $g(\{x + n \alpha\})$ and the sequence $\{x + n \alpha\}$ can be viewed as a source of ``deterministic randomness''.
We refer the reader to \cite{Le-spectralsurvey} for further information on Anzai skew products and information on their importance in ergodic theory. M\"obius disjointness of skew products $T_{\alpha, g}$ received particular attention: For analytic $g$, under a modest additional condition, M\"obius disjointness for $T_{\alpha, g}$ was established by Liu-Sarnak \cite{SarnakLiu}, subsequent results lowered this assumption to $g$ analytic \cite{Wang}, then $C^{\infty}$ \cite{WZX}, then $C^{2 + \varepsilon}$ \cite{KLR} and the current best result requires $g$ to be only $C^{1+ \varepsilon}$ \cite{deFav}.

We are now ready to state our main result.

\begin{theorem}\label{thm:main} Let $\alpha \in \mathbb{R} \backslash \mathbb{Q}$ and let $g : \mathbb{R} \rightarrow \mathbb{R}$ be a $1$-periodic real-analytic function of zero mean. If $T_{\alpha,g}$ is uniquely\footnote{In the class of Anzai skew products strict ergodicity is equivalent to unique ergodicity. Moreover, unique ergodicity implies total unique ergodicity, that is, all-non zero powers remain uniquely ergodic.} ergodic then for every continuous $f : \mathbb{T}^2 \rightarrow \C$ and every $(x,y) \in \mathbb{T}^2$, as $N \rightarrow \infty$,
\begin{equation} \label{eq:pntmainresult}
\frac{1}{N}\sum_{p\leq N}f(T_{\alpha,g}^p(x,y))\log p \rightarrow \int_{\T^2}f (\beta,\gamma) d \beta d \gamma,
\end{equation}
where as usual the letter $p$ stands for prime numbers.
In fact, the convergence is uniform in $(x,y) \in \mathbb{T}^2$.
\end{theorem}

Since Theorem \ref{thm:main} holds for all uniquely ergodic analytic skew-products, we believe that the rate of convergence in \eqref{eq:pntmainresult} can be arbitrarily slow. We expect that the condition ``$T_{\alpha, g}$ uniquely ergodic'' is also necessary in Theorem \ref{thm:main}. Such a converse is implicit in our proof for certain special $\alpha$'s (for example those $\alpha$ that can be expanded into a continued fraction $[0; q_1, q_2, \ldots]$ with all the $q_i$'s having a bounded number of prime factors).

We recall that the system $T_{\alpha, g}$ is uniquely ergodic if and only if there is no measurable solution $\xi : \mathbb{T} \rightarrow \mathbb{C}$ with $|\xi| = 1$ to the equation
\begin{equation} \label{eq:abc}
e^{2\pi i k g(x)} = \frac{\xi(x)}{\xi(x + \alpha)}
\end{equation}
for every $k \in \mathbb{N}$. This implies that if $T_{\alpha, g}$ is uniquely ergodic then $\alpha$ is non-diophantine\footnote{that is, for any given $A > 0$ there are only finitely many $q$ such that $\| q \alpha \| > q^{-A}$.}. Finally, whenever a measurable solution $\xi$ to the equation \eqref{eq:abc} with $k=1$ exists, we say that \textit{$g$ is a multiplicative coboundary for the rotation by $\alpha$}. We refer the reader to Katok \cite{Katok} for a sufficient condition\footnote{If $g(x)=\sum_{m\in\Z}a_me^{2\pi imx}$ is analytic and there is a subsequence $\{q_{n_k}\}$ of denominators for $\alpha$ such that $\|q_{n_k}\alpha\|/a_{q_{n_k}}\to0$ then $g$ is not multiplicatively cohomologous to any constant. As Katok's condition is stable for multiples of $g$, it implies the unique ergodicity of $T_{\alpha, g}$.} in terms of the Fourier coefficients of $g$ that ensures that $g$ is not a multiplicative coboundary for $\alpha$.

We have the following immediate corollary of Theorem \ref{thm:main}.
\begin{corollary}\label{cor:main}Let $\alpha \in \mathbb{R} \backslash \mathbb{Q}$ and $g : \mathbb{R} \rightarrow \mathbb{R}$ be a $1$-periodic real-analytic function of zero mean. If for no $k\in\N$, $kg$ is a multiplicative coboundary for the rotation by $\alpha$ then for any continuous $f : \mathbb{T} \rightarrow \C$ and $x \in \mathbb{T}$, as $N \rightarrow \infty$,
  $$
  \frac{1}{N} \sum_{p \leq N} f \Big ( \sum_{m < p} g(x + m \alpha) \Big )\log p \rightarrow \int_{\mathbb{T}} f(u) du.
  $$
  That is, the sequence $\{g(x)+g(x+\alpha)+\ldots+ g(x+(p_n-1)\alpha)\}$, $n \geq1$, is equidistributed (and where $p_n$ denotes the $n$-th prime number). \end{corollary}

Let us now turn to a discussion of some of the more technical aspects of Theorem \ref{thm:main}. Prime number theorems have been so far established only for dynamical systems $(X,T)$ that are either algebraic (e.g.\ translations on nilmanifolds) or symbolic. A fundamental reason for this is that most of the earlier approaches immediately use Vinogradov's method to reduce the problem to that of understanding sums of the form
\begin{align*}
& \text{(Type I) } \sum_{n \leq N} f(T^{d n} x) \text{ and } \\ & \text{(Type II) } \sum_{n \leq N} f(T^{d_1 n} x) \overline{f(T^{d_2 n} x)}
\end{align*}
for all $x \in X$ and continuous $f$.  For fixed $d_1$ and $d_2$ results on type II sums can be obtained by studying joinings of $T^{d_1}$ and $T^{d_2}$. Unfortunately, in order to obtain information on primes (unlike for the M{\"o}bius disjointness) the variables $d, d_1, d_2$ need to be allowed to grow at least like a small power of $N$. For algebraic dynamical systems (e.g.\ nilsystems) one has a chance of obtaining such information using joinings.
However, for non-algebraic dynamical systems these methods break down.

For a successful application of Vinogradov's method one needs to reach a certain numerical threshold in the evaluation of type I and type II sums, for instance $d \leq N^{1/2 - \varepsilon}$ and $d_1, d_2 \leq N^{1/3 - \varepsilon}$ for any given $\varepsilon > 0$ (see e.g.\ \cite{DFI}). For $T_{\alpha, g}$, in the most optimistic scenario, we can only obtain information for type I and type II sums with $d, d_1, d_2$ that do not exceed $N^{\varepsilon}$ for every fixed $\varepsilon > 0$. This is always insufficient for a successful application of Vinogradov's method.

Instead, we develop a new approach tailored for systems of slow and controlled orbit growth.
The main new idea on the dynamical side is to use the slow orbit growth of the system to approximate it by a periodic system (with a period being a small power of $N$) plus a polynomial phase and then approximate averages along primes by usual ergodic averages. While the coefficients of the polynomial phase do depend on the point $(x,y) \in \mathbb{T}^2$, they are well controlled uniformly over all points $(x,y) \in \mathbb{T}^2$.

On the number theoretic side, controlling the average behavior of these orbits boils down roughly to being able to control expressions of the form
\begin{equation} \label{eq:huxleytype0}
\sum_{z < q} \sum_{y < N} \sup_{\beta} \Big | \sum_{\substack{p \in [y, y + H] \\ p_{q} \in [z, z + H']}} e(p_{q} \beta) \log p  -  \frac{H}{\varphi(q)} \sum_{\substack{(a,q) = 1 \\ a \in [z, z + H']}} e(a \beta) \Big |
\end{equation}
with $e(x) := e^{2\pi i x}$ and $[0,q-1]\ni p_{q} := p \mod{q}$,
or expressions of the form
\begin{equation} \label{eq:huxleytype}
  \sum_{y < N} \sum_{v = 1}^{r} \Big | \sum_{\substack{p \in [y, y + H] \\ p_{q} \equiv v \mod{r}}} \log p - \frac{H}{r} \Big |
\end{equation}
with $1 \leq r \leq q^{1 - \varepsilon}$.
To get a sense of the problem let us focus on \eqref{eq:huxleytype}. A non-trivial estimate for \eqref{eq:huxleytype} with $q = 1$ and $r = 1$ is equivalent to establishing a prime number theorem in $[y, y + H]$ for almost all $y \leq N$. Following Huxley \cite{Huxley}, this is known for $H > N^{1/6 + \varepsilon}$ and $N$ sufficiently large with respect to $\varepsilon > 0$, for any fixed $\varepsilon > 0$. For a general $q > 1$, if we take $r = \lfloor q^{1 - \varepsilon} \rfloor $ and if $p_{q}$ was replaced by $p$ then a non-trivial bound for \eqref{eq:huxleytype} would correspond to a hybrid version of Huxley's theorem in short arithmetic progressions and short intervals. The latter is completely out of reach, and we are helped to a large extent by the fact that we have to understand the distribution of $p_{q}$, rather than $p$, in arithmetic progressions. Our argument will share some commonalities with Huxley's result and in particular, we will be limited by the condition $H / q > N^{1/6 + \varepsilon}$ which is the correct analogue of Huxley's result in short arithmetic progression and short intervals.

The second important input is an extension of a recent result of Matom\"aki-Shao \cite{MS}, on polynomial phases in short intervals, namely
\begin{equation} \label{eq:matotype}
\sum_{N \leq p \leq N + H} e \Big ( \sum_{j \leq k} \alpha_j (p - N)^{j} \Big )
\end{equation}
with $\alpha_j \in \mathbb{R}$.
Their result allows one to take $H > N^{2/3 + \varepsilon}$. For our argument to succeed, it will be crucial to either pass the threshold $H / q > N^{1/6 + \varepsilon}$ in \eqref{eq:huxleytype} or the threshold $H> N^{2/3 + \varepsilon}$ in \eqref{eq:matotype}. Passing either threshold requires one to address the contribution of so-called type III sums. We believe that it is an interesting feature of this problem that such a natural number theoretic obstruction appears in it. We end up passing this threshold by slightly improving the result of Matom{\"a}ki-Shao using ideas of Heath-Brown which allows one to barely handle the contribution of these type III sums.

Clearly, in all of our results, it would be interesting to further relax the assumption on the smoothness of $g$. However, it turns out that Theorem \ref{thm:main} cannot hold for merely continuous $g$.

\begin{theorem}\label{thm:main4}  For every
  $\alpha\in \R\setminus\Q$ there exists a continuous $g : \mathbb{T} \rightarrow \mathbb{T}$ such that the map $T=T_{\alpha,g}$ satisfies the following:
\begin{enumerate}
\item[i.] $T$ is uniquely ergodic,
\item[ii.] $T$ satisfies Sarnak's conjecture,
\item[iii.] $T$ does not satisfy polynomial Sarnak's conjecture; more precisely, for a continuous $f(x,y)=\tilde{f}(y)$,
the sequence $\{\frac1N\sum_{n\leq N}f(T^{n^2}(0,0))\mu(n)\}$ has a non-zero accumulation point,
\item[iv.] there exists a continuous $f(x,y)=\tilde{f}(y)$ such that the sequence
$$
\frac{1}{N}\sum_{p\leq N}f(T^p(0,0))\log p
$$
does not converge.
\end{enumerate}
\end{theorem}

We recall that the only known (totally) strictly ergodic systems for which a prime number theorem fails were constructed by Pavlov \cite{Pa}. His examples are given by some symbolic constructions (subshifts) whose entropy has not been determined. Note also that iii.\ gives a negative answer to polynomial Sarnak's conjecture (that is, a part of  Problem~7.1 in \cite{SarWorkshop} and Conjecture 2.3 in \cite{Ei}). Simultenaously and independently of us, a negative answer to polynomial Sarnak's conjecture has been obtained in \cite{Chin1} in the class of Toeplitz sub-shifts.

It is an open question to determine whether Theorem \ref{thm:main} holds for $g$ which are $C^{\infty}(\mathbb{T})$. Our current proof exploits the fact that if a trigonometric polynomial is large at a point then it is large at a set of large measure, provided that the degree remains under control. When $g$ is analytic, we can approximate $g$ sufficiently well by trigonometric polynomials (of bounded degree), so as to conclude that $g$ inherits the same property. However, such an approximation is no longer possible if $g$ is only required to be $C^{\infty}$.

Before we turn to a description of our proof, we would like to make a few comments on possible extensions of this work:
\begin{enumerate}
\item If $T_{\alpha, g}$ (with $\alpha \not \in \mathbb{Q}$ and $g$ analytic) is minimal but not uniquely ergodic, we are able to show that the set $\{T_{\alpha,g}^{p}(x,y):\: p\text{ is prime}\}$ is dense in $\T^2$ for each $(x,y)\in \T^2$ (the proof will be published elsewhere). This result resembles a result on the distribution of prime orbits of the horocycle flow from \cite{SU}. It would be interesting to determine whether the averages
  $$
  \frac{1}{N} \sum_{p \leq N} f(T^p_{\alpha, g} (x,y)) \log p
  $$
  converge for every $(x,y)$ without any assumption on $T_{\alpha, g}$.

\item A variant of our proof establishes the results of Liu-Sarnak \cite{SarnakLiu} and Wang \cite{Wang} without using either the DDKBSZ (Daboussi-Delange-K{\'a}tai-Bourgain-Sarnak-Ziegler \cite{BSZ, Daboussi, DaboussiDelange, Katai}) criterion or the Matom\"aki-Radziwi{\l\l} theorem \cite{MatoRad}.

\item For certain special $\alpha$'s and under the assumption of the Generalized Riemann Hypothesis it is possible to relax the requirement on the smoothness of $g$ to $g \in C^k(\mathbb{T}^2)$ for some $k > 2$. It is unclear to us if the smoothness can be relaxed conditionally on the Generalized Riemann Hypothesis for all $\alpha$, and it remains an interesting open question to determine even conjecturally the optimal smoothness exponent.

\item It should be possible to extend our work to handle a larger class of rigid systems for which a direct application of Vinogradov's method (of type I and II sums) is ineffectual.
   \end{enumerate}

   \subsection*{Acknowledgments} The three authors would like to thank the American Institute of Mathematics for hosting a workshop on ``Sarnak's Conjecture'' at which this work was begun. Research of AK was partially supported by NSF grant DMS-1956310. Research of ML was
partially supported by Narodowe Centrum Nauki grant UMO-2019/33/B/ST1/00364. MR acknowledges the partial support of NSF (through the grant DMS-1902063) and of a Sloan fellowship.
   We are grateful to Kaisa Matom\"aki for an alternative treatment of a part of Section \ref{sec:eqprimes2}; her argument is described in Subsection \ref{se:alternative}.

\section{Outline of the proof}\label{outline}
Let $q_1 < q_2 < \ldots$ be the sequence of \textit{denominators} of $\alpha$, i.e.\ ``best rational approximations'' of $\alpha$ so that for all $k \geq 2$,
$$
\frac{1}{2 q_{k + 1} q_k} \leq \Big | \alpha - \frac{\ell_k}{q_{k}} \Big | \leq \frac{1}{q_{k + 1} q_{k}},
$$
for some integer valued sequence $\{\ell_k\}_{k\in \N}$. We begin by noticing that since $g$ is real-analytic, it admits a Fourier expansion
$$
g(x) = \sum_{m \in \mathbb{Z}} a(m) e(m x)
$$
with $|a(m)| \leq e^{-\tau m}$ for some $\tau > 0$. For simplicity, we assume that $\tau = 1$.
Moreover, instead of working with $g$, we can work with
$$
\widetilde{g}(x) = \sum_{n \in \mathbb{Z}} g_n(x) \ , \ g_n(x) := \sum_{\substack{q_n \leq |m| \leq \log q_{n + 1} \\ q_n | m}} a(m) e(m x).
$$
Indeed, we show that the maps $T_{\alpha, g}(x,y)$ and $T_{\alpha, \widetilde{g}}(x,y)$ are topologically conjugate, therefore, there exists a continuous invertible map $H: \mathbb{T}^2 \rightarrow \mathbb{T}^2$ such that $$T_{\alpha, g}^{p}(x,y) = H^{-1}(T^{p}_{\alpha, \widetilde{g}}(H(x,y)))$$
for all $p \geq 1$. So Theorem \ref{thm:main} for   $T_{\alpha, g}$ follows from Theorem \ref{thm:main} for ${T_{\alpha,\widetilde{g}}}$. We assume therefore without loss of
generality that $g = \widetilde{g}$. This assumption will be in place throughout the whole paper. 
Since the functions $e_{b,c}(x,y) := e^{2\pi i (b x + c y)}$ are dense in the set of continuous functions on $\mathbb{T}^2$, it suffices to obtain Theorem \ref{thm:main} for $f(x,y) = e_{b,c}(x,y)$. We can assume that $c \neq 0$ since otherwise the result follows from Vinogradov's theorem.
We will also write $e(x) := e^{2\pi i x}$ and $T=T_{\alpha,g}$.

Given  a sufficiently large $N \geq 1$, let $n \in \mathbb{N}$ be the unique integer such that $N \in [q_{n}, q_{n + 1})$.

Roughly, we will relate the behavior of
\begin{equation}\label{eq:beha}
\frac{1}{N} \sum_{p \leq N} e_{b,c}(T^{p}(x,y)) \log p  \ \text{ with } \ \frac{1}{\varphi(q_k)} \sum_{\substack{(m, q_k) = 1 \\ m < q_k}} e_{b,c}(T^{m}(x,y))
\end{equation}
for some $q_k$ with $k \leq n$, depending on $N$ and $g$, and such that $q_{k} \rightarrow \infty$ as $N \rightarrow \infty$.

When $q_k$ is prime, the condition $(m,q_k) = 1$ is redundant and the sum on the right-hand side converges to $0$
since $T$ is uniquely ergodic.
However, for $q_k$ highly composite, the sum over $m$ could be quite lacunary, and it is not obvious that
the unique ergodicity of $T$ is sufficient to ensure that the sum is $o(1)$. Instead, we show that if $q_k$ is replaced by $z_k q_k$ for some small $z_k$ (i.e. $z_k\leq \log^2 q_k$),
then the sum over $m$
can be indeed made to converge to zero. Therefore, in the actual proof we will be relating the sum over primes to a similar sum but with modulus
$z_k q_k$ instead of $q_k$.
For simplicity, we will at first ignore this issue in the outline below and
assume that $\alpha$ is chosen so that all the $q_n$ are prime.
At the end of the outline, we indicate the changes that are necessary to treat all $\alpha$. It is enough to show that for every $\eta>0$ and $N$ sufficiently large (in terms of $1/\eta$),
\begin{equation}\label{eq:aeta}
\sum_{p \leq N} e_{b,c}(T(x,y)) \log p \ll \eta^{1/2} N
\end{equation}
for every $(x,y) \in \mathbb{T}^2$ and every $b,c \in \mathbb{Z}$ with $c \neq 0$.

We establish two important types of approximation which we will repeatedly use:
\begin{itemize}
\item Given $n \geq 1$, let $n^{\ast} \leq n$ be the largest integer such that $q_{n^{\ast}} > e^{q_{n^{\ast} - 1} / 16}$. Then, for $m \leq q_n \min(q_{n + 1} / q_{n^{\ast}} , e^{q_n / 8})$,
  \begin{equation} \label{eq:firstapx}
  e_{b,c}(T^m(x,y)) \approx e_{b,c}(T^{m \mod {q_n}}(x,y)).
  \end{equation}

  In particular, if $e^{q_n / 8} > q_{n + 1}$ then the above holds for all $m \leq q_{n + 1}$.
\item Given any $\delta > 0$, for $m \leq q_{n + 1}^{1 - \delta}$ and $n$ sufficiently large with respect to $\delta$, we have
  \begin{align} \label{eq:secondapx}
  e_{b,c}(T^m(x,y)) & \approx e_{b,c}(T^{m \mod {q_n}}(x , y + P_{n}(x,m)) \\ \nonumber & = e_{b,c}(T^{m \mod{q_n}} (x,y)) e(c P_n(x,m)),
  \end{align}
  where $P_{n}(x,m)$ is a polynomial of degree $\leq \lfloor 1/\delta \rfloor$ and where the second equality follows
  simply from the definition of $T$.
  The polynomial $P_{n}(x,m)$ is given by
  $$
  P_n(x, m) = \sum_{1 \leq j \leq 1 / \delta} a_j(x) m^j
  $$
  with
  $|a_1(x)| \leq e^{-q_n / 16}$ and $|a_j(x)| \leq q_n^{-1} q_{n + 1}^{- j + 1}$
  for all $x$.
\end{itemize}

Note that it is sensible to use these in an iterative fashion. For instance, in some scenarios, we will apply the first approximation twice, and in others, we will first apply the first approximation, followed by the second.

An important parameter for understanding when to use \eqref{eq:firstapx} or \eqref{eq:secondapx} is given by $n^{\ast}$ : the largest integer $n^{\ast} \leq n$ such that $q_{n^{\ast}} \geq e^{q_{n^{\ast} - 1} / 16}$. We will typically localize $e_{b,c}(T^m(x,y))$ into a short interval $m \in [N, N + H]$ for various scales of $H$. In particular, using that
$$
e_{b,c}(T^m(x,y)) = e_{b,c}(T^{m - N}(T^N(x,y))),
$$
it is enough to understand the behavior of $T^{m - N}$, at the price of loosing control on $T^N(x,y)$. If $H \approx q_k$ with $k > n^{\ast}$ then we can appeal to \eqref{eq:firstapx} to show that instead of studying $T^{m - N}$ it's enough to understand $T^{m - N \mod{q_{k - 1}}}$ thus reducing the complexity of the problem. On the other hand, if $H \approx q_{k}^{1 - \eta}$ with $k \leq n^{\ast}$ then we have no choice but to use \eqref{eq:secondapx}. The upshot then is that $e_{b,c}(T^{m - N}(x,y))$ with $m$ varying in each such interval is approximately a polynomial phase of degree $\approx 1 / \eta$ together with a low complexity term (with small period).

We are now ready to discuss the proof of \eqref{eq:aeta}
The proof splits into three main cases, depending on whether $N \in [\exp(q_n^{1/2}), q_{n + 1})$, $N \in [q_n^{6/5 + \eta^2}, \exp(q_n^{1/2})]$
  or $N \in [q_n, q_n^{6/5 + \eta^2}]$.
  The cut-off $\exp(q_n^{1/2})$ is rather arbitrary, while $q_n^{6/5 + \eta^2}$ is significant.
  The case $N \in [q_n, q_n^{6/5 + \eta^2}]$ is further separated into the cases where $q_{n^{\ast}} > N^{2/3 - \eta / 5}$
  and $q_{n^{\ast}} < N^{2/3 - \eta / 5}$. 


  \subsection{The case $N \in [\exp(q_{n}^{1/2}), q_{n + 1})$}\label{ca:one}

    This is the ``easy case'' and we deal with it by decomposing the interval $[0, N]$ into sub-intervals of length $H = \min(N, q_{n + 1}^{3/4})$ (there is no specific importance to the exponent $3/4$ and anything larger than $2/3$ and smaller than $1$ would have worked). As a result, it suffices to show that
    $$
    \sum_{N' \leq p \leq N' + H} e_{b,c}(T^p(x,y)) \log p \ll \eta^{1/2} H
    $$
    for any $N' \leq N$. We write $T^p(x,y) = T^{p - N'}(x', y')$ with $(x', y') = T^{N'}(x,y)$. Splitting $p - N'$ into arithmetic progressions $\pmod{q_n}$ and using \eqref{eq:secondapx}, we can approximate
    $T^{p - N'}(x', y')$ by $T^a(x', y')$ and $P(x', p - N')$ of degree $\leq 5$ in $p - N'$.  In particular, this reduces the problem to showing that
    \begin{equation} \label{eq:redx}
    \sum_{0 \leq a < q_n} e_{b,c}(T^{a}(x', y')) \sum_{\substack{N' \leq p \leq N' + H \\ p \equiv a + N' \pmod{q_n}}} e(c P(p - N')) \ll \eta^{1/2} H.
    \end{equation}
    To understand the short sums over $p \in [m, m + H]$, we can now either appeal to a recent result of Matom\"aki-Shao \cite{MS} or a slight strenghtening of there-off that we will need later (Theorem \ref{thm:nr4}).
    Using that the coefficients of $P(x',p - N')$ are small, we can show that the left-hand side of \eqref{eq:redx} is equal to
    $$
    \Big ( \frac{1}{q_n} \sum_{0 \leq a < q_n} e_{b,c}(T^{a}(x', y')) \Big ) \cdot \sum_{N' \leq k \leq N' + H} e(c P(k)) + O(\eta H)
    $$
    and the result now follows from trivially bounding the sum over $k$ and using the unique ergodicity of $T$ to conclude that the sum over $a$ is $o(1)$
    as $N \rightarrow \infty$.

\subsection{The case $N \in [q_{n}^{6/5 + \eta^2}, \exp(q_{n}^{1/2})]$} \label{ca:sec}

In this case, since $m \leq N \leq e^{q_n / 16}$, we can use \eqref{eq:firstapx} to reduce the problem to showing that
\begin{equation} \label{eq:mainx}
\sum_{p \leq N} e_{b,c}(T^{p_{q}}(x,y)) \log p \ll \eta^{1/2} N,
\end{equation}
where $q := q_n$ and $p_q := p \mod{q} \in [0, q - 1]$. Notice that $p_{q}$ is a simpler object than $p$ but not by a huge amount since $q$ can be as large as $N^{5/6 - c\eta^2}$ for some $c>0$, and trivially $p_{q} = p$ for $q > N$.

We will now apply either \eqref{eq:firstapx} or \eqref{eq:secondapx} to further approximate $e_{b,c}(T^{p_q}(x,y))$ by simpler expressions. We split into two sub-cases depending on the relative sizes of $q_n$ and $q_{n^{\ast}}$, where $n^{\ast}$ is defined as the largest integer $n^{\ast} \leq n$ such that $q_{n^{\ast}} > e^{q_{n^{\ast} - 1} / 16}$. Instead of working with $n^{\ast}$, we could alternatively iterate the approximation \eqref{eq:firstapx} several times until reaching a desirable denominator $q_n$. The use of $q_{n^{\ast}}$ allows to expedite this iteration.

\subsubsection{The case $q_{n^{\ast}} > q_n^{1 - \eta^2}$}
We decompose $p_{q}$ into short intervals of length $H' = q_n^{1/3}$ and split $p_{q}$ into residue classes to modulus $r := q_{n^{\ast} - 1}$. By the definition of $n^{\ast}$, the modulus $r$ is tiny compared to $q_{n^{\ast}} \leq q_n$ and thus compared to $N$. Therefore, splitting into residue classes $\pmod{r}$ does not increase the complexity of the problem. Thus, to establish \eqref{eq:mainx}, we will study the expression  
$$
\frac{1}{q H'} \sum_{z < q} \sum_{a \leq r} \sum_{\substack{p \leq N \\ p_{q} \in [z, z + H'] \\ p_{q} \equiv a \pmod{r}}} e_{b,c}(T^{p_{q} - z}(T^z(x,y))) \log p + O(H').
$$
In the above formula, $0 \leq p_{q} - y \leq q_{n}^{1/3} \leq q_{n^{\ast}}^{3/8}$ for $\eta$ sufficiently small. Therefore, by \eqref{eq:secondapx}, we can approximate the above by
$$
\frac{1}{q H'} \sum_{z < q} \sum_{a \leq r} e_{b,c}(T^{a}(T^z(x,y))) \sum_{\substack{p \leq N \\ p_{q} \in [z, z + H'] \\ p_{q} \equiv a \pmod{r}}} e(c p_q \beta_z) \log p
$$
with $|\beta_z| \leq e^{-q_{n^{\ast} - 1}}$ and $e^{-q_{n^{\ast} - 1}} \leq \eta^4$ provided that $N$ is taken sufficiently large with $\eta$ since $n^{\ast} \rightarrow \infty$ with $n \rightarrow \infty$. Notice that we can exclude $a = 0$ from the summation at the price of an error $\ll N / r$ which is acceptable (recall that we assume for simplicity that all $q_i$ are prime).

We now bound the above as
\begin{align*}
\frac{1}{q H'} \sum_{\substack{z < q \\ 0 < a < r}} \sup_{|\beta| \leq \eta^4} \Big | & \sum_{\substack{p \leq N \\ p_{q} \in [z, z + H'] \\ p_{q} \equiv a \pmod{r}}} e(c p_{q} \beta) \log p - \frac{N}{\varphi(q) \varphi(r)} \sum_{\substack{(v,q) = 1 \\ v \in [z, z + H'] \\ v \equiv a \pmod{r}}} e(c v \beta) \Big | \\ & + \frac{N}{q H'} \sum_{z < q} \frac{1}{\varphi(q) \varphi(r)} \sum_{0 < a < r} e_{b,c}(T^{a + z}(x,y))\sum_{\substack{(v,q) = 1 \\ v \in [z, z + H']}} e(c v \beta_z).
\end{align*}
The second sum is $\ll \eta^{1/2} N$ for all $N$ sufficiently large with respect to $\eta$, by unique ergodicity applied to the sum over $a$. The first sum is also $\ll \eta^{1/2} N$ but this requires a non-trivial arithmetic input. To avoid repetition with a later more involved sub-case we skip the discussion of this number theoretic input. Note that it is important for the argument to work to have the upper bound $|\beta| \leq \eta^4$, since the number theoretic bound cannot hold if for instance $\beta = 1/2$.

\subsubsection{The case $q_{n^{\ast}} \leq q_{n}^{1 - \eta^2}$}  \label{se:similarcase}

In this case, clearly $n^{\ast} \neq n$, therefore, we have $q_{n} \leq e^{q_{n - 1} / 16}$ which means that the approximation \eqref{eq:firstapx} is applicable and we can approximate $e_{b,c}(T^{p_q}(x,y))$ by $e_{b,c}(T^{p_{q} \pmod{r}}(x,y))$, where $r := q_{n - 1}$. In particular, splitting the $p_{q}$ in the sum \eqref{eq:mainx} into progressions $\mod{r}$, it suffices to show that
$$
\sum_{a\leq r}e_{b,c}(T^a(x,y))\Big(\sum_{\substack{p\leq N\\ p_q\equiv a \mod r}}\log p\Big) \ll \eta^{1/2} N.
$$
We now bound this as
$$
\sum_{a \leq r} \Big | \sum_{\substack{p \leq N \\ p_{q} \equiv a \pmod{r}}} \log p - \frac{N}{r} \Big | + \frac{N}{r} \sum_{a \leq r} e_{b,c}(T^{a}(x,y)) .
$$
Unique ergodicity shows that the second sum is $\ll \eta^{1/2} N$ for all sufficiently large $N$, and therefore, it remains to show that the first sum is negligible. Let us now describe the number theoretic tools that go into this. In other words, it will suffice to show that, for any given $\varepsilon > 0$,
\begin{equation} \label{eq:toproveoutline}
\sum_{a \leq r} \Big | \sum_{\substack{p \leq N \\ p_{q} \equiv a \pmod{r}}} \log p - \frac{N}{r} \Big | = o(N)
\end{equation}
uniformly in $r < q^{1 - \epsilon}$ and $q<N^{5/6-\epsilon}$.

To illustrate the core difficulties let us assume that $r \approx N^{5/6 - 2\epsilon}$ which is the hardest case. If $p_{q}$ were replaced by $p$ then this would be a $q$-analogue of Huxley's theorem on prime numbers in almost all short intervals. The latter is completely out of reach since it would require a zero free region for $L(s, \chi)$ better than what is currently known. However, we are helped by the fact that we have to prove this result for $p_{q}$ instead of $p$. Indeed, opening \eqref{eq:toproveoutline} into Dirichlet characters, the problem reduces to bounding
$$
\sum_{a < r} \Big | \Big ( \frac{1}{\varphi(q)} \sum_{\chi \neq \chi_{0} \pmod{q}} \sum_{p \leq N} \chi(p) \log p \Big ) \cdot \Big ( \sum_{\substack{v < q \\ v \equiv a \pmod{r}}} \chi(v) \Big ) \Big |.
$$
It is important that we do not use the triangle inequality on the sum over $\chi$ at this stage.
We separate this expression into two types of characters: the few bad characters $\chi$ for which there is no cancellation in the sum over $p$ and the good characters $\chi$ in which we have a non-trivial amount of cancellations in the sum over $p$. We bound the contribution of the bad characters by
$$
\frac{1}{\varphi(q)} \# \{ \chi \text{ bad}\} \cdot N \cdot \sup_{\chi \neq \chi_{0}} \sum_{a < r} \Big | \sum_{\substack{v < q \\ v \equiv a \pmod{r}}} \chi(v) \Big |.
$$
Applying the Cauchy-Schwarz inequality, orthogonality of additive characters and the completion method gives
$$
\sum_{a < r} \Big | \sum_{\substack{v < q \\ v \equiv a \pmod{r}}} \chi(a) \Big | \ll \sqrt{r q} \cdot d(q) \log q
$$
and so we end up with a final bound
$$
\sqrt{\frac{r}{q}} \cdot \log q \cdot \# \{ \chi \text{ bad}\} \cdot N
$$
which is acceptable as long as $r \leq q^{1 - \varepsilon}$ because $q \approx N^{5/6 - \varepsilon}$ and there are few bad characters (fewer than $\ll (\log N)^{A}$ for some large $A$).

It remains to deal with the contribution of the good characters, that is,
$$
\sum_{a < r} \Big | \frac{1}{\varphi(q)} \sum_{\substack{\chi \neq \chi_{0} \\ \text{good}}} \Big ( \sum_{\substack{p \leq N}} \chi(p) \log p \Big ) \cdot \Big ( \sum_{\substack{v < q \\ v \equiv a \mod{r}}} \chi(v) \Big ) \Big |.
$$
We find phases $\theta_{a} \in \mathbb{R}$ for which the above expression can be re-written as
$$
\sum_{a < r} e^{i \theta_{a}} \cdot \frac{1}{\varphi(q)} \sum_{\substack{\chi \neq \chi_{0} \\ \text{good}}} \Big ( \sum_{\substack{p \leq N}} \chi(p) \log p \Big ) \cdot \Big ( \sum_{\substack{v < q \\ v \equiv a \mod{r}}} \chi(v) \Big ).
$$
In particular, we can re-write this as
$$
\frac{1}{\varphi(q)} \sum_{\substack{\chi \neq \chi_{0} \\ \text{good}}} \Big ( \sum_{p \leq N} \chi(p) \log p \Big ) \cdot \Big ( \sum_{v < q} \chi(a) c(a) \Big )
$$
with $c(a) = e^{i \theta_{a \mod {r}}}$, where $a \mod {r} \in [0, r - 1]$. Whenever we will use such a trick, we will say that we ``used duality''.
We then apply the Cauchy-Schwarz inequality (and the large sieve on the sum over $v$) and the problem reduces to showing that
$$
\frac{1}{\varphi(q)} \sum_{\substack{\substack{\chi \neq \chi_{0} \\ \chi \text{ good}}}} \Big | \sum_{p \leq N} \chi(p) \log p \Big |^2 \ll_{A} \frac{N^2}{(\log N)^{A}}
$$
for some sufficiently large $A > 0$. This is however now an analogue of Huxley's result with the assumption that $L(s, \chi)$ has a good zero free region (because we restrict only to ``good'' characters $\chi$ which is equivalent to assuming that the corresponding $L$-function $L(s, \chi)$ has an enhanced zero-free region). At this stage, we use the same ideas that go into the proof of Huxley's estimate, in particular, his bounds for the frequency of large values of Dirichlet polynomials. We note that there are no known techniques to us that would allow us to handle asymptotically the case $r > N^{5/6 + \varepsilon}$ with $\varepsilon > 0$ fixed and therefore, this is really the best range that we can obtain given the current techniques (short of assuming some unproven hypothesis such as, for example, the generalized Lindel\"of hypothesis).

\subsection{The case $N \in [q_n, q_n^{6/5 + \eta^2}]$}

This is the most delicate case which is further split according to whether $q_{n^{\ast}} > N^{2/3 - \eta / 5}$ or
$q_{n^{\ast}} \leq N^{2/3 - \eta / 5}$. There is an interesting numerological interaction between these two sub-cases: we find that in order to be able to handle both, one either needs to lower the exponent $\tfrac {2}{3} + \varepsilon$ in the result of Matom\"aki-Shao or lower the exponent $\tfrac 16 + \varepsilon$ in our variants of Huxley's theorem. In both cases, the bottleneck are type-III sums which emerge as one crosses this threshold in either problem. We manage to circumvent the problem of fully dealing with these type-III sums since it is sufficient for us to cross the threshold $\tfrac 23$ (or $\tfrac 16$) by an $\eta > 0$ which tends to zero as $N$ tends to infinity (at the price of error term that only save $O(\eta)$). In particular, we appeal to ideas of Heath-Brown \cite{HeathBrown} and bound the contribution of type-III sums using a sieve estimate which is sufficient since $\eta$ eventually tends to zero with $N$ (albeit very slowly). We chose to cross this threshold in the Matom\"aki-Shao theorem since this is more likely to be useful in the number theoretic literature.

\subsection{The case $N \in [q_n, q_n^{6/5 + \eta^2}]$ and $q_{n^{\ast}} > N^{2/3 - \eta / 5}$}

We cover $[0, N]$ with disjoint intervals of length $H=q_{n^\ast}^{1-\eta}$. Thus it's enough to show that
$$
\sum_{p \in [N', N' + H]} e_{b,c}(T^p(x,y)) \log p \ll \eta^{1/2} H
$$
for all $N' \leq N$. The hardest case occurs when $N' \asymp N$, and we assume this for simplicity.
Proceeding as in the first case, we decompose  $N \in [\exp(q_n^{1/2}), q_{n + 1})$ into residue classes $\pmod r$ with $r := q_{n^{\ast} - 1}$ and
  use the approximation \eqref{eq:secondapx} applied to $T^{p - N'}$. As a result, it suffices to bound
  \begin{equation} \label{eq:thirdeqx}
  \frac{1}{r} \sum_{a < r} e_{b,c}(T^{a}(T^{N'}(x,y)))  \sum_{\substack{p \in [N', N' + H] \\ p \equiv a + N' \pmod{r}}} e(c P_{n^{\ast}}(p))
  \end{equation}
  with $P_{n^{\ast}}(p)$ a polynomial of degree $\leq 1/\eta + 1$ and with small coefficients (as described in \eqref{eq:secondapx}).
  At this point, the only difference with
  the previous case \ref{ca:one} is that the length of the interval, $H$, is only guaranteed to be $> N^{2/3 - \eta}$ since $H = q_{n^{\ast}}^{1 - \eta} \geq N^{(1 - \eta) \cdot (2 / 3 - \eta / 5)} > N^{2/3 - \eta}$. If $H$ were $> N^{2/3 + \eta}$, we could appeal to results of Matom\"aki-Shao to conclude immediately. In fact, from the number-theoretic point of view there is a significant difference between intervals of length $N^{2/3 + \eta}$ and $N^{2/3 - \eta}$. The latter requires one to handle the contribution of so-called ``type-III'' sums, a special case of which is
$$
\sum_{\substack{N \leq a b c \leq N + H \\ N^{1/3} \leq a,b,c \leq 2 N^{1/3}}} e ( P_{n^{\star}}(a b c)), \ H = N^{2/3 - \eta}.
$$
Ideally, one would hope to show that these sums are $\ll_{A} H (\log N)^{-A}$ for any $A > 0$. This is possible for example for polynomials of degree $1$ (see \cite{Tang}), but for general polynomials of degree $\asymp \eta^{-1}$ we do not know how to obtain such a saving.
Instead, we appeal to an idea of Heath-Brown \cite{HeathBrown} and use Linnik's identity to bound the contribution of the type-III sums using an upper bound sieve. While this gives rise to a much weaker error term of size $O(\eta \log \frac{1}{\eta}\cdot H / \varphi(r))$, it allows the degree of the polynomial to be of size $1/\eta$ as long as $N$ is sufficiently large with respect to $1/\eta$. This strengthening of the result of M\"atomaki-Shao then allows us to handle \eqref{eq:thirdeqx} just as in the case \ref{ca:one} and we conclude.

  \subsection{The case $N \in [q_n, q_n^{6/5 + \eta^2}]$ and $q_{n^{\ast}} \leq N^{2/3 - \eta / 5}$}

Let $n' < n$ be the largest integer such that $q_{n'} \leq N^{5/6 - 2 \eta^2}$.
Set $q := q_{n'}$ and $H = q N^{1/6 + \eta^2}$. Then, since $q_{n^{\ast}} < N^{2/3 - \eta/ 5}$
and $q_{n' + 1} > N^{5/6 - 2\eta^2}$, we have
$$
q N^{1/6 + \eta^2} \leq q_{n'} \frac{N^{5/6 - 2 \eta^2}}{N^{2/3 - \eta / 5}} \leq q_{n'} \frac{q_{n' + 1}}{q_{n^{\ast}}}
$$
  and, moreover, we have $q_{n^{\ast}} < N^{2/3} \leq N^{5/6 - 2 \eta^2} \leq q_{n' + 1}$
  and therefore $n^{\ast} \leq n'$ so that $N^{5/6 - 2 \eta^2} \leq q_{n' + 1} \leq e^{q_{n'} / 16}$
  and in particular $q N^{1/6 + \eta^2} \leq q e^{q / 16}$.
  Therefore, on intervals of length $H$ we can use the approximation \eqref{eq:firstapx} and write
  \begin{equation} \label{eq:mainxx}
  \sum_{p \in [N', N' + H]} e_{b,c}(T^{p}(x,y)) \log p \approx \sum_{p \in [N', N' + H]} e_{b,c}(T^{p_{q} - N'_{q}}(x', y')) \log p
  \end{equation}
  with $(x', y') = T^{N'}(x,y)$.

\subsubsection{The case $q_{n^{\ast}} < q_{n'}^{1 - \eta^2}$.}

In this case, $n^{\ast} < n'$ so that and we set $r := q_{n' - 1}$. Applying the approximation \eqref{eq:firstapx}
to the right-hand side of \eqref{eq:mainxx}, we can further reduce $p_{q} - N'_{q}$ modulo $r$.
As a result, it's enough to understand on average the behavior of
$$
\sum_{a < r} e_{b,c}(T^{a}(x', y')) \sum_{\substack{p \in [N' , N' + H] \\ p_{q} \equiv a + N'_{q} \pmod{r}}} \log p.
$$
In particular, on each interval $[N', N' + H]$, we have
\begin{equation} \label{eq:maxin}
\sum_{\substack{p \in [N', N' + H] \\ p_{q} \equiv a \pmod{r}}} \log p \sim \frac{H}{r}
\end{equation}
and we conclude by the unique ergodicity of $T$ that for all sufficiently large $N$,
$$
\sum_{p \in [N', N' + H]} e_{b,c}(T^{p}(x,y)) \log p \ll \eta^{1/2} H.
$$
Therefore, it is enough to show that the majority of intervals of length $H$ have the property \eqref{eq:maxin}.
This in turn follows (by Chebyschev's inequality) once we can show that for any given $\varepsilon > 0$ and $H \leq N$,
$$
\sum_{x < N} \sum_{a < r} \Big | \sum_{\substack{p \in [x, x + H] \\ p_{q} \equiv a \pmod{r}}} \log p - \frac{H}{r} \Big | = o(H N)
$$
uniformly in $1 \leq r < q^{1 - \varepsilon}$ and $H / q > N^{1/6 + \varepsilon}$ as $N \rightarrow \infty$.
The proof of this estimate is a slightly more general variant of the estimate according to which
$$
\sum_{a < r} \Big | \sum_{\substack{p  \leq N \\ p_{q} \equiv a \pmod{r}}} \log p - \frac{N}{r} \Big | = o(N)
$$
uniformly $r \leq q^{1 - \varepsilon}$ and $q \leq N^{5/6 - \varepsilon}$ as $N \rightarrow \infty$. Since we discussed the proof
of this earlier, we omit the discussion of the proof of the variant as it is similar.

\subsubsection{The case $q_{n^{\ast}} \geq q_{n'}^{1 - \eta^2}$.}

In this case, we rewrite right hand side of \eqref{eq:mainxx} by splitting $p_{q}$ into
short intervals of length $H' := q^{1/3}$. Note that $q^{1/3} \leq q_{n^{\ast}}^{1/2 - \eta^2}$
since $q \leq N$ and $q_{n^{\ast}} > N^{2/3}$ for $\eta$ sufficiently small. As a result, on each such sum we can
apply the approximation \eqref{eq:secondapx} getting that with $r := q_{n^{\ast} - 1} \leq (\log N)^3$,
$$
\sum_{\substack{p \in [N', N' + H] \\ p_{q} \in [z, z + H']}} e_{b,c}(T^{p_{q} - N'_{q}}(x', y'))
\approx \sum_{a < r} e_{b,c}(T^{a}(x', y')) \sum_{\substack{p \in [N', N' + H] \\ p_{q} \in [z, z + H'] \\ p_{q} \equiv a + N'_{q} \pmod{r}}} e(c (p_{q} - N'_{q}) \beta_{N', z, a})
$$
for some $\beta := \beta_{N', z, a}$ depending on $N', z$ and $a$ and such that $|\beta| \leq \eta^4$ for all sufficiently
large $N$. Once we can show that the sum over $p$ is for most $z,a$ and $N'$ independent of $a$, we can conclude using unique ergodicity on the sum over $a$. Thus, it suffices to show that for the majority of $N', z$ and $v$, we have
$$
\sum_{\substack{p \in [N', N' + H] \\ p_{q} \in [z, z + H'] \\ p_{q} \equiv v + N'_{q} \pmod{r}}} e(p_{q} \beta_{N', z, v}) = \frac{H}{\varphi(q)} \sum_{\substack{(a,q) = 1 \\ a \equiv v \pmod{r} \\ a \in [z, z + H']}} e(a \beta_{N', z, v}) + O(\eta^{1/2} H).
$$
In order to establish this it suffices to show that
$$
\sum_{y < x} \sum_{z < q} \sup_{\substack{\beta \in \mathbb{R} \\ 0 \leq v < r}} \Big | \sum_{\substack{p \in [y, y + H] \\ p_{q} \in [z, z + H'] \\ p_{q} \equiv v \pmod{r}}} e(p_{q} \beta) - \frac{H}{\varphi(q)} \sum_{\substack{(a,q) = 1 \\ a \equiv v \pmod{r} \\ a \in [z, z + H']}} e(a \beta) \Big | = o \Big ( \frac{x H H'}{r}  \Big )
$$
as $N \rightarrow \infty$.
Let us now describe some of the ideas that go into this.
We express the condition $p \in [y, y + H]$ using a contour integral and capture the behavior of the $p_{q}$ using Dirichlet characters. In this way, the problem reduces to obtaining bounds for
$$
\sum_{y < x} \sum_{z < q} \sup_{\substack{\beta \in \mathbb{R}\\ 0 \leq v < r}} \Big | \frac{1}{\varphi(q)} \sum_{\chi \neq \chi_{0}\pmod{q}} \Big ( \frac{H}{\sqrt{N}} \int_{|t| \leq N / H} P(\tfrac 12 + it, \chi) y^{it} dt \Big ) \cdot \Big ( \sum_{\substack{(a,q) = 1 \\ a \equiv v \pmod{r} \\ a \in [z, z + H']}} \chi(a) e(a \beta) \Big ) \Big |.
$$
We notice that expressing the condition $ a\equiv v \pmod{r}$ in terms of additive characters and using the triangle inequality, we can remove the condition $a \equiv v \pmod{r}$ and simply take the supremum over $\beta$ instead of a supremum over $\beta$ and $0 \leq v < r$. Furthermore, the $\beta$ now depends only on $y$ and $z$ and thus we can re-write the above as
$$
\sum_{y < x} \sum_{z < q} \Big | \frac{1}{\varphi(q)} \sum_{\chi \neq \chi_{0}\pmod{q}} \Big ( \frac{H}{\sqrt{N}} \int_{|t| \leq N / H} P(\tfrac 12 + it, \chi) y^{it} dt \Big ) \cdot \Big ( \sum_{\substack{(a,q) = 1 \\ a \in [z, z + H']}} \chi(a) e(a \beta_{y,z}) \Big ) \Big |
$$
for some $\beta_{y,z}$ depending on $y$ and $z$. Finally, using duality, we can express the above as
$$
\sum_{y < x} \sum_{z < q} e^{i \theta_{y,z}} \cdot \frac{1}{\varphi(q)} \sum_{\chi \neq \chi_{0} \pmod{q}} \Big ( \frac{H}{\sqrt{N}} \int_{|t| \leq N / H} P(\tfrac 12 + it, \chi) y^{it} dt \Big ) \cdot \Big ( \sum_{\substack{(a,q) = 1 \\ a \in [z, z + H']}} \chi(a) e(a \beta_{y,z}) \Big )
$$
for some $\theta_{y,z} \in \mathbb{R}$ and
where
$$
P(\tfrac 12 + it, \chi) = \sum_{p \leq N} \frac{\chi(p) \log p}{p^{1/2 + it}}.
$$
We can now proceed in the same way as before separating the tuples $(t, \chi)$ into those which are bad, that is, $P(\tfrac 12 + it, \chi)$ exhibits no cancellations and those which are good, that is, $P(\tfrac 12 + it, \chi)$ is non-trivially small. There are few bad tuples $(t,\chi)$ and in order to control their contribution one needs a non-trivial bound for
$$
\sum_{z < q} \Big | \sum_{\substack{(a,q) = 1 \\ a \in [z, z + H']}} \chi(a) e(a \beta_{y,z}) \Big |.
$$
In order to achieve this, one can use Weyl differencing to eliminate $e(a \beta_{y,z})$ at the cost of now having to estimate a character sum of $\chi(a) \overline{\chi}(a + h)$ on average over $a \in [z, z + H']$. However, this can be accomplished by using the Weyl bound for character sums involving $\chi(a) \overline{\chi}(a + h) \chi(a') \overline{\chi}(a' + h)$.
It remains to show that the contribution of the ``good'' $(t,\chi)$ is acceptable. Here, we use duality to re-write the sum as
$$
\frac{1}{\varphi(q)} \sum_{\substack{\chi \neq \chi_{0} \pmod{q} \\ y < x}} \Big ( \frac{H}{\sqrt{N}} \int_{\substack{|t| \leq N / H \\ (t,\chi) \text{ good}}} P(\tfrac 12 + it, \chi) y^{it} dt \Big ) \cdot \Big ( \sum_{a < q} \chi(a) c(a, y) \Big )
$$
with
$$
c(a,y) := \sum_{\substack{1 \leq z \leq q \\ z \in [a - H', a]}} e^{i \theta_{y,z}} e(a \beta_{y,z}).
$$
We now apply the Cauchy-Schwarz inequality in $(y, \chi)$. Then,
$$
\sum_{y < x} \frac{1}{\varphi(q)} \sum_{\chi} \Big | \sum_{a < q} \chi(a) c(a,y) \Big |^2
$$
is evaluated using the large sieve and the trivial bound $|c(a, y)| \leq H'$. On the other hand, we evaluate
$$
\sum_{y < x} \frac{1}{\varphi(q)} \sum_{\chi \neq \chi_{0}} \Big | \int_{\substack{|t| \leq N / H \\ (t, \chi) \text{ good}}} P(\tfrac 12 + it, \chi) y^{it} dt \Big |^2
$$
by using using orthogonality in $y^{it}$. This reduces the problem to showing that
$$
\sum_{\chi \neq \chi_{0}} \int_{\substack{|t| \leq N / H \\ (t, \chi) \text{ good}}} |P(\tfrac 12 + it, \chi)|^2 dt \ll_{A} \frac{N}{(\log N)^{A}}
$$
for some large $A > 0$, and this can be seen as equivalent to obtaining a hybrid version of Huxley's theorem (in short arithmetic progression and large moduli). Once again the fact that we restrict to $(t, \chi)$ which are good is crucial since it allows us to act as if we had an enlarged zero-free region for $L(s, \chi)$.

\subsection{Extending to the case of general $\alpha$}
Our strategy is to relate sums over primes to sums over reduced residues modulo $q_k$, as in  \eqref{eq:beha}. When the modulus $q_k$ of the reduced residues has few prime factors the sum over reduced residues is easy to estimate using the unique ergodicity of $T$. However, this fails if $q_k$ is ``very'' composite. In fact, we don't know how to deal with the sums on the right-hand side of \eqref{eq:beha}. Instead, for a given $q_k$, we show that there exists a prime number $p_k\in [\log^2q_k, 2\log^2q_k]$\footnote{To be more precise, any prime number $p_k \in [\log^2 q_k, 2 \log^2 q_{k}]$ with $(p_k, q_{k - 1}) = 1$ will work. There always exists at least one such prime number.} such that
\begin{equation} \label{eq:123asd}
\lim_{k\to+\infty} \min_{z_k\in \{q_k,p_kq_k\}} \sup_{(x,y)\in \T^2}\frac{1}{\varphi(z_k)}\Big|\sum_{\substack{i\leq z_k\\(i,z_k)=1}}e_{b,c}(T^i(x,y))\Big|=0.
\end{equation}
Given $k$, let $z_k$ be the integer in $\{q_k, p_k q_k\}$ that minimizes the sum in \eqref{eq:123asd}. We modify our earlier argument so as to relate at every turn the sum over primes in \eqref{eq:beha} to the sum over reduced residues $\pmod{z_k}$ instead of reduced residues $\pmod{q_k}$. This is possible because the approximations \eqref{eq:firstapx} and \eqref{eq:secondapx} remain valid if we replace the modulus $q_k$ by $w q_k$ (for all $w \leq \log^3 q_{k}$ simultaneously).

In fact, we establish a stronger version of \eqref{eq:123asd} showing that the convergence to zero holds uniformly over all divisors  of $z_k$:
\begin{equation}\label{eq:3tf}
\lim_{k\to+\infty}  \min_{z_k\in \{q_k,p_kq_k\}}\max_{d|z_k} \sup_{(x,y)\in \T^2} \ \frac{d}{\varphi(d)} \cdot \frac{1}{z_k} \Big|\sum_{\substack{i\leq z_k\\(i,d)=1}}e_{b,c}(T^i(x,y))\Big|=0.
\end{equation}
The proof and the choice of $z_k$ splits into several cases based on  the relations between $q_{k-1}, q_k$ and $q_{k^\ast}$.

\subsection{The case $k^\ast=k$.} In this case, we take $z_k=q_k$.  We split the interval $[0,q_k]$ into intervals $I$ of length $q_k^{1/2-\epsilon}$.
Then, by \eqref{eq:secondapx}, for every $n \in I$ with $n \equiv a \pmod{q_{k - 1}}$, we have $e_{b,c}(T^{n}(x,y)) \approx e_{b,c}(T^{a}(x_{I},y_{I})) e(c \beta_I (n - z_I))$ with $x_{I}, y_{I}$, $|\beta_I| \leq e^{-q_{k - 1} / 16}$ and $z_I$ depending on the interval $I$. As a result, for every $d | q_{k}$, we have
\begin{equation} \label{eq:asdasdasd}
\sum_{\substack{m \in I \\ (m,d) = 1}} e_{b,c}(T^{m}(x,y)) \approx e(-c z_{I}) \sum_{a \leq q_{k - 1}} e_{b,c}(T^{a}(x_I, y_I)) \sum_{\substack{m \in I \\ (m,d) = 1 \\ m \equiv a \pmod{q_{k - 1}}}} e(c \beta_I m).
\end{equation}
Since $|I| > q_k^{1/5}$, $|\beta_I| \leq e^{-q_{k - 1} / 16}$, $(q_{k - 1}, q_k) = 1$ (because consecutive convergents are co-prime) and $q_{k -1} \leq \log^{100} q_k$, we can show using some simple sieves that
$$
\sup_{|\beta| \leq e^{-q_{k - 1} / 16}} \Big | \sum_{\substack{m \in I \\ (m, d) = 1 \\ m \equiv a \pmod{q_{k - 1}}}} e(m \beta) - \frac{1}{\varphi(q_{k - 1})} \sum_{\substack{m \in I \\ (m, d q_{k - 1}) = 1}} e(m \beta) \Big | \ll \frac{\varphi(d)}{d} \cdot \frac{|I|}{e^{q_{k - 1} / 1000}} .
$$
It is crucial for the validity of this estimate that the supremum over $\beta$ is restricted to small $\beta$. As a result of this estimate, we can re-write the sum on the left-hand side of \eqref{eq:asdasdasd} as
$$
\frac{1}{\varphi(q_{k - 1})} \sum_{\substack{m \in I \\ (m, d q_{k - 1}) = 1}} e(c m \beta_I) \sum_{\substack{a \leq q_{k - 1}}} e_{b,c}(T^a(x_{I}, y_{I}) + O \Big ( \frac{1}{q_{k - 1}} \cdot \frac{\varphi(d)}{d} |I| \Big ).
  $$
  The claim now follows from the unique ergodicity of $T$, because this shows that the sum over $a$ exhibits cancellations.

\subsection{The case $k^\ast<k$.}  Let $a(m)$ denote the $m$th Fourier coefficients of $g$.
Further, let $w(k)=\log \log k$. Let $p_k$ be a prime number in $[\log^2 q_k,2\log^2q_k]$ co-prime to $q_{k - 1}$. We now define $z_k$ as follows:
\begin{enumerate}
\item[Z1.] $z_k:=q_k$ if $q_{k-1}\leq \frac{q_k}{w(k)\log\log q_k}$;
\item[Z2.] $z_k:=p_kq_k$ if $q_{k-1}\geq \frac{q_k}{w(k)\log\log q_k}$ and $q_{k^\ast}\leq \frac{q_{k-1}}{16\log^2 q_k}$;
\item[Z3.] $z_k:=p_kq_k$ if $q_{k-1}\!\geq\! \frac{q_k}{w(k)\log\log q_k}$, $q_{k^\ast}\!\geq\! \frac{q_{k-1}}{16\log^2 q_k}$ and
$
\max_{\substack{|m|\in [q_{k^\ast-1},\log q_{k^\ast}]\\q_{k^\ast-1}|m}}|a(m)|\leq \frac{1}{\log^4 q_k};
$
\item[Z4.] $z_k:=q_k$  if $q_{k-1}\geq \frac{q_k}{w(k)\log\log q_k}$, $q_{k^\ast}\geq \frac{q_{k-1}}{16\log^2 q_k}$ and
\begin{equation}\label{eq:sdk}
\max_{\substack{|m|\in [q_{k^\ast-1},\log q_{k^\ast}]\\q_{k^\ast-1}|m}}|a(m)|\geq \frac{1}{\log^4 q_k}.
\end{equation}
\end{enumerate}
The treatment of cases Z1, Z2 and Z3 is analogous, whereas Z4 uses different methods.
\subsection{Cases Z1, Z2 and Z3} As usual, let $d | z_{k}$. We split the sum
$$
\sum_{\substack{i\leq z_k\\(i,d)=1}}e_{b,c}(T^i(x,y))
$$
into residue classes $\mod q_{k-1}$. In all the cases Z1, Z2 and Z3 it follows that if $\ell\leq z_k$, $\ell\equiv a \mod q_{k-1}$, then $e_{b,c}(T^\ell(x,y)) \approx e_{b,c}(T^a(x,y))$. This implication is not immediate, particularly in the case Z3, but for simplicity we skip the details. We get that the above sum is approximated by
\begin{equation} \label{eq:eqw}
\sum_{a\leq q_{k-1}}e_{b,c}(T^a(x,y)) \Big ( \sum_{\substack{i\leq z_k\\(i,d)=1\\i\equiv a\mod q_{k-1}}}1 \Big )
\end{equation}
Moreover, by the definition of the $z_k$, for every $d|z_k$,
 it follows that in all the cases
\begin{equation} \label{eq:toinf}
\frac{\varphi(d)z_k}{q_{k-1}d}\to+\infty.
\end{equation}
Indeed, in the case Z1 this follows from $q_{k-1}\leq q_k / (w(k) \log\log q_k)$ and in the cases Z2 and Z3, we use that $p_k\geq \log^2q_k$ and $q_k\geq q_{k-1}$ to ensure that \eqref{eq:toinf} holds. We then show, using sieve-methods and by establishing a $q$-analogue of a result of Friedlander \cite[Section 6.10]{Opera}, that if \eqref{eq:toinf} holds, then
$$
\sum_{a\leq q_{k-1}}\Big|\sum_{\substack{i\leq z_k\\(i,d)=1\\i\equiv a\mod q_{k-1}}}1-\frac{\varphi(d)z_k}{q_{k-1}d}\Big|=o \Big ( \frac{\varphi(d)z_k}{d} \Big ).
$$
Therefore, \eqref{eq:eqw} is equal to
$$
\frac{\varphi(d)z_k}{q_{k-1}d}\sum_{a\leq q_{k-1}}e_{b,c}(T^a(x,y)) + o \Big ( \frac{\varphi(d)z_k}{d} \Big )
$$
and the claim follows from unique ergodicity applied to the sum over $a$.

\subsection{The Case Z4} It follows from \eqref{eq:sdk} that we have $e^{-q_{k^\ast-1}}\geq e^{-m} \geq |a(m)|\geq \log^{-4}q_k$ for $m$ divisible by $q_{k^{\ast} - 1}$ and belonging to $[q_{k^{\ast} - 1}, \log q_{k^{\ast}}]$. In particular,
$$
q_{k^{\ast}-1}\leq [\log\log q_k]^2.
$$
We will show that for $H=q_{k^{\ast}}^{1/2-\epsilon}\geq q_k^{1/2-2\epsilon}$ (the inequality follows from the assumptions of this case), we have
$$
\sum_{u<q_k}\Big|\sum_{\substack{\ell\in[u,u+H]\\ (\ell,d)=1}}e_{b,c}(T^{\ell}(x,y))\Big|=o \Big ( \frac{q_k H \varphi(d)}{d} \Big ).
$$
This will then imply that \eqref{eq:3tf} holds (by splitting into  disjoint intervals of length $H$ and summing over them). If $\ell \in[u,u+H]$, $\ell \equiv a+u \mod q_{k^\ast-1}$, then for $(x_u,y_u)=T^u(x,y)$,
$$
e_{b,c}(T^\ell(x,y)) \approx e_{b,c}(T^a(x_u,y_u)) e(c (\ell-u)\beta_u)),
$$
where $\beta_u=g_{k^\ast-1}(x+u\alpha)$ and $g_{k}(x)=\sum_{\substack{m\in [q_k,\log q_{k+1}]\\q_k|m}}a(m) e(m x)$.
Thus,
$$
\sum_{\substack{\ell\in[u,u+H]\\ (\ell,d)=1}}e_{b,c}(T^\ell(x,y))
$$
is approximately
$$
\sum_{a\leq q_{k^\ast-1}}e_{b,c}(T^a(x_u,y_u))\sum_{\substack{\ell\in[u,u+H]\\ (\ell,d)=1\\ \ell\equiv a\mod q_{k^\ast-1}}}e_{c}((\ell-u)\beta_u).
$$
If $\beta_u\geq q_k^{-\epsilon}$, then since $H>q_k^{100\epsilon}$, using some simple sieve estimates, we can show that
$$
\sum_{\substack{\ell\in[u,u+H]\\ (\ell,d)=1\\ \ell\equiv a\mod q_{k^\ast-1}}}e_{c}((\ell-u)\beta_u)=O \Big ( \frac{\varphi(d)H}{dq_{k^\ast-1}^2} \Big ) .
$$
Summing over $a\leq q_{k^\ast-1}$ gives then
$$
\sum_{\substack{\ell\in[u,u+H]\\ (\ell,d)=1}}e_{b,c}(T^\ell(x,y))=O \Big ( \frac{\varphi(d)H}{q_{k^{\ast} - 1} d} \Big),
$$
which is enough since $q_{k^{\ast} - 1} \rightarrow \infty$ as $k \rightarrow \infty$. Therefore, the problem reduces to showing that
$$
\{u\leq q_k\;:\; |\beta_u|\leq q_k^{-\epsilon}\}=o(q_k),
$$
which, by the definition of $\beta_u$, is equivalent to
\be\label{eq:za12}
\{u\leq q_k\;:\; |g_{k^{\ast}-1}(x+u\alpha)|\leq q_k^{-\epsilon}\}=o(q_k),
\ee
uniformly over $x\in \T$. In order to show \eqref{eq:za12}, we will use our assumption that
$$
\sup_{x\in \T}|g_{k^\ast-1}(x)|\geq \frac{1}{\log^4q_k},
$$
which follows from \eqref{eq:sdk}. Since $g_{k^{\ast}-1}$ is a trigonometric polynomial of degree $o(\log q_k)$, it follows from a theorem of Nazarov \cite[Theorem 1.1]{Naz} that as $k \rightarrow \infty$,
$$
Leb(E)=o(1), \text{ where } E:=\{x\in \T\;:\; |g_{k^\ast-1}(x)| \leq q_k^{-\epsilon/2}\}.
$$
Because of the rapid decay of the Fourier coefficients of $g$, we have
$$\sup_{x \in \mathbb{T}} |g'_{k^{\ast} -1}(x)| \rightarrow 0$$ uniformly in $x \in \mathbb{T}$.
It follows that if $I \subset \mathbb{T}$ is an interval of length $\approx q_k^{-\epsilon/2}$ such that $I\cap E^c\neq \emptyset$, then, for all $x \in I$, we have $|g_{k^\ast-1}(x)|\geq q_k^{-\epsilon}$. Let $\{I_i\}_{i \leq w}$ be a covering of $\mathbb{T}$ with intervals of length $\approx q_{k}^{-\varepsilon / 2}$. Since $\text{Leb}(E) = o(1)$ as $k \rightarrow \infty$ for all but at most $o(w)$ indices $i \leq w$, we have $I_{i} \cap E^{c} \neq \emptyset$ and therefore, for such $i$'s for all $x \in I_i$, we have $|g_{k^{\ast} - 1}(x)| \geq q_k^{-\varepsilon}$. Since for each $i$ the number of $m \leq q_k$ such that $\{x + m \alpha \} \in I_i$ is by Denjoy-Koksma inequality equal to $q_k |I_i| + O(1)$, we conclude that the cardinality of the set \eqref{eq:za12} is bounded from above by
$$
\sum_{\substack{i \leq w \\ I_i \cap E^c = \emptyset}} \sum_{\substack{m \leq q_k \\ \{ x + m \alpha \} \in I_i }} 1 \ll q_k \sum_{\substack{i \leq w \\ I_i \cap E^c = \emptyset}} |I_i| = o(q_k).
$$

\subsection*{Plan of the paper}

The paper splits into two parts. In the first part of the paper we establish an ergodic theorem along reduced residue classes, which is required for the proof of our main result Theorem \ref{thm:main}. Specifically, in Section \ref{sec:propan} we establish important properties of analytic cocycles. In Section \ref{sec:principal1} and \ref{sec:principal2} we collect a few number theoretic results on the distribution of reduced residues to large moduli and twisted by additive phases. In Section \ref{sec:ergchar} we establish the main result of this part of the paper, namely that for uniquely ergodic skew products, ergodic sums weighted by principal characters converge.

In the second part of the paper we focus on the proof of our main result Theorem \ref{thm:main}. We start by stating several crucial results on the equidistribution of primes to high moduli and in short arithmetic progressions in Section \ref{sec:eqprimes1}, and several results on exponential sums over primes with polynomial phases in Section \ref{sec:eqprimes2}. In Section \ref{sec:thmmain} we use results from Sections \ref{sec:eqprimes1}, \ref{sec:eqprimes2}, \ref{sec:propan} and \ref{sec:ergchar} to prove Theorem \ref{thm:main}.
Finally, in Section \ref{sec:count} we prove Theorem \ref{thm:main4}. We include below a detailed table of contents.

\subsection*{}

\par
\tableofcontents

\subsection*{Notation} We will denote by $d_{r}(n)$ the $r$-fold divisor function, so that
$$
d_{r}(n) := \sum_{n = n_1 \ldots n_r} 1
$$
and, in particular, $d(n) := d_{2}(n)$. The von Mangoldt function $\Lambda(n)$ is defined as $\log p$ when $n = p^{\alpha}$ with $p$ prime and $\alpha > 0$ and is defined as zero on all the remaining integers. The symbol $\overline{e}$ will denote the modular inverse of $e$ to an appropriate modulus which will typically be clear from the context.

The symbol $f(x) \ll g(x)$ will mean that there exists an absolute constant $C > 0$ such that $|f(x)| \leq C |g(x)|$ for all $x$ in the domain of definition of $f$ and $g$. For instance, if $f,g$ are sequences then this bound will be valid for all positive integers $x$.
When used in a subscript of a sum or integral, the notation $n \sim A$ means that $A \leq n < 2 A$.

The Fourier transform of $f : \mathbb{R} \rightarrow \mathbb{R}$ is defined as
$$
\widehat{f}(x) := \int_{\mathbb{R}} f(u) e(- x u) du.
$$
The Mellin transform of $f : [0, \infty) \rightarrow \mathbb{R}$ is defined as
$$
\widetilde{f}(s) := \int_{0}^{\infty} f(x) x^{s - 1} dx.
$$

Given a real number $x$, we let $\| x \| := \min_{n \in \mathbb{Z}} |x - n|$.

\part{Ergodic theorem along reduced residue classes}

\section{Properties of analytic cocycles}\label{sec:propan}
Fix $\alpha\in \T$ with the sequence of denominators $\{q_n\}$. For $h\in C(\T)$ and $\alpha\in \T$ we use the following notation:
$$
S_{n}(h)(x)=\sum_{j=0}^{n-1}h(x+j\alpha)
$$
for all $n\in\N$. Observe that the \textit{cocycle identity} $S_{n+m}(h)(x)=S_n(h)(x)+S_m(h)(x+n\alpha)$ holds.

Let $g$ be a $1$-periodic real-analytic function of zero mean. Expanding in a Fourier series, we can write
$g(x)=\sum_{m\in \Z}a_me_m(x)$,
where $e_m(x) := e(m x)$ (with $a_0=0$).
Since $g$ is real analytic and $1$-periodic, its Fourier coefficients are decreasing to zero exponentially fast, so without loss of generality, we can assume that for all $m\in\Z$,
 \begin{equation}\label{ujwzrost}|a_m|\leq e^{-\tau' |m|}\text{ with }\tau'<\frac1{10}.\end{equation}

We start with the following lemma:
\begin{lemma}\label{lem:gothed} If $g$ is not a continuous coboundary (i.e. if there is no continuous solution $\xi:\T\to\mathbb{S}^1$ to $e(g)(x)=\xi(x)/\xi(x+\alpha)$ for all $x\in\T$ and where $e(g)(x) := e^{2\pi i g(x)}$), then there exists a subsequence $\{q_{n_k}\}$ such that
$q_{n_k+1}\geq e^{\frac{\tau'q_{n_k}}{2}}$ for all $k\geq1$.
\end{lemma}
\begin{proof}We will show that if such a subsequence does not exist, then $g$ is a continuous coboundary, i.e. assume that for some $C'>0$ and every $s\in \N$,
$$
q_{s+1}\leq C' e^{\frac{\tau'q_{s}}{2}}.
$$
Note first that for $m\in \Z$, if $s$ is unique such that $|m|\in [q_s,q_{s+1})$, then
\be\label{eq:mal}
\|m\alpha\|\geq \|q_s\alpha\|\geq \frac{1}{2q_{s+1}}\geq \frac{1}{2C'}e^{-\frac{\tau'q_{s}}{2}}\geq \frac{1}{2C'}e^{-\frac{\tau'|m|}{2}}.
\ee

 By the Gottschalk-Hedlund theorem, it is enough to show that there exists $C>0$ such that for every $k\in \N$,
$$
|S_k(g)(0)|<C.
$$
Notice that for every $x\in \T$,
$$
S_k(e_m)(x)=e_m(x)\frac{e_m(k\alpha)-1}{e_m(\alpha)-1},
$$
and therefore, $|S_k(e_m(x)|\leq \frac{4}{\|m\alpha\|}$. By the cocycle identity, the bound on $a_m$ and \eqref{eq:mal}, it follows that
\begin{multline*}
|S_k(g)(0)|=|\sum_{m\in \Z}a_m S_k(e_m)(0)|\leq 4\sum_{m\in\Z}e^{-\tau' |m|}\|m\alpha\|^{-1}\leq \\
8C' \sum_{m\in \Z}e^{-\tau' |m|/2}:=C<+\infty.
\end{multline*}
This finishes the proof.
\end{proof}

Next, we will show that the only important frequencies of $g$ come from multiples of denominators. Indeed, for $n\in \N$, let
\be\label{eq:gn}
g_n(x):=\sum_{\substack{|m|\in [q_n,\frac{\log q_{n+1}}{\tau'^2}]\\ q_n|m}}a_me_m(x)\;\;\;\;\;\text{  and  let  }\;\;\tilde{g}:=\sum_{n\in \N} g_n.
\ee

\begin{lemma}\label{lem:tilg}
The function $g-\tilde{g}$ is a continuous coboundary for $\alpha$.
\end{lemma}
\begin{proof} By the Gottschalk-Hedlund theorem the assertion is equivalent to showing that there exists $C>0$ such that for every $k\in \N$,
$$
|S_k(g-\sum_{n\in \N}g_n)(0)|<C.
$$
We have $|S_k(e_m)(0)|\leq \frac{4}{\|m\alpha\|}$ and therefore,
\be\label{eq:abu}
|a_m S_k(e_m)(0)|\leq e^{-\tau' |m|}\frac{4}{\|m\alpha\|}.
\ee
Let $n$ be unique such that $|m|\in [q_n,q_{n+1})$. Then either $q_n|m$ or
\begin{equation}\label{uzu1}
\|m\alpha\|\geq \frac1{6|m|}.\end{equation}
Indeed, if
$|m|=sq_n+r$ with $1\leq r<q_n$ and $s\leq q_{n+1}/q_n$ then either
\begin{itemize}
\item $s\leq \frac{q_{n+1}}{3q_n}$ and then $\|r\alpha\|\geq\frac1{2q_n}$ while $\|sq_n\alpha\|<\frac1{3q_n}$, so $\|m\alpha\|\geq\frac1{6q_n}\geq \frac1{6|m|}$ or
\item $|m|>q_{n+1}/3$ and $\|m\alpha\|\geq\|q_n\alpha\|\geq\frac1{2q_{n+1}}\geq \frac1{6|m|}$
\end{itemize}
and \eqref{uzu1} follows.  Using~\eqref{uzu1} and \eqref{eq:abu}, we obtain
$$
|a_m S_k(e_m)(0)|\ll \frac{1}{m^2}.
$$
Let now $|m|\in[\frac{\log q_{n+1}}{\tau'^2},q_{n+1})$. Then (using $\|m\alpha\|\geq \|q_n\alpha\|\geq 1/(2q_{n+1})$), we have
$$
|a_m S_k(e_m)(0)|\leq e^{-\tau' |m|} \frac{4}{\|m\alpha\|}\leq e^{-\tau' |m|}8q_{n+1}\leq$$$$
8e^{-\tau'\cdot \frac{\log q_{n+1}}{\tau'^2}}q_{n+1}=8\frac 1{q_{n+1}^{\frac1{\tau'}-1}}\leq 8\frac 1{|m|^{\frac1{\tau'}-1}}\ll \frac{1}{m^2}
$$
in view of \eqref{ujwzrost}.
Hence, using the definition of $g_n$, we obtain
$$
|S_k(g-\sum_{n\in \N}g_n)(0)|\ll \sum_{|m|>1}\frac{1}{m^2}<+\infty.
$$
This finishes the proof.
\end{proof}

We call $\widetilde{g}$ the \textit{reduced form of} $g$. By Lemma \ref{lem:tilg}, it follows that it is enough to consider the case $g=\tilde{g}$, i.e.\ when $g$ itself is \textit{reduced}.\footnote{When $g(x)=\widetilde{g}(x)+v(x)-v(x+\alpha)$ with $v:\T\to\R$ continuous, the map $(x,y)\mapsto (x,y+v(x)\text{ mod }1)$ establishes a topological isomorphism between $T_{\alpha, g}$ and $T_{\alpha,\widetilde{g}}$. More generally, if $g$ and $h$ are multiplicatively cohomologous with a continuous transfer function, then the skew products $T_{\alpha, g}$ and $T_{\alpha,h}$ are topologically isomorphic.}   {\bf We make this assumption for the rest of the paper}. Note that the functions $g_n$ are the same for $g$ and its reduced form.

Let $\tau=1/2\min(\tau'^2,\tau'/8)$. For $n\in \N$, let $n^{\ast}\leq n$ be the largest integer such that
\be\label{eq:n,ast}
q_{n^{\ast}}\geq e^{\tau q_{n^\ast-1}}
\ee
and where we set $q_{0} := 0$ so as to guarantee that $n^{\ast}$ always exists.
Notice that by Lemma \ref{lem:gothed},
\be\label{eq:n,ast2}
n^{\ast} \to +\infty\text{ as }n\to +\infty.
\ee
We have the following lemma:

\begin{lemma}\label{lem:an} Let $K_n:=\frac{q_{n+1}}{q_{n^{\ast}}}$. Then, for every $K\in \{1,\ldots, K_n\}$,
\be\label{eq:subz}
\sup_{x\in \T}|S_{Kq_n}(g-g_n)(x)|\to 0 \text{ as } n\to +\infty.
\ee
If additionally $K_n\leq e^{2\tau q_n}$ then, for every $K\in \{1,\ldots, K_n\}$,
\be\label{eq:subz2}
\sup_{x\in \T}|S_{Kq_n}(g)(x)|\to 0 \text{ as } n\to +\infty.
\ee
\end{lemma}
\begin{proof} By the cocycle identity, it follows that
$$
S_{Kq_n}(g-g_n)(x)=\sum_{i=0}^{K-1}S_{q_n}(g-g_n)(x+iq_n\alpha).
$$
The statement follows by showing that, uniformly in $x\in \T$,
$$
K_n|S_{q_n}(g-g_n)(x)|\to 0.
$$
Notice that $S_{q_n}(g-g_n)(x)=\sum_{\substack{m\in \Z\\ |m|\notin [q_n,q_{n+1}]}} a_mS_{q_n}(e_m)(x)$ and
$$
S_{q_n}(e_m)(x)=e_m(x)\frac{e_m(q_n\alpha)-1}{e_m(\alpha)-1}.
$$
Moreover,
$$
\Big|\frac{e_m(q_n\alpha)-1}{e_m(\alpha)-1}\Big|\ll \frac{\|mq_n\alpha\|}{\|m\alpha\|}\leq \min\Big(q_n,\frac{|m|\|q_n\alpha\|}{\|m\alpha\|}\Big).
$$
 Then
$$
K_n|a_mS_{q_n}(e_m)(x)|\leq K_n e^{-\tau' |m|}\min\Big(q_n,\frac{|m|\|q_n\alpha\|}{\|m\alpha\|}\Big).
$$
Let $k\in \N$ be unique such that $|m|\in [q_k,q_{k+1})$. We will separately consider the cases $k<n$ and $k>n$ (notice that $g-g_n$ has no frequencies which are multiplies of $q_n$). Assume first that $k<n$. Then, by \eqref{eq:gn}, $|m|=rq_k$ for some $r\leq \frac{q_{k+1}}{q_k}$ (in fact by \eqref{eq:gn}, $r\leq \frac{\log q_{k+1}}{\tau'^2q_k}$).
It follows that $\|m\alpha\|=r\|q_k\alpha\|\geq\frac{r}{2q_{k+1}}$, and so we get
$$
K_n|a_mS_{q_n}(e_m)(x)|\ll K_n e^{ -\tau' |m|}\frac{|m|\|q_n\alpha\|}{\|m\alpha\|}\leq\\
2K_n e^{ -\tau' |m|}\frac{|m| q_{k+1}}{q_{n+1}r}.
$$
Therefore,
\be\label{eq:sua1}
\sum_{\substack{|m|\in [q_k,q_{k+1}]\\q_k|m}}K_n|a_mS_{q_n}(e_m)(x)|\leq 2\frac{K_n q_{k+1}}{q_{n+1}}\sum_{q_k|m}|m|e^{-\tau'|m|}\leq  2\frac{K_n q_{k+1}}{q_{n+1}} e^{-\frac{\tau' q_k}{2}}.
\ee
So, we have
$$
K_n\sum_{\substack{|m|\in[q_k,q_{k+1}]\\ k<n}}a_mS_{q_n}(e_m)(x)\leq   \frac{2K_n}{q_{n+1}}\sum_{k\leq n^{\ast}-1}q_{k+1}e^{-\frac{\tau'q_k}{2}}+\frac{2K_n}{q_{n+1}}\sum_{n>k> n^{\ast}-1}q_{k+1}e^{-\frac{\tau'q_k}{2}}.
$$
By the definitions of $n^{\ast}$ (see \eqref{eq:n,ast}), $K_n$ and $\tau$, it follows that
$$
\frac{2K_n}{q_{n+1}}\sum_{n>k> n^{\ast}-1}q_{k+1}e^{-\frac{\tau'q_k}{2}}\leq \frac{2}{q_{n^{\ast}}}\sum_{n>k>n^{\ast}-1}e^{-\frac{\tau' q_k}{4}}=o(1)
$$
by \eqref{eq:n,ast2}. Moreover, using the definition of $K_n$ again, and noticing that \eqref{eq:n,ast} holds, so  $q_{n^\ast}$ is exponentially big with respect to $q_{n^\ast-1}$, whence exponentially big with respect to $n^\ast q_{n^\ast-1}$, we get
$$
\frac{2K_n}{q_{n+1}}\sum_{k\leq n^{\ast}-1}q_{k+1}
e^{-\frac{\tau'q_k}{2}}\leq \frac2{q_{n^\ast}}\Big(q_{n^{\ast}}e^{-\frac{\tau'q_{n^\ast-1}}{2}}+
\sum_{k<n^\ast-1}q_{k+1}\Big)\leq \frac{2}{q_{n^{\ast}}}o(q_{n^{\ast}})=o(1)
$$
(recall that the denominators themselves grow exponentially fast). Therefore,
\be\label{eq:smalk}
K_n\Big|\sum_{\substack{|m|\in[q_k,q_{k+1}]\\ k<n}}a_mS_{q_n}(e_m)(x)\Big|=o(1).
\ee
If $k>n$, then by the definition of $K_n$ and using $m\geq q_{n+1}$, we obtain
$$
K_n\Big|a_mS_{q_n}(e_m)(x)\Big|\leq K_nq_ne^{-\tau' |m|}\leq q_{n+1}q_ne^{-\tau' |m|}\leq e^{\frac{-\tau' |m|}{4}}.
$$
Therefore,
$$
K_n\Big|\sum_{\substack{|m|\in[q_k,q_{k+1}]\\ k>n}}a_mS_{q_n}(e_m)(x)\Big|=o(1).
$$
This and \eqref{eq:smalk} finish the proof of \eqref{eq:subz}.

To show \eqref{eq:subz2}, it remains to notice that by the bound on the Fourier coefficients, for every $K\leq e^{2\tau q_n}$, we have
$$
|S_{Kq_n}(g_n)(x)|\leq Kq_n\sup_{x\in\T}|g_n(x)|\leq Kq_ne^{-\frac{\tau'q_n}{2}}=o(1).
$$
Hence, \eqref{eq:subz2} follows by \eqref{eq:subz}. The proof is finished.
\end{proof}
Let $d((x,y),(x',y'))=\|x-x'\|+\|y-y'\|$ where $\| x \| := \min_{n \in \mathbb{Z}} |x - n|$.
From Lemma \ref{lem:an}, we deduce the following: \begin{lemma}\label{cor:jad2} For all $n\in \N$,   $m\leq  q_n\min(q_{n+1}/q_{n^\ast}, e^{2\tau q_n})$ and  $z\in \N$, we have
$$
d(T^m(x,y),T^{m \mod zq_n}(x,y))=o(1)
$$
uniformly in $(x,y)\in\T^2$.
\end{lemma}
\begin{proof}Notice that the statement follows by showing that for every $k\leq \min(K_n, e^{2\tau q_n})$,
$$
d\Big(T^{kq_n}(x,y),(x,y)\Big)=o(1)
$$
uniformly in $(x,y)$.
We have, $T^{kq_n}(x,y)=(x+kq_n\alpha,y+S_{kq_n}(g)(x))$. By Lemma \ref{lem:an}, $|S_{kq_n}(g)(x)|=o(1)$ (uniformly in $x$) and, by the bound on $k$ and the definition of $K_n$,
$$
\|kq_n\alpha\|\leq \frac{2k}{q_{n+1}}\leq \frac{2}{q_{n^\ast}}=o(1)
$$
by \eqref{eq:n,ast2}. This finishes the proof.
\end{proof}

The following proposition is crucial for further analysis.

\begin{proposition}\label{lem:toadd1} For every $\delta>0$ there exists $n_\delta\in \N$ such that for every $n\geq n_\delta$ we can find a function $P_n:\T\times \N\to \R$ such that
$$
P_n(\cdot,k)=\sum_{j=1}^{d(\delta)}a_j(\cdot)k^j,
$$
where $d(\delta)\leq \Big[\frac{1}{\delta}\Big]$, the functions $a_j:\T\to\R$ satisfy (see \eqref{eq:gn})
\be\label{eq:bocoe}
\sup_{x\in \T}|a_j(x)|\leq q_n^{-1}q_{n+1}^{-j+1}, \;\;\;\; a_1(x)=g_n(x),\;\;\;\text{ and }\;\;\sup_{x\in \T}|a_1(x)|\leq e^{-\tau q_n},
\ee
 and, uniformly for every $x\in \T$, $m\leq q_{n+1}^{1-\delta}$ and $w\leq \log^3 q_n$,
 \be\label{eq:polap}
 \Big|S_m(g)(x)-S_{m\mod wq_n}(g)(x)-P_n(x,m)\Big|\to 0 \;\text{ as }\;n\to+\infty.
 \ee
\end{proposition}
\begin{proof} Fix $n\in \N$ and let $m=\tilde{k}wq_n+a$, where $k:=\tilde{k}w\leq \frac{q_{n+1}^{1-\delta}}{q_n}$ and $0\leq a< wq_n$. Notice first that the mean of $g'$ is zero, hence, by the Denjoy-Koksma inequality, $\sup_{x\in\T}|S_{q_n}(g')(x)|=o(1)$. Using $a\leq wq_n$, $w\leq \log^3q_n$ and the cocycle identity (splitting into sums of length $q_n$), we get
 $$
 \sup_{a\leq wq_n}\sup_{x\in\T}|S_{a}(g')(x)|\leq \log^3 q_n+o(q_n)=o(q_n),
 $$
where the $o(q_n)$ terms comes from the last interval of length $\leq q_n$ and we use unique ergodicity to note that $\sup_{s\leq q_n}|S_s(g')(\cdot)|=o(q_n)$.
Therefore,
$$
|S_a(g)(x)-S_a(g)(x+kq_n\alpha)|\leq2\sup_{x\in \T} |S_a(g')(x)|\frac{q_{n+1}^{1-\delta}}{q_nq_{n+1}}=o(1).
$$
Hence, by the cocycle identity: $S_m(g)(x)=S_{kq_n}(g)(x)+S_a(x+kq_n\alpha)$, it is enough to show that for $k \leq \frac{q_{n+1}^{1-\delta}}{q_n}$, we have
$$
|S_{kq_n}(g)(x)-P_n(x,m)|=o(1),
$$
where $P_n(\cdot,\cdot)$ is as in the statement of the proposition. Notice that if we construct $P_n(\cdot,\cdot)$ satisfying the assertions of the proposition, then
$$
\sup_{x\in \T}|P_n(x,m)-P_n(x,kq_n)|=o(1).
$$
Indeed, by the mean value theorem for $j\leq d(\delta)$, $(kq_n+a)^j\leq (kq_n)^j+d(\delta)a(2kq_n)^{j-1}$, which implies
$$
\sum_{j=1}^{d(\delta)}a_j(\cdot)m^j= \sum_{j=1}^{d(\delta)}a_j(\cdot)(kq_n)^j+\sum_{j=1}^{d(\delta)}a_j(\cdot)O_\delta(a (kq_n)^{j-1}).
$$
Remebering that $a\leq q_n\log^3q_n$, for $j=1$, by \eqref{eq:bocoe}, $|a_1(\cdot) O_\delta(a)|=o(1)$, and by the same equation, for every $2\leq j\leq d(\delta)$,
\begin{multline*}
|a_j(\cdot)|O_\delta(a (kq_n)^{j-1})\leq O_\delta\Big(q_{n}^{-1}q_{n+1}^{-j+1}q_n(\log^3 q_n) m^{j-1}\Big)=
O_\delta\Big((\log^3 q_n)q_{n+1}^{-j+1+(1-\delta)(j-1)}\Big)=\\O_\delta((\log^3 q_n)q_{n+1}^{-\delta(j-1)})=o(1),
\end{multline*}
if $n$ is sufficiently large. Writing
$$
|S_m(g)(x)-S_a(g)(x)-P_n(x,m)|$$$$=|
S_{kq_n}(g)(x)-P_n(x,kq_n)+S_a(g)(x+kq_n\alpha)-
S_a(g)(x)+P_n(x,kq_n)-P_n(x,m)|,$$
we see that it is enough to construct $P_n(\cdot,\cdot)$ satisfying the assertions of the theorem and such that  (uniformly) for every $k\leq \frac{q_{n+1}^{1-\delta}}{q_n}$ and every $x\in \T$, we have
\be\label{eq:tysh}
|S_{kq_n}(g)(x)-P_n(x,kq_n)|=o(1).
\ee
Notice that by \eqref{eq:subz} in Lemma \ref{lem:an} for every $\ell \leq \frac{q_{n+1}^{1-\delta}}{q_n}$ (the latter number is $\leq \frac{q_{n+1}}{q_{n^\ast}}=K_n$),
$$
|S_{\ell q_n}(g-g_n)(x)|=o(1),
$$
and therefore, it is enough to show \eqref{eq:tysh} for $g=g_n$.
By the cocycle identity, we have
$$
S_{kq_n}(g_n)(x)=\sum_{i=0}^{k-1}S_{q_n}(g_n)(x+iq_n\alpha)=k S_{q_n}(g_n)(x)+\sum_{i=0}^{k-1}\Big(S_{q_n}(g_n)(x+iq_n\alpha)-S_{q_n}(g_n)(x)\Big).
$$
We now use Taylor expansion of $S_{q_n}(g)(\cdot)$ up to order $d=d(\delta)=[\frac{1}{\delta}]$:
$$
S_{q_n}(g_n)(x+iq_n\alpha)-S_{q_n}(g_n)(x)=\sum_{s=1}^{d-1}\frac{S_{q_n}(g_n^{(s)})(x)}{s!}
[i\|q_n\alpha\|]^s + \frac{S_{q_n}(g_n^{(d)})(\theta_i)}{d!}
[i\|q_n\alpha\|]^d.
$$
Summing over $i\in \{0,\ldots, k-1\}$, denoting $M(s):=\sum_{i=0}^{k-1} i^s$, we get
\begin{multline}\label{eq:krw}
S_{kq_n}(g_n)(x)=\\kS_{q_n}(g_n)(x)+\sum_{s=1}^{d-1}\frac{\|q_n\alpha\|^s S_{q_n}(g_n^{(s)})(x)}{s!}M(s)+\sum_{i=0}^{k-1}\frac{S_{q_n}(g_n^{(d)})(\theta_i)}{d!} [i\|q_n\alpha\|]^d.
\end{multline}
Notice that
\be\label{eq:lsd}
\Big|\sum_{i=0}^{k-1}\frac{S_{q_n}(g_n^{(d)})(\theta_i)}{d!} [i\|q_n\alpha\|]^d\Big|\leq M(d) \|q_n\alpha\|^d\cdot \sup_{x\in \T }|S_{q_n}(g_n^{(d)})(x)|.
\ee
We also have
\be\label{eq:msb}
M(s)=\sum_{i=0}^{k-1}i^s= \frac{1}{s+1}k^{s+1}+O(k^s),
\ee
so using $k\leq \frac{q_{n+1}^{1-\delta}}{q_n}$, we obtain
$$
M(d)\|q_n\alpha\|^d= O\Big(k^{d+1}q_{n+1}^{-d}\Big)=O\Big(q_n^{-d-1} q_{n+1}^{1-\delta d-\delta}\Big)=o(1),
$$
since $\frac1\delta\leq d+1$ (by the definition of $d$), whence $1-\delta d-\delta\leq 0$. Moreover, since $g_n^{(d)}$ is smooth (and has zero mean), it follows by the Denjoy-Koksma inequality that $|S_{q_n}(g_n^{(d)})(x)|=o(1)$ (uniformly in $x\in \T$) if $n$ is large enough. So the RHS of \eqref{eq:lsd} is $o(1)$. Plugging this into \eqref{eq:krw}, we get
\be\label{eq:msb2}
S_{kq_n}(g_n)(x)=kS_{q_n}(g_n)(x)+\sum_{s=1}^{d-1}\frac{\|q_n\alpha\|^s S_{q_n}(g_n^{(s)})(x)}{s!}M(s)+o(1).
\ee
Notice that for every $s\leq d-1$, using $k\leq \frac{q_{n+1}^{1-\delta}}{q_n}$, we obtain
$$
k^s\|q_n\alpha\|^s\leq q_{n+1}^{(1-\delta)s -s}q_n^{-s}=o(1),
$$
and again by the Denjoy-Koksma inequality, it follows that $|S_{q_n}(g_n^{(s)})(x)|=o(1)$ (uniformly in $x\in \T$) if $n$ is large enough. Therefore,
$$
O(k^s)\frac{\|q_n\alpha\|^s S_{q_n}(g_n^{(s)})(x)}{s!}=o(1).
$$
Therefore, using \eqref{eq:msb}, \eqref{eq:msb2} implies that
\begin{multline}\label{fin:eq}
S_{kq_n}(g_n)(x)=kS_{q_n}(g_n)(x)+\sum_{s=1}^{d-1}\frac{\|q_n\alpha\|^s S_{q_n}(g_n^{(s)})(x)}{(s+1)!}k^{s+1}+o(1)=\\kS_{q_n}(g_n)(x)+\sum_{s=1}^{d-1}\frac{\|q_n\alpha\|^s S_{q_n}(g_n^{(s)})(x)}{q_n^{s+1}(s+1)!}[kq_n]^{s+1}+o(1).
\end{multline}
Finally, notice that by the $1/q_n$-periodicity of $g_n$,
$$
|S_{q_n}(g_n)(x)-q_n g_n(x)|\leq \sum_{j=0}^{q_n-1}\Big|g_n(x+j\alpha)-g_n(x+j
\frac{p_n}{q_n})\Big|
=O(q_n\|q_n\alpha\|),
$$
whence
$$
|kS_{q_n}(g_n)(x)-kq_n g_n(x)|\leq \frac{q_{n+1}^{1-\delta}}{q_n}O( q_n\|q_n\alpha\|)=O(q_{n+1}^{-\delta})=o(1).$$
We define $a_1(x):=g_n(x)$, $a_s(x):=\frac{\|q_n\alpha\|^{s-1} S_{q_n}(g_n^{(s-1)})(x)}{q_n^{s}s!}$ (for $s=2,\ldots,d$) and $P_n(x,m):=\sum_{s=1}^da_s(x)m^s$. Then, by \eqref{fin:eq},
$$
S_{kq_n}(g_n)(x)=P_n(x,kq_{n})+o(1).
$$
It remains to bound the coefficients of $P_n(\cdot,\cdot)$. Notice that
$|a_1(x)|=|g_n(x)|\leq e^{-\tau q_n}$, by the bound on the Fourier coefficients of $g$. Moreover, for $s\geq 2$,
$$
|a_s(x)|\leq \frac{1}{q_n^sq_{n+1}^{s-1}}|S_{q_n}(g_n^{(s-1)})(x)|\leq \frac{1}{q_nq_{n+1}^{s-1}},
$$
since by the Denjoy-Koksma inequality, $|S_{q_n}(g_n^{(s)})(x)|=o(1)$.
This finishes the proof.
\end{proof}

Proposition \ref{lem:toadd1}  implies the following corollary:

\begin{corollary}\label{cor:jad} For every $\delta>0$ there exists $n_\delta\in \N$ such that for every $n\geq n_\delta$, every $m\leq q_{n+1}^{1-\delta}$ and every $w\leq \log^3 q_n$,  we have
$$
d\Big(T^m(x,y), T^{m\mod wq_n}(x,y+P_n(x,m))\Big)=o(1),
$$
uniformly over all $(x,y)\in \T^2$, where $P_n$ is the polynomial from Proposition~\ref{lem:toadd1} and $\deg P_n\leq [\frac{1}{\delta}]$.
\end{corollary}
\begin{proof} Recall that $T^m(x,y)=(x+m\alpha,y+S_m(g)(x))$. Then notice that since $w\leq \log^3 q_n$, $\|[m-(m\mod wq_n)]\alpha\|=o(1)$, by the bound on $m$.  It remains to use \eqref{eq:polap}.
\end{proof}

Finally, we state the following general property of (complex) polynomials (see \cite{Naz}).
\begin{theorem}\label{thm:naz}[\cite[Theorem 1.1]{Naz}] There exists a global constant $C>0$ such that if $p(x)=\sum_{k=1}^n a_ke(\lambda_k x)$, then for every Borel subset $E\subset I\subset \T$ ($I$ is an interval), we have
$$
Leb(E)\leq C Leb(I) \Big[\frac{\sup_{E}|p(t)|}{\sup_{I}|p(t)|}\Big]^{\frac{1}{n-1}}.
$$
\end{theorem}

\section{Simple sieve theoretic lemma}
In what follows, $d_k(n)$ denotes the $k$th divisor function $\sum_{n = n_1 \ldots n_k} 1$. In particular, $d(n) := d_2(n)$ denotes the number of $1 \leq d \leq n$ such that $d | n$.
\begin{lemma} \label{le:sieve}
  For $ d | q$ and any $A > 0$, we have
  $$
\mathbf{1}_{(n, d) = 1} = \sum_{\substack{e | n\\ e \leq z}} \lambda_e
  + O \Big ( d(n) \sum_{\substack{p | d \\ p | n \\ p > (\log q)^{A}}} 1 \Big ) + O \Big ( d(n) \mathbf{1}_{\omega(n) > (\log\log q)^2} \Big ),
  $$
  where $z = \exp(A (\log\log q)^3)$ and $\lambda_e$ is defined by setting $\lambda_e = \mu(e)$ when $p | e \implies p \leq (\log q)^{A}$ and $p | e \implies p | d$ and $\omega(e) \leq (\log\log q)^2$ and $\lambda_e = 0$ otherwise.
\end{lemma}
\begin{proof}

  If $n$ has more than $(\log\log q)^2$ prime factors or if $(n,d)$ has a prime divisor greater than $> (\log q)^{A}$ then the result trivially follows.
  Suppose therefore that $n$ has less than $(\log\log q)^2$ prime factors and that $p | (n,d) \implies p \leq (\log q)^{A}$.
  In that case, by the usual inclusion-exclusion,
  $$
  \mathbf{1}_{(n,d) = 1} = \sum_{\substack{e | n \\ p | e \implies p | d \\ p | e \implies p \leq \log^{A} q}} \mu(e).
  $$
  Since $n$ has at most $(\log\log q)^2$ prime factors, we can write the above as
  $$
  \sum_{\substack{e | n \\ p | e \implies p | d \\ p | e \implies p \leq \log^{A} q \\ \omega(e) \leq (\log\log q)^2 \\ e \leq z}} \mu(e),
  $$
  where the condition $e \leq z$ is implied by $p | e \implies p \leq \log^{A} q$ and $\omega(e) \leq (\log\log q)^2$.

\end{proof}

To handle the contribution of the divisor function in Lemma \ref{le:sieve}, we will need the following special case of a result of Shiu.
\begin{lemma}[Shiu's theorem] \label{le:shiu}
  Let $k, \ell > 0$ and $\varepsilon > 0$ be given. Then, for any $y > x^{\varepsilon}$, $q \leq y^{1 - \varepsilon}$ and $(a,q) = 1$,
  $$
  \sum_{\substack{x \leq n \leq x + y \\ n \equiv a \pmod{q}}} d_k(n)^{\ell} \ll_{\varepsilon} \frac{y}{\varphi(q)} \cdot (\log x)^{k^{\ell} - 1}.
  $$
\end{lemma}
\begin{proof}
  This follows from the main theorem of \cite{Shiu}.
\end{proof}
We will also need the following special case of Shiu's theorem when dealing with reduced residues in short intervals.

\begin{lemma} \label{le:shiureduced}
 Let $x \leq q$ and $y > q^{\varepsilon}$. Then,
  $$
  \sum_{\substack{x \leq n \leq x + y \\ (n,q) = 1}} 1 \ll_{\varepsilon} \frac{\varphi(q)}{q} \cdot y.
  $$

\end{lemma}
\begin{proof}
 By Shiu's theorem, the sum is bounded by$$
    y \cdot \exp \Big ( \sum_{\substack{p \leq y}} \frac{\mathbf{1}_{(p,q) = 1} - 1}{p} \Big ) \ll y \prod_{\substack{p \leq y \\ p | q}} \Big( 1 - \frac{1}{p} \Big ).
    $$
    It remains to notice that $\prod_{y < p < q} (1 - 1 / p)^{-1} \ll_{\varepsilon} 1$ to conclude.
\end{proof}

\begin{lemma} \label{le:trivialbounding}
  Let $A > 1000$ and $\varepsilon > 0$ be given. Then, for any $x \leq q$, $q^{\varepsilon} \leq y \leq q$, $r \leq y^{1 - \varepsilon}$ and $a \leq r$ and any $d | q$,
  $$
  \sum_{\substack{x \leq n \leq x + y \\ n \equiv a \pmod{r}}} d(n) \Big ( \sum_{\substack{p | d \\ p | n \\ p > (\log q)^{A}}} 1 + \mathbf{1}_{\omega(n) > (\log\log q)^2} \Big ) \ll \frac{y}{(\log q)^{A / 2}}.
  $$
\end{lemma}
\begin{proof}
  Let $f = (a, r)$. Then $n\equiv a \pmod r$ implies that $f|n$. Using the inequality $d(a b) \leq d(a) d(b)$, we can bound the above expression by
  \begin{equation} \label{eq:lemmaeasy}
  \leq d(f) \sum_{\substack{x / f \leq n \leq x / f + y / f \\ n \equiv a / f \pmod{r / f}}} d(n) \Big ( \sum_{\substack{p | d \\ p | n \\ p > (\log q)^{A}}} 1 + \sum_{\substack{p | d \\ p | f \\ p > (\log q)^{A}}} 1 + \mathbf{1}_{\omega(f n) > (\log\log q)^2} \Big ).
  \end{equation}
  If the middle term is non-zero then $f > (\log q)^{A}$. In that event, using that there are at most $\log q$ primes $p | d$ (since $d \leq q$), we  bound trivially
 $$
  \sum_{\substack{p | d \\ p | n \\ p > (\log q)^{A}}} 1 + \sum_{\substack{p | d \\ p | f \\ p > (\log q)^{A}}} 1 + \mathbf{1}_{\omega(f n) > (\log\log q)^2}\ll \log q.
 $$
 This and Lemma \ref{le:shiu} allow us to bound \eqref{eq:lemmaeasy} by (using $y\leq q$ and $d(f)\ll_{\varepsilon} f^{\varepsilon} \ll (\log q)^{A \varepsilon}$)
  $$
  \ll\frac{\log^2 q}{\varphi(r / f)} d(f) \cdot \frac{y}{f} \ll \frac{\log^2 q}{f^{3/4}} \cdot y \ll \frac{y}{(\log q)^{A / 2}}
  $$
  since $f\geq (\log q)^A$.
In the remaining cases, where the middle term in \eqref{eq:lemmaeasy} is zero, we can bound the contribution of the first term by
$$
\ll d(f)\sum_{\substack{p | d \\ p > (\log q)^{A}}}\sum_{\substack{x / fp \leq n \leq x / fp + y / fp \\ pn \equiv a / f \pmod{r / f}}}d(n)
,$$
which, by Lemma \ref{le:shiu}, leads to
  $$
  \ll d(f) \log q \sum_{\substack{p | d \\ p > (\log q)^{A}}} \Big ( \frac{y}{p f} + 1 \Big )\ll \frac{d(f)}{f} \log q \cdot \frac{y}{(\log q)^{A - 2}} + d(f) \log^2 q \ll \frac{y}{(\log q)^{A - 3}}
  $$
  because $d$ has at most $\log q$ prime divisors and $d(f)\ll_{\varepsilon} f^{\varepsilon}$. Finally, the contribution of the last term in \eqref{eq:lemmaeasy} can be handled by an application of the Cauchy-Schwarz inequality, Lemma \ref{le:shiu} and the fact that for any $A > 0$ and all $q$ sufficiently large,
  \begin{align*}
  \sum_{\substack{x/f \leq n \leq x/f + y/f \\ \omega(f n) > (\log\log q)^{2}}} 1 & \leq \sum_{\substack{x/f \leq n \leq x/f + y/f \\ \omega(f n) > A \log\log q}} 1 \\ & \leq 2^{-A \log\log q} \sum_{x/f \leq n \leq x/f + y/f} d(f) d(n) \ll \frac{d(f)}{f} \cdot y (\log q)^{1 - A \log 2},
  \end{align*}
  using the inequality $2^{\omega(n)} < d(n)$.
\end{proof}

\begin{lemma} \label{le:sieveprimes}
 Let $z \geq 1$ be given.
  There exist real coefficients $\lambda_d$ with $|\lambda_d| \leq 1$ such that
  $$
  \mathbf{1}_{\substack{n \text{ prime}}} \leq \sum_{\substack{d | n \\ d \leq z}} \lambda_d
  $$
  and for any $y > z^2$, $x \geq 1$,
  $$
  \sum_{n \in [x, x + y]} \Big ( \sum_{\substack{d | n \\ d \leq z}} \lambda_d \Big ) \ll \frac{y}{\log z}.
  $$
\end{lemma}
\begin{proof}
  This is a standard combinatorial sieve estimate, see e.g \cite{BrunHooleyFordHalberstam}.
\end{proof}

  We will also need the following simple result.

\begin{lemma} \label{le:sieve2}
  Let $A > 10$ and $d | q$ be given and let $\lambda_{e}$ be the same coefficients as in Lemma \ref{le:sieve}. Then, for any $v \leq z$ such that $p | v \implies p | d$ and $p | v \implies p \leq (\log q)^{A}$,
  $$
  \sum_{\substack{e \leq z \\ v | e}} \frac{\lambda_e}{e} \ll_{A} \frac{1}{\varphi(v)} \cdot \frac{\varphi(d)}{d} + \frac{1}{(\log q)^{A - 1}}.
  $$
  If the condition $p | v \implies p | d$ or $p | v \implies p \leq (\log q)^A$ does not hold then the sum is empty.

  Moreover, for $v = 1$, we have
  $$
  \sum_{\substack{e \leq z}} \frac{\lambda_e}{e} = \frac{\varphi(d)}{d} + O_{A} \Big ( \frac{1}{(\log q)^{A - 1}} \Big ).
  $$
  \end{lemma}

  \begin{proof}
    Notice that the claim is trivial if $v > (\log q)^{A}$ as it just suffices to notice that
    $$
    \sum_{\substack{e \leq z \\ v | e}} \frac{\lambda_e}{e} \ll \frac{1}{v} \sum_{e \leq \exp(A (\log\log q)^3)} \frac{1}{e} \ll_{A} \frac{1}{(\log q)^{A  - 1}}.
    $$
    Therefore, assume now that $v \leq \log^{A} q$. We can further add to the subscript the condition that $p | v \implies p | d$ and $p | v \implies p \leq (\log q)^{A}$ since otherwise the sum is empty. We notice that since the condition $e \leq z$ in the definition of $\lambda_e$ is extraneous and implied by the other two conditions, our sum is equal to
    $$
    \sum_{\substack{p | e v \implies p \leq (\log q)^A \\ p | e v \implies p | d \\ \omega(e v) \leq (\log\log q)^2}} \frac{\mu(v e)}{v e}.
    $$
    By Rankin's bound,
    $$
    \sum_{\substack{p | e \implies p \leq (\log q)^{A} \\ \omega(e) > (\log\log q)^2}} \frac{1}{e} \ll_{A} \frac{1}{(\log q)^{A}}.
    $$
    Therefore, it remains to estimate
    $$
    \sum_{\substack{p | e v \implies p \leq (\log q)^{A} \\ p | e v \implies p | d}} \frac{\mu(e v)}{e v} = \frac{\mu(v)}{v} \cdot \sum_{\substack{p | e \implies p \leq \log^{A} q \\ p | e \implies p | d \\ (e, v) = 1}} \frac{\mu(e)}{e}  = \frac{\mu(v)}{v} \cdot \prod_{\substack{p | d \\ p \nmid v \\ p \leq (\log q)^{A}}} \Big ( 1 - \frac{1}{p} \Big ).
    $$
    Notice that since $p | v \implies p | d$ and $p |v \implies p \leq (\log q)^{A}$, we have
    $$
    \prod_{\substack{p | d \\ p \nmid v \\ p \leq (\log q)^{A}}} \Big (1 - \frac{1}{p} \Big ) = \prod_{\substack{p | d \\ p \leq (\log q)^{A}}} \Big (1 - \frac{1}{p} \Big ) \prod_{p | v} \Big (1 - \frac{1}{p} \Big )^{-1} = \prod_{\substack{p | d \\ p \leq (\log q)^{A}}} \Big ( 1 - \frac{1}{p} \Big ) \frac{v}{\varphi(v)}.
    $$

    It remains to notice that
    $$
    \prod_{\substack{p | d \\ p \leq (\log q)^{A}}} \Big ( 1 - \frac{1}{p} \Big ) = \frac{\varphi(d)}{d} \cdot \Big ( 1 + O \Big ( \frac{1}{\log^{A - 1} q} \Big ) \Big ) = \frac{\varphi(d)}{d} + O \Big ( \frac{1}{(\log q)^{A - 1}} \Big )
    $$
    and the claim follows.
  \end{proof}

  \section{Reduced residues in arithmetic progressions to large moduli} \label{sec:principal1}

  \begin{lemma}\label{lm:n5}
  Let $(r, q) = 1$ and $d | q$. Suppose that $\varphi(d) q / (d r) \rightarrow \infty$. Then,
  \begin{equation} \label{eq:delicate}
\frac{1}{r}  \sum_{a \leq r} \Big | \sum_{\substack{n \leq q \\ (n,d ) = 1 \\ n \equiv a \pmod{r}}} 1 - \frac{\varphi(d)}{d} \cdot \frac{q}{r} \Big | = o \Big ( \frac{\varphi(d)}{d} \cdot \frac{q}{r} \Big ).
  \end{equation}
  \end{lemma}
  \begin{proof}
 Let $\delta > 0$ be given. Fix a large $A$, say $A = 10^{10}$. We will show that for all $d, q, r$ such that $\varphi(d)q / (r d)$ is sufficiently large in terms of $1/ \delta$, the left-hand side of the above equation is $\ll \delta q \varphi(d) / ( d r)$.
   Let $W$ be a smooth function such that $W(x) = 1$ for $0 \leq x \leq 1$ and $W$ is compactly supported in $[-\delta, 1 + \delta]$ and such that $W^{(k)} \ll_{k} \delta^{-k}$ for all $k \geq 1$ (in particular $\widehat{W}(0) = 1 + O(\delta)$). Since
  $$
  \sum_{a \leq r} \Big | \sum_{\substack{q \leq n \leq (1 + \delta) q \\ (n, d) = 1  \\ n \equiv a \pmod{r}}} 1 \Big | \leq \sum_{\substack{q \leq n \leq (1 + \delta) q \\ (n, d) = 1}} 1 \ll \delta \cdot \frac{\varphi(d)}{d} \cdot q,
  $$
it is enough to show that
$$
\frac{1}{r} \sum_{a \leq r} \Big | \sum_{\substack{n \equiv a \pmod{r}}} W \Big ( \frac{n}{q} \Big ) - \widehat{W}(0) \cdot \frac{\varphi(d)}{d} \cdot \frac{q}{r} \Big | = o \Big (\frac{\varphi(d)q}{dr}\Big)
$$
as $\varphi(d) q / (r d) \rightarrow \infty$.

Note that
  \begin{align*}
  \sum_{\substack{a \leq r}} & \Big | \sum_{\substack{n \leq 2 q \\ n \equiv a \pmod{r}}} d(n) \Big ( \sum_{\substack{p | d \\ p | n \\ p > (\log q)^{A}}} 1 + \mathbf{1}_{\omega(n) > (\log\log q)^2} \Big ) \Big | \\ & \leq \sum_{n \leq 2 q} d(n) \sum_{\substack{p | d \\ p | n \\ p > (\log q)^{A}}} 1 + \sum_{\substack{n \leq q \\ \omega(n) > (\log\log q)^2}} d(n)  \ll_{A} \frac{q}{(\log q)^{A/2}}
  \end{align*}
  by Lemma \ref{le:trivialbounding}. Therefore, by Lemma \ref{le:sieve}, it is enough to show that
  $$
\frac{1}{r}  \sum_{a \leq r} \Big | \sum_{\substack{n \equiv a \pmod{r}}} \Big ( \sum_{\substack{e|n \\ e\leq z}}\lambda_e \Big )W \Big ( \frac{n}{q} \Big ) - \widehat{W}(0) \cdot \frac{\varphi(d)}{d} \cdot \frac{q}{r} \Big | = o \Big ( \frac{\varphi(d)}{d} \cdot \frac{q}{r} \Big ).
$$
  After an application of the Cauchy-Schwarz inequality, we see that it suffices to show that
    $$
    \sum_{a \leq r} \Big | \sum_{\substack{n \equiv a \pmod{r}}} \Big ( \sum_{\substack{e | n \\ e \leq z}} \lambda_e \Big ) W \Big ( \frac{n}{q} \Big ) - \widehat{W}(0) \cdot \frac{\varphi(d)}{d} \cdot \frac{q}{r} \Big |^2 = o \Big ( r \cdot \Big ( \frac{\varphi(d)}{d} \cdot \frac{q}{r} \Big )^2 \Big ).
    $$
    By the definition of $\lambda_e$, we can write the main sum as
    \begin{equation} \label{eq:start}
      \sum_{\substack{e \leq z}} \lambda_e \sum_{\substack{e | n \\ n \equiv a \pmod{r}}} W \Big ( \frac{n}{q} \Big ).
    \end{equation}
We notice that this is
$$
\widehat{W}(0) \frac{q}{r} \cdot \sum_{e \leq z} \frac{\lambda_e}{e} + \sum_{\substack{e \leq z}} \lambda_{e} \Big ( \sum_{\substack{e | n \\ n \equiv a \pmod{r}}} W \Big ( \frac{n}{q} \Big ) - \widehat{W}(0) \frac{q}{re} \Big ).
$$
By Lemma \ref{le:sieve2} and the choice of $W$,
$$
\Big | \widehat{W}(0) \frac{q}{r} \cdot \sum_{e \leq z} \frac{\lambda_e}{e} - \frac{q}{r} \cdot \frac{\varphi(d)}{d} \Big | \ll \delta \cdot \frac{q}{r} \frac{\varphi(d)}{d}.
$$
By the Poisson summation,
$$
\sum_{\substack{e | n \\ n \equiv a \pmod{r}}} \widehat{W} \Big ( \frac{n}{q} \Big ) - \widehat{W}(0) \cdot \frac{q}{r e} = \frac{q}{r e} \sum_{\ell \neq 0} e \Big ( - \frac{\ell a \overline{e}}{r} \Big ) \widehat{W} \Big ( \frac{q \ell}{r e} \Big ).
$$
(Note that  $(d,r) = 1$ since $d|q$ and $\lambda_e\neq 0$ implies that $e|d$. Therefore, $\overline{e}$ is well defined.)
We find that it remains to show
\begin{equation} \label{eq:maineq0}
\frac{1}{r} \sum_{a = 1}^{r} \Big | \frac{q}{r} \sum_{e  \leq z} \frac{\lambda_e}{e} \sum_{\ell \neq 0} e \Big ( - \frac{\ell a \overline{e}}{r} \Big ) \widehat{W} \Big ( \frac{q \ell}{r e} \Big ) \Big |^2 = o \Big ( \Big ( \frac{\varphi(d)}{d} \cdot \frac{q}{r} \Big )^2 \Big ) .
\end{equation}
Upon expanding the square and executing the summation over $a$, this is equal to
$$
\frac{q^2}{r^2} \sum_{e_1, e_2 \leq z} \frac{\lambda_{e_1} \lambda_{e_2}}{e_1 e_2} \sum_{\substack{\ell_1, \ell_2 \neq 0 \\ \ell_1 e_2 \equiv \ell_2 e_1 \pmod{r}}} \widehat{W} \Big ( \frac{q \ell_1}{ r e_1} \Big ) \overline{\widehat{W} \Big ( \frac{q \ell_2}{r e_2} \Big )}.
$$
Recall that $z = \exp(A(\log\log q)^3) \ll_{\varepsilon} q^{\varepsilon}$ for any $\varepsilon > 0$. Due to the rapid decay of $\widehat{W}$ we can truncate the sum over $\ell_1$ at $r e_1 q^{-1 + \varepsilon}$ at the price of a completely negligible error term of size $\ll_{A} q^{-A}$. Likewise, we can truncate the sum over $\ell_2$ at $r e_2 q^{-1 + \varepsilon}$. It follows that $|\ell_1 e_2| \leq r e_1 e_2 q^{-1 + \varepsilon} \leq r q^{-1 + 4 \varepsilon} \leq \sqrt{r}$, and similarly, $|\ell_2 e_1| \leq \sqrt{r}$. It follows that the condition $\ell_1 e_2 \equiv \ell_2 e_1 \pmod{r}$ implies $\ell_1 e_2 = \ell_2 e_1$.
We write this as $\ell_1 / e_1 = \ell_2 / e_2 = h / f$ with $(h,f) = 1$ and $f \leq q^{\varepsilon}$. This way we get
$$
\frac{q^2}{r^2} \sum_{\substack{(h,f) = 1 \\ f \leq q^{\varepsilon} \\ f | d}} \Big | \widehat{W} \Big ( \frac{q h}{r f} \Big ) \Big |^2 \cdot \Big ( \sum_{\substack{f | e}} \frac{\lambda_e}{e} \Big )^2.
$$
We notice that owing to the decay of $\widehat{W}$,
$$
\sum_{(h, f) = 1} \Big | \widehat{W} \Big ( \frac{q h}{r f} \Big ) \Big |^2 \ll \frac{r f}{q}.
$$
Therefore, using Lemma \ref{le:sieve2}, the part of the sum with $f \leq (\log q)^{A/2}$ is bounded by
$$
\ll \frac{q^2}{r^2} \sum_{\substack{f \leq (\log q)^{A / 2} \\ f | d}} \frac{r f}{q} \cdot \Big ( \frac{1}{\varphi(f)} \frac{\varphi(d)}{d} \Big )^2 \ll \frac{q}{r} \sum_{f | d} \frac{1}{\varphi(f)} \cdot \frac{f}{\varphi(f)} \cdot \Big ( \frac{\varphi(d)}{d} \Big )^2 \ll \frac{q}{r} \cdot \frac{\varphi(d)}{d}.
$$
On the other hand, the part of the sum with $f > (\log q)^{A / 2}$ is bounded by
\begin{equation} \label{eq:tobound}
\ll \frac{q}{r} \sum_{\substack{f > (\log q)^{A / 2} \\ p | f \implies p | d \\ p | f \implies p \leq (\log q)^{A} \\ f | d}} \frac{\mu^2(f)}{f} \Big ( \sum_{e} \frac{\lambda_{f e}}{e} \Big )^2, 
\end{equation}
the condition $\mu^2(f)$ being implied from the definition of $\lambda_e $ as equal to either $\mu(e)$ or $0$.

We now notice that $\sum_{e} \lambda_{f e} / e \ll  \prod_{p | d} (1 + 1/p) \ll \prod_{p \leq \log q} (1 + 1/p) \ll \log\log q$ and that moreover,
$$
\sum_{\substack{f | d \\ p | f \implies p \leq \log^A q}} \mu^2(f) \ll 2^{\omega(d; \log^{A} q)},
$$
where $\omega(d; y)$ denotes the number of distinct prime factors of $d$ that are $\leq y$. Since $\omega(d, y) \leq \frac{\log y}{\log\log y}$, we conclude that the above expression is $\ll_{\varepsilon} \log^{\varepsilon} q$ for any fixed $\varepsilon > 0$.
Therefore, \eqref{eq:tobound} is
$$
\ll_{A} \frac{q}{r} \cdot \frac{1}{(\log q)^{A / 2}}.
$$
This shows that  \eqref{eq:maineq0} is
$$
\ll_{A} \frac{q}{r} \cdot \frac{\varphi(d)}{d} = o \Big (\Big ( \frac{q}{r} \cdot \frac{\varphi(d)}{d} \Big )^2 \Big ),
$$
finishing the proof.

\end{proof}

\section{Additive exponentials along reduced residues} \label{sec:principal2}

In order to prove the remaining lemma, we will mostly appeal to the following result.
\begin{lemma}\label{lm:n7}
  Let $A > 1000$. Let $d \leq q$ and $r \leq \log^{A} q$ with $0 \leq a < r$. Let $\delta > 100 \varepsilon > 0$. Then, for $H > q^{\delta}$, $q^{-\varepsilon} \leq |\beta| \leq e^{-\tau r}$ and all $y \leq q$, we have
  \begin{equation} \label{eq:mainstat}
 \Big| \sum_{\substack{(n,d) = 1 \\ n \in [y, y + H] \\ n \equiv a \pmod{r}}} e(n \beta)\Big| \ll_{A,\delta} \frac{H}{e^{\tau r / 2}} \cdot \frac{\varphi(d)}{d} + \frac{H}{(\log q)^{A}}.
  \end{equation}
\end{lemma}

\begin{proof}
  We start by introducing a smooth function $W$ such that $W(x) = 1$ for $[ (\log q)^{-A}, 1 - (\log q)^{-A}]$ and $W$ is compactly supported in $[0, 1]$ with $W^{(k)}(x) \ll_{k} (\log q)^{A k}$ for all $x \in \mathbb{R}$.  Then, with a loss of $\ll H (\log q)^{-A}$, we can express the left-hand side of \eqref{eq:mainstat} as
  $$
  \sum_{\substack{(n, d) = 1 \\ n \equiv w \pmod{r}}} W\Big ( \frac{n - y}{H} \Big ) e(n \beta).
  $$
  We express the condition $n \equiv w \pmod{r}$ using additive characters. Therefore, it is enough to bound
  $$
  \sup_{0 \leq w < r} \Big | \sum_{\substack{(n,d) = 1}} W \Big ( \frac{n - y}{H} \Big ) e \Big ( n \Big ( \frac{w}{r} + \beta \Big ) \Big ) \Big |.
  $$
  Using Lemma \ref{le:sieve}, we write
  $$
  \mathbf{1}_{(n,d) = 1} = \sum_{\substack{e| n \\ e \leq z}} \lambda_e + O (d(n) \mathbf{1}_{\omega(n) > (\log\log q)^2})  + O \Big ( d(n) \sum_{\substack{p | d \\ p | n \\ p > (\log q)^{A}}} 1 \Big )
  $$
  with $z = \exp(A (\log\log q)^3)$ and $\lambda_e$ the same sieve coefficients as in Lemma \ref{le:sieve}.
  The contribution of the error term is negligible by Lemma \ref{le:trivialbounding}.
  It therefore suffices to bound
  $$
  \sum_{e \leq z} \lambda_{e} \sum_{n} W \Big ( \frac{n e - y}{H} \Big ) \exp \Big ( 2\pi i n \Big ( \frac{e w}{r} + e \beta \Big ) \Big ).
  $$
  By the Poisson summation, the inner sum is equal to
  $$
  \sum_{\ell} \int_{\mathbb{R}} W \Big ( \frac{e x - y}{H} \Big ) \exp \Big ( 2\pi i x \Big ( \frac{e w}{r} + e \beta \Big ) \Big ) \exp ( - 2 \pi i x \ell) dx.
  $$
  After a change of variable $(e x - y) / H \longleftrightarrow x$ (and writing $d$ instead of $e$ to avoid clashes of notation), this is equal to
  \begin{equation}\label{eq:geometric}
  \sum_{d \leq z} \lambda_d \cdot \frac{H}{d} \sum_{\ell} e \Big ( \frac{y}{d} \Big ( \frac{d w}{r} + d \beta - \ell \Big ) \Big ) \widehat{W} \Big ( \frac{H}{d} \Big ( \frac{d w}{r} + d \beta - \ell \Big ) \Big ).
  \end{equation}
 To analyze this, write
 \begin{equation} \label{eq:diophaprox}
 \Big | \beta + \frac{w}{r} - \frac{a}{v} \Big | \leq \frac{1}{v Q}
 \end{equation}
 with $Q = H q^{- 3 \varepsilon / 4}$, $v < Q$ and $(a, v) = 1$.

 Suppose first that $q^{\varepsilon} < v < H q^{- 3 \varepsilon / 4}$. In that case since \eqref{eq:diophaprox} is $\ll H^{-1} q^{-\varepsilon / 4}$, and since $d \leq q^{\varepsilon / 2}$, we get
 $$
 \Big | d \beta + \frac{d w}{r} - \frac{d a}{v} \Big | \leq \frac{q^{\varepsilon / 4}}{H}.
 $$
 However, since $q^{\varepsilon} \leq v \leq H q^{-3 \varepsilon / 4}$ and $d \leq q^{\varepsilon / 2}$, for any integer $\ell$,
 $$
 \Big | \frac{d a}{v} - \ell \Big | \geq \frac{q^{3 \varepsilon / 4}}{H}.
 $$
 Therefore, combining the above two inequalities,
 $$
 \Big | d \beta + \frac{d w}{r} - \ell \Big | > \frac{q^{3 \varepsilon / 4}}{H}
 $$
 and thus, \eqref{eq:geometric} is negligible by the fast decay rate of $\widehat{W}$ (since $d \leq q^{\varepsilon / 2}$). Therefore, there remains the case of $v < q^\varepsilon$. Notice that writing $w/r - a /v = t / (r v)$ for some $t \in \mathbb{Z}$, we see from \eqref{eq:diophaprox} that if $t \neq 0$ then
 $$
 |\beta| \asymp \frac{|t|}{r v}.
 $$
 Also, if $t = 0$ then we would have $|\beta| \ll H^{-1} q^{3 \varepsilon / 4}$ but this is impossible since we assume that $q^{-\varepsilon} \leq |\beta|$. So, in particular, $t \neq 0$ and it holds that $|\beta| \asymp |t| / (r v)$. Since we also have $|\beta| \leq e^{-\tau r}$, it follows from the previous equation that $v > e^{\tau r} / r$.


 If \eqref{eq:geometric} is non-negligible, that is, if there exists an integer $\ell$ such that
 $$
 \Big |d \beta + \frac{d w}{r} - \ell \Big | \leq \frac{q^{\varepsilon}}{H},
 $$
 then combining this with \eqref{eq:diophaprox}, we get
 $$
 \Big | \frac{d a}{v} - \ell \Big | \leq \frac{2 q^{\varepsilon}}{H}
 $$
 and since $v \leq q^{\varepsilon}$ and $(a,v) = 1$, this implies that $v | d$.

 Conversely, if $v | d$ then \eqref{eq:geometric} is equal to
 $$
 \frac{H}{d} \cdot e \Big ( - y \Big ( \frac{w}{r} + \beta - \frac{a}{v} \Big ) \Big ) \widehat{W} \Big ( H \Big ( \frac{w}{r} + \beta - \frac{a}{v} \Big ) \Big ) + O(q^{-\varepsilon / 4})
 $$
 (with the main term coming from $\ell = a d / v$).
 Therefore, it remains to estimate
 $$
 H \sum_{\substack{d \leq z \\ v | d}} \frac{\lambda_d}{d}
 $$
 and Lemma \ref{le:sieve2} shows that this is
 $$
 \ll_{A} \frac{H}{\varphi(v)} \cdot \frac{\varphi(d)}{d} + \frac{H}{(\log q)^{A / 2}}\ll \frac{H}{e^{\tau r /2 }} \cdot \frac{\varphi(d)}{d} + \frac{H}{(\log q)^{A / 2}}
 $$
 as needed.

\end{proof}

As a fairly immediate consequence of Lemma \ref{lm:n7}, we obtain:

\begin{lemma} \label{lm:n6}
  Let $(r, q) = 1$, $d | q$. Then, for $H > q^{1/5}$ and $r \leq (\log q)^{100}$,  we have
  $$
 \sup_{|\beta|\leq e^{-\tau r}}\Big | \sum_{\substack{n \in [y, y + H] \\ (n, d) = 1 \\ n \equiv a \pmod{r}}} e(n \beta) - \frac{1}{\varphi(r)} \sum_{\substack{n \in [y, y + H] \\ (n, rd) = 1}} e(n \beta) \Big | \ll \frac{\varphi(d)}{d} \cdot \frac{H}{e^{\tau r / 2}} + \frac{H}{(\log q)^{200}}
 $$
 for every $y \leq q$. Moreover, for every $A > 1000$, $q \geq 1$, $H\geq q^{1/10}$, $y\leq q$ and $q'\leq q^2$,
\begin{equation}\label{eq:j3'}
\sum_{\substack{n\in [y,y+H] \\ (n,q')=1}}1 = \frac{H \varphi(q')}{q'} + O_{A} \Big ( \frac{H}{(\log q)^{A}} \Big ).
\end{equation}
\end{lemma}

\begin{proof}
  Suppose first that $q^{-\varepsilon} \leq |\beta| \leq e^{-\tau r}$. In this case, the result follows from Lemma \ref{lm:n7} which shows that both terms are individually bounded by
  $$
  \frac{H}{e^{\tau r / 2}} \cdot \frac{\varphi(d)}{d} + \frac{H}{(\log q)^{200}}.
  $$
  Suppose now that $|\beta | \leq q^{-\varepsilon}$. Cutting into intervals $I = [x,y] \subset [0, q]$ of length $q^{\varepsilon / 4}$, it remains to show that
  $$
  \Big | \sum_{\substack{n \in I \\ (n, d) = 1\\ n \equiv a \pmod{r}}} e((n - x) \beta) - \frac{1}{\varphi(r)} \sum_{\substack{n \in I \\ (n, rd) = 1}} e((n - x) \beta) \Big | \ll \frac{\varphi(d)}{d} \cdot \frac{H}{e^{\tau r / 2}} + \frac{H}{(\log q)^{200}}.
  $$
  Since $|(n - x) \beta| \leq q^{- 3\varepsilon / 4}$ for all $n \in I = [x,y]$, it suffices in fact to show that
  \begin{equation} \label{eq:whatwewant}
  \Big | \sum_{\substack{n \in I \\ (n,d) = 1 \\ n \equiv a \pmod{r}}} 1 - \frac{1}{\varphi(r)} \sum_{\substack{n \in I \\ (n, rd ) = 1}} 1 \Big | \ll \frac{\varphi(d)}{d} \cdot \frac{H}{e^{\tau r / 2}} + \frac{H}{(\log q)^{200}}.
  \end{equation}
  By Lemma \ref{le:sieve} and Lemma \ref{le:trivialbounding}, the first term on the left-hand side is for any $A > 0$ equal to
  \begin{equation} \label{eq:almosteasyfinish}
  \sum_{e \leq v} \lambda_e \sum_{\substack{n \in I \\ n \equiv a \pmod{r} \\ e | n}} 1 + O \Big ( \frac{|I|}{(\log q)^{A}} \Big ),
  \end{equation}
  where $z = \exp(A (\log\log q)^2)$ and $\lambda_e$ the same coefficients as defined in Lemma \ref{le:sieve}.

  Since $\lambda_e$ is supported on integers such that $p | e \implies p | d$, we have $(e, r) = 1$. It follows that \eqref{eq:almosteasyfinish} is, by Lemma \ref{le:sieve2}, equal to
  $$
  \frac{|I|}{r} \sum_{e \leq v} \frac{\lambda_e}{e} + O \Big ( \frac{|I|}{(\log q)^{A}} \Big ) = \frac{|I|}{r} \cdot \frac{\varphi(d)}{d} + O \Big ( \frac{1}{(\log q)^{A - 1}} \Big ).
  $$
  Like-wise a computation (based on Lemma \ref{le:sieve}, Lemma \ref{le:trivialbounding} and Lemma \ref{le:sieve2}) reveals that
  $$
  \frac{1}{\varphi(r)} \sum_{\substack{n \in I \\ (n, rd ) = 1}} 1 = \frac{1}{\varphi(r)} \cdot \frac{\varphi(r d)}{r d} \cdot |I| + O \Big ( \frac{|I|}{(\log q)^{A - 1}} \Big ) = \frac{|I|}{r} \cdot \frac{\varphi(d)}{d} + O \Big ( \frac{|I|}{(\log q)^{A - 1}} \Big ).
  $$
  Choosing $A = 300$, we therefore obtain \eqref{eq:whatwewant}. The bound \eqref{eq:j3'} follows in an identical manner by inserting the sieve weight of Lemma \ref{le:sieve} and appealing to Lemma \ref{le:trivialbounding} and Lemma \ref{le:sieve2}.

\end{proof}

We also have the following corollary (in which we do not assume that $(r,q)=1$):
\begin{corollary}\label{cor:j3}
  Let $r \leq (\log q)^{50}$ and $H' \geq q^{1/3}$. Then, for every $y \leq q$ and every $(a,(r,q))=1$, we have
  \begin{multline} \label{eq:j3}
  \sup_{|\beta|\leq e^{-\tau r}} \Big | \sum_{\substack{(m, q) = 1 \\ m \equiv a \pmod{r} \\ m \in [y, y + H']}} e(m \beta) - \frac{1}{\varphi(r)} \sum_{\substack{(m,\frac{rq}{(r,q)}) = 1 \\ m \in [y, y + H']}} e(m \beta) \Big | \ll \\\frac{(r,q)}{\varphi((r,q))}\cdot\frac{\varphi(q)}{q}\cdot\frac{H'}{e^{\tau r/2}} +\frac{H'}{(\log q)^{100}}.
  \end{multline}
\end{corollary}
\begin{proof} The corollary is a simple consequence of Lemma \ref{lm:n6}. Notice that if $a\leq r $ is such that $(a,(r,q))=1$, then $(m,q)=1$, $m\equiv a \pmod r$ is equivalent to $(m,\frac{q}{(r,q)})=1$, $m\equiv a \pmod r$. Therefore, for every $\beta\in \R$,
 $$
 \sum_{\substack{(m, q) = 1 \\ m \equiv a \pmod{r} \\ m \in [y, y + H']}} e(m \beta)=\sum_{\substack{(m, \frac{q}{(r,q)}) = 1 \\ m \equiv a \pmod{r} \\ m \in [y, y + H']}} e(m \beta).
$$
It remains to use Lemma \ref{lm:n6} for $q'=\frac{q}{(r,q)}$, $d=q'$ and $r'=r$ (notice that $(q',r')=1$).
\end{proof}

\section{Ergodicity of weighted sums}\label{sec:ergchar}
In this section we prove the following result which is also of independent interest.

\begin{proposition}\label{prop:ergchar}For every uniquely ergodic $T=T_{\alpha,g}$, and every $n\in \N$ sufficiently large, there exists
$$
p_n\in [ \log^2 q_n, 2 \log^2 q_n]\cap \cP~\footnote{The set on the RHS is always non-empty by the PNT. We denote by $\cP$ the set of primes.}
$$
such that
\be\label{eq:erg}
\lim_{n\to+\infty} \min_{z_n\in \{q_n,p_nq_n\}}\max_{d|z_n} \sup_{(x,y)\in \T^2}\frac{d}{z_n\varphi(d)}\Big|\sum_{\substack{k\leq z_n\\(k,d)=1}}f(T^k(x,y))\Big|=0.
\ee
for every $f\in C(\T^2)$ with zero mean.
\end{proposition}
\begin{proof}
It is enough to show \eqref{eq:erg} for $f=e_{b,c}$, $b,c\in \Z$. We will consider several cases:\\
\textbf{I.} $n^{\ast}=n$. This implies that $q_n\geq e^{\tau q_{n-1}}$. In this case, we will show \eqref{eq:erg} with $z_n=q_n$. Given a small $\epsilon>0$, let $H\in [\frac{1}{2} q_n^{1/2-\epsilon},q_n^{1/2-\epsilon}]\cap \Z$ be such that $q_{n-1}|H$ (this assumption is only to simplify the notation below) and let $I_i=[iH,(i+1)H)$ for $i=0,\ldots, [\frac{q_n}{H}]-1$. Notice that the interval $[[\frac{q_n}{H}]H,q_n]$ has length $\leq H$, and since $\frac{dH}{q_n\varphi(d)}=o(1)$ for every $d|q_n$,~\footnote{We recall that $\varphi(d)\geq d/(e^\gamma\log\log d+3/\log\log d)$ for $d>2$.} it can be ignored. We have
\be\label{eq:asd21}
\sum_{\substack{k\leq z_n\\(k,d)=1}}e_{b,c}(T^k(x,y))=\sum_{a\leq q_{n-1}}\sum_{i=0}^{[q_n/h]-1}\sum_{\substack{k\in I_i\\(k,d)=1\\k\equiv a\mod q_{n-1}}}e_{b,c}(T^k(x,y))+O(H).
\ee
Let $z_i=iH$ (then $q_{n-1}|z_i$). Then $k\in I_i$ implies that $|k-z_i|\leq q_n^{1/2-\epsilon}$. Let $(x_i,y_i):=T^{z_i}(x_i,y_i)$. By Corollary~\ref{cor:jad} with $q_n$ and $q_{n-1}$, $m=k-z_i$, $\delta=1/2+\epsilon$ and $w=1$, it follows that
$$
T^{k}(x,y)=T^{k-z_i}(x_i,y_i)=T^{k \mod q_{n-1}}(x_i,y_i+P_{n-1}(x_i,k-z_i))+o(1),
$$
where $P_{n-1}$ is a polynomial of degree at most~$1$ (since $\delta=1/2+\epsilon$). Therefore, by Proposition~\ref{lem:toadd1}, $P_{n-1}(x_i,k-z_i)=(k-z_i)\beta_i$, where
\begin{equation}\label{warnabeta}|\beta_i|\leq e^{-\tau q_{n-1}}.\end{equation}
 Using this, we get
\begin{multline}\label{eq:ad1}
\sum_{\substack{k\in I_i\\(k,d)=1\\k\equiv a\mod q_{n-1}}}e_{b,c}(T^k(x,y))=e_{b,c}(T^a(x_i,y_i)) \sum_{\substack{k\in I_i\\(k,d)=1\\k\equiv a\mod q_{n-1}}}e_c((k-z_i)\beta_i) +\\o( \sum_{\substack{k\in I_i\\(k,d)=1\\k\equiv a\mod q_{n-1}}}1 ),
\end{multline}
where $o(\cdot)$ does not depend on $d$. The last term after summing over $i$ and $a\leq q_{n-1}$ is of order $o(\frac{q_n}d\cdot\varphi(d))=o(\frac{\varphi(d)q_n}{d})$, and hence will be ignored.
Let
$$
h_{i,a}:=\sum_{\substack{k\in I_i\\(k,d)=1\\k\equiv a\mod q_{n-1}}}e_c((k-z_i)\beta_i)
$$
and let
$$
v_i=\frac{1}{\varphi(q_{n-1})}\sum_{\substack{(m,dq_{n-1})=1\\ m\in I_i}}e_c((m-z_i)\beta_i).
$$
Summing \eqref{eq:ad1} over $a\leq q_{n-1}$ (and recalling that we ignore the last term), we get
\be\label{eq:ad12}
\sum_{\substack{k\in I_i\\(k,d)=1}}e_{b,c}(T^k(x,y))=v_i\sum_{a\leq q_{n-1}}e_{b,c}(T^a(x_i,y_i))+O\Big(\sum_{a\leq q_{n-1}}|h_{i,a}-v_i|\Big).
\ee
Notice that by \eqref{eq:j3'} with $q'=dq_{n-1}\leq q_n^2$ (and remembering that $|I_i|\geq q_n^{1/3}$),
 $$
|v_i|\leq \frac{1}{\varphi(q_{n-1})}\sum_{\substack{(m,dq_{n-1})=1\\ m\in I_i}}1=\frac{1}{\varphi(q_{n-1})}O\Big(\frac{\varphi(dq_{n-1})H}{dq_{n-1}}\Big)=O\Big(\frac{\varphi(d)H}{dq_{n-1}}\Big),
$$
where we use that $(d,q_{n-1})=1$ (since $d|q_n$ and $(q_n,q_{n-1})=1$). So, by unique ergodicity,
$$
|v_i\sum_{a\leq q_{n-1}}e_{b,c}(T^a(x_i,y_i))|=o\Big(\frac{\varphi(d)H}{d}\Big).
$$
Now, $n=n^\ast$ and~\eqref{warnabeta} allow us to use Lemma~\ref{lm:n6} with $r=q_{n-1}$ which after summing over $a\leq q_{n-1}\leq \log^{50}q_n$, gives
$$
\sum_{a\leq q_{n-1}}|h_{i,a}-v_i|\leq \frac{\varphi(d)}d\frac{H}{e^{\tau q_{n-1}/2}}q_{n-1}+
\frac{Hq_{n-1}}{(\log q_n)^{200}}=
o\Big(\frac{\varphi(d)H}{d}\Big).
$$
Using this and summing \eqref{eq:ad12} over $i$, we get
$$
\sum_{\substack{k\leq q_n\\(k,d)=1}}e_{b,c}(T^k(x,y))=\frac{q_n}{H}o\Big(\frac{\varphi(d)H}{d}\Big)=o\Big(\frac{\varphi(d)q_n}{d}\Big).
$$
This finishes the proof of \textbf{I} (with an arbitrary $p_n$ in the required interval).
\\
\textbf{II.} $n^{\ast}<n$. Notice that by the definition of $n^{\ast}$, $q_n< e^{\tau q_{n-1}}$.
Let $w:\N\to \N$ be a function that goes slowly to $+\infty$, say, $w(n)=\log \log \log n$.

\textbf{II.a.} $q_{n-1}\leq  \frac{q_{n}}{w(n)\log \log q_n}$. In this case, we will show \eqref{eq:erg} with $z_n=q_n$. We split
\be\label{eq:a111}
\sum_{\substack{k\leq q_n\\(k,d)=1}}e_{b,c}(T^k(x,y))=\sum_{a\leq q_{n-1}}\sum_{\substack{k\leq  q_n\\ k\equiv a \mod q_{n-1}\\(k,d)=1}}e_{b,c}(T^k(x,y)).
\ee
By Lemma \ref{cor:jad2} with $n-1$ in place of $n$ (and noticing that $n^\ast=(n-1)^\ast$) and $z=1$, noticing that by assumptions $q_{n}\leq \min(\frac{q_nq_{n-1}}{q_{n^\ast}}, e^{\tau q_{n-1}})$, we get $k\equiv a \mod q_{n-1}$ implies that
\be\label{eq:msa2}
d(T^k(x,y),T^a(x,y))=o(1).
\ee
Notice that for every $d|q_n$,
$$
\frac{\varphi(d)q_n}{q_{n-1}d}\geq \frac{q_n}{q_{n-1}\log\log d}\geq \frac{q_n}{q_{n-1}\log \log q_n}\geq w(n)\to +\infty,
$$
since we are in~\textbf{II.a} case.
Therefore, by Lemma \ref{lm:n5} for $q=q_n$ and $r=q_{n-1}$, \eqref{eq:msa2}, and unique ergodicity, we obtain
\begin{multline*}
\sum_{\substack{k\leq q_n\\(k,d)=1}}e_{b,c}(T^k(x,y))=\sum_{a\leq q_{n-1}}e_{b,c}(T^a(x,y)) [\sum_{\substack{k\leq q_n\\ k\equiv a \mod q_{n-1}\\(k,d)=1}}1]+ o\Big(\frac{\varphi(d)q_n}{d}\Big)=\\
\frac{\varphi(d)q_n}{d}\frac1{q_{n-1}}\sum_{a\leq q_{n-1}}e_{b,c}(T^a(x,y))+O\Big(\sum_{a\leq q_{n-1}}\Big|\sum_{\substack{k\leq q_n\\ k\equiv a \mod q_{n-1}\\(k,d)=1}}1-
\frac{\varphi(d)q_n}{q_{n-1}d}\Big|\Big)+
o\Big(\frac{\varphi(d)q_n}{d}\Big)=\\
o\Big(\frac{\varphi(d)q_n}{d}\Big).
\end{multline*}
This finishes the proof in this case.\\

\textbf{II.b.}  $q_{n-1}\geq  \frac{q_{n}}{w(n)\log \log q_n}$ and $q_{n^\ast}\leq \frac{q_{n-1}}{16\log^2q_{n}}$. In this case, we will show \eqref{eq:erg} with $z_n=p_nq_n$ for {\bf{some}} $p_n\in [\log^2 q_n,2\log^2q_n]\cap \cP$. Namely, let $p_n\in [\log^2 q_n,2\log^2q_n]\cap \cP$ be any number such that $(p_n,q_{n-1})=1$. To see that such a $p_n$ exists, notice that by the prime number theorem
\be\label{prod:pr}
\prod_{p\in [\log^2 q_n,2\log^2q_n]\cap \cP}p\geq \Big(\log^{2}q_n\Big)^{\frac{\log^2q_n}{4\log \log q_n}}=\Big(\log q_n\Big)^{\log^{3/2}q_n}\geq q_n>q_{n-1}
\ee
(recall that $(p_n,q_{n-1})>1$ implies that $p_n|q_{n-1}$).
By the bounds of $p_n$, $p_nq_n<e^{p_n}$, so
\be\label{eq:newbu}
q_{n-1}\leq q_n=\frac{z_n}{p_n}\leq \frac{z_n}{\log z_n}.
\ee
Note that $(p_n,q_{n-1})=(q_n,q_{n-1})=1$, implies  $(z_n,q_{n-1})=1$. Since we are in case~\textbf{II.b.},   $z_n=p_nq_n\leq 4q_n\log^2q_n\leq \frac{q_n q_{n-1}}{q_{n^\ast}}$. Moreover, $z_n\leq 4q_n\log^2 q_n\leq e^{2\tau q_{n-1}}$ (since $q_n\leq e^{\tau q_{n-1}}$). Therefore, we can use Lemma \ref{cor:jad2}, with $n-1$  in place of $n$, $z=1$ to get that for every $k\leq z_n\leq q_{n-1}\min(\frac{q_{n}}{q_{n^\ast}},e^{2\tau q_{n-1}})$, $k \equiv a \mod q_{n-1}$, we have
\be\label{eq:fty}
d(T^k(x,y),T^a(x,y))=o(1).
\ee
Notice that by the bounds on $p_n$, $z_n=p_nq_n\geq q_n\log^2q_n\geq q_n\log^{3/2} z_n$. So, for every $d|z_n$, in view of~\eqref{eq:newbu},
$$
\frac{\varphi(d)z_n}{q_{n-1}d}\geq \frac{z_n}{q_{n-1}\log\log d}\geq \frac{z_n}{q_{n-1}\log \log z_n}\geq \log^{1/2}z_n\to +\infty.
$$ We split
\be\label{eq:glor}
\sum_{\substack{k\leq z_n\\(k,d)=1}}e_{b,c}(T^k(x,y))=\sum_{a\leq q_{n-1}}\sum_{\substack{k\leq z_n\\ k\equiv a \mod q_{n-1}\\(k,d)=1}}e_{b,c}(T^k(x,y)).
\ee
By Lemma \ref{lm:n5} for $q=z_n$ and $r=q_{n-1}$, \eqref{eq:fty} and unique ergodicity, we get
\begin{multline}\label{eq:pidpo}
\sum_{\substack{k\leq z_n\\ (k,d)=1}}e_{b,c}(T^k(x,y))=\sum_{a\leq q_{n-1}}e_{b,c}(T^a(x,y)) [\sum_{\substack{k\leq  z_n\\ k\equiv a \mod q_{n-1}\\(k,d)=1}}1]+ o\Big(\frac{\varphi(d)q_n}{d}\Big)=\\
\frac{\varphi(d)z_n}{q_{n-1}d}\sum_{a\leq q_{n-1}}e_{b,c}(T^a(x,y))+O\Big(\sum_{a\leq q_{n-1}}\Big|\sum_{\substack{k\leq z_n\\ k\equiv a \mod q_{n-1}\\(k,d)=1}}1-\frac{\varphi(d)z_n}{q_{n-1}d}\Big|\Big)+o\Big(\frac{\varphi(d)z_n}{d}\Big)=\\
o\Big(\frac{\varphi(d)z_n}{d}\Big).
\end{multline}
This finishes the proof in this case.\\

\textbf{II.c.}  $q_{n-1}\geq  \frac{q_{n}}{w(n)\log \log q_n}$, $q_{n^\ast}\geq \frac{q_{n-1}}{16\log^2 q_{n}}$ {\bf and} (see \eqref{eq:gn})
\be\label{eq:d111}
\max_{\substack{|m|\in [q_{n^\ast-1},\frac{\log q_{n^\ast}}{\tau'^2}]\\q_{n^\ast-1}|m}}|a_m|\leq \frac{1}{\log^4 q_n}.
\ee
Let  $p_n\in [\log^2 q_n,2\log^2q_n]\cap \cP$ be any number such that $(p_n,q_{n-1})=1$ (analogously to \eqref{prod:pr}, we show that such $p_n$ exists). Let $z_n:=p_nq_n$. By Lemma \ref{lem:an2}, we get that for $k\leq z_n\leq 2q_n\log^2q_n$,  $k\equiv a \mod q_{n-1}$,
$$
d(T^k(x,y),T^a(x,y))=o(1).
$$
The proof follows now the same lines as the proof of \textbf{II.b.}, i.e.\ we repeat~\eqref{eq:glor} and~\eqref{eq:pidpo}.

\textbf{II.d.}  $q_{n-1}\geq  \frac{q_{n}}{w(n)\log \log q_n}$, $q_{n^\ast}\geq \frac{q_{n-1}}{16\log^2q_n}$ and (see \eqref{eq:gn})
$$
\max_{\substack{|m|\in [q_{n^\ast-1},\frac{\log q_{n^\ast}}{\tau'^2}]\\q_{n^\ast-1}|m}}|a_m|\geq \frac{1}{\log^4 q_n}.
$$
We then take $z_n:=q_n$. If $m$ is a number reaching the max above, then $e^{-\tau' q_{n^\ast-1}}\geq e^{-\tau' m}\geq |a_{m}|\geq \frac{1}{\log^4 q_n}$, and this implies that
\be\label{eq:asq1}
q_{n^\ast-1}\leq [\log \log q_n]^2.
\ee
Moreover, by \eqref{eq:gn}, since $a_m$ is a Fourier coefficient of $g_{n^\ast-1}$, by bounding the $L^2$ norm by the supremum norm, we obtain
\be\label{sup:la}
\sup_{x\in\T} |g_{n^{\ast}-1}(x)| \geq \frac{1}{\log^4 q_n}.
\ee

Let $H:=q_{n^{\ast}}^{1/2-\epsilon}\geq q_{n-1}^{1/2-\epsilon}/(16\log^2q_n)^{1/2-\epsilon}\geq q_{n-1}^{1/2-3/(2\epsilon)}\geq q_{n}^{1/2-2\epsilon}$ by our choice of $w$. In this case, we will show \eqref{eq:erg} with $z_n=q_n$. We will show that
\be\label{eq:shoi}
\sum_{u<q_n}\Big|\sum_{\substack{k\in [u,u+H]\\(k,d)=1}}e_{b,c}(T^k(x,y))\Big|=o\Big(\frac{q_nH\varphi(d)}
{d}\Big).
\ee
Then \eqref{eq:erg} will follow, since by the above we can split $[0,q_n]$ (up to error $\leq H$) into disjoint intervals $\{I_i\}$ (of length $H$) satisfying
$$
\Big|\sum_{\substack{k\in I_i\\(k,d)=1}}e_{b,c}(T^k(x,y))\Big|=
o\Big(\frac{H\varphi(d)}{d}\Big).
$$
The result then follows by summing over $i$. Let us show \eqref{eq:shoi}. Let $k\in [u,u+H]$, $k\equiv a+u \mod q_{n^{\ast}-1}$. Note that $k-u\leq H=q_{n^\ast}^{1/2-\epsilon}$, so by Corollary~\ref{cor:jad} with $n$ replaced by $n^\ast$, $w=1$ and  $\delta=1/2+\epsilon$,  and denoting $(x_u,y_u)=T^u(x,y)$, we obtain
$$
T^k(x,y)=T^{k-u}(x_u,y_u)=T^{a}(x_u,y_u+(k-u)\beta_u)+o(1),
$$
as  $\deg P_{n^{\ast}-1}\leq 1$ by our choice of $\delta$. Moreover, by Proposition \ref{lem:toadd1},
\be\label{beu}
\beta_u=g_{n^\ast-1}(x+u\alpha).
\ee
Using this and decomposing into residue classes $\mod q_{n^\ast-1}$, we get
\begin{multline}\label{eq:fde}
\sum_{\substack{k\in [u,u+H]\\(k,d)=1}}e_{b,c}(T^k(x,y))=\\\sum_{a\leq q_{n^\ast-1}}e_{b,c}(T^a(x_u,y_u)[\sum_{\substack{k\in [u,u+H]\\(k,d)=1\\k\equiv a \mod q_{n^\ast-1}}}e_{b,c}((k-u)\beta_u)] +o(\sum_{\substack{k\in [u,u+H]\\(k,d)=1}}1).
\end{multline}
Notice that if $(a,(q_{n^\ast-1},d))>1$, then the above sum is empty. If $(a,(q_{n^\ast-1},d))=1$ and
\be\label{eq:betu}
|\beta_u|\geq \frac{1}{q_n^{\epsilon}},
\ee
then $q_n^{-\epsilon}<|\beta_u|<e^{-\tau q_{n^\ast-1}}$ (see~\eqref{eq:bocoe}) and since $H>q_n^{\frac12-2\epsilon}>q_n^{100\epsilon}$, by Lemma~\ref{lm:n7} with $r=q_{n^{\ast}-1}$ and $q=q_n$, we get
$$
\Big|\sum_{\substack{k\in [u,u+H]\\(k,d)=1\\k\equiv a \mod q_{n^\ast-1}}}e_{b,c}((k-u)\beta_u)\Big|=
o\Big(\frac{\varphi(d)H}{dq_{n^\ast-1}^2}\Big),
$$
where the proper bound $\frac1{(\log q_n)^A}=o\Big(\frac{\varphi(d)}{dq^2_{n^\ast-1}}\Big)$ for the second summand on the RHS in~\eqref{eq:mainstat} follows from~\eqref{eq:asq1}.
Hence, summing over $a\leq q_{n^\ast-1}$, using~\eqref{eq:fde}
and~\eqref{eq:j3'} (which applies since $H>q_n^{1/2-2\vep}>q_n^{1/10}$), we get
$$
\sum_{\substack{k\in [u,u+H]\\(k,d)=1}}e_{b,c}(T^k(x,y))=
o\Big(\frac{\varphi(d)H}{dq_{n^\ast-1}}\Big)+
o\Big(\frac{\varphi(d)H}{d}\Big)=
o\Big(\frac{\varphi(d)H}{d}\Big).
$$
So, by \eqref{eq:betu},
\begin{multline*}
\sum_{u<q_n}\Big|\sum_{\substack{k\in [u,u+H]\\(k,d)=1}}e_{b,c}(T^k(x,y))\Big|=\sum_{u:|\beta_u|\leq \frac{1}{q_n^{\epsilon}}}\Big|\sum_{\substack{k\in [u,u+H]\\(k,d)=1}}e_{b,c}(T^k(x,y))\Big|+
o\Big(\frac{q_n\varphi(d)H}{d}\Big).
\end{multline*}
Since for a fixed $u$, $$\Big|\sum_{\substack{k\in [u,u+H]\\(k,d)=1}}e_{b,c}(T^k(x,y))\Big|\leq\frac Hd\cdot\varphi(d),$$
equation~\eqref{eq:shoi} follows by showing
$$
\Big|\{u\leq q_n\;:\; |\beta_u|\leq q_n^{-\epsilon}\}\Big|=o(q_n).
$$
This however follows by \eqref{sup:la}, \eqref{beu} and Lemma \ref{lem:nazq}. This finishes the proof of~\textbf{II.d.} and hence also the proof of Proposition~\ref{prop:ergchar}.
\end{proof}

\begin{lemma}\label{lem:nazq} Fix $\epsilon>0$. Let $n^\ast <n$ and assume that
\be\label{eq:aaa}
\sup_{x\in\T} |g_{n^{\ast}-1}(x)| \geq \frac{1}{\log^4 q_n}.
\ee
Then (uniformly) for every $x\in \T$,
$$
\Big|\{u\leq q_n\;:\; |g_{n^\ast-1}(x+u\alpha)|\leq q_n^{-\epsilon}\}\Big|=o(q_n).
$$
\end{lemma}
\begin{proof}Notice first that $g_{n^{\ast}-1}$ is a complex polynomial whose number of terms is  at most $2\frac{\log q_{n^\ast}}{q_{n^\ast-1}}=o(\log q_n)$. Let $E:=\{x\in \T\;:\; |g_{n^\ast-1}(x)|<q_n^{-\epsilon/2}\}$. Then, by \eqref{eq:aaa} and Theorem~\ref{thm:naz},
\be\label{eq:lebe}
Leb(E)\leq C\Big[\frac{\log^4 q_n}{q_n^\epsilon}\Big]^{1/o(\log q_n)}=o(1),
\ee
since $\epsilon>0$ is fixed and $n$ grows. Notice that if $z\in E^c$ and $|z'-z|\leq \frac{1}{q_n^{\epsilon/2}}$, then (remembering that $\sup |g'_j|\to0$ due to the exponential decay of the coefficients of $g$)
$$
|g_{n^\ast-1}(z')-g_{n^\ast-1}(z)|\leq \sup_{\theta\in \T}|g'_{n^\ast-1}(\theta)|\frac{1}{q_n^{\epsilon/2}}=
o\Big(\frac{1}{q_n^{\epsilon/2}}\Big),
$$
and so $|g_{n^\ast-1}(z')|\geq \frac{1}{q_n^{\epsilon}}$. Decompose $\T$ into disjoint intervals $\{I_i\}_{i=1}^\ell$ of  equal length $\sim \frac{1}{q_n^{\epsilon/2}}$. By the above, if $I_i\cap E^c\neq \emptyset$, then
$$
\inf_{z\in I_i}|g_{n^\ast-1}(z)|\geq q_n^{-\epsilon}.
$$
Let $J:=\{i\;:\; I_i\cap E^c\neq \emptyset\}$. By \eqref{eq:lebe}, $|J|=\ell-o(\ell)$. By the Denjoy-Koksma inequality, for every $i\leq \ell$ and every $x\in \T$,
$$
\Big|\sum_{m\leq q_n}\chi_{I_i}(x+m\alpha)-q_nLeb(I_i)\Big|\leq 2,
$$
which implies that
$$\Big|\sum_{m\leq q_n}\chi_{I_i}(x+m\alpha)\Big|\geq q_n^{1-\epsilon/2}-2.$$
 Therefore, and since $|I_i|\sim \frac{1}{q_n^{\epsilon/2}}$, so  $\ell=[q_n^{\epsilon/2}]$, we obtain
$$
\Big|\{u\leq q_n\;:\; |g_{n^\ast-1}(x+u\alpha)|\geq q_n^{-\epsilon/2}\}\geq (\ell -o(\ell))q_n^{1-\epsilon/2}=q_n-o(q_n).
$$
This finishes the proof of Lemma \ref{lem:nazq}.
\end{proof}

\begin{lemma}\label{lem:an2}Let $n$ be such that $n^\ast<n$, $q_{n-1}\geq q_n^{1/2}$ and
\be\label{eq:asdfg}
\sup_{\substack{|m|\in [q_{n^\ast-1},\frac{\log q_{n^\ast}}{\tau'^2}]\\q_{n^\ast-1}|m}}|a_m|\leq \frac{1}{\log^4 q_n}.
\ee
Then, for every $m\leq 10q_n\log^2 q_n$,
$$
d(T^{m}(x,y),T^{m \mod q_{n-1}}(x,y))=o(1),
$$
uniformly over $m$ and $(x,y)\in\T^2$.
\end{lemma}
\begin{proof}It is enough to show that for every $k\leq \frac{10q_n\log^2 q_n}{q_{n-1}}$,
$$
d(T^{kq_{n-1}}(x,y),(x,y))=o(1),
$$
uniformly over $k$ and $(x,y)\in \T^2$. Notice that
$$
\|kq_{n-1}\alpha\|\leq\frac{10q_n\log^2q_n}{q_{n-1}}\frac{1}{q_n}=o(1),
$$
since $q_{n-1}\geq q_n^{1/2}$. Therefore, we only need to show that for $k\leq \frac{10q_n\log^2 q_n}{q_{n-1}}$, we have
$$
|S_{kq_{n-1}}(g)(x)|=o(1).
$$
Recall that $g=\widetilde{g}=\sum_{\ell\geq1}g_\ell$. Clearly, by the cocycle identity, we have
$$
\sup_{x\in\T}|S_{kq_{n-1}}(g)(x)|\leq k\sup_{x\in\T}|S_{q_{n-1}}(g)(x)|\leq k\sum_{\ell\geq1}\sup_{x\in\T}|S_{q_{n-1}}(g_\ell)(x)|.$$
But by the definition of $n^{\ast}$ (see \eqref{eq:n,ast}), for $s\in[n^\ast,n]$ the interval
$[q_{s-1},\log q_s/(\tau'^2)]$ is empty as
$$
\frac{\log q_s}{\tau'^2}<\frac{\tau q_{s-1}}{\tau'^2}<\frac{(1/2)\tau'^2 q_{s-1}}{\tau'^2}=\frac12 q_{s-1}.$$
Therefore, there are no frequencies in $[q_{n^\ast},q_n]$ ($g_\ell=0$ for $\ell\in[n^\ast,n]$). Moreover, for the frequencies  at least $q_n$, we apply the exponential decaying rate of Fourier coefficients to obtain
$$\sum_{\ell\geq n}|S_{q_{n-1}}(g_\ell)|=q_{n-1}O(e^{-\tau q_n}),$$
and, clearly,
$$
\frac{10 q_n \log^2 q_n}{q_{n-1}}\cdot q_{n-1} o(e^{-\tau q_n})=o(1).$$
It follows that it is enough to show that
$$
\frac{10q_n \log^2 q_n}{q_{n-1}}\sum_{\ell<n^\ast}\sup_{x\in\T }|S_{q_{n-1}}(g_\ell)(x)|=o(1).
$$

Recall that for every $m\in \Z$,
\be\label{eq:sw2}
S_{q_{n-1}}(e_m)(x)=e_m(x)\frac{e_m(q_{n-1}\alpha)-1}{e_m(\alpha)-1}.
\ee
Moreover,
\be\label{eq:sw3}
\Big|\frac{e_m(q_{n-1}\alpha)-1}{e_m(\alpha)-1}\Big|\ll \frac{\|mq_{n-1}\alpha\|}{\|m\alpha\|}\leq \min\Big(q_{n-1},\frac{m\|q_{n-1}\alpha\|}{\|m\alpha\|}\Big).
\ee
We will separately consider the cases $S_{q_{n-1}}(g_{n^\ast-1}(\cdot))$ and $\sum_{\ell<n^\ast-1}|S_{q_{n-1}}(g_\ell)(x)|$. Let first $\ell<n^{\ast}-1$ and let $|m|\in [q_\ell, q_{\ell+1}]$. Then
$$
|S_{q_{n-1}}(e_m)(x)|\leq \frac{|m|\|q_{n-1}\alpha\|}{\|m\alpha\|}\leq  \frac{2|m| q_{\ell+1}}{q_{n}}.
$$
Notice that since $\ell+1<n^{\ast}$, it follows by the definition of $n^\ast$ and  $n^\ast<n$ that $q_{\ell+1}\leq q_{n^\ast-1}\leq [\log q_{n^\ast}]^2\leq [\log q_{n-1}]^2$. Therefore,
$$
|a_mS_{q_{n-1}}(e_m)(x)|\leq 2|ma_m|\frac{\log^{2}q_{n-1}}{q_n}\leq e^{-\tau |m|/2} \frac{\log^{2}q_{n-1}}{q_n}.
$$
So, since $q_{n-1}\geq q_n^{1/2}$,
$$
\frac{10q_n \log^2 q_n}{q_{n-1}}|a_mS_{q_{n-1}}(e_m)(x)|\leq 10e^{-\tau |m|/2}\frac{\log^4 q_n}{q_{n-1}}\ll e^{-\tau m/2} \frac{1}{q_{n-1}^{1/2}}.
$$
Therefore,
$$
\frac{10q_n\log^2 q_n}{q_{n-1}}\sum_{\ell<n^\ast-1}\sup_{x\in\T }|S_{q_{n-1}}(g_\ell)(\cdot)|=O(q_{n-1}^{-1/2})=o(1).
$$
Notice that we did not use \eqref{eq:asdfg} in this case. It remains to bound
$$
\sup_{x\in \T}|S_{q_{n-1}}(g_{n^\ast-1})(x)|.
$$
Let $|m|\in [q_{n^\ast-1}, q_{n^\ast}]$.  Notice that by \eqref{eq:sw2} and \eqref{eq:sw3},
$$
|S_{q_{n-1}}(e_m)(x)|\leq \frac{2|m| q_{n^\ast}}{q_{n}}.
$$
Using \eqref{eq:asdfg}, $|a_m|\leq e^{-\tau |m|}$ and $n^{\ast}<n$,
$$
|a_mS_{q_{n-1}}(e_m)(x)|\leq  2|m||a_m|^{1/4} |a_m|^{3/4}\frac{q_{n-1}}{q_{n}}\leq e^{-\tau |m|/8} \frac{q_{n-1}}{q_n\log^{3}q_n}.
$$
Therefore,
$$
\frac{10q_n\log^2 q_n}{q_{n-1}}|a_mS_{q_{n-1}}(e_m)(x)|\leq  10e^{-\tau |m|/8}\frac{1}{\log q_n}.
$$
Summing over $|m|\in [q_{n^\ast-1}, q_{n^\ast}]$ gives
$$
\frac{10q_n \log^2 q_n}{q_{n-1}}S_{q_{n-1}}(g_{n^\ast-1})(x)=o(1).
$$
This finishes the proof.
\end{proof}

\part{Equidistribution along primes}

\section{Number theoretic lemma}

Here, we will collect a number of standard lemmas that will be frequently used in the upcoming sections. We also collect a number of more mundane lemmas that would otherwise obstruct the flow of the argument.

\begin{lemma}[The hybrid large sieve]
  Let $D(s, \chi) = \sum_{n \leq N} a(n) \chi(n) n^{-s}$. Then,
  \begin{equation} \label{eq:toboundlarge}
  \sum_{\chi \pmod{q}} \int_{|t| \leq T} |D(\tfrac 12 + it, \chi)|^2 dt \ll (\varphi(q) T + N) \sum_{\substack{n \leq N \\ (n, q) = 1}} \frac{|a(n)|^2}{n}.
  \end{equation}
\end{lemma}
\begin{proof}
  This is \cite[Theorem 6.4]{MontgomeryTopics}.
\end{proof}

\begin{lemma}[Classical large sieve]
  Let $D(\chi) = \sum_{n \leq N} a(n) \chi(n)$. Then,
  $$
  \sum_{\chi \pmod{q}} |D(\chi)|^2 \ll ( \varphi(q) + N ) \sum_{\substack{n \leq N \\ (n, q) = 1}} |a(n)|^2.
  $$
\end{lemma}
\begin{proof}
  This is \cite[Theorem 6.2]{MontgomeryTopics}.
\end{proof}

\begin{lemma}[Mean-value theorem] \label{le:mvt}
  Let $D(s) = \sum_{n \leq N} \alpha_n n^{-s}$ be a Dirichlet polynomial.
  Then,
  $$
  \int_{|t| \leq T} |D(it)|^2 dt \ll (T + N) \sum_{n \leq N} |\alpha_n|^2.
  $$
  \end{lemma}
  \begin{proof}
    Let $\Phi$ be a smooth non-negative function such that $\Phi(x) \gg 1$ for $|x| \leq 1$ and $\text{supp }\widehat{\Phi} \subset [-1,1]$. Then,
    $$
    \int_{|t| \leq T} |D(it)|^2 dt \leq\ \int_{\mathbb{R}} |D(it)|^2 \Phi \Big ( \frac{t}{T} \Big ) dt = \sum_{n,m} \alpha_n \overline{\alpha_m} T \widehat{\Phi} \Big ( T \log \frac{n}{m} \Big ).
      $$
      Writing $n = m + h$, we obtain that the contribution of terms with $|h| > N / T$ is zero. Therefore, the above is equal to
      $$
   \ll   T \sum_{n \leq N} |\alpha_n|^2  + T \sum_{0 < h < N / T} \sum_{n \leq N - h} |\alpha_n| \cdot |\alpha_{n + h}|
   $$
   and applying the inequality $|\alpha_{n} \alpha_{n + h}| \leq |\alpha_{n}|^2 + |\alpha_{n + h}|^2$, we obtain
   $$
   T \sum_{n \leq N} |\alpha_n|^2 + T \Big ( \frac{N}{T} + 1 \Big ) \sum_{n \leq N} |\alpha_n|^2
   $$
   which gives the claim.
  \end{proof}

  \begin{lemma}[Vaughan's identity] \label{le:vaughan}
    For $n > z \geq 1$,
    $$
    \Lambda(n) = \sum_{\substack{d | n \\ d \leq z}} \mu(d) \ln \frac{n}{d} - \sum_{\substack{d  c | n \\ d, c \leq z}} \mu(d) \Lambda(c) + \sum_{\substack{d c | n  \\ d > z, c > z}} \mu(d) \Lambda(c).
    $$
    \end{lemma}
    \begin{proof}
      See \cite[Proposition 13.4]{IwaniecKowalski}.
    \end{proof}

\begin{lemma}[Heath-Brown identity]
  For any integer $k \geq 1$,
  $$
  - \frac{\zeta'}{\zeta}(s) = \sum_{j = 1}^{k} (-1)^j \binom{k}{j} \zeta(s)^{j - 1} \zeta'(s) M(s)^j - \frac{\zeta'}{\zeta}(s) \cdot (1 - \zeta(s) M(s))^k,
  $$
  where
  $$
  M(s) = \sum_{n \leq z} \frac{\mu(n)}{n^s}.
  $$
\end{lemma}
\begin{proof}
  This is a trivial consequence of the binomial theorem.
\end{proof}
\begin{lemma}[Linnik identity] \label{le:linnik}
  We have
  $$
  - \sum_{k} \frac{(-1)^k}{k} \cdot d^{\star}_{k, z}(n) = \frac{1}{\alpha} \cdot \mathbf{1}_{\substack{n = p^{\alpha} \\ n > z}},
  $$
  where $d_{k,z}^{\star}(n)$ counts the number of representations of $n$ as $n_1 \ldots n_k$ with $n_i$ such that $p | n_i \implies p > z$ and $n_i > 1$ for all $i = 1, \ldots, k$.
\end{lemma}
\begin{proof}
  Let $P(s) = \prod_{p \leq z} (1 - p^{-s})$. Consider then
  $$
  \log (\zeta(s)P(s)) = \log(1 - (1 - \zeta(s)P(s))) = - \sum_{k} \frac{(-1)^k}{k} \cdot (1 - \zeta(s)P(s))^{k}.
  $$
  The lemma follows on comparing the coefficients of the Dirichlet polynomials on the left-hand side and the right-hand side.
\end{proof}

  \begin{lemma} \label{le:technical}
    Let $\mathcal{E}$ be a subset of tuples of the form $(t, \chi)$ with $|t| \leq x$ and $\chi$ a character $\pmod{q}$. Let $D(s, \chi)= \sum_{n \leq x} a(n) \chi(n) n^{-s}$ be a Dirichlet polynomial such that $\sum_{n \leq x} \frac{|a(n)|^2}{n} \ll (\log x)^{500}$. Then, for $\varphi(q) \leq H \leq x$,
    \begin{align*}
     \frac{1}{\varphi(q)} \sum_{\chi \neq \chi_{0}} \sum_{y < x} & \Big | \int_{\substack{|t| \leq x \\ (t,\chi) \not \in \mathcal{E}}} D(\tfrac 12 + it, \chi) \cdot \frac{(y + H)^{1/2 + it} - y^{1/2 + it}}{1/2 + it} dt \Big |^2 \\ & \ll \frac{H^2 \log x}{\varphi(q)} \sum_{\chi \neq \chi_{0} \pmod{q}} \int_{\substack{|t| \leq (x / H) (\log x)^{1000} \\ (t,\chi) \not \in \mathcal{E}}} |D(\tfrac 12 + it, \chi)|^2 dt +  \frac{H^2 x}{\varphi(q) (\log x)^{500}}.
    \end{align*}
  \end{lemma}
  \begin{proof}
    This result is essentially standard and is implicit for instance in \cite[Lemma 14]{MatoRad}. We will repeat the proof here for the convenience of the reader. We start by splitting $y$ into dy-adic intervals $100 H < 2^{-L - 1} x \leq y \leq 2^{-L} x$ with $0 \leq L \leq \log x$ and the interval $[0, 100 H]$. We will first handle the contribution coming from the $y \in [2^{-L - 1} x, 2^{-L} x]$ and then discuss the remaining (easier) case of $y \in [0, 100 H]$.

    First notice that
    \begin{align*}
    \frac{(y + H)^{s} - y^s}{s} & = \frac{y}{2 H} \int_{H / y}^{3 H / y} y^s \cdot \frac{(1 + u)^s - 1}{s} du \\ & - \frac{ y + H}{2 H} \int_{0}^{2 H /(y + H)} (y + H)^{s} \cdot \frac{(1 + u)^s - 1}{s} du.
    \end{align*}
    Using this identity, we see that
    \begin{align*}
    \Big | \int_{\substack{|t| \leq x \\ (t,\chi) \not \in \mathcal{E}}} D(s, \chi ) \cdot & \frac{(y + H)^{s} - y^{s}}{s} ds \Big |
    \ll \Big | \int_{\substack{|t| \leq x \\ (t,\chi) \not \in \mathcal{E}}} D(s, \chi) y^s \cdot \frac{(1 + u)^s - 1}{s} ds \Big |\\ & + \Big | \int_{\substack{|t| \leq x \\ (t,\chi) \not \in \mathcal{E}}} D(s, \chi) (y + H)^s \cdot \frac{(1 + v)^s - 1}{s} ds \Big | \ , \ s = \tfrac 12 + it
    \end{align*}
    for some $|u|,|v| \ll H / y \ll 2^{L} H / x$. The treatment of the second term involving $(y + H)^{s}$ is identical because after a change of variable $y \mapsto y - H$, the variable $y$ is still localized in an interval of length $H$ starting at a point $\gg H$, since $2^{-L - 1} x \geq 100 H$. For this reason, we will omit this term from further discussion.

    It remains therefore to bound
    \begin{equation} \label{eq:oupla}
    \sum_{L} \frac{1}{\varphi(q)} \sum_{\chi \neq \chi_{0}} \int_{\mathbb{R}} \Big | \int_{\substack{|t| \leq x \\ (t, \chi) \not \in \mathcal{E}}} D(s, \chi) y^s \cdot \frac{(1 + u)^s  - 1}{s} ds \Big |^2 \Phi \Big ( \frac{y}{2^{-L} x} \Big ) d y \ , \ |u| \ll \frac{2^L H}{x}
    \end{equation}
    with $\Phi$ some smooth non-negative function such that $\Phi(x) \geq 1$ for all $x \in [1/2, 1]$ (notice that this expression is an upper bound for the sum over $y < x$). Expanding the square, interchanging the integral signs and using the bound
    $$
    \frac{(1 + u)^s - 1}{s} \ll \min \Big ( \frac{H 2^{L}}{x} , \frac{1}{1 + |t|} \Big ),
    $$
   we conclude that the integral over $y$ in \eqref{eq:oupla} is
    \begin{align*}
    \ll \int_{\substack{|t| \leq x \\ (t, \chi) \not \in \mathcal{E}}} & \int_{\substack{|u| \leq x \\ (u, \chi) \not \in \mathcal{E}}} |D(\tfrac 12 + iu, \chi) D(\tfrac 12 + it, \chi)| \\ & \times \min \Big ( \frac{H 2^{L}}{x} , \frac{1}{1 + |t|} \Big ) \min \Big ( \frac{H 2^{L}}{x}, \frac{1}{1 + |u|} \Big ) \Big | \int_{\mathbb{R}} y^{1 + it - iu} \Phi \Big ( \frac{y}{2^{-L} x} \Big ) dy \Big | dt du.
    \end{align*}
    By the integration by parts, this is
    \begin{align*}
    \ll 2^{-2L} x^2 & \int_{\substack{|t|, |u| \leq x \\ (t, \chi), (u,\chi) \not \in \mathcal{E}}} |D(\tfrac 12 + iu, \chi)D(\tfrac 12 + it, \chi)| \\ &  \times \min \Big ( \frac{H 2^{L}}{x} , \frac{1}{1 + |u|} \Big ) \min \Big ( \frac{H 2^L}{x}, \frac{1}{1 + |t|} \Big ) \cdot \frac{dt du}{1 + |t - u|^2}.
    \end{align*}
    Using the inequality $2 a b \leq a^2 + b^2$ (applied to each of the $D(\cdot) \min(\ldots)$) then gives the bound
    $$
    \ll 2^{-2L} x^2 \int_{\substack{|t| \leq x \\ (t, \chi) \not \in \mathcal{E}}} |D(\tfrac 12 + it, \chi)|^2 \cdot \min \Big ( \frac{H^2 2^{2L}}{x^2}, \frac{1}{1 + |t|^2} \Big ) dt.
    $$
    The part of the integral with $|t| \leq (x / H) (\log x)^{1000}$ gives after summing over $L$ and $\chi$, a contribution which is
    $$
    \ll \frac{H^2 \log x}{\varphi(q)} \sum_{\chi \neq \chi_{0} \pmod{q}} \int_{\substack{|t| \leq (x / H) (\log x)^{1000} \\ (t, \chi) \not \in \mathcal{E}}} |D(\tfrac 12 + it, \chi)|^2 dt.
    $$
    It therefore remains to bound the part with $|t| \geq (x / H) (\log x)^{1000}$ which is
    \begin{equation} \label{eq:badpart}
    \ll \sum_{0 \leq L \leq \log x} 2^{-2L} x^2 \cdot \frac{1}{\varphi(q)} \sum_{\chi \neq \chi_{0} \pmod{q}} \int_{\substack{|t| > (x / H) (\log x)^{1000}}} |D(\tfrac 12 + it, \chi)|^2 \cdot \frac{dt }{1 + |t|^2}.
    \end{equation}
    Dissecting the range $t$ over dy-adic intervals, we see that
    \begin{align*}
    & \frac{1}{\varphi(q)} \sum_{\chi \pmod{q}} \int_{|t| > (x / H) (\log x)^{1000}} |D(\tfrac 12 + it, \chi)|^2 \cdot \frac{dt}{1 + |t|^2} \\ & \ll \frac{H^2}{x^2} \cdot \frac{1}{(\log x)^{2000}} \sum_{R} 2^{-2 R} \cdot \frac{1}{\varphi(q)} \sum_{\chi \pmod{q}} \int_{|t| \sim 2^{R} (x / H) (\log x)^{1000}} |D(\tfrac 12 + it, \chi)|^2 dt
    \end{align*}
    by the large sieve, the assumptions on the coefficients of $D(\cdot)$ and by $\varphi(q)\leq H$, this is
  $$
   \ll \frac{1}{\varphi(q)} \frac{H^2}{x^2} \cdot \frac{1}{(\log x)^{2000}} \cdot \sum_{R} 2^{-2 R} \cdot \Big ( \frac{\varphi(q) 2^R x}{H} (\log x)^{1000} + x \Big ) (\log x)^{500}
  $$$$
    \ll \frac{H^2}{\varphi(q) x} \cdot\frac{1}{(\log x)^{500}}.
$$
    This shows that \eqref{eq:badpart} is
    $$
    \ll \frac{H^2 x}{\varphi(q) (\log x)^{500}}
    $$
    as required.

    Finally, it remains to deal with the contribution of $y \in [0, 100 H]$. This is sligtly easier and so we will be briefer.
    First, it suffices to use
$$
    \Big | \int_{\substack{|t| \leq x \\ (t, \chi) \not \in \mathcal{E}}} D(s, \chi)\cdot \frac{(y + H)^s - y^s}{s} ds \Big |^2 \ll
 $$$$
    \Big | \int_{\substack{|t| \leq x \\ (t, \chi) \not \in \mathcal{E}}} D(s, \chi) \cdot \frac{(y + H)^s}{s} ds \Big |^2 + \Big | \int_{\substack{|t| \leq T \\ (t, \chi) \not \in \mathcal{E}}} D(s, \chi) \cdot \frac{y^{s}}{s} ds \Big |^2,
    $$
    where $s = \tfrac 12 + it$.
    Once again we can focus on the second term involving $y^{s}$ since the treatment of the first term with $(y + H)^{s}$ is similar because after the change of variable $y + H \mapsto y$, the variable $y$ still belongs to an interval of length $\gg H$ ending at a point which is $\gg H$.

    Therefore, it remains to estimate
    $$
    \frac{1}{\varphi(q)} \sum_{\chi \neq \chi_{0}} \int_{\mathbb{R}} \Big | \int_{\substack{|t| \leq x \\ (t, \chi) \not \in \mathcal{E}}} D(s, \chi) \cdot \frac{y^s}{s} ds \Big |^2 \Phi \Big (\frac{y}{100 H} \Big ) dy
    $$
    with $\Phi$ a smooth, non-negative, compactly supported function such that $\Phi(x) \geq 1$ for $x \in [0,1]$. Expanding the square, we get
    \begin{equation} \label{eq:easypeasy}
    \frac{1}{\varphi(q)} \sum_{\chi \neq \chi_{0}} \int_{\substack{|t|, |u| \leq x \\ (t, \chi), (u, \chi) \not \in \mathcal{E}}} \frac{D(\tfrac 12 + it, \chi)}{\tfrac 12 + i t} \frac{\overline{D(\tfrac 12 + iu, \chi)}}{\tfrac 12 - iu} \int_{\mathbb{R}} y^{1 + it - iu} \cdot \Phi \Big ( \frac{y}{100 H} \Big ) dy dt du.
    \end{equation}
    By the integration by parts,
    $$
    \int_{\mathbb{R}} y^{1 + it - iu} \Phi \Big ( \frac{y}{100 H} \Big ) dy \ll \frac{H^2}{1 + |t - u|^2}.
    $$
    Therefore, \eqref{eq:easypeasy} is
    $$
    \ll \frac{H^2}{\varphi(q)} \sum_{\chi \neq \chi_{0}} \int_{\substack{|t| \leq x \\ (t, \chi) \not \in \mathcal{E}}} |D(\tfrac 12 + it, \chi)|^2 \cdot \frac{dt}{1 + |t|^2}.
    $$
    Using that $H \leq x$, we can now bound this by
    $$
    \ll \frac{H^2}{\varphi(q)} \sum_{\chi \neq \chi_{0}} \int_{\substack{|t| \leq x (\log x)^{1000} / H \\ (t, \chi) \not \in \mathcal{E}}} |D(\tfrac 12 + it, \chi)|^2 dt + \frac{H^2}{\varphi(q)} \int_{\substack{(\log x)^{1000} \leq |t| \leq x}} |D(\tfrac 12 + it, \chi)|^2 \cdot \frac{dt}{1 + |t|^2}.
    $$
    Splitting the second term into dy-adic intervals $2^{L} \leq |t| \leq 2^{L + 1}$ and applying the hybrid large sieve, we see that the contribution of the second term is
    $$
    \ll \frac{H^2}{\varphi(q)}  \sum_{(\log x)^{1000} \leq 2^{L}} 2^{-2L} \Big ( \varphi(q) 2^{L} + x \Big ) \cdot (\log x)^{500} \ll H^2 (\log x)^{-500} + \frac{H^2 x}{\varphi(q)} \cdot (\log x)^{-1500}
    $$
    and this is sufficient.
  \end{proof}

  \begin{lemma}[Cancellations in Dirichlet polynomials over almost primes] \label{le:canceldirpoly}
    Let $A > 10$ be given. Let $\chi$ be a character of conductor $\leq (\log N)^{A}$ and
    $t$ be such that $(\log N)^{A^2} \leq |t| \leq N^{A}$.
    Then, uniformly in $1 \leq w \leq \sqrt{N}$,
    $$
    \Big | \sum_{\substack{n \sim N \\ p | n \implies p > w}} \frac{\mu(n) \chi(n)}{n^{1/2 + it}} \Big | \ll \frac{\sqrt{N}}{(\log N)^{A}} \text{ and } \Big | \sum_{\substack{n \sim N \\ p | n \implies p > w}} \frac{\chi(n)}{n^{1/2 + it}} \Big | \ll \frac{\sqrt{N}}{(\log N)^{A}}.
    $$
    In addition,
    $$
    \Big | \sum_{p \sim P} \frac{\chi(p) \log p}{p^{1/2 + it}} \Big | \ll \frac{\sqrt{P} \mathbf{1}_{\chi = \chi_0}}{1 + |t|} + \frac{\sqrt{P}}{(\log P)^A}.
    $$
  \end{lemma}
  \begin{proof}

    The third bound follows from the Korobov-Vinogradov zero-free region \cite[Chapter 9, Notes]{MontgomeryLectures}
    and contour integration as in \cite[Lemma 2]{MRShort}.
    We will only describe the proof of the first bound, since the proof of the second one is identical.

    Let $\varepsilon \in (0, \tfrac{1}{1000})$. The proof splits into two cases.

    \textbf{Case 1: $w > \exp((\log N)^{2/3 + \varepsilon}$.}
    By Ramar\'e's identity,
    \begin{align*}
      \sum_{\substack{n \sim N \\ p | n \implies p > w}} \frac{\mu(n) \chi(n)}{n^{1/2 + it}} & = \sum_{w < p \leq \sqrt{N}} \frac{\mu(p)\chi(p)}{p^{1/2 + it}} \sum_{\substack{q | m \implies q > w \\ m \sim N / p \\ (m,p) = 1}} \frac{\mu(m) \chi(m)}{m^{1/2 + it}} \cdot \frac{1}{\omega^{\star}(m; w)},
    \end{align*}
    where
    $$
    \omega^{\star}(m; w) = \sum_{w \leq p \leq \sqrt{N}} 1.
    $$
    We partition $p$ into dy-adic range $w \leq P \leq \sqrt{N}$
    and we express the condition $m p \sim N$ using a contour integral so that the above expression can be re-written as
    \begin{align*}
    \sum_{w \leq P \leq \sqrt{N}} \frac{1}{2\pi i} \int_{|u| \leq (\log N)^{3A}} \sum_{N / 4 P \leq m \leq 2 N / P} \frac{\mu(m)\chi(m)}{m^{1/2 + iu+ it}} \sum_{\substack{p \sim P \\ p \nmid m}} \frac{\mu(p)\chi(p)}{p^{1/2 + it + iu}} & \cdot \frac{N^{\sigma + iu} du}{\sigma + iu} \ , \ \sigma := \frac{1}{\log N} \\ & + O \Big ( \frac{\sqrt{N}}{(\log N)^{A}} \Big ).
    \end{align*}
    We now conclude by using
    $$
    \sum_{\substack{p \sim P \\ p \nmid m}} \frac{\chi(p)}{p^{1/2 + it + iu}} \ll \frac{\sqrt{P}}{(\log N)^{5 A}}
    $$
    and the trivial bound on the Dirichlet polynomial over $n$.

    \textbf{Case 2: $w < \exp((\log N)^{2/3 + \varepsilon})$.}
    On the other hand, if $w \leq \exp((\log N)^{2/3 + \varepsilon})$ then we notice that
    $$
    \sum_{\substack{n \sim N \\ p | n \implies p > w}} \frac{\mu(n) \chi(n)}{n^{1/2 + it}} =  \sum_{\substack{n \sim N \\ \omega(n) \leq (\log N)^{1/100}}} \frac{\mu(n) \chi(n)}{n^{1/2 + it}} \sum_{\substack{d | n \\ p | d \implies p \leq w \\ d \leq N^{1/100}}} \mu(d) + O \Big ( \frac{\sqrt{N}}{(\log N)^{A}} \Big ),
    $$
    where the condition $d \leq N^{1/100}$ is implied from the fact that $d$ has at most $(\log N)^{1/100}$ prime factors, and all of them are less than $w$.
    Interchanging the sum over $d$ and $n$, and trivially bounding the contribution of the integers $n$ with more than $(\log N)^{1/100}$ prime factors, we get that the first sum is equal to
    $$
    \sum_{\substack{p | d \implies p \leq w \\ d \leq N^{1/100}}} \frac{\mu(d)\chi(d)}{d^{1/2 + it}}
    \sum_{\substack{n \sim N \\ (n, d) = 1}} \frac{\mu(n) \chi(n)}{n^{1/2 + it}} + O \Big ( \frac{\sqrt{N}}{(\log N)^{A}} \Big ) \ll_{A} \frac{\sqrt{N}}{(\log N)^{A}}
    $$
    and this is $\ll_{A} \sqrt{N} / (\log N)^{A}$ using cancellations in the sum over $n$.

  \end{proof}

Recall that $p_{q} := p \mod {q}$, so $p_q \in [0, q - 1]$.

  \begin{lemma} \label{lem:us1}
For any $\varepsilon > 0$ and intervals $I \subset [0,N]$, $J \subset [0, q]$
    such that $|I| \geq q \cdot N^{5 \varepsilon}$ and $|J| > q^{\varepsilon}$, we have
    $$
    \sum_{\substack{p \in I \\ p_{q} \in J}} \log p \ll_{\varepsilon} \frac{|J|}{q} \cdot |I|.
    $$
  \end{lemma}
  \begin{proof}

    Since $I \subset [1, N]$, we have
    $$
    \sum_{\substack{p \in I \\ p_{q} \in J}} \log p \ll \log N \sum_{\substack{p \in I \\ p_{q} \in J}} 1.
    $$
    It will therefore suffice to prove the bound
    $$
    \sum_{\substack{p \in I \\ p_{q} \in J}} 1 \ll_{\varepsilon} \frac{|J|}{q} \cdot \frac{|I|}{\log N}.
    $$

      We separate the proof into two cases. First, consider the case where $|J| > q^3 / N^{2 - 4 \varepsilon}$.
    Let $\lambda_d$ denote the sieve coefficients coming from Lemma \ref{le:sieveprimes} so that
    $$
    \mathbf{1}_{p \in I} \leq \sum_{d \leq z} \lambda_d
    $$
    with $z = N^{\varepsilon}$. Therefore,
    $$
    \sum_{\substack{p \in I \\ p_{q} \in J}} 1 \ll \sum_{\substack{n \in I \\ n_{q} \in J}} \Big ( \sum_{\substack{ d | n \\ d \leq z}} \lambda_d \Big ).
    $$
    Opening the later sum in characters, we find that it is equal to
    $$
    \frac{1}{\varphi(q)} \sum_{\chi \pmod{q}} \Big ( \sum_{n \in I} \chi(n) \Big ( \sum_{\substack{ d| n \\ d \leq z}} \lambda_d \Big ) \Big ) \cdot \Big ( \sum_{a \in J} \overline{\chi}(a) \Big ).
    $$
    We notice that by the Polya-Vinogradov inequality, for $\chi \neq \chi_{0}$,
    $$
    \sum_{n \in I} \chi(n) \Big ( \sum_{\substack{d | n \\ d \leq z}} \lambda_d \Big ) = \sum_{d \leq z} \lambda_d \chi(d) \sum_{\substack{n \in I / d}} \chi(n) \ll z \sqrt{q} \log q.
  $$
  Moreover, by the large sieve,
  $$
  \frac{1}{\varphi(q)} \sum_{\chi \pmod{q}} \Big | \sum_{a \in J} \chi(a) \Big | \ll \sqrt{|J|}.
  $$
  Therefore, the contribution of the non-principal characters is
  $$
  \ll z \sqrt{q} \log q \sqrt{|J|}.
  $$
  Finally, the contribution of the principal character is
  $$
  \ll \Big ( \frac{1}{\varphi(q)} \sum_{\substack{n \in J \\ (n, q) = 1}} 1 \Big ) \sum_{\substack{n \in I \\ (n, q) = 1}} \Big ( \sum_{\substack{d | n \\ d \leq z}} \lambda_d \Big ) \ll \frac{|J|}{q} \sum_{n \in I} \Big ( \sum_{\substack{d | n \\ d \leq z}} \mu(d) \Big ) \ll \frac{|J|}{q} \cdot \frac{|I|}{\log z}
  $$
  by Lemma \ref{le:sieveprimes} and Lemma \ref{le:shiureduced}.
  This gives a final bound of the form
  $$
  \ll_{\varepsilon} \frac{|J|}{q} \cdot \frac{|I|}{\log N} + N^{\varepsilon} \sqrt{q} \log q \sqrt{|J|} \ll_{\varepsilon} \frac{|J|}{q} \cdot \frac{|I|}{\log N}  $$
  by our assumption that $|J| > q^3 / N^{2 - 4 \varepsilon}$.

  Let us now consider the case $|J| < q^3 / N^{2 - 4 \varepsilon}$ in which case necessarily $N \leq q^2$. We cover the interval $I$ with $\ll |I| / q$ disjoint intervals $I_1, \ldots, I_{k} \subset [0, 2N]$ of length $q$.  On each such sub-interval we notice that
  $$
  \sum_{\substack{p \in I_j \\ p_{q} \in J}} 1 = \sum_{p \in I_j^{\star}} 1,
  $$
  where $I_j^{\star}$ is an interval of length $|J|$. Since $N \leq q^2$ and $|J| > q^{\varepsilon}$, it follows by the Brun-Titchmarsh theorem that
  $$
  \sum_{p \in I_{j}^{\star}} 1 \ll_{\varepsilon} \frac{|J|}{\log N}.
  $$
  Summing back over all $I_j$, this gives the required bound.
\end{proof}

\section{Hybrid Huxley's results} \label{sec:eqprimes1}


In this section we will prove the following ``hybrid'' version of Huxley's theorem.

\begin{theorem} \label{thm:nr1}
  Let $\varepsilon, \xi \in (0, \tfrac{1}{1000})$. Suppose that $H / q > x^{1/6 + \varepsilon}$ and $H \leq x$. Then, for $r \leq q^{1-\xi}$ with $(r, q) = 1$, we have
  \begin{equation*}
  \sum_{y < x} \sum_{v = 1}^{r} \Big | \sum_{\substack{p \in [y, y + H] \\ p_{q} \equiv v \pmod{r}}} \log p - \frac{H}{r} \Big | \ll_{\varepsilon,\xi} \frac{H x}{(\log x)^{100}}.
  \end{equation*}
Moreover, if $H=x$ then
$$
\sum_{v = 1}^{r} \Big | \sum_{\substack{p \in [0,  H] \\ p_{q} \equiv v \pmod{r}}} \log p - \frac{H}{r} \Big | \ll_{\varepsilon,\xi} \frac{H }{(\log H)^{100}}.$$
\end{theorem}

Notice that taking $r = 1$ recovers the original result of Huxley in almost all short intervals. On the other end, taking $H = x$ and thinking of $p_{q}$ as $p$, one would recover a version of Huxley's theorem in arithmetic progressions to large moduli. We notice that such a version of Huxley's theorem (with $p$ in place of $p_{q}$) cannot be proven for arbitrary moduli $q$ using the current technology (because of the weakness of the zero-free region for $L(s, \chi)$) and we heavily exploit the fact that we are looking at the distribution in residue classes of $p_{q} := p \pmod{q} \in [0, q - 1]$ instead of $p$.

Using a rather similar proof, but with different input on the character sums, we will also prove the following variant of Theorem \ref{thm:nr1}.

\begin{theorem} \label{thm:nr2}
  Let $\varepsilon \in (0 , \tfrac{1}{1000})$ be given. Suppose that $(H / q) > x^{1/6 + \varepsilon}$ and $H \leq x$. Then, for $q^{1/2 - 1/10} \geq H' \geq q^{1/100}$, we have
  $$
  \sum_{y \leq x} \sum_{z < q} \sup_{\substack{\beta \in \mathbb{R} \\ 0 \leq v < r}} \Big | \sum_{\substack{p \in [y, y + H] \\ p_{q} \equiv v \pmod{r} \\ p_{q} \in [z, z + H']}} e(p_{q} \beta) \log p - \frac{H}{\varphi(q)} \sum_{\substack{(a,q) = 1 \\ 0 \leq a < q \\ a \equiv v \pmod{r} \\ a \in [z, z + H']}} e( a \beta) \Big | \ll_{\varepsilon} \frac{x H H'}{(\log x)^{100}}.
  $$
Moreover, if $H=x$ then
\begin{equation}\label{Hrownex}
\sum_{z < q} \sup_{\substack{\beta \in \mathbb{R} \\ 0 \leq v < r}} \Big | \sum_{\substack{p \leq H \\ p_{q} \equiv v \pmod{r} \\ p_{q} \in [z, z + H']}} e(p_{q} \beta) \log p - \frac{H}{\varphi(q)} \sum_{\substack{(a,q) = 1 \\ 0 \leq a < q \\ a \equiv v \pmod{r} \\ a \in [z, z + H']}} e( a \beta) \Big | \ll_{\varepsilon} \frac{ H H'}{(\log x)^{100}}.\end{equation}
\end{theorem}

\begin{remark}
  Notice that the result is non-trivial only for $r \leq (\log x)^{100}$.
  \end{remark}

We will be helped to a very large extent by the fact that we are working with $p_{q}$ instead of $p$. This has roughly the effect of a convolution, and off-loads the problem of obtaining cancellations in $\sum_{p \leq x} \chi(p) \log p$ onto the problem of obtaining cancellations in $\sum_{n} \chi(n)$ which is substantially easier. Theorem \ref{thm:nr2} can be thought of as the analogue (for $p_{q}$ instead of $p$) of the Fourier Uniformity problem for primes in the Huxley range. The latter remains an outstanding challenge.

\subsection{Lemma on large values of Dirichlet polynomials}

We say that a set $\mathcal{S}$ consisting of tuple  $(t, \chi)$ is \textit{well-spaced} if whenever  $(t, \chi) , (t', \chi) \in \mathcal{S}$ we
have either $t = t'$ or $|t - t'| \geq 1$.

\begin{lemma} \label{le:huxleylargevalues}
  Let $D(s, \chi) = \sum_{n \leq N} a(n) \chi(n) n^{-s}$.
  Let
  $$
  G = \sum_{n \leq N} \frac{|a(n)|^2}{n}.
  $$
  Let $\mathcal{S}$ be a set of well-spaced tuples $(t, \chi)$ such that for each $(t, \chi) \in \mathcal{S}$ we have $|t| \leq T$, $\chi \pmod{q}$ and
  $
  |D(\tfrac 12 + it, \chi)| > V.
  $
  Then  $|\mathcal{S}| \ll (\log q T)^2 \cdot (G N V^{-2} + G^3 N q T V^{-6})$.
\end{lemma}
\begin{proof}
  This is \cite[Lemma 10.2]{Harman}.
  \end{proof}

  \begin{lemma} \label{le:fourthmoment}

    Let $D(s, \chi) = \sum_{n} f(n) \chi(n) n^{-s} V (n / N)$ with either $f(n) = 1$ or $f(n) = \log n$ and $V$ a fixed smooth, compactly supported in $[0, \infty)$, function. Let $A >0 $ and assume that $N \leq (q T)^{A}$.
    Let $\mathcal{S}$ be a collection of well-spaced tuples $(t, \chi)$ with $|t| \leq T$ and $\chi \neq \chi_0 \pmod{q}$. Then,
    $$
    \sum_{(t, \chi) \in \mathcal{S}} |D(\tfrac 12 + it, \chi)|^4 \ll_{A} (q T) (\log q T)^{6}.
    $$
  \end{lemma}
  \begin{proof}
    Notice that
    $$
    D(s, \chi) = \frac{1}{2\pi i} \int_{(1/\log N)} L(s + w, \chi) \widetilde{V}(w) N^w dw,
    $$
    where $\widetilde{V}(w) := \int_{0}^{\infty}V(x) x^{w - 1} dx $ is the Mellin transform of $V$.  By Holder's inequality and the decay of $\widetilde{V}$, for all $A > 0$,
    $$
    |D(s, \chi)|^4 \ll_{A} \int_{\mathbb{R}} \Big | L \Big ( s + \frac{1}{\log N} + i u, \chi \Big ) \Big |^{4} \cdot \frac{du}{1 + |u|^{A}}.
    $$
    Therefore, it remains to show that for $|u| \leq (q T)^{\varepsilon}$, for any $\varepsilon > 0$,
    \begin{equation} \label{eq:discrete}
    \sum_{(t, \chi) \in \mathcal{S}} |L(\tfrac 12 + it + \frac{1}{\log N} + i u, \chi)|^4 \ll (q T) \cdot (\log q T)^{6} .
    \end{equation}
    In order to do this notice that by sub-harmonicity,
    $$
    |L(\tfrac 12 + it + \frac{1}{\log N} + iu)|^4 \leq \frac{1}{|D|}\iint_{D} |L(\tfrac 12 + it + \frac{1}{\log N} + x + iy + iu, \chi)|^4 dx dy,
    $$
    where $D$ is a disk of radius $1/\log (q T)$. Therefore, \eqref{eq:discrete} is bounded by
    $$
    (\log (q T))^2 \int_{-2/\log (q T)}^{2/\log (q T)} \sum_{\chi \pmod{q}} \int_{-T - (q T)^{\varepsilon}}^{T + (q T)^{\varepsilon}} |L(\tfrac 12 + it + x, \chi)|^4 dt dx.
    $$
    The result now follows from \cite[Theorem 10.1]{MontgomeryTopics}.
  \end{proof}

  \begin{lemma} \label{le:largevaluessmooth}

    Let $D(s, \chi) = \sum_{n} f(n) \chi(n) n^{-s} V (n / N)$ with either $f(n) = 1$ or $f(n) = \log n$ and $V$ a fixed, smooth, compactly supported function.
    Let $\mathcal{S}$ be a set of well-spaced $(t, \chi)$ such that for $(t, \chi) \in \mathcal{S}$ we have $|t| \leq T$, $\chi \neq \chi_0 \pmod{q}$ and
    $
    |D(\tfrac 12 + it, \chi)| > V
    $
    Let $A > 0$ and assume that $N \leq (q T)^{A}$. Then,
    $$
    |\mathcal{S}| \ll_{A} q T (\log q T)^{6} \cdot V^{-4}.
    $$
  \end{lemma}

  \begin{proof}
    This is an immediate consequence of Lemma \ref{le:fourthmoment}.
  \end{proof}


  \begin{lemma} \label{le:exceptional}
    Let $\{a(n)\}$ be a sequence of complex numbers. Suppose that $|a(n)| \ll d_{r}(n) (1 + \log n)$ for some $r \geq 2$ and all $n \geq 1$.
    Let $\mathcal{E} = \mathcal{E}(A; T; q; x; \varepsilon)$ be the set of well-spaced tuples $(t, \chi)$ with $|t| \leq T$ and $\chi \pmod{q}$ for which
    $$
    \sup_{(q T)^{\varepsilon} \leq M \leq x} \frac{(\log x)^{A}}{M^{1/2}} \Big | \sum_{n \leq M} a(n) \chi(n) n^{-\tfrac 12 - it} \Big | \gg 1.
    $$
    Then, $|\mathcal{E}(A; T; q; x; \varepsilon)| \ll  (\log x)^{32 r^2 / \varepsilon^2 + 8 A / \varepsilon + 4}$.
  \end{lemma}
  \begin{proof}

    Let $R$ denote the cardinality of $\mathcal{E}(A; T; q; x; \varepsilon)$. By the pigeonhole principle, there exists a $(q T)^{\varepsilon} \leq N = 2^k \leq x$ and a subset $\mathcal{E}' \subset \mathcal{E}$ of cardinality $\gg R / \log x$ such that for all $(t, \chi) \in \mathcal{E}'$,
  $$
  \sup_{(q T)^{\varepsilon} \leq u \leq N} \Big | \sum_{n \leq u} a(n) \chi(n) n^{-1/2 - it} \Big | \gg \frac{\sqrt{N}}{(\log x)^{A}}
  $$
  for all $(t, \chi) \in \mathcal{E}'$.
  Let $k$ be the smallest integer $> 3/\varepsilon$. Let $\beta_{n}$ denote coefficients such that
  $$
  \beta_{n} = \sum_{n = n_1 \ldots n_k} a(n_1) \ldots a(n_{k}).
  $$
  Note that $|\beta_{n}| \ll_{\varepsilon} d_{rk}(n) (1 + \log n)^{k}$ because $|a(n)| \leq d_r(n) ( 1 + \log n)$. Moreover,
  $$
  \Big | \sum_{n \leq u} a(n) \chi(n) n^{-1/2 - it} \Big |^{k} = \Big | \sum_{n \leq u^k} \beta_{n} \chi(n) n^{-1/2 - it} \Big |.
  $$
  And in particular,
  $$
  \frac{N^{k/2}}{(\log x)^{A k}} < \sup_{u < N} \Big | \sum_{n < u^k} \beta_{n} \chi(n) n^{-1/2 - it} \Big | = \sup_{u < N^k} \Big | \sum_{n < u} \beta_n \chi(n) n^{-1/2 - it} \Big |.
  $$
  Let $M = N^k$ so that $M > (q T)^{\varepsilon k} > (q T)^3$. It suffices therefore to estimate the number of tuples $(t, \chi)$ for which
  $$
  \sup_{u < M} \Big | \sum_{\ell < u} \beta_{\ell} \chi(\ell) \ell^{-1/2 - it} \Big | > \frac{\sqrt{M}}{(\log x)^{A k}}.
  $$

  Notice that the supremum over $u\leq M$ can be easily removed using a contour integral, i.e.\ writing for some small $\delta > 0$,
  $$
  \sum_{n \leq u} \beta_n \chi(n) n^{-1/2 - it} = \frac{1}{2\pi i} \int_{-1 / \log M - i M^{\delta}}^{1/\log M + i M^{\delta}} \Big ( \sum_{n \leq M} \beta_n n^{-s} \chi(n) n^{-1/2 - it} \Big ) \cdot \frac{u^{s}}{s} ds + O(M^{1 - \delta + o(1)}),
  $$
  so that
  $$
  \frac{\sqrt{M}}{(\log x)^{A k}} \ll \int_{-M^{\delta}}^{M^{\delta}} \Big | \sum_{n \leq M} \beta_{n} n^{-s} \chi(n) n^{-1/2 - it} \Big | \cdot \frac{dv}{1 + |v|} \ , \ s = \frac{1}{\log M} + i v.
  $$
  Let $\mathcal{T}$ be the set of tuples $(t, \chi)$ for which the above holds, so that $|\mathcal{T}| \gg R / \log x$ and
  \begin{equation} \label{eq:abouttogodual}
  \frac{R}{\log x} \cdot \frac{\sqrt{M}}{(\log x)^{A k}} \ll \sum_{(t, \chi) \in \mathcal{T}} \int_{-M^{\delta}}^{M^{\delta}} \Big | \sum_{n \leq M} \gamma_{n} n^{- 1/2 - it - iv} \chi(n) \Big | \cdot \frac{dv}{1 + |v|}
  \end{equation}
  for some coefficients $|\gamma_n| \ll |\beta_{n}| \ll d_{rk}(n) ( 1 + \log n)^k$.
  We find phases $\theta_{t, \chi, v} \in \mathbb{R}$ so that the right-hand side can be re-written as
\begin{align*}
\sum_{(t, \chi) \in \mathcal{T}} & \int_{-M^{\delta}}^{M^{\delta}} e^{i \theta_{t, \chi, v}} \sum_{n \leq M} \gamma_n n^{-1/2 - it - i v} \chi(n) \cdot \frac{dv}{1 + |v|} = \\ & \sum_{n \leq M} \gamma_n n^{-1/2} \int_{-M^{\delta}}^{M^{\delta}} \sum_{(t, \chi) \in \mathcal{T}} e^{i \theta_{t, \chi, v}} n^{-it - iv} \chi(n) \cdot \frac{dv}{1 + |v|}.
\end{align*}
By Cauchy's inequality and the bound $|\gamma_n| \leq d_{rk}(n) (1 + (\log n)^k)$, the above is
$$
\ll \Big ( (\log x)^{(rk)^2 + 2k} \Big )^{1/2} \cdot \Big ( \sum_{n \leq M} \Big | \int_{-M^{\delta}}^{M^{\delta}} \sum_{(t,\chi) \in \mathcal{T}} e^{i \theta_{t, \chi, v}} n^{- it - iv} \chi(n) \cdot \frac{dv}{1 + |v|} \Big |^2 \Big )^{1/2}.
$$
Expanding the square in the right-hand side, we get
\begin{equation} \label{eq:poisex}
\int_{-M^{\delta}}^{M^{\delta}} \int_{-M^{\delta}}^{M^{\delta}} \sum_{\substack{(t, \chi) \in \mathcal{T} \\ (t', \chi') \in \mathcal{T}}}
\sum_{n \leq M} n^{ - i (t - t') - i (v - u)} \chi(n) \overline{\chi'}(n) \cdot \frac{dv}{1 + |v|} \cdot \frac{du}{1 + |u|}.
\end{equation}
By the Poisson summation, the above is bounded by
\begin{align*}
M \sum_{\substack{(t, \chi) \in \mathcal{T} \\ (t', \chi) \in \mathcal{T}}}
\int_{-M^{\delta}}^{M^{\delta}} \int_{-M^{\delta}}^{M^{\delta}} \frac{1}{1 + |t - t' + v - u|}\cdot \frac{du}{1 + |u|} \cdot & \frac{dv}{1 + |v|} + O_{\varepsilon}((q T)^{1/2 + 1/100} R^2) \\ & \ll M R (\log T)^3
\end{align*}
since $M > (q T)^3$ and $R \leq T q$ by a trivial bound.
Plugging this into \eqref{eq:abouttogodual}, we obtain
$$
\frac{R}{\log x} \cdot \frac{\sqrt{M}}{(\log x)^{A k}} \ll_{\varepsilon} (\log x)^{(r k)^2 / 2 + k + 3} \cdot \sqrt{M} \sqrt{R}.
$$
Simplifying this inequality, yields
  $$
  R \ll (\log x)^2 \cdot (\log x)^{2 (rk)^2 + 2 A k + 2} \leq (\log x)^{32 r^2 / \varepsilon^2 + 8 A / \varepsilon + 4}.
  $$

\end{proof}

\subsection{Hybrid Huxley's theorem on a typical set of characters}

In this section we will establish a result of Huxley type on the set of $(t, \chi)$ lying outside of the exceptional set defined in the corollary below. The exceptional set defined in the lemma below is no longer required to be well-spaced.

\begin{corollary} \label{cor:exceptional}
  Let $\varepsilon, A, B, T, q, x > 0$ be given.
  Let $V$ be a smooth function, compactly supported in $[1,2]$ with $V^{(k)}(y) \ll_{k} (\log x)^{B k}$ for all $k \geq 1$ and $y \in \mathbb{R}$.
  Let $\mathcal{E}_{V}(A; T; q; x; \varepsilon)$ be the set of $(t, \chi)$ with $|t| \leq T$ and $\chi \pmod{q}$ for which there exists an $(q T)^{\varepsilon} < N < x$ (allowed to depend on $(t, \chi)$) such that either
  \begin{equation} \label{eq:largest}
  \Big | \sum_{n} \frac{\mu(n) \chi(n)}{n^{1/2 + it}} \cdot V \Big ( \frac{n}{N} \Big ) \Big |  \geq \frac{\sqrt{N}}{(\log x)^{A}} \text{ or } \Big | \sum_{n} \frac{f(n) \chi(n)}{n^{1/2 + it}} \cdot V \Big ( \frac{n}{N} \Big ) \Big | \geq \frac{\sqrt{N}}{(\log x)^{A}}
  \end{equation}
  with $f(n) = 1$ or $f(n) = \log n$. Then,
  $$\sum_{\chi} \int_{(t,\chi) \in \mathcal{E}_{V}(A; T; q; x; \varepsilon)}1\; dt \ll_{B} (\log x)^{32 / \varepsilon^2 + 8 A / \varepsilon + 4 + 2B}.
  $$
\end{corollary}
\begin{proof}

  First notice that
  $$
  \sum_{n} \frac{a(n) \chi(n)}{n^{1/2 + it}} \cdot V \Big ( \frac{n}{N} \Big ) = \frac{1}{2\pi} \int_{\mathbb{R}} \sum_{N \leq n \leq 2N} \frac{a(n) \chi(n)}{n^{1/2 + it + iu}} \cdot \widetilde{V}(i u) N^{i u} du,
  $$
  where $\widetilde{V}(s) := \int_{\mathbb{R}} V(y) y^{s - 1} dy$ is the Mellin transform of $V$. Second, the Mellin transform $\widetilde{V}(i u)$ has rapid decay and already for $|u| > (\log x)^{2B}$ it is bounded by $\ll_{A, \varepsilon, B} (\log x)^{-A}$. Therefore, one can restrict to $|u| \leq (\log x)^{2B}$. Consequently,
  $$\sum_{\chi} \int_{(t, \chi) \in \mathcal{E}_{V}(A; T; q; x; \varepsilon)} 1 dt \ll (\log x)^{2B} \sum_{\chi} \int_{(t, \chi) \in \mathcal{E}(A; T; q; x; \varepsilon)} 1\;dt,$$
  where $\mathcal{E}(A; T; q; x; \varepsilon)$ is the set of those $(t, \chi)$ at which either
  $$
  \Big | \sum_{N \leq n \leq 2N} \frac{\mu(n) \chi(n)}{n^{1/2 + it}} \Big |  \geq \frac{\sqrt{N}}{(\log x)^{A}} \text{ or } \Big | \sum_{N \leq n \leq 2N} \frac{f(n) \chi(n)}{n^{1/2 + it}} \Big | \geq \frac{\sqrt{N}}{(\log x)^{A}}
  $$
for some $N\in [(qT)^{\varepsilon}, x]$.
Let
\begin{equation}
  M_{N}(s, \chi) = \sum_{n \leq N} \mu(n) \chi(n) n^{-s} \text{ and } D_{i, N}(s) = \sum_{n \leq N} f_{i}(n) \chi(n) n^{-s}
\end{equation}
with $f_{1}(n) = 1$ and $f_{2}(n) = \log n$. Note that if for instance the left-hand side of \eqref{eq:largest} holds then either $|M_{N}(\tfrac 12 + it)| \gg \sqrt{N} (\log N)^{-A}$ or $|M_{2 N} (\tfrac 12 + it)| \gg \sqrt{N} (\log N)^{-A}$.

  Cover $[-T, T]$ by intervals $I$ of unit length. For each interval $I$ and character $\chi$, let $(t_{I}, \chi)$ denote the tuple that maximizes
  \begin{equation} \label{eq:max}
  \max(|D_{1, N}|, |D_{2,N}|, |M_{N}|, |D_{1, 2N}|, |D_{2,2N}|, |M_{2N}|) ( \tfrac 12 + it, \chi)
  \end{equation}
  as $t$ ranges over $I$ and $\chi$ ranges over all characters $\pmod{q}$.
  In the very unlikely case that there are two or more choices for $t_{I}$, we pick one arbitrarily.
  Therefore, for each $I$ and $\chi$ there is a unique $(t_{I}, \chi)$ that maximizes \eqref{eq:max}.

  Let $\mathcal{T}$ be the subset of $\{(t_{I}, \chi) : I  , \chi \pmod{q}\}$ for which \eqref{eq:max} is $\gg \sqrt{N} (\log N)^{-A}$. Then,
  $$
  \sum_{\chi} \int_{t \in \mathcal{E}(A; T; q; x; \varepsilon)} dt \leq \sum_{I} \sum_{(t_{I}, \chi) \in \mathcal{T}} \int_{\substack{t \in I}} dt \ll |\mathcal{T}|.
  $$
  Taking every other interval $I$, we can separate $\mathcal{T}$ into a union of two well-spaced sets $\mathcal{T}_{1} \cup \mathcal{T}_{2}$. Applying Lemma \ref{le:exceptional} to $\mathcal{T}_{1}, \mathcal{T}_{2}$, establishes the result.
  \end{proof}

The following lemma is a hybrid version of Huxley's theorem to large progressions and short intervals, under (essentially) the assumption of a good prime number theorem. This latter assumption is encapsulated in our requirement that $(t, \chi) \not \in \mathcal{E}_{V}(A; T; q; x; \varepsilon)$, where $\mathcal{E}_{V}(A; T; q; x; \varepsilon)$ is the same set as in Corollary \ref{cor:exceptional}.

\begin{lemma} \label{le:huxley}
  Let $x \geq 1$ be given.
Let
  $$
  P(s; \chi) = \sum_{p} \frac{\chi(p) \log p}{p^s} \cdot W \Big ( \frac{p}{x} \Big )
  $$
  with $W$ a smooth function such that $W^{(k)} (y) \ll_{k} \delta^{-k}$ for all $k \geq 1$ and $y \in \mathbb{R}$
  with $\delta \leq (\log x)^{750}$, and $W$ supported in $[1,2]$.
Let $T = 2 x (\log x)^{1000} / H$ and $\varepsilon \in (0, \tfrac{1}{10000})$ be given. Suppose that $(H / q) > x^{1/6 + \varepsilon}$ and $H \leq x$. Let $A > 4000 / \varepsilon $ be given. Then,
  \begin{equation} \label{eq:hb}
  \sum_{\substack{\chi \neq \chi_{0} \pmod{q}}} \int_{\substack{(t, \chi) \not \in \mathcal{E}(A; T; q; x; \varepsilon/1000) \\ |t| \leq T / 2}} |P(\tfrac 12 + it,\chi)|^2 \ dt \ll \frac{x}{(\log x)^{A / 2}},
  \end{equation}
where $\mathcal{E}(A; T; q; x; \varepsilon / 1000)$ is the same set as the set $\mathcal{E}_{W}(A; T; q; x; \varepsilon / 1000)$ defined in Corollary \ref{cor:exceptional}.
\end{lemma}
\begin{proof}

  Applying Heath-Brown's identity with $k = 3$ and $z = (H / q)^{2 - 3 \varepsilon / 4} > (2x)^{1/3}$, so that the coefficients of $(1 - \zeta(s) M(s))^{k}$ are zero on integers $n \in [1, 2x]$,
  allows us to write $P(\tfrac 12 + it, \chi)$ as a linear combination of Dirichlet polynomials of the form
  $$
  \sum_{\substack{\substack{n_1, \ldots, n_i \leq 2 x \\ n_4, \ldots, n_j \leq (H / q)^{2 - 3 \varepsilon / 4}}}} \frac{\mu(n_1) \ldots \mu(n_i) f(n_4)}{(n_1 \ldots n_i \cdot n_4 \ldots n_j)^{1/2 + it}} W \Big ( \frac{n_1 \ldots n_i \cdot n_4 \ldots n_j}{x} \Big )
  $$
  with $1 \leq i\leq 3, 4\leq j \leq 6$ and $f(n) = 1$ or $f(n) = \log n$. Let $V$ be a partition of unity, that is, a smooth compactly supported function with support in $[1, 2]$ and
  such that
  $$1 = \sum_{N \in \mathcal{N}} V \Big ( \frac{n}{N} \Big )$$
  for all integers $n \geq 1$ and with $N$ running along a set of integers $\mathcal{N}$ such that $|\mathcal{N} \cap [-X, X]| \ll \log X$ for any $X > 100$.
  We introduce such a partition of unity on each of the variables $n_1, \ldots, n_j$. Finally, we separate variables in $W(n_1 \ldots n_i \cdot n_4 \ldots n_j / x)$ by opening $W$ as a Mellin transform. As a result, we can bound $|P(\tfrac 12 + it)|$ as a linear combination of at most $(\log x)^{100}$ expressions of the form
  \begin{equation} \label{tobeintegrated}
  \int_{\mathbb{R}} |\widetilde{W}(i u)| \cdot | \prod_{i \in I} N_i(\tfrac 12 + iu + it)| du
\end{equation}
with $I$ a subset of $\{1, \ldots, 6\}$ and
  with
  \begin{align*}
    N_i(s, \chi) & := \sum_{n \sim N_i} \frac{f_i(n) \chi(n)}{n^{s}} \cdot V \Big ( \frac{n}{N_i} \Big ) \ , \ f_{i}(n) = \log n \text{ or } f_{i}(n) = 1 \ , \ 1 \leq i \leq 3 \\
    N_i(s, \chi) & := \sum_{n \sim N_i} \frac{\mu(n) \chi(n)}{n^s} \cdot V \Big ( \frac{n}{N_i} \Big ) \ , \ 4 \leq i \leq 6.
  \end{align*}
  and where $N_i \in \mathcal{N}$ are such that $N_1, N_2, N_3 \leq x$ and $N_4, N_5, N_6 \leq (H / q)^{2 - 3 \varepsilon / 4}$.

  Finally, since $\widehat{W}(u)$ decays rapidly starting with $|u| > (\log x)^{751}$, integrating \eqref{tobeintegrated} over $t$ and applying the Cauchy-Schwarz inequality allows us to remove the integration over $u$, at the price of increasing the integration over $|t| \leq T / 2$ to integration up to $|t| \leq T$ (since $T / 2 > (\log x)^{1000} > (\log x)^{751}$ always).

  These preliminary transformations allow us to bound \eqref{eq:hb} by
  \begin{equation} \label{eq:hb2}
  (\log x)^{1000} \sup_{\substack{N_4, N_5, N_{6} \leq (H / q)^{2 - 3 \varepsilon / 4} \\ N_1,N_2, N_{3} \leq x}} \sum_{\chi \neq \chi_{0}} \int_{\substack{(t, \chi) \not \in \mathcal{E}(A; T; q; x; \varepsilon/1000) \\ |t| \leq T}} |(N_1 \ldots N_6)(\tfrac 12 + it, \chi)|^2 dt,
  \end{equation}
  where $N_i \in \mathcal{N}$ and $N_i(s, \chi)$ are as above.
  We will now obtain a satisfactory bound for each of the possible cases.

\subsubsection{A first reduction} Suppose first that there exists an $N_i$ with $(x q / H)^{\varepsilon / 1000} < N_i < H / q$. In that case write $N_1 \ldots N_{6}(s, \chi)$ as $N_i(s, \chi) R(s, \chi)$ and apply an $L^{\infty}$ bound on $N_i(s, \chi)$ and the large sieve on $R(s,\chi)$. This shows that the contribution of such a term is
  $$
  \ll (\log x)^{2000} \cdot \frac{N_{i}}{(\log x)^{2A}} \cdot \Big ( \frac{x q}{H} \cdot (\log x)^{1000} + \frac{x}{N_i} \Big ) \ll \frac{x}{(\log x)^{2A - 3000}}
  $$
  and therefore, this is acceptable provided that $A$ is sufficiently large.  Thus, it remains to deal with a Dirichlet polynomial of the form $M(s,\chi) (\prod_{i \in I} N_i)(s,\chi)$, where $I$ is a subset of $\{1, \ldots, 6\}$ and $M(s,\chi)$  is of length at most $(x q / H)^{3\varepsilon/500}$ and $N_i$ are the same Dirichlet polynomials as before, but they have length $> H / q$.

  \subsubsection{The large values argument}

  Let $\mathcal{T}_{j, \sigma}$
 be the set of $(t, \chi)\not \in \mathcal{E}$ such that for $t \in \mathcal{T}_{j, \sigma}$,
  $$
  |N_{j}(\tfrac 12 + it, \chi)| \asymp N_j^{\sigma - \tfrac 12} \ , \ |N_{i}(\tfrac 12 + it, \chi)| \ll N_i^{\sigma  - \tfrac 12}
  $$
  for all $i \neq j$.
  Notice that since $(t, \chi) \not \in \mathcal{E}$, we have $\sigma \leq 1 - \frac{A \log\log x}{\log x}$.

  Moreover, by the pigeonhole principle, there exist  $1 \leq j \leq 6$ and $\tfrac 12 \leq \sigma \leq 1 - \frac{A \log\log x}{\log x}$ such that \eqref{eq:hb2} is bounded by
  \begin{equation} \label{eq:hb3}
    \ll (\log x)^{1001} \sum_{\chi \neq \chi_{0} \pmod{q}} \int_{\substack{|t| \leq x ( \log x)^{1000} / H \\ (t, \chi) \in \mathcal{T}_{j, \sigma}}} |M(\tfrac 12 +it, \chi) \prod_{i \in I} N_i(\tfrac 12 + it, \chi)|^2 dt,
  \end{equation}
  where $I$ is a subset of $\{1, \ldots, 6\}$.

  \subsubsection{The case $N_j \leq (H / q)^{2 - \varepsilon / 2}$ and $\sigma \leq \tfrac 34$}

  In this case, we bound the expression \eqref{eq:hb3} by an $L^{\infty}$ bound applied to $N_j(s, \chi)$ and an $L^2$ bound applied to the remaining Dirichlet polynomials. This shows that the contribution of this case is
  $$
  \ll N_j^{1/2} \cdot (\log x)^{2000} \cdot \Big ( \frac{q x}{H} + \frac{x}{N_j} \Big ) \ll (\log x)^{2000} \cdot \Big ( \frac{H}{q} \Big )^{1 - \varepsilon / 4} \cdot \frac{x q}{H} + (\log x)^{2000} \cdot \frac{x}{N_j^{1/2}}
  $$
  and therefore, we see that this is $\ll x^{1 - \varepsilon / 100}$ which is completely sufficient.

  \subsubsection{The case $N_j \leq (H / q)^{2 - \varepsilon / 2}$ and $\sigma > \tfrac 34$}

  In this case, we bound \eqref{eq:hb3} by
  \begin{equation} \label{eq:hb4}
  (\log x)^{2000} \cdot M \cdot (x / M)^{2 \sigma - 1} \cdot |\mathcal{T}_{j, \sigma}|,
  \end{equation}
  where $M \leq (x q / H)^{3 \varepsilon / 500}$ is the length of the Dirichlet polynomial $M(s, \chi)$ and
  where, without loss of generality, we can assume that the set $\mathcal{T}_{j, \sigma}$ is well-spaced (by first bounding the integral over this set by the local maxima). By Lemma \ref{le:huxleylargevalues} applied to the Dirichlet polynomial $N_j(s, \chi)^g$ with $g \in \mathbb{N}$, this is
  \begin{equation} \label{eq:hb101}
  \ll M \cdot (x/ M)^{2 \sigma - 1} \cdot (\log x)^{3002} \cdot \Big ( N_j^{(2 - 2 \sigma)g} + \frac{x q}{H} \cdot N_j^{(4 - 6 \sigma)g} \Big ).
  \end{equation}
  We choose $g$ so that
  $$
  x \Big ( \frac{q}{H} \Big )^{2} x^{\varepsilon / 24} \leq N_j^g \leq x^{1 - \varepsilon / 24} \text{ and } N_j^{g + 1} \geq x^{1 - \varepsilon / 24}.
  $$
  Such a choice is possible since $N_j \leq (H / q)^{2 - \varepsilon / 2} \leq (H / q)^{2} x^{-\varepsilon / 12}$.

  We now find that the first term in \eqref{eq:hb101} is
  $$
  \ll (\log x)^{3002} \cdot M N_{j}^g \Big ( \frac{x}{M N_j^{g}} \Big )^{2 \sigma - 1} \ll \frac{x}{(\log x)^{A / 2}},
  $$
  since the maximum is attained at $\sigma = 1 - \frac{A \log\log x}{\log x}$ and $x / M N_j^g \gg x^{\varepsilon / 1000}$.

  The contribution of the second term is (bounding by the value at $\sigma = \tfrac{3}{4}$ and $\sigma = 1$)
  $$
  \ll (\log x)^{3002} \cdot \Big ( \frac{x^2 q}{H} \cdot N_j^{-2 g} + \frac{x^{3/2} q \sqrt{M}}{H} N_j^{-g/2} \Big ) \ll x^{1 - \varepsilon / 1000},
  $$
  where in the first expression we used that $N_j^{2g} \geq N_j^{g + 1} \geq x^{1 - \varepsilon / 24}$ and in the second the inequality $x^{1 + \varepsilon / 24} \cdot (q / H)^2 \leq N_j^g$ and the fact that $\sqrt{M} x^{-\varepsilon / 48} \leq x^{- \varepsilon / 500}$.

  \subsubsection{The case $N_j > (H / q)^{2 - \varepsilon / 2}$ and $\sigma \leq \tfrac 34$}

  In this case, $N_j(s, \chi)$ must correspond to a polynomial with smooth coefficients (since all the Dirichlet polynomials with non-smooth coefficients are of length $\leq (H / q)^{2 - 3 \varepsilon / 4}$). In particular, we bound \eqref{eq:hb3} in the same way as in \eqref{eq:hb4} but now apply the stronger
  bound
  $$
  |\mathcal{T}_{j, \sigma}| \ll (\log x)^{10} \cdot \frac{x q}{H} \cdot N_j^{2 - 4 \sigma}
  $$
  which is a consequence of Lemma \ref{le:largevaluessmooth}. This gives us the bound
  $$
  \ll (\log x)^{3010} \cdot \frac{x q}{H} \cdot M (x / M)^{2 \sigma - 1} \cdot N_j^{2 - 4 \sigma}.
  $$
  Since $\sigma \leq \tfrac 34$, we get (evaluating the above at $\sigma = \tfrac 12$ and $\sigma = \tfrac 34$)
  $$
  \ll (\log x)^{3010} \cdot \frac{x q}{H} \cdot \Big ( M + M (x / M)^{1/2} \cdot N_j^{-1} \Big ) \ll x^{1 - \varepsilon / 100} + \sqrt{M} x^{3/2} \cdot \Big ( \frac{q}{H} \Big )^{3 - 3 \varepsilon / 4}
  $$
  and since $H / q > x^{1/6 + \varepsilon}$ and $M \leq x^{3 \varepsilon / 500}$, the above is
  $
  \ll x^{1 - \varepsilon / 100} + x^{1 - 2 \varepsilon}.
  $

  \subsubsection{The case $N_j > (x q / H)^{1/2 + \varepsilon}$}

  In this case, $N_j$ must once again correspond to a smooth Dirichlet polynomial. In particular, writing the Dirichlet polynomial $M(s, \chi) \prod_{i \in I} N_i(s, \chi)$ as $N_j(s, \chi) R(s, \chi)$ and applying the Cauchy-Schwarz inequality, we can bound \eqref{eq:hb2} as
  \begin{align*}
  \ll (\log x)^{2000} & \cdot \Big ( \sum_{\chi \neq \chi_{0} \pmod{q}} \int_{|t| \leq x (\log x)^{1000} / H} |N_j(\tfrac 12 + it, \chi)|^4 dt \Big )^{1/2} \\ & \times \Big ( \sum_{\chi \neq \chi_{0} \pmod{q}} \int_{|t| \leq x (\log x)^{1000} / H} |R(\tfrac 12 + it, \chi)|^4 dt \Big ) .
  \end{align*}
  By the large sieve and Lemma \ref{le:fourthmoment}, this is
  $$
  \ll (\log x)^{4000} \cdot \Big ( \frac{x q}{H} \Big )^{1/2} \cdot \Big ( \frac{x q}{H} + \Big ( \frac{x}{N_j} \Big )^2 \Big )^{1/2} \ll x^{1 - \varepsilon / 100}
  $$
 since $N_j > (x q / H)^{1/2 + \varepsilon}$.

  \subsubsection{The case $(H / q)^{2 - \varepsilon / 2} < N_j < (x q / H)^{1/2 + \varepsilon}$ and $\sigma > \tfrac 34$}

  We apply the bound of \eqref{eq:hb4} and then use Lemma \ref{le:huxleylargevalues} applied to $N_j(s, \chi)^2$ to see that \eqref{eq:hb3} is
  $$
  \ll M (x / M)^{2 \sigma - 1} \cdot (\log x)^{3002} \cdot \Big ( N_j^{2 (2 - 2 \sigma)} + \frac{x q}{H} \cdot N_j^{2 (4 - 6 \sigma)} \Big ).
  $$
  The contribution of the first term is
  $$
  \ll M N_j^2 \cdot \Big ( \frac{x}{N_j^2 M} \Big )^{2 \sigma - 1} \cdot (\log x)^{3002} \ll \frac{x}{\log^{A / 2} x}
  $$
  since $M N_j^2 \leq x^{1 - 1 / 12}$ and $\sigma \leq 1 - \frac{A \log\log x}{\log x}$. On the other hand, since $N_j > x^{1/3 - \varepsilon}$, the contribution of the second term is maximized at $\sigma = \tfrac{3}{4}$, and thus is
  $$
  \ll M (x / M)^{1/2} (\log x)^{3002} \cdot \frac{x q}{H} \cdot N_j^{-1} \ll \sqrt{M} x^{3/2} (\log x)^{3002} \cdot \Big ( \frac{q}{H} \Big )^{3 - \varepsilon / 2}
  $$
  and since $H / q > x^{1/6 + \varepsilon}$ and $M \leq (x q / H)^{3 \varepsilon / 500}$, this is
  $
\ll x^{1 - \varepsilon}
$
which is more than enough.

\end{proof}

\subsection{Lemma on character sums}

We will also need a number of results on character sums. For the proof of Theorem \ref{thm:nr1} we will need the following lemma which is a consequence of Poisson summation and the large sieve for additive characters. The proof is a little bit laborious due to our choice of using sharp cut-offs.
\begin{lemma}\label{lem:character}
For any $\chi \neq \chi_0 \pmod{q}$ and $(r, q) = 1$, we have
  \begin{equation} \label{eq:easy}
    \sum_{v = 1}^{r} \Big | \sum_{\substack{a < q \\ a \equiv v \pmod{r}}} \chi(a) \Big | \ll (r q)^{1/2} \cdot d(q) \log q.
  \end{equation}

\end{lemma}
\begin{remark}
Notice that this is essentially optimal as the best error term that we expect for the sum over $a$ is $\sqrt{q / r}$.
\end{remark}
\begin{proof}
  Let $\chi^{\star} \mod e$ be a primitive character inducing $\chi$, so that $q = e f$. Therefore,
  \begin{align*}
    \sum_{\substack{a < q \\ a \equiv v \pmod{r}}} \chi(a) = \sum_{\substack{a < q \\ (a, f) = 1 \\ a \equiv v \pmod{r}}} \chi^{\star}(a)
    = \sum_{d | f} \mu(d) \chi^{\star}(d) \sum_{\substack{a < q / d \\ d a \equiv v \pmod{r}}} \chi^{\star}(a).
  \end{align*}
  Therefore, \eqref{eq:easy} is bounded by
  $$
  \leq \sum_{d | f} \sum_{v = 1}^{r} \Big | \sum_{\substack{a < q / d \\ d a \equiv v \pmod{r}}}. \chi^{\star}(a) \Big |.
  $$
  Applying the Cauchy-Schwarz inequality, we bound this by
  $$
  r^{1/2} \sum_{d | f}  \Big ( \sum_{v = 1}^{r} \Big |  \sum_{\substack{a < q / d \\ d a \equiv v \pmod{r}}} \chi^{\star}(a) \Big |^2 \Big )^{1/2}.
  $$
  We now express the condition $d a \equiv v \pmod{r}$ using additive characters, so that the second inner sum is equal to
  \begin{equation}\label{eq:orthogonality}
  \sum_{v = 1}^{r} \Big | \frac{1}{r} \sum_{0 \leq \ell < r} e \Big ( - \frac{\ell v}{r} \Big ) \sum_{a < q / d} \chi^{\star}(a) e \Big ( \frac{d \ell a}{r} \Big ) \Big |^2 = \frac{1}{r} \sum_{0 \leq \ell < r} \Big | \sum_{a < q / d} \chi^{\star}(a) e \Big ( \frac{d \ell a }{r} \Big ) \Big |^2.
  \end{equation}
  We now use the completion method to write
  \begin{equation} \label{eq:completion}
  \sum_{a < q / d} \chi^{\star}(a) e \Big ( \frac{d \ell a}{r} \Big ) = \frac{q}{d e} \cdot \frac{1}{r} \sum_{x < e r} \Big ( \sum_{y < e r} \chi^{\star}(y) \exp \Big ( \frac{2\pi i y x}{e r} \Big ) e \Big ( \frac{d \ell y}{r} \Big ) \Big ) \cdot \frac{1}{q/d} \overline{\sum_{1 \leq y < q / d} \exp \Big ( \frac{2\pi i x y}{e r} \Big ) }.
  \end{equation}
  Writing $y = a e + b r$ with $a \mod{r}$ and $b \mod{e}$ (recall that $(r,e)=1$ since $e | q$ and $(r,q) = 1$), we find that the sum over $y < e r$, divided by $r$, is equal to
  \begin{align*}
  \frac{1}{r} \sum_{\substack{b \mod e}} \chi^{\star}(b r) \exp \Big ( & \frac{2 \pi i x b}{e} \Big ) \sum_{a \mod {r}} \exp \Big ( \frac{2\pi i d \ell a e}{r} + \frac{2 \pi i x a}{r} \Big ) \\ & =
  \mathbf{1}_{x \equiv - d \ell e \mod{r}} \cdot \chi^{\star}(r) \sum_{b \mod e} \chi^{\star}(b) \exp \Big ( \frac{2 \pi i x b}{e} \Big ) .
  \end{align*}
  Therefore, \eqref{eq:completion} can be re-written as
  $$
\frac{q}{de} \sum_{\substack{1 \leq x < e r \\ x \equiv - d \ell e \mod{r}}} \chi^{\star}(r) G_{\chi^{\star}}(x) F(x),
 $$
 where
  $$
  F(x) := \frac{1}{q/d} \overline{\sum_{1 \leq y < q / d} \exp \Big ( \frac{2 \pi i x y}{e r} \Big )} \ll \min \Big (1, \frac{d e r / q}{x} \Big ) \text{ and }
  G_{\chi^{\star}}(x) := \sum_{b \mod e} \chi^{\star}(b) \exp \Big ( \frac{2 \pi i b x}{e} \Big ).
  $$
  By \cite[Lemma 5.4]{Exceptional}, we have $G_{\chi^{\star}}(x) \ll \sqrt{e}$.
  Therefore, the above sum is
  $$
  \ll \sqrt{e} \cdot \frac{q}{d e} \sum_{\substack{1 \leq x < e r \\ x \equiv - d \ell e \mod{r}}} \min \Big ( 1, \frac{d e r / q}{x} \Big ).
  $$
  Thus, we get a total bound for \eqref{eq:easy} of
  $$
  \sum_{d | f} r^{1/2} \cdot \Big ( \frac{e}{r} \sum_{\ell = 1}^{r} \Big | \frac{q}{d e} \sum_{\substack{1 \leq x < e r \\ x \equiv - d \ell e \mod{r}}} \min \Big ( 1, \frac{d e r / q}{x} \Big ) \Big |^2 \Big )^{1/2}.
  $$
  Since $(d e , r) = 1$, we can re-write the above as
  \begin{align} \label{eq:blabla}
  \sum_{d | f} r^{1/2} \cdot \Big ( \frac{e}{r} \sum_{y = 1}^{r} \Big | \frac{q}{d e} \sum_{\substack{1 \leq x < e r \\ x \equiv y \mod{r}}} \min \Big ( 1, \frac{d e r / q}{x} \Big ) \Big |^2 \Big )^{1/2}.
  \end{align}
  If $d e / q \geq 1$ then
  $$
  \sum_{\substack{1 \leq x < e r \\ x \equiv y \mod{r}}} \min \Big ( 1 , \frac{d e r / q}{x} \Big ) \ll \frac{d e}{q} \cdot \log q
  $$
  and so \eqref{eq:blabla}
  is $\ll d(q) (r e)^{1/2} \log q$. Consider now the case of $d e / q < 1$  and $d e r / q \geq 1$.
  Splitting the sum over $x$ into sub-sums of length $d e r / q$, we can bound \eqref{eq:blabla} by
  $$
 \sum_{d | f}  r^{1/2} \cdot \Big ( \frac{q^2}{d^2 e r} \sum_{y = 1}^{r} \Big | \sum_{0 \leq \ell < q^2} \frac{1}{\ell + 1} \sum_{\substack{\ell d e r / q \leq x \leq (\ell + 1) d e r / q \\ x \equiv y \mod{r}}} 1  \Big |^2 \Big )^{1/2}.
  $$
  By Cauchy's inequality, this is less than
  $$
  \sum_{d | f} r^{1/2} \cdot \Big ( \frac{q^2}{d^2 e r} \cdot (\log q) \sum_{1 \leq \ell \leq q^2} \frac{1}{\ell} \sum_{1 \leq y \leq r} \Big | \sum_{\substack{\ell d e r / q \leq x \leq (\ell + 1) d e r / q  \\ x \equiv y \mod{r}}} 1 \Big |^2 \Big )^{1/2}.
  $$
  The sum over $x$ is now bounded by $\ll 1$ (since $d e / q \leq 1$) and the sum over $y$ is constrained to an interval containing at most $\frac{d e r}{q} \geq 1$ terms (since otherwise the sum over $x$ is empty). This gives a final bound of
  $$
  \ll \sum_{d | f} r^{1/2} \cdot \Big ( \frac{q^2}{d^2 e r} \cdot (\log q)^2 \cdot \frac{d e r}{q} \Big )^{1/2} \ll \sum_{d | f} \Big ( \frac{r q}{d} \Big )^{1/2} \cdot \log q
  $$
  which is sufficient.
  Finally, it remains to handle the case when $d e r / q \leq 1$. In this case, we can bound \eqref{eq:blabla} by
  \begin{align*}
  \sum_{d | f} r^{1/2} \cdot \Big ( \frac{e}{r} \sum_{y = 1}^{r} \Big | \sum_{\substack{0 \leq j \leq e}} \frac{1}{j + y / r} \Big |^2 \Big )^{1/2}
  & \ll \sum_{d | f} r^{1/2} \cdot \Big ( \frac{e}{r} \sum_{y = 1}^{r} \frac{r^2}{y^2} + e \log^2 q \Big )^{1/2} \\ & \ll \sum_{d | f} \Big ( r \sqrt{e} + ( e r)^{1/2} \log q \Big )
  \end{align*}
  and it remains to notice that since $d e r / q \leq 1$, we have $\sqrt{r} \leq \sqrt{q} / \sqrt{d e}$ and hence, $r \sqrt{e} \leq \sqrt{r} \sqrt{q / d}$ which is sufficient. The claim is therefore verified in all cases.

\end{proof}

  For the proof of Theorem \ref{thm:nr2}, we will also need the following estimate for character sums. The proof depends on van der Corput's inequality and the Weil bound.
    \begin{lemma} \label{le:character2}
Let $\varepsilon, \delta > 0$ and $q \geq 1$.   We have for $H' \ll q^{1/2 - \delta}$ and $\chi$ a non-principal character $\pmod{q}$,
$$ \sum_{z < q} \sup_{\substack{\beta \in \mathbb{R}}} \Big | \sum_{\substack{a < q \\ a \in [z, z + H']}} \chi(a) e(a \beta) \Big | \ll_{\varepsilon} q^{1 - \delta/4 + \varepsilon} H' + q H'^{3/4}.
     $$
  \end{lemma}

  \begin{proof}
    By the Cauchy-Schwarz inequality,  it suffices to bound
    $$
    q^{1/2} \cdot \Big ( \sum_{z < q} \sup_{\beta} \Big | \sum_{\substack{a < q \\ a \in [z, z + H']}} \chi(a) e(a \beta) \Big |^2 \Big )^{1/2}.
    $$
    By van der Corput's inequality \cite[Lemma 1]{MontgomeryLectures},
    $$
    \Big | \sum_{\substack{a < q \\ a \in [z, z + H']}} \chi(a) e(a \beta) \Big |^2 \ll H'^{3/2} + \frac{H'}{\sqrt{H'}} \sum_{\substack{0 < h < \sqrt{H'}}} \Big | \sum_{\substack{a < q \\ a \in [z, z + H']}} \chi(a) \overline{\chi(a + h)} \Big |.
    $$
    The first term gives a total contribution of $\ll q^{1/2} H'^{3/4}$. The second term gives a contribution of
    $$
    \ll q^{1/2} H'^{1/2} \Big ( \frac{1}{\sqrt{H'}} \sum_{0 < h < \sqrt{H'}} \sum_{z < q} \Big | \sum_{\substack{a < q \\ a \in [z, z + H']}} \chi(a) \overline{\chi}(a + h) \Big | \Big )^{1/2}.
    $$
    By another application of the Cauchy-Schwarz inequality, we get
    \begin{equation} \label{eq:beforebeforeweil}
    q^{3/4} H'^{1/2} \Big ( \frac{1}{\sqrt{H'}} \sum_{0 < h < \sqrt{H'}} \sum_{z < q} \Big | \sum_{\substack{a < q \\ a \in [z, z + H']}} \chi(a) \overline{\chi}(a + h) \Big |^2 \Big )^{1/4}.
    \end{equation}
    Expanding, we see that
    \begin{equation} \label{eq:beforeweil}
    \sum_{z < q} \Big | \sum_{\substack{a < q \\ a \in [z, z + H']}} \chi(a) \overline{\chi}(a + h) \Big |^2
     \ll q H' + \sum_{\substack{0 < a \neq a' < H' \\ z < q }} \chi(z + a) \overline{\chi}(z + a') \chi(z + a + h) \overline{\chi}(z + a' + h),
   \end{equation}
   where we alter the terms with $z > q - H'$ giving rise to an additional error $\ll H'^3 \ll q H'$.
     By \cite[Lemma 7]{Burgess},
     $$
     \sum_{z < q} \chi(z + a) \overline{\chi}(z + a') \chi(z + a + h) \overline{\chi}(z + a' + h) \ll 8^{\omega(q)} q^{1/2} (q, h (a - a') (a - a' + h)).
     $$
     Therefore, \eqref{eq:beforeweil} is
     $$
     \ll q H' + \sqrt{q} 8^{\omega(q)} \cdot H' \sum_{0 < w < H'} (q, h w (w + h)).
     $$
     We notice that any $n < H'^{5/2}$ has at most $q^{\varepsilon}$ representations as $h w (w + h)$. Thus,
     $$\frac{1}{\sqrt{H'}} \sum_{0 < h < \sqrt{H'}} \sum_{0 < w < H'} (q, h w (w + h)) \leq \frac{q^{\varepsilon}}{\sqrt{H'}} \sum_{0 < n < H'^{3/2}} (q, n). $$
     Splitting according to the possible values $d = (q, n)$, we find that the above is
     $$
     \ll \frac{q^{\varepsilon}}{\sqrt{H'}} \sum_{\substack{d | q \\ d \leq H'^{5/2}}} d \sum_{\substack{0 \leq n \leq H'^{5/2} \\ d | n}} 1 \ll \frac{q^{\varepsilon}}{\sqrt{H'}} \sum_{\substack{d | q \\ d \leq H'^{5/2}}} d \Big ( \frac{H'^{5/2}}{d} + 1 \Big ) \ll H'^2 q^{2 \varepsilon} + \min \Big ( H'^2, \frac{q}{\sqrt{H'}} \Big ) q^{2 \varepsilon}.
     $$
     Therefore,
     $$
     \frac{1}{\sqrt{H'}} \sum_{0 < h < \sqrt{H'}} \sum_{z < q} \Big | \sum_{\substack{a < q \\ a \in [z, z + H']}} \chi(a) \overline{\chi}(a + h) \Big |^2 \ll q H' + q^{1/2 + 3 \varepsilon}  H'^3 \ll q H' + q^{1 + 3 \varepsilon - \delta} H'^2.
     $$
     This gives rise to the final bound $\ll q^{1 + \varepsilon - \delta / 4} H' + q H'^{3/4}$
     which is sufficient.
   \end{proof}

  \subsection{Proof of Theorem \ref{thm:nr1}}

  We are now ready to prove Theorem \ref{thm:nr1}. This depends on a combination of Lemma \ref{le:huxley} and \ref{lem:character}.
  Notice first that in the range $q \leq (\log x)^{A}$ for any fixed $A > 0$, Theorem \ref{thm:nr1} is an immediate consequence of the following variant of Huxley's theorem due to Koukoulopoulos.
  \begin{theorem} \label{thm:koukoulopoulos}
  Let $A > 0$ and $\varepsilon > 0$ be given. Let $x \geq H > x^{1/6 + \varepsilon}$. Then, uniformly in $q \leq (\log x)^{A}$ and $(a,q) =1 $,
$$
\sum_{y < x} \Big | \sum_{\substack{p \in [y, y + H]  \\ p \equiv a \pmod{q}}} \log p - \frac{H}{\varphi(q)} \Big | \ll \frac{H x}{(\log x)^{A}}.
$$
  \end{theorem}
\begin{proof}
  This follows by taking $Q = (\log x)^{A}$ in \cite[Theorem 1.2]{Koukoulopoulos} and dropping all but one term.
\end{proof}
  Therefore, it suffices to establish the variant stated below.
  \begin{theorem}
    Let $\varepsilon, \xi \in (0, \tfrac{1}{1000})$. Let $C(\varepsilon) = 10^{100} / \varepsilon^3$. Suppose that $(H / q) > x^{1/6 + \varepsilon}$, $H \leq x$ and $(\log x)^{100 C(\varepsilon) / \xi} < q$. Then, for all $r \leq q^{1 - \xi}$ with $(r,q) = 1$, we have
    \begin{equation} \label{eq:thminit}
    \sum_{y < x} \sum_{v = 1}^{r} \Big | \sum_{\substack{p \in [y, y + H] \\ p_{q} \equiv v \pmod{r}}}\log p - \frac{H}{r} \Big | \ll_{\varepsilon, \xi} \frac{H x}{(\log x)^{100}}.
    \end{equation}
  \end{theorem}
  \begin{proof}

  Let $0 \leq W \leq 1$ be a smooth function such that $W(v) = 1$ for $0 \leq v \leq 1$, and compactly supported in $[- (\log x)^{-500}, 1 + (\log x)^{-500}]$ and such that $W^{(k)}(y) \ll_{k} (\log x)^{500k }$ for all $k \geq 1$ and all $y \in \mathbb{R}$. At the price of a negligible error term of size $\ll H x / (\log x)^{500}$, we can bound \eqref{eq:thminit} by
$$
\sum_{y < x} \sum_{v = 1}^{r} \Big | \sum_{\substack{p \in [y, y + H] \\ p_{q} \equiv v \pmod{r}}}\log p \cdot W \Big ( \frac{p}{x} \Big ) - \frac{H}{r} \Big | \ll_{\varepsilon, \xi} \frac{H x}{(\log x)^{100}}.
$$

    We start by expressing the congruence condition using characters, this allows to bound our main expression by
    \begin{align*}
   \sum_{y < x} \sum_{v = 1}^{r} \Big | \frac{1}{\varphi(q)} \sum_{\chi  \neq \chi_{0} \pmod{q}} \Big (  \sum_{p \in [y, y + H]} \chi(p) \log p W \Big ( \frac{p}{x} \Big ) \Big ) \Big ( \sum_{\substack{a < q \\ a \equiv v \pmod{r}}} \overline{\chi}(a) \Big ) \Big |  + \frac{x H}{(\log q)^{100}},
    \end{align*}
    where the contribution of the principal character is estimated using Theorem \ref{thm:koukoulopoulos} and the second part of Lemma \ref{lm:n6}.
    We now open the sum over primes using a contour integral, getting that
    $$
    \sum_{p \in [y, y + H]} \chi(p) \log p \cdot W \Big ( \frac{p}{x} \Big ) = \int_{|t| \leq x} \sum_{p} \frac{\chi(p) \log p}{p^{s}} \cdot W \Big ( \frac{p}{x} \Big ) \cdot \frac{(y + H)^{s} - y^{s}}{s} dt + O (1)
    $$
    with $s = \tfrac 12 + it$. The total contribution of the error term is $\ll q x \ll H x^{7/8}$ which is sufficient.
    Let
    $$
    P(s, \chi) = \sum_{p} \frac{\chi(p) \log p}{p^{s}} \cdot W \Big ( \frac{p}{x} \Big )
    $$
    and $\mathcal{E}_{W}(A; x; q; x; \varepsilon / 1000)$ with $A = 10^{10} / \varepsilon$ be the same set as in Corollary \ref{cor:exceptional}. We will abbreviate the notation by dropping the subscript $W$ from $\mathcal{E}_{W}$.

    We separate $(t, \chi)$ according to whether $(t, \chi) \in \mathcal{E}(A; x; q; x; \varepsilon / 1000)$ or $(t,\chi) \not \in \mathcal{E}(A; x; q; x; \varepsilon / 1000)$. In the first case, we notice that
    $$
    |P(\tfrac 12 + it, \chi)| \ll \sqrt{x}
    $$
    and that
    $$
    \sum_{y < x} \Big | \frac{(y + H)^{1/2 + it} - y^{1/2 + it}}{1/2 + it} \Big | \ll \sum_{y < x} \Big | \int_{y}^{y + H} x^{-1/2 + it} dx \Big | \ll H \sqrt{x}.
    $$
    Combining this with Lemma \ref{lem:character} and Corollary \ref{cor:exceptional}, we find that that contribution of the $(t, \chi) \in \mathcal{E}$ is
    \begin{align*}
    \ll \frac{1}{\varphi(q)} \sum_{\chi} & \int_{(t, \chi) \in \mathcal{E}(A; x; q; x; \varepsilon / 1000)} \sqrt{x} \cdot H \sqrt{x} \cdot dt \sup_{\chi \neq \chi_{0} \pmod{q}} \sum_{v = 1}^{r} \Big | \sum_{\substack{a < q \\ a \equiv v \pmod{r}}} \chi(a) \Big | \\ & \ll  \frac{x H}{\varphi(q)} \cdot (\log x)^{10^{20} / \varepsilon^2} \cdot \sqrt{r q} \cdot d(q) \log q.
    \end{align*}
    We therefore get
    $$
    \ll \sqrt{\frac{r}{q}} (\log x)^{C(\varepsilon) / 4} \cdot x H \cdot d(q) \log q
    $$
    and this gives an acceptable contribution since $r \leq q^{1 - \xi}$ and $q > (\log x)^{100 C(\varepsilon) / \xi}$.

    It now remains to handle the contribution of the non-exceptional $(t,\chi)$, that is, those $|t|\leq x$ and $\chi \pmod{q}$ for which $(t, \chi) \not \in \mathcal{E}(A; x; q; x; \varepsilon / 1000)$. Therefore, we need to bound
    $$
    \sum_{y < x} \sum_{v = 1}^{r} \Big | \frac{1}{\varphi(q)} \sum_{\chi \neq \chi_{0}} \Big ( \int_{\substack{|t| \leq x \\ (t, \chi)\not \in \mathcal{E}}} P(\tfrac 12 + it, \chi)\cdot  \frac{(y + H)^{1/2 + it} - y^{1/2 + it}}{1/2 + it} dt \Big ) \cdot \Big ( \sum_{\substack{a < q \\ a \equiv v \pmod{r}}} \overline{\chi}(a) \Big ) \Big |.
    $$
    We now introduce phases $\theta_{v,y} \in \mathbb{R}$ for which we can re-write the above expression as
    $$
    \sum_{y < x} \sum_{v = 1}^{r} e^{i \theta_{v,y}} \cdot \frac{1}{\varphi(q)} \sum_{\chi \neq \chi_{0}}\Big (  \int_{\substack{|t| \leq x \\ (t, \chi) \not \in \mathcal{E}}} P(\tfrac 12 + it, \chi) \cdot \frac{(y + H)^{1/2 + it} - y^{1/2 + it}}{1/2 + it} \Big ) \cdot \Big ( \sum_{\substack{a < q \\ a \equiv v \pmod{r}}} \overline{\chi}(a) \Big ).
    $$
    Notice that
    $$
    \sum_{v = 1}^{r} e^{i \theta_{v,y}} \sum_{\substack{a < q \\ a \equiv v \mod{r}}} \chi(a) = \sum_{a < q} \chi(a) e^{i \theta_{a \mod r, y}} = \sum_{a < q} \chi(a) c(a, y, r),
    $$
    where $c(a, y, r) = e^{i\theta_{a \mod r, y}}$ depend only on $a,r, y$ and have absolute value $1$.
    Therefore, we have re-written our main expression as
    $$
    \frac{1}{\varphi(q)} \sum_{\chi \neq \chi_{0}} \sum_{y< x} \Big ( \int_{\substack{|t| \leq x \\ (t, \chi) \not \in \mathcal{E}}} P(\tfrac 12 + it, \chi) \cdot \frac{(y + H)^{1/2 + it} - y^{1/2 + it}}{1/2 + it} dt \Big ) \cdot \Big ( \sum_{a < q} \overline{\chi}(a) c(a,y,r) \Big ).
    $$
    We now apply the Cauchy-Schwarz inequality and the large sieve which give us
    $$
    \sqrt{q x} \cdot \Big ( \frac{1}{\varphi(q)} \sum_{\chi \neq \chi_{0}} \sum_{y < x} \Big | \int_{\substack{|t| \leq x \\ (t, \chi) \not \in \mathcal{E}}} P(\tfrac 12 + it, \chi) \cdot \frac{(y + H)^{1/2 + it} - y^{1/2 + it}}{1/2 + it} dt \Big |^2 \Big )^{1/2}.
    $$
    By Lemma \ref{le:technical},
    \begin{align*}
   & \frac{1}{\varphi(q)} \sum_{\chi \neq \chi_{0} \pmod{q}} \sum_{y < x} \Big | \int_{\substack{|t| \leq x \\ (t, \chi) \not \in \mathcal{E}}} P(\tfrac 12 + it, \chi) \cdot \frac{(y + H)^{1/2 + it} - y^{1/2 + it}}{1/2 + it} dt \Big |^2 \\ & \ll \frac{H^2 \log x}{\varphi(q)} \sum_{\chi \neq \chi_{0}} \int_{\substack{|t| \leq (x / H) (\log x)^{1000} \\ (t, \chi) \not \in \mathcal{E}}} |P(\tfrac 12 + it. \chi)|^2 dt + \frac{H^2 x}{\varphi(q)(\log x)^{500}}.
    \end{align*}
    The error term gives an acceptable contribution.
    Therefore, we end up with the problem of showing that
    $$
    H \sqrt{q x \log x} \cdot \Big ( \frac{1}{\varphi(q)} \sum_{\chi \neq \chi_{0}} \int_{\substack{|t| \leq (x / H) (\log x)^{1000} \\ (t, \chi) \not \in \mathcal{E}}} |P(\tfrac 12 + it, \chi)|^2 dt \Big )^{1/2}
    $$
    is $\ll x H (\log x)^{-100}$.
    At this point appealing to Lemma \ref{le:huxley} (and using that $(t, \chi) \not \in \mathcal{E}(A;x;q;x;\varepsilon / 1000)$ implies $(t, \chi) \not \in \mathcal{E}(A; T ; q ; x;\varepsilon / 1000)$ with $T = 2 x (\log x)^{1000} / H$) gives a bound that is
    $
    \ll H \cdot x (\log x)^{-A / 4}
    $
    and this is completely sufficient.
  \end{proof}

  \subsection{Proof of Theorem \ref{thm:nr2}}
  We will only prove the first statement since the proof of the second assertion \eqref{Hrownex} will be identical.

  We can prove Theorem \ref{thm:nr2} by largely following the outline of the proof of Theorem \ref{thm:nr1} but using Lemma \ref{le:character2} instead of Lemma \ref{lem:character}. Once again if $q \leq (\log x)^{B}$ for some fixed $B > 0$ then Theorem \ref{thm:nr2} follows from Theorem \ref{thm:koukoulopoulos}. We quickly describe these details below.

  \begin{proof}[Proof of Theorem \ref{thm:nr2} for $q \leq (\log x)^{B}$]
    Since $q$ is small, $p_{q}$ takes on at most $(\log x)^{B}$ values. We can therefore, by the triangle inequality, separate the sum according to the value of $p \equiv w \pmod{q}$ which fixes the values $p_{q} = w$. This gives an upper bound of the form
    $$
    \sum_{y \leq x} \sum_{z \leq q} \sum_{\substack{0 \leq w < q \\ w \equiv v \pmod{r} \\ w \in [z, z + H']}} \Big | \sum_{\substack{p \in [y, y + H] \\ p \equiv w \pmod{q}}} \log p - \frac{H}{\varphi(q)} \Big |
    $$
    and the result is now an immediate consequence of Theorem \ref{thm:koukoulopoulos}.
\end{proof}

  Therefore, it suffices to prove the following slightly weaker variant.
  \begin{theorem}
    Let $\varepsilon \in (0 , \tfrac{1}{1000})$ be given. Let $C(\varepsilon) = 10^{100} / \varepsilon^3$. Suppose that $(H / q) > x^{1/6 + \varepsilon}$ and $H \leq x$. Then, for $q^{1/2 - 1/10} \geq  H' > (\log x)^{10 C(\varepsilon)}$, we have
    \begin{equation} \label{eq:mainthmeqeq}
    \sum_{y < x} \sum_{z < q} \sup_{\substack{\beta \in \mathbb{R} \\ 0 \leq v < r}} \Big | \sum_{\substack{p \in [y, y + H] \\ p_{q} \equiv v \pmod{r} \\ p_{q} \in [z, z + H']}} e(p_{q} \beta) \log p - \frac{H}{\varphi(q)} \sum_{\substack{(a,q) = 1 \\ 0 \leq a < q \\ a \equiv v \pmod{r} \\ a \in [z, z + H']}} e(a \beta) \Big | \ll_{A, \varepsilon} \frac{x H H'}{(\log x)^{100}}.
    \end{equation}
  \end{theorem}

  \begin{proof}

    As in the previous proof (i.e.\ proof of Theorem \ref{thm:nr1}) let $W$ be a smooth function such that $W(v) = 1$ for $v \in [0,1]$, $W$ is compactly supported in $[-(\log x)^{-500}, 1 + (\log x)^{-500}]$ and $W^{(k)} (y) \ll_{k} (\log x)^{500 k}$ for all $k \geq 1$ and all $y \in \mathbb{R}$. At the price of a negligible error term of size $\ll x H H' (\log x)^{-500}$, we can bound \eqref{eq:mainthmeqeq} by
    $$
    \sum_{y < x} \sum_{z < q} \sup_{\substack{\beta \in \mathbb{R} \\ 0 \leq v < r}} \Big | \sum_{\substack{p \in [y, y + H] \\ p_{q} \equiv v \pmod{r} \\ p_{q} \in [z, z + H']}} e(p_{q} \beta) \log p \cdot W \Big ( \frac{p}{x} \Big )  - \frac{H}{\varphi(q)} \sum_{\substack{(a,q) = 1 \\ a \equiv v \pmod{r} \\ a \in [z, z + H']}} e(a \beta) \Big |.
    $$
    We proceed just as before opening the expression into characters. This gives us the following bound
    $$
    \sum_{y < x} \sum_{z < q} \sup_{\substack{\beta \in \mathbb{R} \\ 0 \leq  v < r}} \Big | \frac{1}{\varphi(q)} \sum_{\chi \neq \chi_{0}} \Big ( \sum_{p \in [y, y + H]} \chi(p) \log p  W \Big ( \frac{p}{x} \Big ) \Big ) \Big ( \sum_{\substack{a < q \\ 0 \leq a < q \\ a \equiv v \pmod{r} \\ a \in [z, z + H']}} \chi(a) e(a \beta) \Big ) \Big | + \frac{x H H'}{(\log x)^{100}},
    $$
    where the contribution of the principal character is estimated using Theorem \ref{thm:koukoulopoulos}.
    Expressing the condition $a \equiv v \pmod{r}$ using additive characters and using the triangle inequality, we see that we can bound the above expression by
    $$
    \leq \sum_{y < x} \sum_{z < q} \sup_{\beta \in \mathbb{R}} \Big | \frac{1}{\varphi(q)} \sum_{\chi \neq \chi_0} \Big ( \sum_{p \in [y, y + H]} \chi(p) \log p \cdot W \Big ( \frac{p}{x} \Big ) \Big ) \cdot \Big ( \sum_{\substack{a < q \\ a \in [z , z + H']}} \chi(a) e(a \beta) \Big ) \Big |.
    $$
    Since we take a supremum over $\beta$, we can instead write $\beta$ as a function of $y$ and $z$ so that the above expression takes the form
    $$
    \sum_{y < x} \sum_{z < q} \Big | \frac{1}{\varphi(q)} \sum_{\chi \neq \chi_{0}} \Big ( \sum_{p \in [y, y + H]} \chi(p) \log p \cdot W \Big ( \frac{p}{x} \Big ) \Big ) \cdot \Big ( \sum_{\substack{a < q \\ a \in [z, z + H']}} \chi(a) e(a \beta_{y,z}) \Big ) \Big |.
    $$
    Finally, we pick phases $\theta_{y,z} \in \mathbb{R}$ for which the above can be re-written as
    \begin{equation} \label{eq:tobounder}
    \sum_{y < x} \sum_{z < q} e^{i\theta_{y,z}} \cdot \frac{1}{\varphi(q)} \sum_{\chi \neq \chi_{0}} \Big ( \sum_{p \in [y, y + H] } \chi(p) \log p W \Big ( \frac{p}{x} \Big ) \Big ) \cdot \Big ( \sum_{\substack{a < q \\ a \in [z, z + H']}} \chi(a) e(a \beta_{y,z}) \Big ) .
    \end{equation}
    Just as before, using a contour integral, we write
    $$
    \sum_{p \in [y, y + H]} \chi(p) \log p \cdot W \Big ( \frac{p}{x} \Big ) = \frac{1}{2\pi} \int_{|t| \leq x} P(\tfrac 12 + it, \chi) \cdot \frac{(y + H)^{1/2 + it} - y^{1/2 + it}}{1/2 + it} \cdot dt + O(1),
    $$
    where
    $$
    P(s, \chi) = \sum_{p} \frac{\chi(p) \log p}{p^{s}} \cdot W \Big ( \frac{p}{x} \Big ).
    $$
    The total contribution of the error term is $\ll q x H' \ll x^{7/8} H H'$ and therefore negligible.
    We now look at the contribution of $(t, \chi) \in \mathcal{E}(C; x; q; x; \varepsilon / 1000)$ with $C = 10^{10} / \varepsilon$.
    Using the bounds
    $
    |P(\tfrac 12 + it, \chi)| \ll \sqrt{x},
    $
    we see that the contribution of $(t, \chi) \in \mathcal{E}$ to \eqref{eq:tobounder} is
    \begin{equation} \label{eq:tobounder2}
    \ll \frac{1}{\varphi(q)} \sum_{\chi \neq \chi_{0}} \int_{(t, \chi) \in \mathcal{E}} \sqrt{x} \cdot \sum_{y \leq x} \Big | \frac{(y + H)^{1/2 + it} - y^{1/2 + it}}{1/2 + it} \Big | \cdot \sum_{z < q} \Big | \sum_{\substack{a \in [z, z + H'] \\ 0 \leq a < q }} \chi(a) e(a \beta_{y,z}) \Big | dt .
    \end{equation}
    Applying Lemma \ref{le:character2}, then the trivial bound,
    $$
    \sum_{y < x} \Big | \frac{(y + H)^{1/2 + it} - y^{1/2 + it}}{1/2 + it} \Big | \ll H \sqrt{x}
    $$
    and finally Corollary \ref{cor:exceptional}, we see that
    \eqref{eq:tobounder2} is
    $$
    \ll \frac{q}{\varphi(q)} \cdot x H H' \cdot \Big ( \frac{(\log x)^{C(\varepsilon) / 100}}{(H')^{1/4}} + (\log x)^{C(\varepsilon) / 100} q^{- 1 / 90} \Big )
    $$
    which is negligible since $q^{1/2 - 1/10} \geq H' > (\log x)^{10 C(\varepsilon)}$.

    Therefore, it remains to handle the contribution of the non-exceptional $(t, \chi)$, that is,
    \begin{align} \label{eq:tobounder3}
    \sum_{y < x} \frac{1}{\varphi(q)} \sum_{\chi \neq \chi_{0}} \int_{\substack{|t| \leq x \\ (t, \chi) \not \in \mathcal{E}}} P(\tfrac 12 + it, \chi) & \frac{(y + H)^{1/2 + it} - y^{1/2 + it}}{1/2 + it} dt \\ \nonumber & \times \sum_{0 \leq a < q} \chi(a) \Big ( \sum_{\substack{z \in [a - H', a] \\ z < q}} e^{i \theta_{z,y}} e(a \beta_{z,y}) \Big ).
    \end{align}
    Let
    $$
    c(a, y) := \sum_{\substack{z \in [a - H', a] \\ z < q}} e^{i\theta_{z,y}} e(a \beta_{z,y}).
    $$
    By an application of the Cauchy-Schwarz  inequality, \eqref{eq:tobounder3} is
    \begin{align*}
    \ll & \Big ( \frac{1}{\varphi(q)} \sum_{\chi \neq \chi_{0} \pmod{q}} \sum_{y < x} \Big | \int_{\substack{|t| \leq x \\ (t, \chi) \not \in \mathcal{E}}} P(\tfrac 12 + it, \chi) \cdot \frac{(y + H)^{1/2 + it} - y^{1/2 + it}}{1/2 + it} dt \Big |^2 \Big )^{1/2} \\ & \times \Big ( \sum_{y < x} \frac{1}{\varphi(q)} \sum_{\chi \neq \chi_{0} \pmod{q}} \Big | \sum_{a < q} \chi(a) c(a, y) \Big  |^2 \Big )^{1/2}.
    \end{align*}
    We estimate the second term by applying the large sieve. This shows that the second term is
    $$
    \ll x \sum_{\substack{a < q\\ (a, q) = 1}} |c(a, y)|^2 \ll x \varphi(q) (H')^2.
    $$
    To estimate the first term we appeal to Lemma \ref{le:technical}. This shows that
    \begin{align*}
    \frac{1}{\varphi(q)}  \sum_{\chi \neq \chi_{0} \pmod{q}} & \sum_{y < x} \Big | \int_{\substack{|t| \leq x \\ (t, \chi) \not \in \mathcal{E}}} P(\tfrac 12 + it, \chi) \cdot \frac{(y + H)^{1/2 + it} - y^{1/2 + it}}{1/2 + it} dt \Big |^2 \\ & \ll \frac{H^2 \log x}{\varphi(q)} \sum_{\chi \neq \chi_{0} \pmod{q}} \int_{\substack{|t| \leq x ( \log x)^{1000} / H \\ (t, \chi) \not \in \mathcal{E}}} |P(\tfrac 12 + it, \chi)|^2 dt + \frac{H^2 x}{\varphi(q) (\log x)^{500}}.
    \end{align*}
    By Lemma \ref{le:huxley}, this is
    $$
    \ll \frac{H^2 x}{\varphi(q) (\log x)^{500}}.
    $$
    Combining all our previous estimates we conclude that \eqref{eq:tobounder3} is
    $$
    \ll \frac{x H H'}{(\log x)^{250}}
    $$
    as needed.

  \end{proof}

\section{Extension of results of Matom\"aki-Shao} \label{sec:eqprimes2}

We will need the following extension of a recent result of Matom\"aki-Shao \cite{MS}.

\begin{theorem}\label{thm:nr4}
  Let $\tau > 0$, $\eta \in (0, 10^{-7})$ and $k \geq 1$ be given.
  Let $N > H > N^{2/3 - \eta}$. Then, for all $N > N_0(\eta, k)$, uniformly in $r \leq (\log N)^{100}$, $(a,r) = 1$ and uniformly in polynomials $g(n) = \sum_{i = 1}^{k} \gamma_i (n - N)^{i}$ with $|\gamma_i| \ll H^{-i + 1}$ for all $i = 2, \ldots, k$ and $|\gamma_1| \leq e^{-\tau r} \leq \eta^4$, we have
  \begin{equation} \label{eq:matoshaothm}
  \Big | \sum_{\substack{p \equiv a \pmod{r} \\ N \leq p \leq N + H}} e(g(p)) \log p  - \frac{1}{\varphi(r)} \sum_{N \leq n \leq N + H} e(g(n)) \Big | \ll \eta \log \frac{1}{\eta} \cdot \frac{H}{\varphi(r)}.
  \end{equation}
\end{theorem}

The proof separates into the oscillatory case in which the main term $\sum_{N \leq n \leq N + H} e(g(n))$ exhibits cancellations and the non-oscillatory case in which the main term is large.

\begin{proposition}[Oscillatory case] \label{prop:oscilatory} Let $\eta \in (0, 10^{-7})$ be given. Let $N > H > N^{2/3 - \eta}$. Let $g(n) = \sum_{i = 1}^{k} \gamma_i (n - N)^{k}$. If for all $q \leq (\log N)^{B}$, with $B$ sufficiently large in terms of $k$ and $1/\eta$, there exists an $i \in \{1, \ldots, k\}$ such that
  $$
  \| q \gamma_{i} \| \geq \frac{(\log N)^{B}}{H^i}
  $$
  then, for all $r \leq (\log N)^{100}$ and all $(a, r) = 1$,
  \begin{equation} \label{eq:matoshaoproper}
  \Big | \sum_{\substack{N \leq p \leq N + H \\ p \equiv a \pmod{r}}} e(g(p)) \log p \Big | \ll \eta \log \frac{1}{\eta} \cdot \frac{H}{\varphi(r)}
  \end{equation}
  and
  $$
  \Big | \sum_{\substack{N \leq n \leq N + H}} e(g(n)) \Big | \ll \frac{H}{(\log N)^{100}}.
  $$

\end{proposition}

In \cite{MS}, Matom\"aki-Shao obtain under the same assumptions, cancellations in the left-hand side of \eqref{eq:matoshaoproper} for $H > N^{2/3 + \varepsilon}$ and $r = 1$. In contrast to our Proposition \ref{prop:oscilatory}, they obtain savings of an arbitrary power of the logarithm. We push their result slightly past the $N^{2/3}$ threshold, but at the cost of much weaker, barely non-trivial error terms.

We now state the much easier ``non-oscillatory case''.

\begin{proposition}[Non-oscillatory case] \label{prop:nonoscilatory}
  Let $B, \tau > 0$ and $\eta \in (0, 10^{-7})$ be given. Let $N > H > N^{2/3 - \eta}$. 
  Then, for all $N$ sufficiently large with $B$, uniformly in polynomials $g(n) = \sum_{i = 1}^{k} \gamma_i (n - N)^{i}$ such that,
  $|\gamma_1| \leq e^{-\tau r} \leq \eta^4$, $|\gamma_i| \ll H^{-i + 1}$ for $i = 2, \ldots, k$
  and for which there exists a $q \leq (\log N)^{B}$ such that $\|q \gamma_i \| \leq (\log N)^{B} H^{-i}$ for all $i \leq k$,
  $$
  \Big | \sum_{\substack{N \leq p \leq N + H \\ p \equiv a \pmod{r}}} e(g(p)) \log p - \frac{1}{\varphi(r)}\sum_{\substack{N \leq n \leq N + H}} e(g(n)) \Big | \ll \frac{\eta H}{\varphi(r)}
  $$
  for all $r \leq (\log N)^{100}$ and $(a,r) = 1$.
\end{proposition}

This is a simple consequence of the Siegel-Walfisz theorem in short intervals.

\begin{proposition}[Siegel-Walfisz in short intervals] \label{prop:siegelwalfiszshort}
  Let $\varepsilon > 0$ and $A > 0$ be given. Then, for $(a, r) = 1$, $r \leq (\log N)^{A}$ and $H > N^{7/12 + \varepsilon}$,
  $$
  \sum_{\substack{N \leq p \leq N + H \\ p \equiv a \pmod{r}}} \log p = \frac{H}{\varphi(r)} + O_{A, \varepsilon} \Big ( \frac{H}{\varphi(r) (\log N)^{A}} \Big ).
  $$
  \end{proposition}
  \begin{proof}
    This follows from setting $Q = (\log x)^{A}$ in the main result of \cite{Perelli}.
    \end{proof}

In the ``non-oscillatory case'' an additional assumption on the size of the coefficients of the polynomial $g(n)$ is important, since for example the conclusion of Proposition \ref{prop:nonoscilatory} fails for the polynomial $g(n) = (n - N) / 2$.

\subsection{The Type-I and Type-II information}

The proof of Proposition \ref{prop:oscilatory} will largely rely on the type-I and type-II information obtained by Matom\"aki-Shao in \cite{MS}. We will need slight generalizations of these type-I and type-II estimates to allow for an extra congruence condition. We quickly sketch below the necessary modifications in this subsection. We broke down the results of \cite{MS} into many smaller propositions to make checking simpler. Throughout, given a sequence $\{\alpha_n\}$, we will use the notation
$$
\| \alpha \|_{p} := \Big ( \sum_{M \leq m \leq 2M} |\alpha_m|^{p} \Big )^{1/p}
$$
to denote its $L^p$ norm.

First, we will need the following variant of the Weyl bound.

  \begin{lemma}\label{le:weyl}
    Let $k > 0$ and $g(n) = \sum_{i = 1}^{k} \gamma_{i} (n - N)^{i}$.
    If for all $q \leq (\log N)^{B}$ with $B$ sufficiently large in terms of $k$ there exists an $i \in \{1, \ldots, k\}$ such that
    $$
    \| q \gamma_i \| \geq \frac{(\log N)^{B}}{H^{i}}
    $$
    then, for all $r \leq (\log N)^{100}$, and all $0 \leq a < r$,
    $$
    \sum_{\substack{N \leq n \leq N + H \\ n \equiv a \pmod{r}}} e(g(n)) \ll \frac{H}{(\log N)^{1000}}.
    $$
  \end{lemma}
  \begin{proof}
    Pick $C$ sufficiently large in terms of $k$
    so that if for every $q \leq (\log N)^{C}$ there exists an $i \in \{1, \ldots, k\}$ such that
    $\| q \gamma_{i} \| > (\log N)^{C} / H^i$ then
    $$
    \sum_{N \leq n \leq N + H} e(g(n)) \ll \frac{H}{(\log N)^{1000}}.
    $$
    The existence of such a $C > 0$ follows from Weyl's bound  (see \cite[Theorem 2 in Chapter 2]{MontgomeryLectures}).
    We claim that $B = 2C + 100$ is admissible.

    We express the condition $n \equiv a \pmod{r}$ using additive characters so that it is enough to bound
    $$
    \sup_{0 \leq \ell < r} \Big | \sum_{N \leq n \leq N + H} e \Big ( g(n) + \frac{n \ell}{r} \Big ) \Big |.
    $$

    Let $(\ell, r)$ be the tuple that maximizes the above expression. If for all $q \leq (\log N)^{C}$ we have
    $\| q (\gamma_1 + \ell / r) | \geq (\log N)^{C} / H$ then we are done by taking $B = 2C + 100$ and using the Weyl's bound as above.

    Suppose therefore that for some $q \leq (\log N)^{C}$ and $0 \leq \ell < r$, we have
    \begin{equation} \label{eq:weyl}
    \| q (\gamma_1 + \ell / r) \| \leq (\log N)^{C} / H,
    \end{equation}
    then $\gamma_1 + \ell / r = a / q + O((\log N)^{C} / H)$, hence $\gamma_1 = a' / (r q) + O((\log N)^{C} / H)$.
    This however would imply that
    $\| q \gamma_{1} \| \leq (\log N)^{2C + 100} / H$ for some $q \leq (\log N)^{C + 100}$.
    So taking $B = 2C + 100$, it follows that if for each $q \leq (\log N)^{B}$ there exists an $i \in \{1, \ldots, k\}$ such
    that $\| q \gamma_{i} \| > (\log N)^{B} /H^i$ then \eqref{eq:weyl} cannot hold, and hence, by the Weyl bound, we obtain a
    saving of $H (\log N)^{-1000}$.

  \end{proof}

With this lemma in hand we begin with the type-I information.

\begin{proposition} \label{prop:typeI}

   Let $A > 1000$, $\eta \in (0, 10^{-6})$, $k \geq 1$ and $N^{2/3 - \eta} \leq H \leq N$ be given.
  Let $g(n) = \sum_{i = 1}^{k} \gamma_i (n - N)^{i}$ be a polynomial of degree $k \geq 1$. Let $f(\ell) = 1$ or $f(\ell) = \log \ell$.
  Suppose that $M \leq H (\log N)^{- B}$ for some $B$ sufficiently large with respect to $A$ and $k$. Then there exist a constant $C > 0$ sufficiently large with respect to $A$ and $k$ such that if for all $q \leq (\log N)^{C}$ there exists an $i \in \{1, \ldots, k\}$ such that
  $$
  \| q \gamma_{i} \| \geq \frac{(\log N)^{C}}{H^i},
  $$
  then, for all $r \leq (\log N)^{100}$,  complex coefficients $\alpha_m$ supported on $[M, 2M]$ and $(a,r) = 1$,
  $$
  \sum_{\substack{ m \sim M \\ N \leq \ell m \leq N + H \\ \ell  m \equiv a \pmod{r}}} e(g(\ell m)) \alpha_{m} f(\ell) \ll \frac{H / \sqrt{M}}{(\log N)^{A}} \cdot \| \alpha \|_{2}.
  $$
\end{proposition}
\begin{proof}
  Following Matom\"aki-Shao, we write the sum as
  $$
  \sum_{m \sim M} \alpha_m \sum_{\substack{N / m \leq \ell \leq N / m + H / m \\ \ell m \equiv a \pmod{r}}} e(g(\ell m)) f(\ell)
  $$
  and we apply the Cauchy-Schwarz inequality which leads to the problem of bounding
  $$
  \Big ( \sum_{m \sim M} |\alpha_m|^2 \Big )^{1/2} \cdot \Big ( \sum_{m \sim M} \Big | \sum_{\substack{N / m \leq \ell \leq N / m + H / m \\ \ell m \equiv a \pmod{r}}} e(g(\ell m)) f(\ell) \Big |^2 \Big )^{1/2}
  $$
  From here on, we proceed in the same way as Matom\"aki-Shao starting with the second display of the proof of their Proposition 2.1, with the only difference that we use Lemma \ref{le:weyl} instead of their Lemma 3.1.
  \end{proof}

  We will also need information on the type-II sums,
  $$
  \sum_{N \leq \ell m \leq N + H} \alpha_{\ell} \beta_{m} e(g(n m)).
  $$
  \begin{proposition} \label{prop:typeII}
    Let $A > 1000$, $\eta \in (0, 10^{-6})$ and $N^{2/3 - \eta} \leq H \leq N$ be given.
    Suppose that $\max(N / M, M) \leq H (\log N)^{-B}$ for $B > 1000$ sufficiently large with $A$ and $k$.
    Suppose that for all $q \leq (\log N)^{C}$ with $C$ sufficiently large with respect to $A$ and $k$,
    there is an $i \in \{1, \ldots, k\}$ such that\footnote{We set $\gamma_{k + 1} = 0$}
    $$
    \| q (i \gamma_{i} + (i + 1) N \gamma_{i + 1}) \| \geq \frac{(\log N)^{C}}{H^i}.
    $$
    Then, for any sequence of complex numbers $\{\alpha_{m}\}$ and $\{\beta_{n}\}$ supported on respectively $[M, 2M]$ and $[N / 4M, 4 N / M]$, we have, uniformly in $r \leq (\log N)^{100}$, $(a,r) = 1$,
    \begin{align*}
    \sum_{\substack{M \leq m \leq 2M \\ N \leq m n \leq N + H \\ m n \equiv a \pmod{r}}} & \alpha_m \beta_{n} e(g(m n)) \ll \frac{H}{(\log N)^{A}} \cdot \frac{M^{1/4}}{N^{1/2}} \cdot \| \alpha \|_{4} \cdot \| \beta \|_{2}.
    \end{align*}
  \end{proposition}
  \begin{proof}
    Expressing the condition $\ell m \equiv a \pmod{r}$ using Dirichlet characters, we see that it is enough to bound
    $$
    \sup_{\chi \pmod{r}} \Big | \sum_{\substack{m \sim M \\ N \leq \ell m \leq N + H}} \alpha_{\ell} \chi(\ell) \beta_m \chi(m) e(g(\ell m)) \Big |.
    $$
    The result now follows by going through the proof of Proposition 2.2 in \cite{MS} with $\delta = (\log N)^{-A - 32}$.
  \end{proof}

  The diophantine condition in Proposition \ref{prop:typeII} excludes from consideration those $g$ for which
  \begin{equation} \label{eq:bad}
  \Big | \sum_{N \leq \ell m \leq N + H} \alpha_{\ell} \beta_{m} e (g (\ell m)) \Big | \approx \Big | \sum_{N \leq \ell m \leq N + H} \alpha_{\ell} \beta_{m} (\ell m)^{it} \Big |
  \end{equation}
  for some $|t| \leq N^{k + 1} / H^{k + 2}$. In order to handle these remaining cases, we need additional information on the sequences $\{\alpha_{\ell}\}$ and $\{\beta_{\ell}\}$. In particular, we will assume that either $\alpha_{\ell}$ or $\beta_{\ell}$ admits a bilinear structure. First let us establish a rigorous version of \eqref{eq:bad}. This result is implicit in \cite{MS}.

  \begin{proposition} \label{prop:preHarman}
    Let $A > 1000$, $D > 1000 A$, $\eta \in (0, 10^{-6})$, $k \geq 1$ and $N^{2/3 - \eta} \leq H \leq N$ be given.
    Let $g(n) = \sum_{i = 1}^{k} \gamma_i (n - N)^{i}$ be a polynomial of degree $k$.
    Suppose that for some $C$, there exists a $q \leq (\log N)^{C}$ such that for all $i \in \{1, \ldots, k\},$\footnote{We set $\gamma_{k + 1} = 0$.}
    $$
    \| q (i \gamma_i + (i + 1) N \gamma_{i + 1}) \| \leq \frac{(\log N)^{C}}{H^i}.
    $$
    Let $\{\alpha_{\ell}\}$ be a sequence of complex numbers supported on integers $\ell$ not having prime factors $\leq k! (\log N)^{C}$.
    Then, for $N$ sufficiently large with respect to $C$ and $k$, either of the following holds:
    \begin{enumerate} \item
      There exists $B > 0$ sufficiently large with respect to $A,C$ and $k$ such that for $H' = H (\log N)^{-B}$ and all $r \leq (\log N)^{100}$, $(a,r) = 1$,
    \begin{align*}
      \Big | \sum_{\substack{N \leq n \leq N + H \\ n \equiv a \pmod{r}}} \alpha_{n} e(g(n)) \Big | \ll k! (\log N)^{100} \cdot \frac{H}{H'} \sup_{\substack{\chi \pmod{k! q r}  \\ N \leq N' \leq N + H - H' \\ N (\log N)^{D} / H' \leq |t| \leq (N / H)^{k + 2}}} & \Big |\sum_{\substack{ N' \leq n \leq N' + H'}} \alpha_n \chi(n) n^{it} \Big | \\ & + \frac{H}{(\log N)^{A}} \sum_{N \leq n \leq N + H} |\alpha_n|.
    \end{align*}
  \item There exists an $E$ sufficiently large with respect to $A, C, D$ and $k$, and a $q \leq (\log N)^{E}$ such that for all $i \in \{1, \ldots, k\}$, we have $\| q \gamma_i \| \leq (\log N)^{E} / H^j$.
    \end{enumerate}
  \end{proposition}

  \begin{proof}

    By the triangle inequality,
    $$
    \Big | \sum_{\substack{N \leq n \leq N + H \\ n \equiv a \pmod{r}}} \alpha_n e(g(n)) \Big | \leq \sum_{(v, k! q) = 1} \Big | \sum_{\substack{N \leq n \leq N + H \\ n \equiv a \pmod{r} \\ n \equiv v \pmod{k! q}}} \alpha_n e(g(n)) \Big |
    $$
    because $\alpha_n$ is supported on integers having no prime factors $\leq k! (\log N)^{C}$. Cover $[N, N + H]$ with $\ll (\log N)^{B}$ disjoint short intervals $I$ of length $H'$, we bound the above expression by
    $$
    \sum_{(v,k! q) = 1} \sum_{I} \Big | \sum_{\substack{n \in I \\ n \equiv a \pmod{r} \\ n \equiv v \pmod{k! q}}} \alpha_n e(g(n)) \Big |.
    $$
    By the argument in \cite{MS} following equation (4.2), given such an interval $I = [N', N' + H']$ of length $H'$ and given $(v, k! q) = 1$, we have for all
    $n \in I$, and $n\equiv v \pmod{k! q}$,
    $$
    e(g(n)) = \nu n^{it} + O((\log N)^{-A - C} / k!)
    $$
    provided that $B$ is taken sufficiently large with respect to $A,C$ and $k$, and where $|\nu| = 1$, $t = 2 \pi N' (\beta_1 + a/q)$ for some $a \in \mathbb{Z}$ with $|t| \leq (\log N)^{O(C)} (N / H)^{k + 1}$, and with the coefficients $\beta_i$ defined by
    \begin{equation} \label{eq:binomeasy}
    \gamma_j = \sum_{i = j}^{k} \binom{i}{j} (N - N')^{i - j} \beta_i.
    \end{equation}
    In particular, it follows from this that if $|t| \leq (\log N)^{D} N / H' = (\log N)^{D + B} N / H$ then
    $$
    \| q \beta_1 \| \ll \frac{(\log N)^{D + B}}{H}.
    $$
    And since (as shown in equation (4.4) of \cite{MS}) for all $j \in \{1, \ldots, k\}$,
    $$
    \Big | \Big ( \beta_j + \frac{a_j}{q j} \Big ) - \frac{(-1)^{j - 1}}{j N'^{j - 1}} \Big ( \beta_1 + \frac{a}{q} \Big ) \Big | \leq (\log N)^{O(C)} \cdot H^{-j}
    $$
    for some $a_j \in \mathbb{Z}$, it follows from $t = 2\pi N' (\beta_1 + a/q)$ and the assumption $|t| \leq (\log N)^{D + B} N / H$ that for $j \in \{2, \ldots, k\}$,
    $$
    \| k! q \beta_j \| \leq \frac{(\log N)^{O(C)}}{H^j} + \frac{(\log N)^{D + B}}{N^{j - 1} H}.
    $$
    Finally, from \eqref{eq:binomeasy}, we get for all $j \in \{1, \ldots, k\}$,
    $$
    \| k! q \gamma_j \| \leq  2^k \frac{(\log N)^{D + O(C) + B}}{H^j}.
    $$
    Picking $E = D + K C + B + \log k$ with $K$ sufficiently large, we note that $E$ depends on $A,D,C$ and $k$,
    and that there exists a $q \leq (\log N)^{E}$ such that
    $$
    \| q \gamma_j \| \leq \frac{(\log N)^{E}}{H^j}
    $$
    for all $j \in \{1, \ldots, k\}$.

  \end{proof}

  Finally, to rule out the possibility that the bilinear form $\alpha_m \beta_{n}$ resonates with $(m n)^{it}$, we will use the following result of Baker, Harman and Pintz. Note that in order to apply it, one of the sequences $\{\alpha_{m}\}$ or $\{\beta_{n}\}$ appearing in \eqref{eq:bad} needs to have an additional bilinear structure.

   \begin{proposition} \label{prop:Harman}
    Let $A > 1000$, $D > 1000A$, $k > 0$ and $\eta \in (0, 2 \times 10^{-6})$ be given. Let $N^{2/3 - \eta} \leq H \leq N$.
    Let $\{\alpha_{k}\}, \{\beta_{\ell}\}, \{\gamma_{v}\}$ be three sequences of complex numbers supported respectively on $[K, 2K]$, $[L, 2L]$ and $[V, 2V]$ with $K L V \asymp N$.
    Suppose that, for $|u| \leq N (\log N)^{D / 2} / H$, we have
    $$
    \Big | \sum_{\substack{V \leq v \leq 2V}} \frac{\gamma_{v}}{v^{1/2 + iu}} \Big | \ll (\log N)^{-10 A} \Big ( \sum_{V \leq v \leq 2V} |\gamma_{v}|^2 \Big )^{1/2}.
    $$
    Suppose that $\max(K / L, L / K) \leq N^{1/3 - 3 \eta}$ and $V \leq N^{5/9 - 2 \eta}$, then
    $$
    \sum_{\substack{K \leq k \leq 2K \\ L \leq \ell \leq 2L \\ V \leq v \leq 2V \\ N \leq k \ell v \leq N + H}} \alpha_{k} \beta_{\ell} \gamma_{v} \ll \frac{H}{(\log N)^{A}} \cdot \frac{1}{\sqrt{N}} \cdot \| \alpha \|_{2} \| \beta \|_{2} \| \gamma \|_{2}.
    $$
  \end{proposition}
  \begin{proof}
    This follows from the case $g = 1$ of \cite[Lemma 7.3]{Harman} (alternatively see \cite[Lemma 9]{BHP}) since for $\theta > 2/3 - \eta$,
    $$
    \gamma := \min \Big ( 4 \theta - 2, \frac{4 \theta - 1}{3} , \frac{24 \theta - 13}{3} \Big ) = \frac{4 \theta - 1}{3} \geq \frac{5}{9} - 2 \eta.
    $$
    See also for e.g.\ \cite[Lemma 2.3]{MS} for the details of this deduction. Note that \cite[Lemma 2.3]{MS} is more restrictive than necessary and stated with the exponent $\tfrac{4}{9}$ instead of the exponent $\tfrac{5}{9}$.
  \end{proof}

  \subsection{The oscillatory case}

In this subsection we will prove Proposition \ref{prop:nonoscilatory}.
 Therefore, we will assume that for all $q \leq (\log N)^{B}$, with $B$ sufficiently large with respect to $k$, there exists an index $i \in \{1, \ldots, k\}$ such that
  \begin{equation} \label{eq:diophantine}
  \| q \gamma_{i} \| \gg \frac{(\log N)^{B}}{H^{i}}
\end{equation}
and where $\gamma_i$ are coefficients of the polynomial $g(n) = \sum_{i = 1}^{k} \gamma_i (n - N)^{i}$.
  In this situation, if $B$ is sufficiently large in terms of $k$ then it follows from Lemma \ref{le:weyl} that for all $r \leq (\log N)^{100}$ and $0 \leq a < r$,
  $$
  \sum_{\substack{N \leq n \leq N + H \\ n \equiv a \pmod{r}}} e (g(n)) \ll \frac{N}{(\log N)^{1000}}.
  $$
  Removing the $\log p$ weight, it therefore remains to show that if $B$ in \eqref{eq:diophantine} is sufficiently large, then
  \begin{equation} \label{eq:toboundshao}
  \sum_{\substack{N \leq p \leq N + H \\ p \equiv a \pmod{r}}} e(g(p)) \ll \eta \log \frac{1}{\eta} \cdot \frac{H}{\varphi(r) \log N}.
  \end{equation}
  Notice that we can assume that $N^{2/3 - \eta} \leq H \leq N^{2/3 + \eta}$. If $H > N^{2/3 + \eta}$ then the conclusion follows from \cite[Theorem 1.3]{MS}.

  Taking $z = N^{1/3 + 100 \eta}$ and using Linnik's identity (Lemma \ref{le:linnik}), we bound the left-hand side of \eqref{eq:toboundshao} by
  \begin{equation} \label{eq:linnik}
  \Big | \sum_{\substack{N \leq n \leq N + H \\ p | n \implies p > z \\ n \equiv a \pmod{r}}} e(g(n)) \Big | + \Big | \sum_{\substack{N \leq n m \leq N + H \\ p | n,m \implies p > z \\ n,m > z \\ n m \equiv a \pmod{r}}} e(g(n m)) \Big | + O(N^{2/3 - 2 \eta}),
\end{equation}
where the $O(N^{2/3 - 2 \eta})$ accounts for the modifications on integers $n$ with $n = p^{\alpha}$ with $\alpha \geq 2$.
  We notice that the second sum falls exactly within the scope of applicability of Proposition \ref{prop:typeII}. Indeed, we can write this sum as a linear combination of expressions of the form
  $$
  \sum_{\substack{m \sim M \\ N \leq n m \leq N + H \\ n,m > z \\ n m \equiv a \pmod{r}}} \alpha_n \alpha_m e(g(n m))
  $$
  with $\sqrt{N} < M < N / z = N^{2/3 - 100 \eta}$ and $\alpha_n$ the indicator function of integers $n$ such that $p | n \implies p > z$. Since $z^2 > M$, this means that $\alpha_n$ is in fact the indicator function of prime numbers.

  The next lemma establishes cancellations in this bilinear sum.
  \begin{lemma}
    Let $\eta \in (0, 10^{-7})$.
    Let $\sqrt{N} < M < N^{2/3 - 100 \eta}$. Then, for $N^{2/3 - \eta} \leq H \leq N^{2/3 + \eta}$,
    $$
    \sum_{\substack{p \sim M \\ N \leq p q \leq N + H \\ p q \equiv a \pmod{r}}} e(g (p q)) \ll \frac{H}{(\log N)^{1000}}
    $$
    provided that $B$ in \eqref{eq:diophantine} is taken to be sufficiently large.
  \end{lemma}
  \begin{proof}
    By the integration by parts, we see that it is enough to prove the same result for a sum weighted by $\log p$ and $\log q$.
    Applying Vaughan's identity (Lemma \ref{le:vaughan}) reduces the problem to bounding type-I and type-II sums. The type-I sums are of the form
    $$
    \sum_{\substack{n m \sim M \\ N \leq m n q \leq N + H \\ m \leq z^2 \\ m n q \equiv a \pmod{r}}} e(g(n m q)) \beta_{m} f(n) \log q
    $$
    for some divisor-bounded coefficients $\beta_m$, with $z := N^{1/12 - \eta - \eta^2}$ and with $f(n) = 1$ or $f(n) = \log n$.
    By Proposition \ref{prop:typeI}, this is $\ll_{C} H (\log N)^{-10^6}$ provided that $B$ in \eqref{eq:diophantine} is chosen sufficiently large.

    Therefore, it remains to obtain a similar saving in the type-II sums of the form
    \begin{equation} \label{eq:typeIIpractice}
    \sum_{\substack{u v w \sim M \\ N \leq u v w q \leq N + H \\ u \sim U, v \sim V , w \sim W \\ u v w q \equiv a \pmod{r}}} e(g(u v w q)) \Lambda(v) \mu(w) \log q
    \end{equation}
    for $U V W \asymp M$ powers of two with $V, W \geq N^{1/12 - \eta - \eta^2}$.
    Since $\sqrt{N} \leq M \leq N^{2/3 - 100 \eta}$, Proposition \ref{prop:typeII} establishes that \eqref{eq:typeIIpractice} is $\ll H / (\log N)^{10^7}$ if for $C > 0$ sufficiently large with respect to $k$ and for all $q \leq (\log N)^{C}$ there is an $i \in \{1, \ldots, k\}$ such that
    $$
    \| q (i \gamma_{i} + (i + 1) N \gamma_{i + 1}) \| \geq \frac{(\log N)^{C}}{H^i}.
    $$
    Therefore, we can assume that there exists a $q \leq (\log N)^{C}$ such that for all $i \in \{1,2, \ldots, k\}$,
    \begin{equation} \label{eq:resonates}
    \| q (i \gamma_{i} + (i + 1) N \gamma_{i + 1} ) \| \leq \frac{(\log N)^{C}}{H^i}.
    \end{equation}
    In \eqref{eq:typeIIpractice}, write $u = u_1 u_2$ and $w = w_1 w_2$ with $u_1, w_1$ such that all the prime factors of $u_1, w_1$ are $\leq k! (\log N)^C$ and all the prime factors of $u_2, w_2$ are $> k! (\log N)^{C}$.
    We note that if $w_1 > N^{\eta^6}$ or $u_1 > N^{\eta^6}$ then the integer $u w v$ has more than $\exp(\log N / (\log\log N)^2)$ distinct prime factors. The contribution of such integers to \eqref{eq:typeIIpractice} is
    $$
    \ll \sum_{N / 4 M \leq q \leq 4 N / M} \log q \sum_{\substack{H / q \leq n \leq N / q + H / q \\ \omega(n) > 10^9 \log\log N}} d_3(n) \ll (\log N)^{-10^8}\sum_{\substack{N / 4 M \leq q \leq 4 N / M \\ H / q \leq n \leq N / q + H / q}} d_3(n) e^{\omega(n)}
    $$
    and by Shiu's theorem (Lemma \ref{le:shiu}) applied to the sum over $n$, we see that the above is $\ll H (\log N)^{-10^6}$.

    It remains therefore to obtain an upper bound for
    $$
    \sum_{\substack{u_1, w_1 \leq N^{\eta^6} \\ p | w_1 \implies p \leq k! (\log N)^{C}}} \mu(w_1) \sum_{\substack{u_2 w_2 v \sim M / (u_1 w_1) \\ N / (u_1 w_1) \leq v u_2 w_2 q \leq N / (u_1 w_1) + H / (u_1 w_1) \\ v \sim V , w_2 \sim W / w_1, u_2 \sim U / u_1 \\ p | u_2, w_2 \implies p > k! (\log N)^{C} \\ v u_1 w_1 \equiv a \overline{u_2 w_2}\pmod{r}}} \log q \Lambda(v) \mu(w_2) e(g(v u_1 w_1 u_2 w_2 q))
    $$
    in the case when \eqref{eq:resonates} holds.
    By Proposition \ref{prop:preHarman}, it suffices to show that there exists a $D > 10^{10}$ such that for every $w_1 \leq N^{\eta^6}$,
    $$
    \sum_{\substack{u_2 w_2 v \sim M / (u_1 w_1) \\ N / (u_1 w_1) \leq v u_2 w_2 q \leq N / (u_1 w_1) + H' \\ v \sim V , w_2 \sim W / w_1, u_2 \sim U / u_1 \\ p | u_2, w_2 \implies p > k! (\log N)^{C}}} \Lambda(v) \mu(w_2) \log q \chi(v u_2 w_2 q) (v u_2 w_2 q)^{it} \ll H' \cdot (\log N)^{-10^6}
    $$
    for $H' = (H / (u_1 w_1)) (\log N)^{-B}$, $B$ a sufficiently large constant depending on $k$, $\chi$ of conductor $\leq k! (\log N)^{F}$ with $F$ sufficiently large with respect to $k$ and $N (\log N)^{D} / H' \leq |t| \leq (N / H)^{k + 2}$.

    Suppose that $V > U W / (u_1 w_1)$ (the case of $U W / (u_1 w_1) < V$ is essentially identical as it amounts to swapping the roles of $v$ and $u_2 w_2$).
    Then since $N^{1/2 - \eta^4} \leq U V W / (u_1 w_1) \leq N^{2/3 - 100 \eta}$, we have $V \geq N^{1/4 - \eta^3}$ and also $V \leq N^{2/3 - 100 \eta}$. Moreover, $U W \leq N^{1/3 - 50 \eta + \eta^4}$ and $N^{1/3 + 100 \eta} \leq N / M \leq \sqrt{N}$. Therefore,
    $$
    \max \Big ( \frac{N / M}{V}, \frac{V}{N / M} \Big) \leq \max (N^{1/4 + \eta^3}, N^{1/3 - 200 \eta} ) = N^{1/3 - 200 \eta}
    $$
    and we also have $N^{1/12 - 2 \eta} \leq U W \leq N^{1/3 - 50 \eta + \eta^4}$. Note, moreover, that by Lemma \ref{le:canceldirpoly}, for every $|u| \leq N (\log N)^{D / 2} / H$, we have, for $A = 10^{6}$ and $W, V > N^{\varepsilon}$,
    $$
    \Big | \sum_{\substack{w \sim W \\ p | w \implies p > k! (\log N)^{C}}} \frac{\mu(w) \chi(w)}{w^{1/2 + it}} \Big | \ll \frac{\sqrt{W}}{(\log N)^{10 A}} \ \ , \ \  \Big | \sum_{\substack{v \sim V}} \frac{\Lambda(v) \chi(v)}{v^{1/2 + iu - it}} \Big | \ll \frac{\sqrt{V}}{ (\log N)^{10 A}}
    $$
    since $|t| > N (\log N)^{D} / H$ and $D$ is much larger than $10 A$.
    Therefore, Proposition \ref{prop:Harman} is applicable and gives the required saving.



  \end{proof}



  In order to handle the contribution of the first sum in \eqref{eq:linnik}, we will use the following lemma.
  We refer the expert reader to subsection \ref{se:alternative} for a quicker alternative treatment relying on
  Harman's book \cite{Harman}.

  \begin{lemma} \label{le:buchstablemma}
    Let $\eta \in (0, \tfrac{1}{1000})$.
    Let $w = N^{\eta^4}$ and $v = N^{\eta^2}$ and $y = N^{1/3 - 100 \eta}$, $z = N^{1/3 + 100 \eta}$ and $N^{2/3 - \eta} \leq H \leq N^{2/3 + \eta}$. Then,  there exist coefficients $\lambda_d$ with $|\lambda_d| \leq 1$ such that for all $N$ sufficiently large with respect to $1/\eta$, and $(a,r) = 1$, $r \leq \log^{100} N$,
    \begin{align*}
    & \sum_{\substack{N \leq n \leq N + H \\ n \equiv a \pmod{r} \\ p | n \implies p \geq z}} e(g(n)) = \sum_{\substack{N \leq n \leq N + H \\ n \equiv a \pmod{r}}} e(g(n)) \Big ( \sum_{\substack{d | n \\ d \leq v}} \lambda_d \Big ) - \sum_{\substack{N \leq p n \leq N + H \\ w \leq p < z \\ p n \equiv a \pmod{r}}} e(g(p n)) \Big ( \sum_{\substack{d | n \\ d \leq v}} \lambda_d \Big ) \\ & + \sum_{\substack{N \leq n p_1 p_2 \leq N + H \\ w \leq p_1 < p_2 < y \\ n p_1 p_2 \equiv a \pmod{r}}} e(g(n p_1 p_2)) \Big ( \sum_{\substack{d | n \\ d \leq v}} \lambda_d \Big ) + \sum_{3 \leq k \leq \eta^{-4}} (-1)^k \sum_{\substack{N \leq n p_1 \ldots p_k \leq N + H \\ w \leq p_1 < p_2 < \ldots < p_k < y \\ p | n \implies p \geq w \\ n p_1 \ldots p_k \equiv a \pmod{r}}} e(g(n p_1 \ldots p_k)) \\ & + O \Big ( \eta \log \frac{1}{\eta} \frac{H}{\varphi(r) \log N} \Big ).
    \end{align*}
  \end{lemma}

  \begin{proof}

    Iterating Buchstab's identity twice, we see that
    \begin{align*}
      \sum_{\substack{N \leq n \leq N + H \\ n \equiv a \pmod{r} \\ p | n \implies p \geq z}} e(g(n)) & = \sum_{\substack{N \leq n \leq N + H \\ n \equiv a \pmod{r} \\ p | n \implies p \geq w}} e(g(n)) - \sum_{\substack{w \leq p < z \\ N \leq p n \leq N + H \\ p n \equiv a \pmod{r} \\ q | n \implies q \geq w}} e(g(p n)) + \sum_{\substack{w \leq p < q < z \\ N \leq p q n \leq N + H \\ p q n \equiv a \pmod{r} \\ t | n \implies t \geq q \\ t \text{ prime}}} e(g(p q n)).
    \end{align*}
    We will show that at the cost of an error term of size $O(\eta \log(1/\eta) H / (\varphi(r) \log N))$, we can restrict the sum over $w \leq p < q < z$ to $w \leq p < q < y$. Indeed, we notice that the contribution of the integers with $y < q < z$ is bounded by
    \begin{align*}
      \ll \sum_{\substack{y \leq q \leq z \\ w \leq p < z}} \sum_{\substack{N \leq p q n \leq N + H \\ n \equiv a \overline{p q} \pmod{r} \\ t | n \implies t > y \\ t \text{ prime }}} 1 \ll \sum_{\substack{y \leq q \leq z \\ w \leq p \leq z}} \frac{H}{p q \varphi(r) \log N} \ll \eta \log \frac{1}{\eta} \cdot \frac{H}{\varphi(r) \log N}
    \end{align*}
    by the Brun-Titchmarsh theorem \cite[Theorem 6.6]{IwaniecKowalski}.

    On the remaining sum
    $$
    \sum_{\substack{w \leq p < q < y \\ N \leq p q n \leq N + H \\ p q n \equiv a \pmod{r} \\ t | n \implies t \geq q \\ t \text{ prime}}} e(g(n)),
    $$
    we apply Buchstab's identity $\log N / \log w$ times, and this shows that this sum is equal to
    $$
    \sum_{2 \leq k \leq \eta^{-4}} (-1)^k \sum_{\substack{w \leq p_1 < \ldots < p_k < y \\ N \leq p_1 \ldots p_k n \leq N + H \\ p_1 \ldots p_k n \equiv a \pmod{r} \\ t | n \implies t \geq w \\ t \text{ prime}}} e(g(n p_1 \ldots p_k)).
    $$

    It remains to express the condition $p | n \implies t \geq w$ using a sieve on the terms with $k \in \{0,1,2\}$.
    Let $\mathcal{T}$ be the subset of integers $n \in [1, N]$ with the property that all the prime factors of $n$ are less than $w$ and $n$ has at most $100 \lfloor \log_{i + 1} N \rfloor$ distinct prime factors in the interval
    $$
    I_{i} := \Big [ \exp \Big ( \frac{\log N}{(\log_{i} N)^2} \Big ) , \exp \Big ( \frac{\log N}{(\log_{i + 1} N)^2} \Big ) \Big ]  \ , i = 1 , 2 \ldots, J
    $$
    with $J$, the smallest integer such that $\eta^{-4} \leq \log_{J + 1} N \leq \exp(\eta^{-4})$, and where $\log_{1} N := \infty$ so that $\log N / \log_{1} N = 0$. Moreover, letting $w' := \exp(\log N / (\log_{J + 1} N)^2)$, we also require that $n \in \mathcal{T}$ has at most
    $$
    100 \Big \lfloor \sum_{w' \leq p \leq w} \frac{1}{p} \Big \rfloor
    $$
    distinct prime factors in the interval $[w', w]$. Note that $\sum_{w' \leq p \leq w} p^{-1} \geq \frac{1}{2} \log \eta^{-1}$.

    Notice that if $n \in \mathcal{T}$ then in fact $n \leq N^{\eta^2} =: v$. Let also $\mathcal{T}'$ denote the set of integers that can be written as $n = a b $ with $p | a \implies p \leq w $ and $p | b \implies p > w$, and such that $a \in \mathcal{T}$ and $b$ has at most $100 \log(1 / \eta)$ prime factors.

    We notice that on the set $n \in \mathcal{T}'$, we have
    $$
    \mathbf{1}_{p | n \implies p > w} = \sum_{\substack{d | n \\ p|d \implies p \leq w}} \mu(d) = \sum_{\substack{d | n \\ p | d \implies p \leq w \\ d \leq v , d \in \mathcal{T}}} \mu(d) =: \sum_{\substack{d | n \\ d \leq v}} \lambda_d
    $$
    since any divisors $d$ of $n$ with the property that all the prime factors of $d$ are $\leq w$ is a divisor of $a$ and therefore, an element of $\mathcal{T}$, and hence $\leq v$. Here, $\lambda_d$ is defined by setting $\lambda_d = \mu(d)$ whenever $d \in \mathcal{T}$ and $\lambda_d = 0$ otherwise. Moreover,
    \begin{equation} \label{eq:sieveobservation}
    0 \leq \mathbf{1}_{p | n \implies p \geq w} \leq \sum_{\substack{d | n \\ d \leq v}} \lambda_d.
    \end{equation}

    Therefore, we have for $k \in \{0,1,2\}$ the identity
    \begin{align*}
    \sum_{\substack{N \leq n p_1 \ldots p_k \leq N + H \\ w \leq p_1 < \ldots < p_k < y \\ n p_1 \ldots p_k \equiv a \pmod{r}\\ p | n \implies p \geq w}} e(g(n p_1 \ldots p_k)) =  \sum_{\substack{N \leq n p_1 \ldots p_k \leq N + H \\ w \leq p_1 < \ldots < p_k < y \\ n p_1 \ldots p_k \equiv a \pmod{r} \\ n \in \mathcal{T}'}} & e(g(n p_1 \ldots p_k)) \Big ( \sum_{d | n} \lambda_d \Big ) \\ & + O \Big ( \sum_{\substack{N \leq n p_1 \ldots p_k \leq N + H \\ w \leq p_1 < \ldots < p_k < y \\ n p_1 \ldots p_k \equiv a \pmod{r} \\ p | n \implies p \geq w \\ n \not \in \mathcal{T}'}} 1 \Big ).
    \end{align*}
    Furthermore, by \eqref{eq:sieveobservation}, this is equal to
    \begin{align*}
    \sum_{\substack{N \leq n p_1 \ldots p_k \leq N + H \\ w \leq p_1 < \ldots < p_k < y \\ n p_1 \ldots p_k \equiv a \pmod{r} }} & e(g(n p_1 \ldots p_k)) \Big ( \sum_{d | n} \lambda_d \Big ) + O \Big ( \sum_{\substack{N \leq n p_1 \ldots p_k \leq N + H \\ w \leq p_1 < \ldots < p_k < y \\ n p_1 \ldots p_k \equiv a \pmod{r} \\ n \not \in \mathcal{T}'}} \sum_{\substack{d | n \\ d \leq v}} \lambda_d \Big ).
    \end{align*}

    It therefore remains to show that the sum over $n \not \in \mathcal{T}'$ above is negligible for each $k \in \{0,1,2\}$. Since $\lambda_d$ is supported on integers all of whose prime factors are $\leq w$, and $p_1, \ldots, p_k > w$, we have
    $$
    \sum_{\substack{d | n \\ d \leq v}} \lambda_d = \sum_{\substack{d | n p_1 \ldots p_k \\ d \leq v}} \lambda_d.
    $$
    Moreover, the number of representations of a given integer $m$ as $n p_1\ldots p_k$ with $w < p_1 , \ldots, p_k < z$ and all of the prime factors of $n$ less than $w$ is
    $\ll (\log z / \log w)^k$. Therefore,
    \begin{equation} \label{eq:sieveboundelim}
    \sum_{0 \leq k \leq 2} \sum_{\substack{N \leq n p_1 \ldots p_k \leq N + H \\ n p_1 \ldots p_k \equiv a \pmod{r} \\ w \leq p_1 < \ldots < p_k < z \\ n \not \in \mathcal{T}'}} \Big ( \sum_{\substack{d | n \\ d \leq v}} \lambda_d \Big ) \ll \sum_{\substack{N \leq n \leq N + H \\ n \equiv a \pmod{r} \\ n \not \in \mathcal{T}'}} \Big ( \frac{\log z}{\log w} \Big )^2 \cdot \Big ( \sum_{\substack{d | n \\ d \leq v}} \lambda_d \Big )
  \end{equation}
  and it remains to show that this is $\ll \eta H / (\varphi(r) \log N)$.

  Let $f$ be a completely multiplicative function with $f(p) = 1$ for $p \leq 100$ and $f(p) = 2$ for $p > 100$. Then, by the union bound,
  $$
  \mathbf{1}_{n \not \in \mathcal{T}'}  \ll \Big ( \sum_{i = 1}^{J} 2^{-100 \log_{i + 1} N} + 2^{-100 \log(1/\eta)} \Big ) f(n).
  $$
  Therefore, \eqref{eq:sieveboundelim} is
  $$
  \eta^{40} \sum_{\substack{N \leq n \leq N + H \\ n \equiv a \pmod{r}}} f(n) \sum_{\substack{d | n \\ d \leq v}} \lambda_d = \eta^{40} \sum_{\substack{d \leq v \\ (d,r) = 1}} \lambda_d f(d) \sum_{\substack{N \leq n \leq N + H \\ n \equiv \overline{d} a \pmod{r}}} f(n).
  $$
  By \cite[Main Theorem]{Ramachandra}, we have
  \begin{align*}
  \sum_{\substack{N \leq d n \leq N + H \\ n \equiv \overline{d} a \pmod{r}}} f(n) & = \frac{1}{\varphi(r)} \sum_{\chi \pmod{r}} \chi(d) \overline{\chi}(a) \sum_{\substack{N \leq d n \leq N + H}} f(n) \chi(n) \\ & = \frac{H}{d \varphi(r)} \cdot \Big ( \frac{\varphi(r)}{r} \Big )^2 \cdot P (\log N) + O_{A}\Big ( \frac{H}{(\log N)^{A}} \Big )
  \end{align*}
  with $P$ a linear polynomial. Therefore, \eqref{eq:sieveboundelim} is
  $$
  \ll \eta^{40} \cdot \frac{H \log N}{\varphi(r)} \cdot \Big (\frac{\varphi(r)}{r} \Big )^2 \sum_{\substack{d \leq v}} \frac{\lambda_d f(d)}{d}
  $$
  and by definition of $\lambda_d$, we have
  \begin{align*}
    \sum_{\substack{d \leq v \\ (d,r) = 1}} \frac{\lambda_d f(d)}{d} & = \prod_{i = 1}^{J} \Big ( \sum_{\substack{p | n \implies p \in I_i \\ \Omega(n; I_i) \leq 100 \lfloor \log_{i + 1} N \rfloor}} \frac{\mu(d) f(d)}{d} \Big )  \\
    & = \prod_{i = 1}^{J} \Big ( \prod_{\substack{p \in I_i \\ (p, r) = 1}} \Big ( 1 - \frac{2}{p} \Big ) + \Big ( \frac{1}{(\log_{i} N)^{10}} \Big ) \Big )
  \end{align*}
  by Chernoff's bound. Since the Euler product is always larger than the error term, we can bound the above by
  $$
  \ll \Big ( \frac{r}{\varphi(r)} \Big )^2 \cdot \prod_{p \leq v} \Big ( 1 - \frac{2}{p} \Big ) \ll \Big ( \frac{r}{\varphi(r)} \Big )^2 \cdot \frac{\eta^{-8}}{\log^2 N}.
  $$
  It follows that \eqref{eq:sieveboundelim} is
  $$
  \ll  \frac{\eta^{10} H}{\varphi(r) \log N}
  $$
  as needed.

\end{proof}

  We notice that the first two terms,
  $$
  \sum_{\substack{N \leq n \leq N + H \\ n \equiv a \pmod{r}}} e(g(n)) \Big ( \sum_{\substack{d | n \\ d \leq v}} \lambda_d \Big ) \text{ and } \sum_{\substack{N \leq p n \leq N + H \\ w \leq p \leq z \\ p n \equiv a \pmod{r}}} e(g(n)) \Big ( \sum_{\substack{d | n \\ d \leq v}} \lambda_d \Big )
  $$
  fall within the scope of Proposition \ref{prop:typeI}, and in particular, it follows that these terms are $\ll_{A} H (\log N)^{-10^{6}}$ provided that $B$ in \eqref{eq:diophantine} is sufficiently large with respect to $k$.
  We notice that the case $k = 2$ also falls within the scope of Proposition \ref{prop:typeI} since $d p_1 p_2 \leq N^{2/3 - 100 \eta}$. Therefore, we can assume that $k \geq 3$. We localize the variable $n$ in a dy-adic interval $R$. We notice that if $R \geq N^{1/3 + \eta + \eta^2}$ then Proposition \ref{prop:typeI} is once again applicable. We can therefore assume that $R \leq N^{1/3 + \eta + \eta^2}$.

  It therefore remains to show that for each $3 \leq k \leq \eta^{-4}$,
$$
\sum_{\substack{N \leq n p_1 \ldots p_k \leq N + H \\ n \sim R \\ w \leq p_1 < p_2 < \ldots < p_k < y \\ n p_1 \ldots p_k \equiv a \pmod{r} \\ p | n \implies p \geq w}} e(g(n p_1 \ldots p_k)) \ll H \cdot (\log N)^{-{10^6}}.
$$
We then localize each variable $p_i$ in a dyadic interval $[P_i, 2P_i]$ with $w \leq P_i \leq y$ powers of two.  Subsequently, we use contour integral to resolve the condition $p_i < p_{i + 1}$ for $i = 1, 2, \ldots, k- 1$. All these operations introduce logarithmic losses (in total $(\log N)^{O_{\eta}(1)}$) and in particular, it is enough to show that for every $1 \leq R \leq N^{1/3 + \eta + \eta^2}$, $P_i \in [w,y]$ for $i = 1,2, \ldots, k$ with $R P_1 \ldots P_k \asymp N$, we have for some $A > 0$ sufficiently large with respect to $1 / \eta$,
$$
\sum_{\substack{n \sim R, p_i \sim P_i \\ N \leq n p_1 \ldots p_k \leq N + H \\ n p_1 \ldots p_k \equiv a \pmod{r} \\ p | n \implies  p > w}} p_1^{i t_1} \ldots p_{k}^{i t_k} e(g(n p_1 \ldots p_k)) \ll_{A} \frac{H}{(\log N)^{A}}
$$
with $|t_i| \leq y^{1 + 1/100}$ for all $i = 1, 2, \ldots, k$.

Let $\ell$ be the first index such that $P_1 \cdots P_{\ell} > N^{1/3 + 50\eta}$. Then, necessarily $\ell \geq 2$ and $P_1 \cdots P_{\ell} \leq N^{2/3 - 50 \eta}$ since $P_1 \cdots P_{\ell - 1} \leq N^{1/3 + 50 \eta}$ and $P_{\ell} \leq N^{1/3 - 100 \eta}$. Therefore, grouping together the variables $p_1, \ldots, p_{\ell}$ on one side, and the variables $p_{\ell + 1}, \ldots, p_k$, $n$ on the other side, we obtain a bilinear form to which Proposition \ref{prop:typeII} is applicable. Consequently, we can assume that there exists a large constant $C > 0$ depending on $1/\eta$ and $k$ such that\footnote{As usual we set $\gamma_{k + 1} = 0$.}
$$
\| q (i \gamma_i + (i + 1) N \gamma_{i + 1}) \| \leq \frac{(\log N)^{C}}{H^i}
$$
for every $i \in \{1, 2, \ldots, k\}$.

By Proposition \ref{prop:preHarman}, it remains to verify that for $A, B, F$ sufficiently large with respect to $1/\eta$ and $k$, $H' = H (\log N)^{-B}$ and $(N / H') (\log N)^{A^2} \leq |t| \leq (N / H)^{k + 2}$, and $\chi$ of conductor $\leq k! (\log N)^{F}$,
$$
\sum_{\substack{n \sim R,p_i \sim P_i \\ N \leq n p_1 \ldots p_k \leq N + H' \\ p | n \implies p > w}} n^{i t} p_1^{i t_1 + i t} \ldots p_{k}^{i t_k + i t} \chi(n p_1 \ldots p_k) \ll \frac{H'}{(\log N)^{A}}.
$$
Importantly, we notice that $t$ is much larger than the remaining $t_1, \ldots, t_k$. Therefore, by Lemma \ref{le:canceldirpoly},
we have, for $|u| \leq (N / H') (\log N)^{A^2 - 1}$,
$$
\Big | \sum_{\substack{R \leq n \leq 2R \\ p | n \implies p > w}} \frac{\chi(n)}{n^{1/2 + it + iu}} \Big | \ll_{A} \frac{\sqrt{N}}{\log^{A} N}
$$
as long as $R > N^{100 \eta^2}$ and similarly for $|u| \leq (N / H') (\log N)^{A^2 - 1}$,
$$
\Big | \sum_{P_j \leq p \leq 2P_j} \frac{\chi(p)}{p^{1/2 + it + i t_j + iu}} \Big | \ll_{A} \frac{\sqrt{P_j}}{\log^{A} P_j}
$$
for all $j = 1, \ldots, k$ since $P_j \geq w$.

First, let us show that we can assume that $R > N^{\eta}$. In the case $k = 3$ this is clear: indeed, if $R < N^{\eta}$ then $P_1 \ldots P_3 R \leq N^{1 - 100 \eta}$ which is impossible. Let us assume therefore that $k \geq 4$ and that $R \leq N^{\eta}$. In that case, we group together the variable $n$ with the longest variable among the $p_i$'s. This leads to a situation in which we have four variables, all of length $> N^{\eta^2}$, all exhibiting cancellations, and all but the one shorter than $\tfrac 13 - 100 \eta$ (the outlier is still shorter than $\tfrac 13$). It follows then from Lemma \ref{le:combi} that we can group these variables in a way so that Proposition \ref{prop:Harman} is applicable.

In the remaining case, when $R > N^{\eta}$ and $k \geq 3$, we still find ourselves in the situation in which we have at least four variables, and all of them exhibit non-trivial cancellations. Therefore, we conclude again by using Lemma \ref{le:combi} below and Proposition \ref{prop:Harman}.

\begin{lemma}\label{le:combi} Let $\eta \in (0, 10^{-5})$. Let $k \geq 4$.
  Let $a_1 + \ldots + a_k = 1$ be a sequence of real numbers with $0 < a_i < \tfrac 13 - 100 \eta$ for $i=1,\ldots,k-1$, and $0 < a_k < \tfrac 13 + 100 \eta$. Then, there exists a partition of $\{1, \ldots, k\}$ into three disjoint non-empty subsets $I,J,K$ such that
  $$
  \Big | \sum_{i \in I} a_i - \sum_{j \in J} a_j \Big | \leq \frac{1}{3} - \eta \text{ and } \Big | \sum_{k \in K} a_k \Big | \leq \frac{5}{9} - 2 \eta.
  $$
\end{lemma}
\begin{proof}
  Suppose first that $k = 4$.
  Either $a_1 + a_2 < 5/9 - 2 \eta$ or $a_3 + a_4 < \frac{5}{9} - 2 \eta$. In the first case, we take $K = \{1, 2\}$ and $I = \{3\}, J = \{4\}$. In the second case, we take $K = \{3, 4\}$ and $I = \{1\}, J = \{2\}$.  Suppose now that $k = 5$. If for any two $\ell \neq j$ we have $a_{\ell} + a_{j} \leq \frac{1}{3} - 100 \eta$ then we collapse $a_{\ell} + a_j$ into one element and appeal to the result with $k = 4$. Therefore, we can assume that $a_{\ell} + a_j \geq \frac{1}{3} - 100 \eta$ for any two $\ell \neq j$.  In particular, $\tfrac{1}{3} -100 \eta \leq a_1 + a_2, a_3 + a_4 \leq \tfrac{2}{3} - 200 \eta$, so we take $I = \{1, 2\}, J = \{3, 4\}$ and $K = \{5\}$. Suppose now that $k = 6$. Once again we can assume that for any two $\ell \neq j$ we have $a_{\ell} + a_j > \tfrac{1}{3} - 100 \eta$. Therefore, $\tfrac 13 - 100 \eta \leq a_1 + a_2, a_3 + a_4, a_5 + a_6 \leq \tfrac{2}{3} - 100 \eta$. Moreover, at least one of $a_1 + a_2$ or $a_3 + a_4$ has to be $\leq \tfrac{5}{9} - 100 \eta$, say $a_1 + a_2$. In that case we pick $K = \{1, 2\}$ and $I = \{3, 4\}$ and $J = \{5, 6\}$. Finally, suppose that $k \geq 7$. In that case, as before, we can assume that for any $\ell \neq j$ we have $a_{\ell} + a_{j} \geq \tfrac {1}{3} - 100 \eta$, as otherwise, we are back to the case $k - 1$ which we can assume to be proven. Then either $a_6$ or $a_7$ is greater than $\tfrac{1}{6} - 50 \eta$ because $a_6 + a_7 > \tfrac{1}{3} - 50 \eta$. Without loss of generality, assume that it is $a_7$. This however leads to an impossible situation as
$$
3 \times \Big ( \frac{1}{3} - 100 \eta \Big ) + \frac{1}{6} - 50 \eta < \sum_{i = 1}^{6} a_i + a_7 + \ldots + a_k = 1,
$$
so the case $k \geq 7$ reduces to the earlier case with $k - 1$ variables.
\end{proof}

      \subsection{The non-oscillatory case}
We will prove Proposition \ref{prop:nonoscilatory}.
Suppose therefore that there exist  $B > 0$ and $ q \leq (\log N)^{B}$ such that for all $i \in \{1, \ldots, k\}$ we have $\| q \gamma_i \| \leq (\log N)^{B} / H^{i}$. In this situation, we write
      $$
      \gamma_{i} = \frac{a_{i}}{q} + \theta_{i}
      $$
      with $|\theta_{i}| \leq (\log N)^{B} H^{-i}$.
      We split into progressions $\pmod{q r}$ and we apply Proposition \ref{prop:siegelwalfiszshort} and the integration by parts (using that the derivative of $\sum \theta_{i} (n - N)^{i}$ is $\ll H^{-1} (\log N)^{B}$). This gives us
      \begin{align} \label{eq:nonoscillating}
      \sum_{\substack{N \leq p \leq  N + H \\ p \equiv a \pmod{r}}} e(g(p)) \log p = \frac{1}{\varphi(q r)} \sum_{\substack{(x,q r ) = 1 \\ x \equiv a \pmod{r}}} & e \Big ( \sum_{i = 1}^{k} \frac{(x - N)^i a_i}{q^i} \Big ) \sum_{\substack{N \leq n \leq N + H}} e \Big ( \sum_{i = 1}^{k} (n - N)^{i} \theta_{i} \Big ) \\ \nonumber & + O_{A} \Big ( \frac{H}{(\log N)^{A}} \Big ).
      \end{align}
      We now notice that the assumptions of the theorem imply that $a_{i} = 0$ for all $i \geq 2$. Moreover, if $a_1 \neq 0$ then without loss of generality we can assume that $(a_1, q) = 1$. We also notice that since $|\gamma_1| \leq e^{-\tau r}$, we have $q > e^{\tau r}$ since $a_1 \neq 0$. In particular, for any $A > 1000$,
      $$
      \frac{1}{\varphi(r q)} \cdot e \Big ( - \frac{N a_1}{q} \Big ) \sum_{\substack{(x, qr) = 1 \\ x \equiv a \pmod{r}}} e \Big ( \frac{x a_1}{q} \Big ) \ll_{A} \frac{1}{e^{\tau r / 2}} \cdot \frac{\varphi(q r)}{q r} \cdot q r  + \frac{q r}{(\log q)^{A}}
      $$
      by Lemma \ref{lm:n7}. Therefore, \eqref{eq:nonoscillating} is
      $$
      \ll_{A} \frac{H}{e^{\tau  r / 2}} + \frac{H}{(\log q)^{A}} \ll \frac{\eta H}{\varphi(r)}
      $$
      as needed.

      Finally, in the remaining case when $a_1 = 0$, we obtain precisely the statement of the theorem, as in that case one can apply Proposition \ref{prop:siegelwalfiszshort} with the choice $q = 1$ and $\theta_{i} = \gamma_{i}$ for all $i \geq 1$.

\subsection{An alternative argument} \label{se:alternative}

We describe here an alternative arrangement of our argument that was communicated to us by Kaisa Matom\"aki and which relies on Harman's  book \cite{Harman}. Write
$$
S(A, z) := \sum_{\substack{x \leq n \leq x + H  \\ n \equiv a \pmod{r}  \\ p | n \implies p > z }} e(g(n))
$$
and, as usual, let
$$
S(A_p, z) := \sum_{\substack{x \leq p n \leq x + H \\ p n \equiv a \pmod{r} \\ q | n \implies q > z}} e(g(n p)).
$$
Set $y = x^{1/3 - 100 \eta}$ and $z = x^{1/3 + 500 \eta}$. By Buchstab's identity and a sieve upper bound, we have
\begin{align*}
\sum_{\substack{x \leq p \leq x + H \\ p \equiv a \pmod{r}}} e(g(p)) & = S(A, \sqrt{x}) = S(A, y) - \sum_{y < p < \sqrt{x}} S(A_p, p) \\ & = S(A,y) - \sum_{z < p < \sqrt{x}} S(A_p, p) + O \Big ( \frac{\eta H}{\varphi(r) \log x} \Big ) \\ & = S(A,y) - \sum_{z < p < \sqrt{x}} S(A_p, y) + O \Big ( \frac{\eta H}{\varphi(r) \log x} \Big )
\end{align*}
since the implicit variable in $S(A_p, p)$ with $p > z$ is necessarily a prime.

These sums can be now decomposed into appropriate type I/II sums by using \cite[Theorem 3.1]{Harman} in the ``non-resonating case'' when $e(g(m)) \not \approx m^{it}$ , and using \cite[Lemma 7.5]{Harman} in the ``resonating case'' when $g(m) \approx m^{it}$. The \cite[Lemma 7.5]{Harman} is stated for intervals $H = x^{7/12}$ but a minor variant also works in the case of intervals of length $x^{2/3 - \eta}$ with $z = x^{1/3 - 100 \eta}$.
This saves us from having to prove Lemma \ref{le:buchstablemma} and decomposing into type I and type II sums as this is then done in \cite[Theorem 3.1]{Harman} and \cite[Lemma 7.5]{Harman}. As a result, this arrangement of the proof would save a few pages (beginning with Lemma \ref{le:buchstablemma} and ending at the ``Non-oscillating case'').




\section{Proof of Theorem \ref{thm:main}}\label{sec:thmmain}
In this section we will prove our main equidistribution result for analytic skew products. Fix $\alpha\in \T$, $g\in C^{\omega}(\T)$ of zero mean and $T(x,y)=(x+\alpha,y+g(x))$. Notice that since the characters form a linearly dense set, it is enough to show Theorem~\ref{thm:main} for $f(x,y)=e_{b,c}(x,y)=e^{2\pi i (bx+cy)}$ for all $b,c\in\Z$. From now on, we also fix $b,c\in \Z$.

We will show that for all $0<\eta<1$ and every sufficiently large $N$, we have
\be\label{eq:smae}
\sum_{p\leq N}e_{b,c}(T^p(x,y))\log p=O(\eta^{1/2} N).
\ee
Then Theorem \ref{thm:main} will immediately follow from \eqref{eq:smae} (since $\eta>0$ is arbitrary). Fix $1>\eta>0$ and set $\epsilon:=\frac{\eta^2}{1000}$ and $\xi:=\epsilon^{10}$ (see Theorems~\ref{thm:nr1} and~\ref{thm:nr2}). Assume that $N\in \N$ and let $n\in \N$ be unique such that $q_n\leq N<q_{n+1}$. The proof of the theorem will split into several cases:\\
\textbf{Case A.} $N\geq e^{q_n^{1/2}}$. Let $H:=\min(N,q_{n+1}^{3/4})$~\footnote{In this case, we can take $1-\delta$ (for fixed $\delta>0$) instead of $3/4$.} and $m\leq N$. We will show that
\be\label{sum:h}
\Big|\sum_{p\in [m,m+H]}e_{b,c}(T^p(x,y))\log p\Big|=O(\eta^{1/2} H),
\ee
then \eqref{eq:smae} follows by summing over disjoint intervals of length $H$. Let $z_n\in \{q_n,p_nq_n\}$ come from Proposition~\ref{prop:ergchar} (recall that $p_n\leq 2\log^2 q_n$). Let $p\in [m,m+H]$. Since $p-m\leq H\leq q_{n+1}^{2/3-\eta}$, by Corollary \ref{cor:jad} ~(applied to $p-m$) with $\delta=1/5$, $w\in \{1,p_n\}$ (in both cases, $|w|\leq 2\log^2q_n\leq \log^3 q_n$) and  $(x_m,y_m)=T^m(x,y)$, we have
\be\label{eq:tep}
T^p(x,y)=T^{p-m}(x_m,y_m)=\\
o(1)+T^{(p-m)\mod z_n}(x_m, y_m+P_n(x_m,p-m)),
\ee
where the degree of $P_n$ is bounded by $5$.
Assume that $a\in\{0,1\,\ldots,z_n-1\}$. First notice that if $p\equiv m+a\mod z_n$ with $(m+a,z_n)>1$, then $p\leq z_n$ (in fact, $p|z_n$) and hence such residue classes can be ignored as their contribution to the LHS of~\eqref{eq:smae} is of order $\sum_{p\leq z_n}\log p=:\theta(z_n)\sim z_n\leq 2q_n\log^2q_n=o(N)$, where we have used the PNT and $N\geq e^{q_n^{1/2}}$. We hence consider only  $a\leq z_n$ such that  $(m+a,z_n)=1$. By \eqref{eq:tep}, we have
\begin{multline}\label{eq:as1}
\sum_{\substack{p\in [m,m+H]\\ p-m\equiv a\mod z_n}}e_{b,c}(T^p(x,y))\log p= o(\sum_{\substack{p\in [m,m+H]\\ p-m\equiv a\mod z_n}}\log p)+\\
e_{b,c}(T^a(x_m,y_m))\sum_{\substack{p\in [m,m+H]\\ p-m\equiv a\mod z_n}}e_{c}(P_n(x_m,p-m))\log p
\end{multline}
(where $o(\cdot)$ does not depend on $a$). Set
$$
h_{m,a}:=\sum_{\substack{p\in [m,m+H]\\ p-m\equiv a\mod z_n}}e_{c}(P_n(x_m,p-m))\log p,
$$
and let
$$v_{m}:=\frac{1}{\varphi(z_n)}\sum_{k\in [m,m+H]}e_{c}(P_n(x_m,k-m)).
$$
By \eqref{eq:as1} and summing over all $a\leq z_n$ for which $(m+a,q_n)=1$, we get
\begin{multline}\label{eq:jgh1}
\sum_{p\in [m,m+H]}e_{b,c}(T^p(x,y))=o(\sum_{p\in [m,m+H]}\log p)+\\
v_{m}\sum_{\substack{(m+a,z_n)=1\\a\leq z_n}}e_{b,c}(T^a(x_m,y_m))+O\Big(\sum_{\substack{(m+a,z_n)
=1\\a\leq z_n}}|h_{m,a}-v_{m}|\Big).
\end{multline}
Again, we can ignore the term $o(\sum_{p\in [m,m+H]}\log p)$  in what follows: after summing over disjoint intervals of length $H$, the joint error term, by a use of the prime number theorem, contributes $o(N)$ in~\eqref{eq:smae}.

Notice that by Proposition~\ref{lem:toadd1}, it follows that $P_n(x_m,n-m)$ is a polynomial (of degree $\leq 3$)  whose coefficients satisfy~\eqref{eq:bocoe}. Since
$N<q_{n+1}$ and $H=\min(N,q_{n+1}^{3/4})\geq N^{3/4}\geq m^{3/4}$, it follows that we can apply Theorem~\ref{thm:nr4}\footnote{In this case, we in fact could appeal to the results of Matom\"aki-Shao and we don't need the full strength of Theorem \ref{thm:nr4}. We will use Theorem \ref{thm:nr4} in its strongest form in case {\bf B.2.1}.} with $g$ replaced with $P_n$, $N$ with $m$ and $r$ with $z_n$ (note that in view of~\eqref{eq:bocoe}, also the assumptions on $\beta_i$ are satisfied).
 Therefore,
$$
\sum_{\substack{(m+a,z_n)=1\\a\leq z_n}}|h_{m,a}-v_{m}|=O\Big(\eta \log\Big(\frac{1}{\eta}\Big) H\Big)=O(\eta^{1/2} H)
$$
as (remembering that $z_n$ is of order at most $q_n\log^2 q_n$) by taking $A>3$ in Theorem~\ref{thm:nr4}, we have $z_n\leq (\log m)^A$ for $m\geq e^{q_n^{1/3}}$, so the theorem applies, and this range of $m$ is sufficient to cope with~\eqref{eq:smae}.

Consider now the set $C:=\{0\leq a<z_n:\:(m+a,z_n)=1\}\subset\{a'+\ell z_n:\: 0\leq a'<z_n, (a',z_n)=1,\ell\in\Z\}$. If $m=uz_n+t$ with $0\leq t<z_n$ then a number $a\in C$ either satisfies $t+a=a'$ or $t+a=a'+z_n$ (with $0\leq a'<z_n$, $(a',z_n)=1$), in any case we obtain a bijection $a\mapsto a'$. But, by its definition, $z_n$ is a time of (uniform) rigidity, so $d(T^a(x_m,y_m),T^{a'}(T^{-t}(x_m,y_m))=o(1)$.
Moreover (trivially),
$|v_{m}|\leq\frac{H}{\varphi(z_n)}$,
so
$$
v_{m}\sum_{\substack{(m+a,z_n)=1\\a\leq z_n}}e_{b,c}(T^a(x_m,y_m))=o(H),
$$
where we use the bound
$$
\frac{1}{\varphi(z_n)}\sum_{\substack{(m+a,z_n)=1\\a\leq z_n}}e_{b,c}(T^a(x_m,y_m))=
\frac{1}{\varphi(z_n)}\sum_{\substack{(a',z_n)=1\\a'\leq z_n}}e_{b,c}(T^{a'}(T^{-t}(x_m,y_m))]+o(1)=o(1),
$$
which follows from Proposition \ref{prop:ergchar}  (with $d=z_n$).

Putting the above bounds together to \eqref{eq:jgh1}, yields
$$
\sum_{p\in [m,m+H]}e_{b,c}(T^p(x,y))\log p=O(\eta^{1/2} H).
$$
This gives \eqref{eq:smae} and finishes the proof of \textbf{Case A.}\\

\textbf{Case B.} $N\leq e^{q_n^{1/2}}$. Let $n^\ast\leq n$ be the largest number such that  (see \eqref{eq:n,ast})
\be\label{eq:nast}
q_{n^\ast}\geq e^{\tau q_{n^\ast-1}}.
\ee

\textbf{B.1.} $N^{5/6-\epsilon}\geq q_n$. Denote $p_{q_n}:=p \mod q_n$. Notice that since $N\leq \min(e^{q_n^{1/2}},q_{n+1})$, by Lemma~\ref{cor:jad2} (note that if $e^{q_n^{1/2}}\geq q_{n+1}$ then still $N\leq q_n\frac{q_{n+1}}{q_{n^\ast}}$ and $e^{\tau q_n}\geq q_{n+1}$) it follows that
$$
d\Big(T^p(x,y),T^{p_{q_n}}(x,y)\Big)=o(1).
$$
Therefore,
$$
\sum_{p\leq N}e_{b,c}(T^p(x,y))\log p=o(N)+\sum_{p\leq N}e_{b,c}(T^{p_{q_n}}(x,y))\log p,
$$
and therefore below we will consider the last sum. We further split this case in two subcases:\\

\textbf{B.1.1.} $q_{n^{\ast}}>q_n^{1-\epsilon}$. In this case we use Theorem \ref{thm:nr2} with $H:=N$, $q:=q_n$, $r:=z_{n^\ast-1}$ (where $z_{n^\ast-1}$ comes from Proposition~\ref{prop:ergchar}), $H':=q^{1/3}$.
Note that by the definition of $n^\ast$, $r=z_{n^{\ast}-1}\leq q_{n^\ast-1}^2\leq \log^3 q_{n^\ast}\leq \log^3 q_n\leq \log^3 N$. Notice moreover that $H/q=N/q\geq \frac{N}{N^{5/6-\epsilon}}= N^{1/6+\epsilon}$.

We call an interval  $I\subset [0,q]$ ``good''
if it satisfies
\begin{equation}\label{goodboys}\sup_{\beta}\Big|\sum_{\substack{p\leq N\\ p_{q}\equiv v \mod r\\ p_{q}\in I}}e(p_{q}\beta)\log p- \frac{N}{\varphi(q)}\sum_{\substack{(m,q)=1\\ m\equiv v \mod r \\m\in I}}e(m\beta)\Big|\ll\frac{N|I|}{q\log^{60}N}.
\end{equation}
Otherwise, we call $I$ ``bad''.  For a good $I\subset [0,q]$, summing over $v\leq r$ and using $r\leq \log^3 N$, we obtain
\be\label{eq:p11}
\sum_{\substack{(v,(r,q))=1\\v\leq r}}\sup_{\beta}\Big|\sum_{\substack{p\leq N\\ p_{q}\equiv v \mod r\\ p_{q}\in I}}e(p_{q}\beta)\log p- \frac{N}{\varphi(q)}\sum_{\substack{(m,q)=1\\ m\equiv v \mod r \\m\in I}}e(m\beta)\Big|=o\Big(\frac{N|I|}{q}\Big).
\ee
We now consider  intervals $[j,j+H']$ (with $j\leq q$) of length $H'$. Recall that all the assumptions of Theorem~\ref{thm:nr2} are satisfied (with $H:=N$, $q:=q_n$, $r:=z_{n^\ast-1}$ and $H'=q^{1/3}$), where we are in the situation $x=H$, so the second part of this theorem applies.
Dividing in \eqref{Hrownex} both sides by $q$, we obtain that the LHS is bounded by $(1/\log^{100}N) O(NH'/q)$. The number $K$ of $j\leq q$ of those intervals $[j,j+H']$ which are bad, i.e.\ for which the LHS in~\eqref{goodboys} is bounded from below by $\frac1{\log^{60} N}O(NH'/q)$ is hence at most  $q\frac{\log^{60} N}{\log^{100}N}=q\frac1{\log^{40}N}$, whence $K=o(q)$. It follows that the number of good intervals $I$  is $q-o(q)$.
By considering these intervals in arithmetic progressions $[s+tH,s+(t+1)H']$ (with $s\leq H'$), we must see the same proportion of good intervals along at least one such arithmetic progression.
It follows that we can decompose $[0,q]=\bigcup_{i=1}^\ell I_i$, where all the intervals $I_i$ are pairwise disjoint, $|I_i|=H'$ for $2\leq i\leq \ell-1$ and $|I_1|,|I_\ell|\leq H'$ and all but $o(\ell)$ of the intervals $\{I_i\}$ satisfy~\eqref{eq:p11}. By \eqref{eq:p11} (summing over good $i\leq \ell$), we have
\be\label{eq:p12}
\sum_{i \text{ is good}}\sum_{\substack{(v,(r,q))=1\\v\leq r}}\sup_{\beta}\Big|\sum_{\substack{p\leq N\\ p_{q}\equiv v \mod r\\ p_{q}\in I_i}}e(p_{q}\beta)\log p- \frac{N}{\varphi(q)}\sum_{\substack{(m,q)=1\\ m\equiv v \mod r \\m\in I_i}}e(m\beta)\Big|=o\Big(N\Big).
\ee
Notice that
\begin{multline}\label{mult:jul}
\sum_{p\leq N}e_{b,c}(T^{p_{q}}(x,y))\log p=\sum_{i\text{ is good }}\sum_{\substack{p\leq N\\ p_{q}\in I_i}}e_{b,c}(T^{p_{q}}(x,y))\log p+\\
\sum_{i\text{ is bad}}\sum_{\substack{p\leq N\\ p_{q}\in I_i}}e_{b,c}(T^{p_{q}}(x,y))\log p.
\end{multline}
Since the cardinality of bad $i\leq \ell$ is $o(\ell)$, by Lemma \ref{lem:us1} (with $I=[0,N]$ and $J=I_i$) for each bad $I_i$, it follows that
$$
\Big|\sum_{i\text{ is bad}}\sum_{\substack{p\leq N\\ p_{q}\in I_i}}e_{b,c}(T^{p_{q}}(x,y))\log p\Big|\leq \sum_{i\text{ is bad}}\sum_{\substack{p\leq N\\ p_{q}\in I_i}}\log p\ll
o(\ell)\frac{H'}{q}N=o(N).$$
Therefore, we will only consider the first sum on the RHS of  \eqref{mult:jul}. Fix a good $i\leq \ell$.
Let $I_i=[u_i,u_i+H']$ and let $p_{q}\in [u_i,u_i+H']$. Then
$$|p_q-u_i|\leq H'= q^{1/3}\leq q_{n^\ast}^{1/(3(1-\epsilon))}\leq q_{n^\ast}^{1/3+10\epsilon}$$ (since we are in case \textbf{B.1.1.}) Therefore, by Corollary~\ref{cor:jad} with $\delta=4/7$ and  $n=n^{\ast}$, it follows that if we denote $(x_i,y_i)=T^{u_i}(x,y)$ and take $r=z_{n^{\ast}-1}$, then
$$
T^{p_q}(x,y)=T^{p_q-u_i}(x_i,y_i)=T^{(p_q-u_i)\mod r}\Big(x_i, y_i+P_{n^{\ast}}(x_i, p_q-u_i)\Big) +o(1).
$$
Moreover,  $P_{n^{\ast}}(x_i,\cdot)$ is a degree $1$ polynomial, and so by the definition of $a_1(\cdot)$ it follows that $P_{n^{\ast}}(x_i,p_q-u_i)=(p_q-u_i)\beta_i$, where $\beta_i:=g_{n^{\ast}}(x_i)$. Then, by \eqref{eq:bocoe}, $|\beta_i|\leq e^{-\tau q_{n^\ast-1}}\leq e^{-(\tau/2) r}$ (since $r=z_{n^\ast-1}\leq q_{n^\ast-1}\log^2 q_{n^\ast-1}$, see Proposition \ref{prop:ergchar}).
Therefore,
\begin{multline}\label{mult4}
\sum_{\substack{p\leq N\\ p_{q}\in I_i}}e_{b,c}(T^{p_{q}}(x,y))\log p=\sum_{a\leq r}\sum_{\substack{p\leq N, p_{q}\in I_i\\ p_q-u_i\equiv a\mod r}}e_{b,c}(T^{p_{q}}(x,y))\log p=\\
\sum_{a\leq r}e_{b,c}(T^a(x_i,y_i))\sum_{\substack{p\leq N,p_{q}\in I_i\\ p_q-u_i\equiv a\mod r}}e_{c}((p_q-u_i)\beta_i)\log p+ o(\sum_{\substack{p\leq N\\ p_{q}\in I_i}} \log p).
\end{multline}
The last term after summing over $i\leq \ell$ is $o(\theta(N))=o(N)$ and hence can be ignored.
Let
$$
h_{i,a}:=\sum_{\substack{p\leq N, p_{q}\in I_i\\ p_q\equiv u_i+a\mod r}}e_{c}((p_q-u_i)\beta_i)\log p,
$$
and let
$$
v_{i,a}:=\frac{N}{\varphi(q)}\sum_{\substack{(m,q)=1\\ m\equiv u_i+a \mod r\\ m\in I_i}}e_c((m-u_i)\beta_i). $$
Notice that if $(u_i+a,(r,q))>1$ (in particular $(r,q)>1$), then $p_q\equiv u_i+ a \mod r$ implies that $(r,q)|p$, which implies that $p=(r,q)\leq q<N^{5/6}$ and hence this can be ignored (after summing over $i\leq\ell$, it gives the contribution to the first summand on the RHS in~\eqref{mult4} at most $N^{5/6}\log N=o(N)$). Therefore, we will only consider those residue classes for which  $(u_i+a,(r,q))=1$. Let $a_0:= (1-u_i) \mod (r,q)$ (we could choose any $a_0$ such that $(a_0+u_i,(r,q))=1$).

By~\eqref{mult4} and the triangle inequality, it follows that
\begin{multline}\label{eq:ma}\sum_{p\leq N, p_q\in I_i}e_{b,c}(T^{p_q}(x,y))\log p=
v_{i,a_0}\sum_{\substack{a\leq r\\ (u_i+a,(r,q))=1}}e_{b,c}(T^a(x_i,y_i))+\\
\sum_{\substack{a\leq r\\ (u_i+a,(r,q))=1}}|v_{i,a}-v_{i,a_0}|+\sum_{\substack{a\leq r\\ (u_i+a,(r,q))=1}}|h_{i,a}-v_{i,a}|.
\end{multline}
Notice that by the definitions of $h_{i,a}$,$v_{i,a}$ and \eqref{eq:p12},
$$
\sum_{i\text{ is good}}\sum_{\substack{a\leq r\\ (z_i+a,(r,q))=1}}|h_{i,a}-v_{i,a}|=o(N),
$$
since when $a$ goes over $0,\ldots,r-1$, $u_i+a$~mod~$r$ runs over the same set.
Moreover, recall that $r\leq \log^3q$, so we can apply~\eqref{eq:j3} (with $\tau$ replaced by $\tau/2$ and $y=0$)  and we have
$\frac{r(r,q)}{\varphi((r,q))e^{\tau r/4}}\leq r(\log r) e^{-\tau r/4}=o(1)$ since $\frac{r(r,q)}{\varphi((r,q))}\leq r^2= o(e^{\tau r/4})$ and $\frac{q\log^3 q}{\varphi(q)\log^{100}q}=o(1)$.
Now, it follows from~\eqref{eq:j3} that
$$
\sum_{\substack{a\leq r\\ (z_i+a,(r,q))=1}}|v_{i,a}-v_{i,a_0}|\ll r\frac{N}{\varphi(q)}\Big[\frac{(r,q)\varphi(q)}{q\varphi((r,q))}\frac{H'}{e^{\tau r/4}}+ \frac{H'}{\log^{100}q}\Big]=o\Big(\frac{NH'}{q}\Big).
$$
Finally, by \eqref{eq:j3}, using $\varphi(q)\log^{100} q\geq rq$ and \eqref{eq:j3'} (with $\frac{rq}{(r,q)}$ in place of $q'$),
\begin{multline*}
|v_{i,a_0}|\ll \frac{N}{\varphi(q)}\Big[\frac{(r,q)\varphi(q)}{q\varphi((r,q))}\frac{H'}{e^{\tau r/4}}+ \frac{H'}{\log^{100}q}\Big]+ \frac{N}{\varphi(q)\varphi(r)}\Big|\sum_{\substack{(m,\frac{rq}{(r,q)})=1\\m\in I_i }}e_c(m\beta_i)\Big|\ll \\
\frac{NH'(r,q)}{rq\varphi((r,q))}+ \frac{N}{\varphi(q)\varphi(r)} 2\frac{|I_i|\varphi\Big(\frac{rq}{(r,q)}\Big)}{\frac{rq}{(r,q)}}\ll \frac{NH'(r,q)}{rq\varphi((r,q))},
\end{multline*}
where we used $\varphi\Big(\frac{rq}{(r,q)}\Big)=\varphi(\frac{r}{(r,q)})\varphi(q)\leq \frac{\varphi(r)\varphi(q)}{\varphi((r,q))}$.
By Proposition~\ref{prop:ergchar} (applied to $n^\ast-1$ instead of $n$) with  $d=(r,q)$, where $r=z_{n^\ast-1}$, it follows that
\begin{multline*}
|v_{i,a_0}\sum_{\substack{a\leq r\\ (u_i+a,(r,q))=1}}e_{b,c}(T^a(x_i,y_i))|= \\O\Big(\frac{NH'}{q}\Big) \Big|\frac{(r,q)}{r\varphi((r,q))}\sum_{\substack{a\leq r\\ (u_i+a,(r,q))=1}}e_{b,c}(T^a(x_i,y_i))\Big|=o\Big(
\frac{NH'}{q}\Big).
\end{multline*}
Using the above estimates and summing over $i\leq \ell$ in \eqref{eq:ma}, it follows that (recall that $\ell\leq \frac{q}{H'}+1$).
$$
\sum_{p\leq N}e_{b,c}(T^{p_q}(x,y)\log p=o(N).
$$
This finishes the proof in this case.\\

\textbf{B.1.2.} $q_{n^\ast}\leq q_n^{1-\epsilon}$. If $q_{n-1}\leq q_n^{1-\epsilon/2}$, let $p_n:=1$. If $q_{n-1}\geq q_n^{1-\epsilon/2}$, let $p_n\in \cP$ be a prime number in the interval $[\frac{q_n^{\epsilon/2}}{2},q_{n}^{\epsilon/2}]$ such that $(p_n,q_{n-1})=1$. Notice that such $p_n$ always exists since by the prime number theorem, for a sufficiently small $\epsilon'>0$,
$$
\prod_{p\in[\frac{q_n^{\epsilon/2}}{2},q_{n}^{\epsilon/2}]\cap \cP}p\geq \Big(\frac12q^{\epsilon/2}_n\Big)^{q_n^{\epsilon/2-\epsilon'}}\geq q_n>q_{n-1}.
$$
 Let $H:=N$, $q=p_nq_n$ and $r=q_{n-1}$. Notice that by the bound on $p_n$ (and remembering that $q_n\leq N$, so $N^{3\epsilon/4}>q_m^{2\epsilon/3}$),
$$
\frac{H}{q}=\frac{N}{p_nq_n}\geq \frac{N}{q_n^{1+2\epsilon/3}}\geq N^{1/6+\epsilon/4},
$$
since we are in \textbf{Case  B1.} Note that  $(q_n,q_{n-1})=1$ and so by the definition of $p_n$, $(q,r)=1$. Moreover, by the definition of $p_n$, in both cases,
$$
q^{1-\epsilon/2}=q_{n}^{1-\epsilon/2}p_n^{1-\epsilon/2}\geq q_{n-1}^{1-\epsilon^2/3}=r^{1-\epsilon^2/3}.
$$
Therefore, $r\leq q^{1-\xi}$ (recall that $\xi=\epsilon^{10}$). Hence, the assumptions of Theorem~\ref{thm:nr1} are satisfied (we use it for $H=x$ and $y=0$). This implies that (since $H=N$)
\be\label{eq:asw1}
\sum_{v=1}^{r} \Big|\sum_{\substack{p\leq N\\p_q~\equiv v \mod r }}\log p - \frac{N}{r}\Big|=o(N).
\ee
Notice that since $n^{\ast}\leq n-1$ (as $q_{n}^\ast\leq q_n^{1-\epsilon}$ since we are in case \textbf{B.1.2.}), the definition of $p_n$ implies
$$p_q\leq q=p_nq_n\leq q_n\max(1,\frac{q_{n-1}}{q_{n}^{1-\epsilon}})\leq q_n \frac{q_{n-1}}{q_{n^{\ast}}}.
$$
Similarly, by the definition of $p_n$ and the definition of $n^{\ast}$ (recalling that $n^{\ast}\leq n-1$),  $p_q\leq q=p_nq_n\leq q_n^{\epsilon/2} q_n\leq e^{2\tau q_{n-1}}$. Therefore,
$$
p_q\leq q_{n-1}\min\Big(\frac{q_{n}}{q_{n^{\ast}}},e^{2\tau q_{n-1}}  \Big).
$$
So, by Lemma~\ref{cor:jad2} with $n-1$ in place of $n$ (since $n^\ast<n$, it follows that $(n-1)^\ast=n^\ast$), $z=1$, $m=p_q$ and remembering that $r=q_{n-1}$, we get
$$
d(T^{p_q}(x,y),T^{p_q \mod r}(x,y))=o(1).
$$
Therefore,
$$
\sum_{p\leq N}e_{b,c}(T^{p_q}(x,y))\log p=\sum_{v\leq r}e_{b,c}(T^v(x,y))[\sum_{\substack{p\leq N\\p_q\equiv v \mod r}}\log p]+o(N).
$$
Moreover, by \eqref{eq:asw1} and unique ergodicity,
\begin{multline*}
\sum_{v\leq q_{n-1}}e_{b,c}(T^v(x,y))[\sum_{\substack{p\leq N\\p_q\equiv v \mod q_{n-1}}}\log p]=\\
\frac{N}{q_{n-1}}\sum_{v\leq q_{n-1}}e_{b,c}(T^v(x,y))+O\Big(\sum_{v\leq q_{n-1}}\Big|\sum_{\substack{p\leq N\\p_q\equiv v \mod q_{n-1}}}\log p-\frac{N}{q_{n-1}}\Big|\Big)=o(N).
\end{multline*}

This finishes the proof in this case and hence also completes the proof of case {\bf B1}.
\bigskip

\textbf{B2.} $N^{5/6-\epsilon}\leq q_n$. We will split the proof into several subcases.\\
\textbf{B2.1.} $q_{n^{\ast}}>N^{2/3-\eta/5}$. Let $r_{n^\ast}:=z_{n^\ast-1}$, where $z_{n^\ast-1}$ comes from Proposition \ref{prop:ergchar}. Let $H:=m r_{n^\ast}$, where $m$ is the largest such that $m r_{n^\ast}\leq q_{n^{\ast}}^{1-\eta}$. Notice that by definition $H\geq\frac12 q_{n^\ast}^{1-\eta}\geq \frac{1}{2}N^{(1-\eta)(2/3-\eta/5)}\geq N^{2/3-\eta}$. We partition the interval $[0,N]$ into consecutive disjoint intervals $I_i$ of length $H$. Let $I_i=I=[z,z+H]$. Notice that by the definition of $H$ it follows that $z=\ell_z r_{n^\ast}$. Denote $(x_z,y_z):=T^z(x,y)$. Let $p\in I$. Notice that  by the definition of $H$, $|p-z|\leq H\leq q_{n^\ast}^{1-\eta}$. So, by Corollary \ref{cor:jad} (with $\delta$ replaced by $\eta$ and using that $z$ is a multiple of $r_{n^\ast}$, so $p-z$ mod~$r_{n^\ast}$ equals $p$ mod~$r_{n^\ast}$), we get
$$
T^p(x,y)=T^{p-z}(x_z,y_z)= T^{p \mod r_{n^\ast}}(x_z,y_z+P_{n^\ast-1}(x_z,p-z))+o(1),
$$
where $P(x_z,\cdot):=P_{n^\ast-1}(x_z,\cdot-z)$ is a polynomial of degree $\leq [\frac{1}{\eta}]$ with coefficients satisfying \eqref{eq:bocoe}.
Let
$$
h_{a,I}:=\sum_{\substack{p\in I\\ p\equiv a\mod r_{n^\ast}}}e_c(P(x_z,p))\log p
$$
(we set $h_{a,I}=0$ whenever no $p\in I$ equals $a$ mod~$r_{n^\ast}$).
Then
$$
\sum_{p\in I}e_{b,c}(T^p(x,y))\log p=\sum_{\substack{(a,r_{n^\ast})=1\\a\leq r_{n^\ast}}}e_{b,c}(T^a(x_z,y_z)) h_{a,I}+o\Big(\sum_{p\in I} \log p\Big).
$$
Denote
$$
v_{I}:=\frac{1}{\varphi(r_{n^\ast})}\sum_{n\in I}e_c(P(x_z,p)).
$$
Then
$$
\sum_{\substack{(a,r_{n^\ast})=1\\a\leq r_{n^\ast}}}e_{b,c}(T^a(x_z,y_z)) h_{a,I}=$$$$
v_{I}\sum_{\substack{(a,r_{n^\ast})=1\\a\leq r_{n^\ast}}}e_{b,c}(T^a(x_z,y_z)) +O\Big(\sum_{\substack{(a,r_{n^\ast})=1\\a\leq r_{n^\ast}}}|h_{a,I}-v_{I}|\Big).
$$
Recall that $H\leq q_{n^\ast}^{1-\eta}$. Therefore, by \eqref{eq:bocoe}, the coefficients of $P(x_z,\cdot)$ satisfy the assumptions of Theorem~\ref{thm:nr4} (in which $N$ is replaced by $z$, cf.\ the definition of $P(x_z,\cdot)$, where obviously $H>z^{\frac23-\eta}$). Hence, applying this theorem to each relevant $a$ and summing over them, yields
$$
\sum_{\substack{(a,r_{n^\ast})=1\\a\leq r_{n^\ast}}}|h_{a,I}-v_{I}|=O(\eta^{1/2} H).
$$
Moreover,
$
|v_{I}|\leq  \frac{H}{\varphi(r_{n^\ast})}$.
Putting the above estimates together, we get
$$
\sum_{p\in I}e_{b,c}(T^p(x,y))\log p=o\Big(\sum_{p\in I} \log p\Big)+O(\eta^{1/2} H)+O\Big(H{\frac{1}{\varphi(r_{n^\ast})}}\sum_{(a,r_{n^\ast})=1}
e_{b,c}(T^a(x_z,y_z))\Big),
$$
and the last summand is $o(H)$ by Proposition~\ref{prop:ergchar} (with $d=r_{n^\ast}$). The proof is finished by summing over $I$.

\textbf{B2.2.} $q_{n^{\ast}}\leq N^{2/3-\eta/5}$. Let $n'<n$ be the largest number such that
\be\label{eq:n'}q_{n'}<N^{5/6-2\epsilon}.
\ee
 We consider two cases:
\\

\textbf{B2.2.1.} $q_{n^\ast}\geq q_{n'}^{1-\epsilon}$. Let  $z_{n^{\ast}-1}\in \{q_{n^\ast-1},p_{n^{\ast}-1}q_{n^\ast-1}\}$ be the number for which the minimum in \eqref{eq:erg} is obtained. Let $H=q_{n'}N^{1/6+\epsilon}$, $q:=q_{n'}$, $r:=z_{n^{\ast}-1}$ and $H':=q_{n'}^{1/3}$. Note that by the definition of $n^\ast$, $r=z_{n^{\ast}-1}\leq q_{n^\ast-1}^2\leq \log^3 q_{n^\ast}\leq \log^3 N$. Notice moreover that, by~\eqref{eq:n'}, it follows that $H\leq N^{1-\epsilon}$. Moreover, $H/q= N^{1/6+\epsilon}$. For $I\subset [0,N]$ and $J\subset[0,q]$, we call the pair $(I,J)$ ``good'' if
\be\label{eq:cas3}
\sum_{\substack{v\leq r\\ (v,(r,q))=1}}\sup_{\beta}\Big|\sum_{\substack{p\in I\\p_q\equiv v \mod r\\ p_q\in J}}e(p_q\beta)\log p- \frac{H}{\varphi(q)}\sum_{\substack{(a,q)=1\\ a\equiv v\mod r\\ a\in J}}e(a\beta)\Big|=o\Big(\frac{|I||J|}{q}\Big).
\ee
Otherwise, the pair $(I,J)$ is called ``bad''. We use Theorem \ref{thm:nr2} with $x=N$ and $H, H',q,r$ (defined above).

This will give us intervals $[y,y+H]$ of length $H$ from which we are interested in those for which the LHS sum $\sum_{z<q}\sup_{\beta\in\R,v<r}|\ldots|\ll_\epsilon \frac{HH'}{(\log N)^{100}}$. Most of them will satisfy this requirement. More than that, we can
decompose $[0,N]=\bigcup_{i=1}^\ell I_i$, where $\{I_i\}_{i=1}^{\ell}$ are pairwise disjoint, $|I_i|=H$ for $2\leq i< \ell$ and $|I_1|,|I_\ell|\leq H$ (we additionally assume that $|I_1|\geq H/2$), where most of the $I_i$ will satisfy the above requirement. Then, we can fix such an $I_i$ and repeat the same procedure by considering intervals $[z,z+H']$, where now we require that on this interval $\sup_{\beta\in\R,v<r}|\ldots|\ll_\epsilon \frac{HH'}{q(\log N)^{100}}$. For most $z$ we will see this requirement satisfied, and in fact
we can decompose $[0,q]:=\bigcup_{j=1}^{\ell'} J^i_j$,  where $\{J^i_j\}_{i=1}^{\ell'}$ are pairwise disjoint, $|J^i_j|=H'$ for $2\leq j< \ell'$ and $|J^i_1|,|J^i_{\ell'}|\leq H'$, and for most of the $j$ we will see the requirement satisfied. Finally, summing over $v\leq r$ (and using $r\leq (\log N)^{100}$), will yield a bound $\ll_\epsilon \frac {HH'}{q}\frac{\log^3 N}{\log^{100}N}$. For the remaining $i$, we can still perform the same procedure, which will give us pairs of the form $(I_i,J^i_j)$, where
(by Theorem~\ref{thm:nr2}),
we get that the cardinality of ``bad'' pairs $(I_i,J^i_j)$ is at most $o(\ell \cdot \ell')$. We will also call the pairs of the form $(I_1,J^1_j)$, $(I_\ell,J^\ell_j)$ and $(I_i, J^i_1)$, $(I_i,J^i_{\ell'})$ bad. Notice that adding the new bad pairs give that the total cardinality of bad pairs is  $2\cdot \ell'+2\ell+o(\ell \cdot \ell')=o(\ell\cdot\ell')$, since $\ell,\ell'\to +\infty$.
We have
\begin{multline}\label{eq:asasa}
\sum_{p\leq N}e_{b,c}(T^p(x,y))\log p=
\sum_{i\leq \ell, j\leq \ell'}\sum_{\substack{p\in I_i\\ p_q\in J^i_j}}e_{b,c}(T^p(x,y))\log p=\\
\sum_{(I_i,J^i_j)\text{ is good}}\sum_{\substack{p\in I_i\\ p_q\in J^i_j}}e_{b,c}(T^p(x,y))\log p+\sum_{(I_i,J^i_j)\text{ is bad}}\sum_{\substack{p\in I_i\\ p_q\in J^i_j}}e_{b,c}(T^p(x,y))\log p.
\end{multline}
Moreover, by the bound on the cardinality of bad pairs, by Lemma \ref{lem:us1} (for each bad pair) and since $\ell \leq 2N/H$ and $\ell'\leq 2q/H'$,
$$
\sum_{(I_i,J^i_J)\text{ is bad}}\sum_{\substack{p\in I_i\\ p_q\in J^i_j}}e_{b,c}(T^p(x,y))\log p\ll o(\ell\cdot \ell')\frac{HH'}{q}=o(N).
$$
Therefore, the second term on the RHS of \eqref{eq:asasa} will be ignored. Fix $2\leq i \leq \ell-1$ and $2\leq j\leq \ell'-1$ such that the pair $(I_i,J^i_j)$ is good.


Let $I_i=[u_i,u_i+H]$. Notice that by the definition of $n'$, it follows that $q_{n'+1}\geq N^{5/6-2\epsilon}$, and by assumptions, $q_{n^\ast}\leq N^{2/3-\eta/5}$ (so $q_{n'+1}/q_{n^\ast}\geq N^{1/6+\epsilon}$). Therefore, $n>n'\geq n^{\ast}$ and so, by the definition of $n^{\ast}$, $N^{1/6+\epsilon}\leq N^{5/6-2\epsilon}\leq q_{n'+1}\leq e^{\tau q_{n'}}$. So,
$$
H=qN^{1/6+\epsilon}=q_{n'}N^{1/6+\epsilon}\leq q_{n'}\min(\frac{q_{n'+1}}{q_{n^{\ast}}}, e^{\tau q_{n'}}).
$$
Therefore, and using $n'^{\ast}=n^{\ast}$ (since $n>n'\geq n^{\ast}$), for $p\in I_i$, we have $p-u_i\leq H\leq q_{n'}\min(\frac{q_{n'+1}}{q_{n'^{\ast}}}, e^{\tau q_{n'}})$ and by using Lemma~\ref{cor:jad2} with $n=n'$, $z=1$ and $m=p-u_i$, we get

$$
T^p(x,y)=T^{p-u_i}(T^{u_i}(x,y))=
T^{(p-u_i)\mod q}(T^{u_i}(x,y))+o(1)=$$
$$
T^{(p-u_i)\mod q+ (u_i\mod q)}(T^{u_i-(u_i\mod q)}(x,y))+o(1)=$$$$
T^{p_q}(x_{i},y_i)+o(1)+o(1)=T^{p_q}(x_i,y_i)+o(1),
$$
where $(x_i,y_i)=T^{u_i-(u_i\mod q)}(x,y)$. Therefore,
$$
\sum_{\substack{p\in I_i\\ p_q\in J^i_j}}e_{b,c}(T^p(x,y))\log p=\sum_{\substack{p\in I_i\\ p_q\in J^i_j}}e_{b,c}(T^{p_q}(x_i,y_i))\log p+o\Big(\sum_{\substack{p\in I_i\\ p_q\in J^i_j}}\log p\Big).
$$
Notice that the last term after summing over $j,i$ contributes $o(N)$ to \eqref{eq:asasa} and hence can be ignored. Moreover, splitting into residue classes $\mod r$, we get
\begin{multline}\label{eq:tgh2}
\sum_{\substack{p\in I_i\\ p_q\in J^i_j}}e_{b,c}(T^{p_q}(x_i,y_i))\log p=\sum_{v\leq r}\sum_{\substack{p\in I_i, p_q\in J^i_j\\ p_q\equiv v \mod r}}e_{b,c}(T^{p_q}(x_i,y_i))\log p=\\ \sum_{\substack{v\leq r\\(v,(r,q))=1}}\sum_{\substack{p\in I_i, p_q\in J^i_j\\ p_q\equiv v \mod r}}e_{b,c}(T^{p_q}(x_i,y_i))\log p+\sum_{\substack{v\leq r\\(v,(r,q))>1}}\sum_{\substack{p\in I_i, p_q\in J^i_j\\ p_q\equiv v \mod r}}e_{b,c}(T^{p_q}(x_i,y_i))\log p.
\end{multline}
If $v\leq r$ is such that $(v,(r,q))>1$, then $p_q\equiv v \mod r$ implies that $(v,(r,q))|p$, which is only possible if $(v,(r,q))=p$. In particular, this means that $p<r$. However, by the assumptions, $r\leq \log^3N$ and
hence, $p\in I_i$ and  $(v,(r,q))=p$ is only possible if $i=1$ since the intervals $I_i$ are disjoint and have length at least $H/2\geq N^{1/6+\epsilon}/2$. But, by definition, we consider a good pair $(I_i,J^i_j)$ which implies that $i\geq 2$. This implies that the second sum in \eqref{eq:tgh2} is empty. Let  $p\in I_i$ be such that $p_q\in J^i_j=[z_j^i,z_j^i+H']$. Notice that $p_q-z_{j}^i\leq H'\leq q_{n'}^{1/3}\leq q_{n^\ast}^{\frac{1}{3(1-\epsilon)}}$ (since we are in case {\bf B2.2.1}). Let $(\tilde{x}_{i,j},\tilde{y}_{i,j}):=T^{z^i_j}(x_i,y_i)$. Applying Corollary \ref{cor:jad} with $n+1=n^\ast$, $\delta=1-\frac{1}{3(1-\epsilon)}$ and $w\leq \log^3 q_{n^\ast-1}$ satisfying $wq_{n^\ast-1}= z_{n^\ast-1}$ (see Proposition \ref{prop:ergchar}),
we get that  (using $r=z_{n^\ast-1}$)
\begin{multline}\label{eq:tpg}
T^{p_q}(x_i,y_i)=T^{p_q-z_j^i}(T^{z_j^i}(x_i,y_i))=\\
T^{(p_q-z_j^i) \mod r}(\tilde{x}_{i,j},\tilde{y}_{i,j}+P_{n^\ast-1}(\tilde{x}_{i,j},p_q-z_{j}^i))+o(1).
\end{multline}
Moreover, since $\delta>4/7$ ($\epsilon$ is small), it follows that $\deg P_{n^\ast-1}\leq 1$, and so by Proposition \ref{lem:toadd1}, it follows that $P_{n^\ast-1}(\tilde{x}_{i,j},p_q-z_{j}^i)=(p_q-z_{j}^i)\beta_{ij}$, for some $|\beta_{ij}|\leq e^{-\tau q_{n^\ast-1}}\leq e^{-\tau r/2}$. Therefore, using \eqref{eq:tpg}, if we denote $(x_{i,j},y_{i,j}):=T^{-z^i_j \mod r}(\tilde{x}_{i,j},\tilde{y}_{i,j})$ (so $T^{(p_q-z^i_j)\mod r}(\tilde{x}_{i,j},\tilde{y}_{i,j})=T^{p_q\mod r}(x_{i,j},y_{i,j})+o(1)$ in view of Corollary~\ref{cor:jad}), we obtain
\begin{multline}\label{eq:caa2}
\sum_{\substack{v\leq r\\ (v,(r,q))=1}}\sum_{\substack{p\in I_i,p_q\in J^i_j\\p_q\equiv v \mod r}}e_{b,c}(T^{p_q}(x_i,y_i))\log p=\\
\sum_{\substack{v\leq r\\ (v,(r,q))=1}}e_{b,c}(T^v(x_{i,j},y_{i,j})h_{ij,v} +o(\sum_{\substack{p\in I_i,p_q\in J^i_j}}\log p),
\end{multline}
where
$$
h_{ij,v}:=\sum_{\substack{p\in I_i,p_q\in J^i_j\\p_q\equiv v \mod r}}e_c((p_q-z_j^i)\beta_{ij})\log p.
$$
Notice that after summing over $j$ and $i$,
$$
\sum_{i,j}o\Big(\sum_{\substack{p\in I_i,p_q\in J^i_j}}\log p\Big)=o\Big(\sum_{p\leq N}\log p\Big)=o(N),
$$
and hence this term can be ignored. Let
$$
u_{ij,v}:=\frac{H}{\varphi(q)}\sum_{\substack{(a,q)=1\\a\equiv v \mod r\\ a\in J_j^i}}e_c((a-z_j^i)\beta_{ij}).
$$
Then, by the triangle inequality,
\begin{multline}\label{eq:lasmu}
\Big|\sum_{\substack{v\leq r\\ (v,(r,q))=1}}e_{b,c}(T^v(x_{i,j},y_{i,j})h_{ij,v}\Big|\leq
\Big|u_{ij,1}\sum_{\substack{v\leq r\\ (v,(r,q))=1}}e_{b,c}(T^v(x_{i,j},y_{i,j}))\Big|+\\\sum_{\substack{v\leq r\\ (v,(r,q))=1}}|h_{ij,v}-u_{ij,v}|+
\sum_{\substack{v\leq r\\ (v,(r,q))=1}}|u_{ij,v}-u_{ij,1}|.
\end{multline}
By \eqref{eq:j3} (with $\tau/2$ instead $\tau$), using $\varphi(q)\log^{100} q\geq qr$, $|J_j^i|= H'\geq q^{1/3}$ and $J_j^i\subset [0,q]$, we obtain
$$
|u_{ij,1}|\ll \frac{H}{\varphi(q)}\Big[\frac{(r,q)\varphi(q)}
{q\varphi((r,q))}\frac{H'}{e^{\tau r/2}}+ \frac{H'}{\log^{100}q}\Big]+ \frac{H}{\varphi(q)\varphi(r)}
\Big|\sum_{\substack{(m,\frac{rq}{(r,q)})=1\\m\in J_j^i }}e_c(m\beta_{ij})\Big|\leq$$
$$
\frac{HH'(r,q)}
{q\varphi((r,q))e^{\tau r/2}}+ \frac{HH'}{\varphi(q)\log^{100}q}+ \frac{H}{\varphi(q)\varphi(r)}
\Big|\sum_{\substack{(m,\frac{rq}{(r,q)})=1\\m\in J_j^i }}e_c(m\beta_{ij})\Big|\leq
$$
$$
\frac{HH'(r,q)}
{qr\varphi((r,q))r}+ \frac{HH'}{qr}+ \frac{H}{\varphi(q)\varphi(r)}
\Big|\sum_{\substack{(m,\frac{rq}{(r,q)})=1\\m\in J_j^i }}e_c(m\beta_{ij})\Big|\leq $$$$\frac{HH'(r,q)}
{qr\varphi((r,q))}+  \frac{H}{\varphi(q)\varphi(r)}
\Big|\sum_{\substack{(m,\frac{rq}{(r,q)})=1\\m\in J_j^i }}e_c(m\beta_{ij})\Big|.
$$
Moreover, by \eqref{eq:j3'} (with $\frac{rq}{(r,q)}$ in place of $q'$),
$$
\Big|\sum_{\substack{(m,\frac{rq}{(r,q)})=1\\m\in J_j^i }}e_c(m\beta_{ij})\Big|\leq \frac{|J^i_j|\varphi\Big(\frac{rq}{(r,q)}\Big)}{\frac{rq}{(r,q)}},$$
so finally
$$
|u_{ij,1}|\ll \frac{HH'(r,q)}{rq\varphi((r,q))},$$
where we used $\varphi\Big(\frac{rq}{(r,q)}\Big)=\varphi(\frac{r}{(r,q)})\varphi(q)\leq \frac{\varphi(r)\varphi(q)}{\varphi((r,q))}$.

By the definition of $r=z_{n^{\ast}-1}$, using Proposition \ref{prop:ergchar} with $r=z_{n^\ast-1}$ and $d=(r,q)$, it follows that
\begin{multline*}
|u_{ij,1}\sum_{\substack{v\leq r\\ (v,(r,q))=1}}e_{b,c}(T^v(x_{i,j},y_{i,j}))|=\\O\Big(\frac{HH'}{q}\Big)\Big|\frac{(r,q)}{r\varphi((r,q))}\sum_{\substack{v\leq r\\ (v,(r,q))=1}}e_{b,c}(T^v(x_{i,j},y_{i,j}))\Big|=o(\frac{HH'}{q}).
\end{multline*}

Moreover, by \eqref{eq:cas3} (since $(I_i,J^i_j)$ is good),
$$
\sum_{\substack{v\leq r\\ (v,(r,q))=1}}|h_{ij,v}-u_{ij,v}|=o\Big(\frac{HH'}{q}\Big).
$$
Finally, by Corollary~\ref{cor:j3}, see~\eqref{eq:j3}, summing over $v\leq r$ and using
$\frac{r(r,q)}{\varphi((r,q))e^{\tau r/2}}\leq r\log r e^{-\tau r/2}=o(1)$ and $\frac{q}{\varphi(q)\log^{100}q}=o(1)$,
$$
\sum_{\substack{v\leq r\\ (v,(r,q))=1}}|u_{ij,v}-u_{ij,1}|\ll r\cdot\frac{H}{\varphi(q)}\Big[\frac{(r,q)\varphi(q)}{q\varphi((r,q))}\frac{H'}{e^{\tau r/2}}+ \frac{H'}{\log^{100}q}\Big]=o\Big(\frac{HH'}{q}\Big).
$$

Therefore and by \eqref{eq:lasmu} and \eqref{eq:caa2} (ignoring the last term in \eqref{eq:caa2}),
$$
\sum_{\substack{p\in I_i\\ p_q\in J^i_j}}e_{b,c}(T^p(x,y))\log p=o\Big(\frac{HH'}{q}\Big).
$$
Summing over all $j$, we get
$$
\Big|\sum_{p\in I_i}e_{b,c}(T^p(x,y))\log p\Big|\leq \frac{q}{H'}o\Big(\frac{HH'}{q}\Big)=o(H).
$$
Summing over all $i$, yields
$$
\sum_{p\leq N}e_{b,c}(T^p(x,y))\log p=o(N)
$$
which finishes the proof.

\textbf{B2.2.2.} $q_{n^\ast} \leq q_{n'}^{1-\epsilon}$. In this case, we will constantly use that this implies that $n^{\ast}=n'^{\ast}$. If $q_{n'-1}\leq q_{n'}^{1-\epsilon/2}$, let $p_{n'}:=1$. If $q_{n'-1}\geq q_{n'}^{1-\epsilon/2}$, let $p_{n'}\in \cP$ be a prime number in the interval $[\frac{q_{n'}^{\epsilon/{10}}}{2},q_{n'}^{\epsilon/{10}}]$ such that $(p_{n'},q_{n'-1})=1$. As in {\bf B1.1.2.}, notice that such $p_{n'}$ always exists since by the prime number theorem,
$$
\prod_{p\in[\frac{q_{n'}^{\epsilon/{10}}}{2},q_{n'}^{\epsilon/10}]\cap \cP}p\geq \frac{1}{4}\Big(q_{n'}\Big)^{q_{n'}^{1/2}}>q_{n'-1}.
$$
 Let  $q:=p_{n'}q_{n'}$, $H:=qN^{1/6+\epsilon}$ and $r:=q_{n'-1}\geq q_{n^\ast}$. Notice that $H/q= N^{1/6+\epsilon}$, and by \eqref{eq:n'},
$$
H=p_{n'}q_{n'}N^{1/6+\epsilon}\leq q_{n'}^{1+\epsilon/10}N^{1/6+\epsilon}\leq N^{(5/6-2\epsilon)(1+\epsilon/10)+1/6+\epsilon}<N^{1-\epsilon/2}.
$$
Moreover, since $q_{n^{\ast}}\leq q_{n'}^{1-\epsilon}$ and $\xi=\epsilon^{10}$,
$$
q=p_{n'}q_{n'}\geq q_{n'-1}^{1+\epsilon/20}=r^{1+\epsilon/20}>r^{\frac{1}{1-\xi}},
$$
so $r<q^{1-\xi}$, and
$(q,r)=1$ (by the definition of $p_{n'}$ and $(q_{n'},q_{n'-1})=1$). Thus, the assumptions of Theorem~\ref{thm:nr1} are satisfied with $x=N$ and $H,q,r$. Therefore, for some $V\leq H$, we can decompose $[0,N]=\bigcup_{i=1}^\ell I_i \cup [0,V]$, where $I_i:=[(i-1)H+V,iH+V]$ and moreover for ``most of'' $i\leq \ell$ (that is, for $\ell-o(\ell)$),
\be\label{eq:cas}
\sum_{v=1}^{r}\Big|\sum_{\substack{p\in I_i\\p_q\equiv v \mod r}}\log p- \frac{H}{r}\Big|=o(H).
\ee
By the bound on $H$, we also have
$$
\sum_{p\in [0,V]}1=O(V)=O(H)=o(N),
$$
and hence the interval $[0,V]\subset [0,N]$ can be ignored. Let $p\in I_i=[u_i,u_i+H]$.  Then $p-u_i\leq H\leq q_{n'}^{1+\epsilon/10}N^{1/6+\epsilon}$ (by the definitions of $H$ and $p_{n'}$). Moreover, by the definition of $n'$ (see \eqref{eq:n'}) and since we are in \textbf{B.2.2.} (using also $\epsilon\leq\eta^2/1000$),
\begin{multline*}
H\leq q_{n'}q_{n'}^{\epsilon/10}N^{1/6+\epsilon}\leq q_{n'}N^{\epsilon/10(5/6-2\epsilon)+1/6+\epsilon}\leq\\ q_{n'}N^{1/6+2\epsilon}\leq q_{n'} \frac{N^{5/6-3\epsilon}}{N^{2/3-\eta/5}}\leq q_{n'} \frac{q_{n'+1}}{q_{n^\ast}}.
\end{multline*}
By \eqref{eq:nast} and since $n'>n^{\ast}$, it follows that $q_{n'+1}\leq e^{\tau q_{n'}}$. Moreover, by the definition of $n'$ (see \eqref{eq:n'}), $q_{n'+1}\geq N^{5/6-2\epsilon}$. Therefore,
$$
 H\leq q_{n'}^{1+\epsilon/10}N^{1/6+\epsilon}\leq q_{n'}q_{n'+1}\leq q_{n'}e^{\tau q_{n'}}.
$$
Putting together the two above bounds on $H$, we get $$H\leq q_{n'}\min( \frac{q_{n'+1}}{q_{n^\ast}}, e^{\tau q_{n'}}).$$
Since $p-u_i\leq H$, by Lemma~\ref{cor:jad2} with $n=n'$ (we may use the lemma since $q_{n^\ast}=q_{n'^\ast}$), $m=p-u_i$ and $z=p_{n'}$ (recall that $q=p_{n'}q_n$),
\begin{multline*}
T^p(x,y)=T^{p-u_i}T^{u_i}(x,y)=T^{[(p-u_i)\mod q ]}(T^{u_i}(x,y))+o(1)=\\
T^{p_q}(T^{u_i-[u_i\mod q]}(x,y))+o(1)=T^{p_q}(x_i,y_i)+o(1),
\end{multline*}
where $(x_i,y_i)=T^{u_i-[u_i\mod q] }(x,y)$. Therefore,
\be\label{eq:lse2}
\sum_{p\leq N}e_{b,c}(T^p(x,y))\log p=\sum_{i=1}^{\ell}\sum_{p\in I_i}e_{b,c}(T^{p_q}(x_i,y_i))\log p+o(N).
\ee
Moreover,
$$
\sum_{p\in I_i}e_{b,c}(T^{p_q}(x_i,y_i))\log p=\sum_{v=1}^r\sum_{\substack{p\in I_i\\p_q\equiv v \mod r}}e_{b,c}(T^{p_q}(x_i,y_i))\log p.
$$
Note that since we are in {\bf B2.2.2.} and by the definition of $p_{n'}$ (note that if $q_{n'-1}\geq q_{n'}^{1-\epsilon/2}$ then $\frac{q_{n'-1}}{q_{n'}^{1-\epsilon}}>q_{n'}^{\epsilon/2}>p_{n'}$),
$$
q=p_{n'}q_{n'}\leq q_{n'} \max\Big(1,\frac{q_{n'-1}}{q_{n'}^{1-\epsilon}}\Big)\leq q_{n'}\max\Big(1,\frac{q_{n'-1}}{q_{n^\ast}}\Big).
$$
Since $n>n'>n^{\ast}$  (and $q_{n^{\ast}}\to +\infty$ ), by the definition of $n^\ast$ it follows that  $q_{n'}\leq e^{\tau q_{n'-1}}$. Hence, $q\leq q_{n'}^{1+\epsilon/10}\leq e^{2\tau q_{n'-1}}$. Therefore, by using Lemma~\ref{cor:jad2} with $m=p_q\leq q\leq q_{n'-1}\min(\frac{q_{n'}}{q_{n^\ast}},e^{2\tau q_{n'-1}})$, $n=n'-1$ and $z=1$, we  get that $p_q\equiv v \mod r$ (recall that $r=q_{n'-1}$) implies that
$$
d(T^{p_q}(x_i,y_i),T^v(x_i,y_i))=o(1).
$$
Therefore,
\begin{multline}\label{eq:haj}
\sum_{\substack{p\in I_i\\p_q\equiv v \mod r}}e_{b,c}(T^{p_q}(x_i,y_i))\log p=\\e_{b,c}(T^v(x_i,y_i))\sum_{\substack{p\in I_i\\p_q\equiv v \mod r}}\log p+o\Big(\sum_{\substack{p\in I_i\\p_q\equiv v \mod r}}\log p\Big).
\end{multline}
Let
$$
h_{i,v}:=\sum_{\substack{p\in I_i\\p_q\equiv v \mod r}}\log p.
$$
Then, by \eqref{eq:haj}, summing over $v$,
$$
\sum_{p\in I_i}e_{b,c}(T^{p_q}(x_i,y_i))\log p=\frac{H}{r}\sum_{v=1}^re_{b,c}(T^v(x_i,y_i))+O\Big(\sum_{v=1}^r
\Big|h_{i,v}-\frac{H}{r}\Big|\Big)+o\Big(\sum_{p\in I_i} \log p\Big).
$$
By unique ergodicity, $\frac{H}{r}\sum_{v=1}^re_{b,c}(T^v(x_i,y_i))=o(H)$. Using \eqref{eq:cas} (for all $i\leq\ell$ but $o(\ell)$), we get $$O\Big(\sum_{v=1}^r\Big|h_{i,v}-\frac{H}{r}\Big|\Big)=o(H).$$
 Summing over $i$, and using \eqref{eq:lse2}, we get
$$
\sum_{p\leq N}e_{b,c}(T^p(x,y))\log p=o(N).
$$
This finishes the proof.

\part{Counterexamples}

\section{Counterexamples}\label{sec:count}
In what follows, for  every irrational rotation $\alpha$, we will construct a continuous cocycle $g=g_\alpha:\T\to \R$ such that the Anzai skew product $T(x,y)=T_{\alpha, g}(x,y)=(x+\alpha, y+g(x))$ is uniquely ergodic and there exists  $f\in C(\T^2)$ such that
$$
\lim_{N\to+\infty}\frac{1}{\pi(N)}\sum_{p\leq N}f(T^p(0,0)) \text{ does not exist.}\footnote{We recall that
the limit $\lim_{N\to+\infty}\frac{1}{\pi(N)}\sum_{p\leq N}f(T^p(0,0))$ exists if and only if the limit
$\lim_{N\to+\infty}\frac{1}{N}\sum_{p\leq N}f(T^p(0,0))\log p$ does.}
$$
More generally, our result applies to all $A\subset \N$ which are {\em almost  sparse}.
\begin{definition}\label{def:as}A set $A\subset \N$ is called {\em almost sparse} if  the following three conditions hold:
\begin{enumerate}
\item[i.]  $\lim_{N\to +\infty}\frac{|A\cap[1,N]|}{|A\cap[1,2N]|}$ exists and is positive;
\item[ii.] there exists a sequence of sets $B_N\subset A\cap [1,N]$, $N\geq1$, satisfying $$\lim_{N\to +\infty}\frac{|B_N|}{|A\cap[1,N]|}=0$$ and
$$
\min_{k,l\in (A\cap [0,N])\setminus B_N,\\ k\neq l}|k-l|\to+\infty\text{ as } N\to +\infty;
$$
\item[iii.] $|A\cap [N,2N]\setminus B_{2N}|\neq\emptyset$ eventually.
\end{enumerate}
\end{definition}

\begin{remark} Note that if the limit in i.\ is $<1$ then iii.\ holds automatically.

Note also that very sparse sequences will automatically satisfy i.\ and ii.\ while iii.\ in general is not satisfied, cf.\ $A=\{2^n:\:n\geq1\}$. A reason to add condition iii.\ is that we aim at presenting a universal construction which yields a counterexample for {\bf all} irrational rotations. If $A$ is very sparse (like lacunary sequences) then one can also give a relevant construction in which an irrational rotation is adapted to $A$.\end{remark}

In order to see that the set $\cP$ is almost sparse\footnote{We cannot expect more than that: indeed, the twin prime conjecture implies that, arbitrarily far, there are primes which differ by $2$. Unconditionally, recent results of Zhang and Maynard show that there are infinitely many primes with bounded gaps (with the gap $\leq 249$).}, let us first notice that by the prime number theorem the limit in i.\ exists and equals 1/2. To obtain ii.\ recall:

\begin{theorem} [V.\ Brun, 1919] Given $a$ an even natural number, set $D_a(N):=\{p\leq N:\:p,p+a\in\cP\}$. Then
$$
\Big|D_a(N)|\ll_a\frac{N(\log\log N)^2}{(\log N)^2}.$$\end{theorem}

In Brun's theorem we have implicit constants $C_a$, we now select a slowly increasing $c(N)\to\infty$, so that depending on the constants $C_a$, the set $B_N:=\bigcup_{a\leq c(N)}D_a(N)$ yields ii.

However, the class of almost sparse sets is far beyond the set of prime numbers, cf.\ the remark below to see another classical class of subsets along which an equidistribution is of interest.

\begin{remark}If $P\in\Z[x]$ is a non-constant polynomial with integer coefficients, $deg\, P\geq 2$ and $A:=\{P(n)\}_{n\in \N}$, then $A$ is almost sparse. Indeed, let $P(x)=c_{r}x^{r}+\ldots+c_0$ with $c_r\neq0$, $r\geq 2$. Assume WLOG that $c_r>0$. Fix $\epsilon>0$, then there exists $M>0$ such that for $x>M$, we have
$$\frac{c_rx^r}{1+\epsilon}\leq P(x)\leq (1+\epsilon)c_rx^r.$$
If we set $\alpha_N:=|\{ n>M:\:P(n)<N\}|$, $\beta_N:=|\{n>M:\: \frac{c_rn^r}{1+\epsilon}<N\}|$ and $\gamma_N:=|\{n>M:\: c_rn^r(1+\epsilon)<N\}|$, then
$$
\frac1{(1+\epsilon)^{2/r}}2^{-1/r}\leq \frac{\gamma_N}{\beta_{2N}}\leq \frac{\alpha_N}{\alpha_{2N}}\leq \frac{\beta_N}{\gamma_{2N}}\leq (1+\epsilon)^{2/r} 2^{-1/r},
$$
so $\lim_{N\to\infty}\alpha_N/\alpha_{2N}=2^{-1/r}$. Moreover, $P(\cdot)$ is eventually increasing with $P(n+1)-P(n)\to\infty$ (since $\deg P\geq 2$), so the existence of $B_N$ follows. Notice finally that the assumption $\deg P\geq 2$ is necessary:\footnote{Degree~1 polynomials yield sets $A$ satisfying i.\ and iii.\ but not ii.}  below, we will show the existence of uniquely ergodic Anzai skew products which are NOT equidistributed along $A$, and such absence of equidistribution  does not hold for instance for $P(x)=x$.
\end{remark}

With the above definition, our main result will be now:
\begin{theorem}\label{thm:main2} Let $A\subset \N$ be an almost sparse set and let $\alpha\in \R\setminus \Q$. There exists $g=g_{A,\alpha}\in C(\T)$ such that $T=T_{\alpha,g}:\T^2\to \T^2$ given by $T(x,y)=(x+\alpha,y+g(x))$ is uniquely ergodic and there exists $f\in C(\T^2)$ such that
$$
\lim_{N\to +\infty}\frac{1}{|A\cap [0,N]|}\sum_{n\in A, n\leq N} f(T^n(0,0)) \text{ does not exist}.
$$
\end{theorem}

Theorem~\ref{thm:main2} should be compared with Bourgain's theorem which asserts that if $A=\cP$ then for every $g, \alpha$ as above, and every $f\in C(\T^2)$ the limit exists for Lebesgue-a.e.\ $(x,y)\in \T^2$. From now on, the set $A$ and $\alpha$ are fixed, so we omit them in the formulations below. Theorem~\ref{thm:main2} is a consequence of the following lemma:

\begin{lemma}\label{lem:phi} There exist $g\in C(\T)$ such that $T_{\alpha, g}$ is uniquely ergodic and  an increasing sequence $\{M_n\}$ of natural numbers such that, for every $\epsilon>0$, we can find $n_0$ for which for every $n\geq n_0$ and every $k\in A\cap[M_n,2M_n]\setminus B_{2M_n}$, we have
$$
S_k(g)(0)\in (-\epsilon,\epsilon) \text{ if } n\text{ is even }
$$
and
$$
S_k(g)(0)\in (1/2-\epsilon,1/2+\epsilon) \text{ if } n\text{ is odd.}
$$
\end{lemma}
Notice that the assertion is non-trivial provided that, as we have assumed in iii., $A\cap[M_n,2M_n]\setminus B_{2M_n}$ is not empty.
We will prove Lemma \ref{lem:phi} in a separate subsection. Before we do that, let us show how it implies the main theorem.

\begin{proof}[Proof of Theorem \ref{thm:main2}] Notice that if the limit exists then, by ii., also does
$$
\lim_{N\to +\infty}\frac{1}{|A\cap [0,N]|}\sum_{k\in A\setminus B_N, k\leq N} f(T^k(0,0)),
$$
as the sets $B_N$ have density which goes to $0$ (relatively on $A$).
Moreover, because of i., $\lim_{N\to +\infty} \frac{|A\cap[1,N]\setminus B_{2N}|}{|A\cap[1,N]|}=1$. Therefore,
\begin{multline*}
\lim_{n\to +\infty}\frac{1}{|A\cap [M_n,2M_n]|}\sum_{k\in A\cap[M_n,2M_n]\setminus B_{2M_n}} f(T^k(0,0))=\\
\lim_{n\to +\infty}\frac{|A\cap [0,2M_n]|}{|A\cap [M_n,2M_n]|}\frac{1}{|A\cap [0,2M_n]|}\sum_{k\in A\cap[0,2M_n]\setminus B_{2M_n}} f(T^k(0,0))-\\
\lim_{n\to +\infty}\frac{|A\cap [0,M_n]|}{|A\cap [M_n,2M_n]|}\frac{1}{|A\cap [0,M_n]|}\sum_{k\in A\cap[0,M_n]\setminus B_{2M_n}} f(T^k(0,0))
\end{multline*}
also exists. Moreover, by iii., the summation on the LHS summand is non-trivial.

But for any $k\in [M_n,2M_n]$ with $n$ even, we have $T^k(0,0)=(k\alpha,S_k(g)(0))\in \T\times(-\epsilon,\epsilon)$ and for every $k\in [M_n,2M_n]$ with $n$ odd, we have $T^k(0,0)=(k\alpha,S_k(g)(0))\in \T\times(1/2-\epsilon,1/2+\epsilon)$. It is therefore enough to take any $f\in C(\T^2)$ of the form $f(x,y)=\tilde{f}(y)$, with $\tilde{f}\in C(\T)$ satisfying $\tilde{f}(0)=0$ and $\tilde{f}(1/2)=1$. Then, along even $n$, the limit equals $0$ and along odd $n$, it is equal to $1$. Hence, the limit does not exist. This finishes the proof.
\end{proof}
\subsection{Proof of Lemma \ref{lem:phi}}
A general  idea behind the construction of $\varphi$ comes from \cite{Fr-Le}.

Let $\{B_n\}$ be the sequence of sets coming from ii. in the definition of almost sparse set and let
\be\label{eq:eps}
\epsilon_n:=\Big(\min_{k,l\in A\cap [0,n]\setminus B_n,\\ k\neq l}|k-l|\Big)^{-1}.
\ee
By ii., it follows that $\epsilon_n\to 0$. Therefore, there exists a sequence $\{k_n\}$ such that
\be\label{eq:summ}\sum_{n=1}^{+\infty} \epsilon_{q_{k_n}}<+\infty.
\ee
We can also WLOG assume that $\alpha-\frac{p_{k_n}}{q_{k_n}}>0$ and (by taking a further subsequence) that $k_{n+1}>k_n^2$.

Let $f_n=f_{k_n}:\T\ \to \R$ be the following function: for every $w\in\{0,\ldots,q_{k_n}-1\}$,
\be\label{eq:def1}f_n(x):=L_{n,w}\Big(x-\frac{wp_{k_n}}{q_{k_n}}\Big)\;\text{  for  }\;x\in\Big[\frac{wp_{k_n}}{q_{k_n}}, \frac{wp_{k_n}}{q_{k_n}}+\frac{1}{q_{k_n+1}}\Big],
\ee
\be\label{eq:def2}f_n(x):=L_{n,w}\Big(\frac{wp_{k_n}+1}{q_{k_n}}-x\Big)\;\text{  for  }\;x\in \Big[\frac{wp_{k_n}+1}{q_{k_n}}-\frac{1}{q_{k_n+1}},\frac{wp_{k_n}+1}{q_{k_n}}\Big],
\ee
and
\be\label{eq:boundfn}
f_n(x):=\frac{L_{n,w}}{q_{k_n+1}}\;\text{ for }\;x\in \Big[ \frac{wp_{k_n}}{q_{k_n}}+\frac{1}{q_{k_n+1}},
\frac{wp_{k_n}+1}{q_{k_n}}-\frac{1}{q_{k_n+1}}\Big].
\ee
Moreover, we assume that $L_{n,w}>0$.

We define $g:\T\to\R$ by setting
\be\label{eq:defp}
g(x):=\sum_{n=1}^{+\infty}\Big(f_n(x+\alpha)-f_n(x)\Big).
\ee
We have the following:
\begin{lemma} \label{l:wyst} Assume that
\begin{multline}\label{eq:cond}
 \sum_{n=1}^{+\infty }\frac{\max_{w\in \{0,\ldots,q_{k_n}-1\}} L_{n,w}}{q_{k_n}q_{k_n+1}}<+\infty\text{ and }\\\sum_{n=1}^{+\infty} \frac{\max_{w\in \{0,\ldots,q_{k_n}-1\}}|L_{n,w}-L_{n,w+1}|}{q_{k_n+1}}<+\infty,
\end{multline}
where $L_{n,q_{k_n}}=L_{n,0}$.
Then $g$ is continuous.
\end{lemma}
\begin{proof} Notice that
$$
|f_n(x+\alpha)-f_n(x)|\leq \Big|f_n(x+\alpha)-f_n(x+\frac{p_{k_n}}{q_{k_n}})\Big|
+\Big|f_n(x+\frac{p_{k_n}}{q_{k_n}})-f_n(x)\Big|.
$$
Furthermore, either (for some $j$) $x+p_{k_n}/q_{k_n},x+\alpha\in [j/q_{k_n},(j+1)/q_{k_n})$ or $x+p_{k_n}/q_{k_n}\in [j/q_{k_n},(j+1)/q_{k_n})$ and $x+\alpha\in [(j+1)/q_{k_n},(j+2)/q_{k_n})$ and then, by the definition of $f_n$, it follows that
$$
 \Big|f_n(x+\alpha)-f_n(x+\frac{p_{k_n}}{q_{k_n}})\Big|\leq \max(L_{n,w}, L_{n,w'})\frac{1}{q_{k_n}q_{k_n+1}},
$$
where $j=wp_{k_n}$ mod~$q_{k_n}$ and $j+1=w'p_{k_n}$ mod~$q_{k_n}$.
Finally, if $x\in [wp_{k_n}/q_{k_n}, (wp_{k_n}+1)/q_{k_n})$ then
$x+p_{k_n}/q_{k_n}\in [(w+1)p_{k_n}/q_{k_n}, ((w+1)p_{k_n}+1)/q_{k_n})$, so taking into account the bound on $f_n$ given by~\eqref{eq:boundfn}, we obtain
$$
\Big|f_n(x+\frac{p_{k_n}}{q_{k_n}})-f_n(x)\Big|\leq |L_{n,w}-L_{n,w+1}| \frac{1}{q_{k_n+1}},
$$
whence, from \eqref{eq:cond} and \eqref{eq:defp}, it follows that $g$ is continuous on $\T$.
\end{proof}
From now on, we assume that $g$ is defined for parameters for which~\eqref{eq:cond} holds. We will show that there exists a sequence $\{L_{n,w}\}$ satisfying~\eqref{eq:cond} and such that the statement of Lemma~\ref{lem:phi} holds.

Notice that by \eqref{eq:defp} and since $f_n(0)=0$ for every $n$, it follows that for every $k\in \N$, we have
\be\label{eq:formphi}
S_k(g)(0)=\sum_{n=1}^{+\infty}f_n(k\alpha).
\ee
The definition of $\{L_{n,w}\}$ is inductive.  Assume we have defined $L_{\ell,w}$ for $w\in\{0,\ldots, q_{k_\ell}-1\}$ and $\ell<n$ so that $\max_{w\in\{0,\ldots, q_{k_\ell}-1\}}L_{\ell,w}\leq 12q_{k_\ell+1}$ and
$\max_{w\in \{0,\ldots,q_{k_\ell}-1\}}|L_{\ell,w}-
L_{\ell,w+1}|<\max(12\epsilon_{q_{k_\ell}}q_{k_\ell+1},
24q_{k_\ell+1}/q_{k_\ell})$. We will now define $L_{n,w}$ for $w\in\{0,\ldots, q_{k_n}-1\}$ so that
\be\label{eq:max2}
\max_{w\in\{0,\ldots, q_{k_n}-1\}}L_{n,w}\leq 12q_{k_n+1}\ee
and
\be\label{eq:max2a}\max_{w\in \{0,\ldots,q_{k_n}-1\}}|L_{n,w}-L_{n,w+1}|<
\max(12\epsilon_{q_{k_n}}q_{k_n+1},24q_{k_n+1}/q_{k_n}).
\ee
This, by~\eqref{eq:summ} (and the obvious fact that $\sum_r1/q_r<+\infty$), immediately implies that~\eqref{eq:cond} holds and therefore, in view of Lemma~\ref{l:wyst},  $g$ is continuous. Let $w_0<w_1<\ldots <w_{t}$ be all the elements of the set $(A\cap [\frac{q_{k_n}}{2},q_{k_n}])\setminus B_{q_{k_n}}$ (cf.\ iii.\ of Definition~\ref{def:as}). By~\eqref{eq:eps}, it follows that for every $i\in\{0,\ldots,t-1\}$, we have
\be\label{eq:las}
w_{i+1}-w_i\geq \epsilon_{q_{k_n}}^{-1}.
\ee
Let $w\in \{w_0,\ldots, w_t\}$. Then $w\alpha\in [\frac{wp_{k_n}}{q_{k_n}},\frac{wp_{k_n}+1}{q_{k_n}})$ and moreover
\be\label{eq:bounda}
\Big|w\alpha-\frac{wp_{k_n}}{q_{k_n}}\Big|\geq \frac{w}{2q_{k_n}q_{k_n+1}}\geq \frac{1}{4q_{k_n+1}}
\ee
and analogously
\be\label{eq:boundb}
\Big|w\alpha-\frac{wp_{k_n}}{q_{k_n}}\Big|\leq \frac{w}{q_{k_n}q_{k_n+1}}\leq \frac{1}{q_{k_n+1}}.
\ee
Therefore and since $\alpha>\frac{p_{k_n}}{q_{k_n}}$, we know that $f_n(w\alpha)$ is given by \eqref{eq:def1}. Let
\be\label{eq:rwn}
r_{w,n}:=\sum_{m=1}^{n-1}f_m(w\alpha)\text{ mod }1\in [0,1).
\ee
We define $L_{n,w}$ by setting
\be\label{eq:lnw1}
f_n(w\alpha)=L_{n,w}\Big(w\alpha-\frac{wp_{k_n}}{q_{k_n}}\Big)=
2-r_{w,n}
\ee
if $n$ is even and
\be\label{eq:lnw2}
f_n(w\alpha)=L_{n,w}\Big(w\alpha-\frac{wp_{k_n}}{q_{k_n}}\Big)=
3/2-r_{w,n}
\ee
if $n$ is odd. By~\eqref{eq:bounda},~\eqref{eq:boundb} and since $r_{w,n}<1$, it follows that $L_{n,w}\in [\frac{1}{12}q_{k_n+1}, 12q_{k_n+1}]$. In this way we have defined $\{L_{n,w_i}\}_{i=0}^t$. Note also that $L_{n,w_i}>L_{n,w_{i+1}}$ for $i=0,\ldots,t-1$. Now, for any $s\in [w_i,w_{i+1}]$ (with $i=0,\ldots,t-1$), we define inductively
\be\label{d:ind1}
L_{n,s+1}=L_{n,s}+\frac{L_{n,w_{i+1}}-L_{n,w_{i}}}{w_{i+1}-w_i}.
\ee
Note that $L_{n,s}>L_{n,s+1}$ and by iterating~\eqref{d:ind1}, we obtain $$L_{n,w_i+(w_{i+1}-w_i)}=L_{n,w_i}+(w_{i+1}-w_i)
\frac{L_{n,w_{i+1}}-L_{n,w_{i}}}{w_{i+1}-w_i}=L_{n,w_{i+1}},$$
so this is indeed an extension of the definition of $L_{n,w_i}$ to $L_{n,s}$.
By~\eqref{eq:las} and the bound on $(L_{n,w})$, it follows that for every $s$, we have
$$
|L_{n,s+1}-L_{n,s}|<12q_{k_n+1}\epsilon_{q_{k_n}}.
$$
Finally, we complete the definition of $L_{n,w}$ by setting
\be\label{d:in1a}
L_{n,s+1}=L_{n,s}+\frac{L_{n,w_{0}}-
L_{n,w_{t}}}{q_{k_n}-w_{t}+w_0}\ee
for $s=w_t,w_t+1,\ldots, \ldots, q_{n_k},\ldots,q_{n_k}+w_{0}-1$. As before, we verify that this definition yields an extension of the definition of $L_{n,w_i}$ to all of $L_{n,s}$. Moreover, for $s=w_t,\ldots, q_{n_k},\ldots, q_{n_k}+w_0-1$, by~\eqref{d:in1a}, the bound on $L_{n,w_i}$ and $w_0\geq q_{k_n}/2$, we obtain
$$
|L_{n,s+1}-L_{n,s}|<24q_{k_n+1}/q_{k_n}.$$
This finishes the inductive step of the construction.

We will now show that Lemma \ref{lem:phi} holds for $M_n:=\frac{q_{k_n}}{2}$. WLOG we assume that $n$ is even and we will use \eqref{eq:lnw1}, the proof in case $n$ is odd follows the same steps using \eqref{eq:lnw2}.

Recall that by the definition of the sequence $\{w_i\}$, we have $A\cap [\frac{q_{k_n}}{2},q_{k_n}]\setminus B_{q_{k_n}}=\{w_i\}_{i=0}^t$.
Moreover, by \eqref{eq:defp},~\eqref{eq:rwn} and~\eqref{eq:lnw1}, it follows that mod~1, we have
\begin{multline*}
S_{w_i}(g)(0)=\sum_{l=1}^{+\infty}f_l(w_i\alpha)=
\sum_{l=1}^{n-1}f_l(w_i\alpha)+f_n(w_i\alpha)+
\sum_{l=n+1}^{+\infty}f_l(w_i\alpha)=\\
\Big(2+\sum_{l=n+1}^{+\infty}f_l(w_i\alpha)\Big) \text{ mod }1=\sum_{l=n+1}^{+\infty}f_l(w_i\alpha).
\end{multline*}
Therefore, to finish the proof of the lemma,  it is enough to show that
$$
\sum_{l=n+1}^{+\infty}f_l(w_i\alpha)<\epsilon.
$$
Fix $l\geq n+1$. Since $\Big|\alpha-\frac{p_{k_l}}{q_{k_l}}\Big|<
\frac{1}{q_{k_l}q_{k_l+1}}$, we have $$\Big|w_i\alpha-\frac{w_ip_{k_l}}{q_{k_l}}\Big|<
\frac{w_i}{q_{k_l}q_{k_l+1}}\leq \frac{q_{k_n}}{q_{k_l}q_{k_l+1}}<\frac1{q_{k_l}+1}.$$
Hence, the formula for $f_l(w_i\alpha)$ is given by~\eqref{eq:def1}. Therefore  and by \eqref{eq:max2}, we obtain
$$
f_l(w_i\alpha)\leq L_{l,w_i}\frac{q_{k_n}}{q_{k_l}q_{k_l+1}}\leq \frac{12q_{k_n}}{q_{k_l}}
$$
and hence
$$
\sum_{l=n+1}^{+\infty}f_l(w_i\alpha)\leq12q_{k_n} \sum_{l\geq n+1}\frac{1}{q_{k_l}}<\epsilon,
$$
since the sequence $\{k_n\}$ satisfies $k_{n+1}>k_n^2$ and $(q_n)$ grows exponentially fast. This finishes the proof.

\subsection{How to make this construction uniquely ergodic?}
In order to show that the equidistribution along an almost sparse set $A$ does not hold, we only use  our knowledge about~\eqref{eq:rwn} and the fact that the Lipschitz constants  $L_{n,w}$ satisfy certain growths restrictions, cf.~\eqref{eq:max2} and~\eqref{eq:max2a}. Our idea is now to proceed with an interchanged construction in which
$$
q_{k_n}< q_{\ell_n}< q_{k_{n+1}}$$
(remembering that we can sparse $k_n$ and $\ell_n$ as much as we need) and it is  ``time'' $q_{\ell_n}$ which will guarantee that the construction is ergodic (hence uniquely ergodic). In fact, we will show that no non-zero integer multiple of $g$ is multiplicatively cohomologous to a constant, which guarantees that $T_{\alpha,g}$ is uniquely ergodic and the only eigenvalues of it are numbers $e^{2\pi i m\alpha}$, $m\in\Z$.

We define $$h_n(x)=2q_{\ell_n}\Big(x-\frac{j}{q_{\ell_n}}\Big)\text{ if } x\in\Big[\frac{j}{q_{\ell_n}},\frac{j}{q_{\ell_n}}+
\frac1{2q_{\ell_n}}\Big)$$ and
$$h_n(x)=2q_{\ell_n}\Big(\frac{j+1}{q_{\ell_n}}-x\Big)\text{ if } x\in\Big[\frac{j}{q_{\ell_n}}+\frac1{2q_{\ell_n}},
\frac{j+1}{q_{\ell_n}}\Big),$$
for $j=0,\ldots,q_{\ell_n}-1$. Then $h_n$ is Lipschitz continuous, with Lipschitz constant $L'_n=2q_{\ell_n}$ (that is, contrary to the definition of $f_n$, the Lipschitz constant does not depend on the interval $[j/q_{\ell_n},(j+1)/q_{\ell_n})$; $h_n$ is $1/q_{\ell_n}$-periodic). As before, we easily check that the assumptions of Lemma~\ref{l:wyst} are satisfied and $|h_n(x+\alpha)-h_n(x)|\leq L'_n\frac1{q_{\ell_n}q_{\ell_n+1}}$.
We define $K_n:=[q_{\ell_n+1}/2q_{\ell_n}]$ which yields the point $K_nq_{\ell_n}\alpha$  ``close'' to $\frac1{2q_{\ell_n}}$ and guarantees that the distribution of
$$
S_{K_nq_{\ell_n}}(h_n(\cdot+\alpha)-h_n(\cdot))=h_n(\cdot+K_nq_{\ell_n}
\alpha)-h_n(\cdot)$$
is ,,close'' to the distribution of $h_n(\cdot+\frac1{2q_{\ell_n}})-h_n(\cdot)=h_n(\cdot+\frac12)-
h_n(\cdot)$. Note, what will be crucial for our final argument, that {\bf mod~1}
\be\label{eq:nieblisko}
h_n(\cdot+\frac12)-h_n(\cdot)\text{ is not close to any constant}.\ee

We define $$g(x):=\sum_{n=1}^{+\infty}\Big(f_n(x+\alpha)+h_n(x+\alpha)-
f_n(x)-h_n(x)\Big),$$
where $r_{w,n}$ (needed to define  $f_n$) are given by
$$
r_{w,n}:=\sum_{m=1}^{n-1}(f_m(w\alpha)+h_m(w\alpha)).$$
We need to precise how we choose $k_n<\ell_n<k_{n+1}$.
We have
$$
S_{K_nq_{\ell_n}}\Big(\sum_{m=1}^n(f_m(x+\alpha)-f_m(x))+
\sum_{m=1}^{n-1}(h_m(x+\alpha)-h_m(x))\Big)=$$
$$
\sum_{m=1}^{n}(f_m(x+K_nq_{\ell_n}\alpha)-f_m(x))+\sum_{m=1}^{n-1}
(h_m(x+K_nq_{\ell_n}\alpha)-h_m(x)).$$
Now, $K_nq_{\ell_n}\alpha$ is as close to $\frac1{2q_{\ell_n}}$ as we need, so we can make the above sum uniformly as small as we need (by choosing $\ell_n$).  Similarly,
$$
S_{K_nq_{\ell_n}}\Big(\sum_{m=n+1}^\infty(f_m(x+\alpha)-f_m(x))+
\sum_{m=n+1}^{\infty}(h_m(x+\alpha)-h_m(x))\Big)=$$
$$
\sum_{m=n+1}^{\infty}(f_m(x+K_nq_{\ell_n}\alpha)-f_m(x))+
\sum_{m=n+1}^{\infty}
(h_m(x+K_nq_{\ell_n}\alpha)-h_m(x)).$$
Proceeding as in the proof of Lemma~\ref{l:wyst} and using~\eqref{eq:max2a}, we obtain (for some $w\in\{0,\ldots,q_{k_m}-1\}$)
$$
|f_m(x+K_nq_{\ell_n}\alpha)-f_m(x)|\leq $$$$ K_nq_{\ell_n}\frac1{q_{k_m}q_{k_m+1}}+|L_{m,w}-L_{m,w+K_nq_{\ell_n}}|
\frac1{q_{k_m+1}}\leq$$
$$
K_nq_{\ell_n}\Big(\frac1{q_{k_m}q_{k_m+1}}+
\frac{\max(12\epsilon_{q_{k_m}}q_{k_m+1},24q_{k_m+1}/q_{k_m})}
{q_{k_m+1}}\Big).$$
By sparsing the sequence $\{k_n\}$ (e.g.\ we need much stronger assumption than~\eqref{eq:summ}), we can achieve that
$K_nq_{\ell_n}\sum_{m\geq n+1}\Big(\frac1{q_{k_m}q_{k_m+1}}+
\frac{\max(12\epsilon_{q_{k_m}}q_{k_m+1},24q_{k_m+1}/q_{k_m})}
{q_{k_m+1}}\Big)$ is as small as we need. We obtain the same goal for the second series as
$$
|h_m(x+K_nq_{\ell_n}\alpha)-h_m(x)|\leq K_nq_{\ell_n}\frac{L'_m}{q_{\ell_m}q_{\ell_m+1}}=\frac{2K_nq_{\ell_n}}
{q_{\ell_m+1}}.$$

A conclusion of these considerations is that the distribution of $S_{K_nq_{\ell_n}}(g)$ is close to the distribution of $S_{K_nq_{\ell_n}}(h_n(\cdot+\alpha)-h_n(\cdot))$ which by~\eqref{eq:nieblisko} is not close to any constant. This means that $T_{\alpha,g}$ is uniquely ergodic.

Finally, given $n$, we have
$$
S_{w_i}(g)(0)=$$$$\sum_{m=1}^{n-1}(f_m(w_i\alpha)+h_m(w_i\alpha))+
f_n(w_i\alpha)+h_n(w_i\alpha)+\sum_{m\geq n+1}(f_m(w_i\alpha)+h_m(w_i\alpha)).$$
Since $|w_i\alpha-\frac{w_ip_{\ell_n}}{q_{\ell_n}}|<
\frac{w_i}{q_{\ell_n}q_{\ell_n+1}}\leq \frac{q_{k_n}}{q_{\ell_n}q_{\ell_n+1}}$, the third summand $h_m(w_i\alpha)$ is as small as we need, and we finish the proof as at the end of Lemma~\ref{l:wyst}.

\subsection{How to make this construction uniquely ergodic and to satisfy Sarnak's conjecture?}
According to \cite{Ku-Le0}, (see also Theorem 4.1 in \cite{Ku-Le1}) to obtain that $T=T_{\alpha,g}$ is disjoint from M\"obius,  it is sufficient, for all $a,b\in\Z$, $a^2+b^2>0$, and all $(r,s)=1$ (enough to consider pairs of different prime numbers), to have
\be\label{w-eknaaop}
\psi(x)=\psi_{a,b,r,s}(x):=aS_r(g)(r\cdot)+bS_s(g)(s\cdot+c)\ee
is  a (multiplicative) coboundary for no $c\in\T$ (this condition implies the so called AOP property which is sufficient for the M\"obius disjointness; in fact, it yields orthogonality to any multiplicative function).

As we need to consider only countably many cocycles $\psi$, we can repeat the construction from the previous section, where we automatically obtain that  no $A$-equidistribution property holds, while to obtain that no $\psi$ is a (multiplicative) coboundary will be guaranteed by ``reserving''  a subsequence $\{\ell_{n_k}\}$ of $(\ell_n)$ depending on $\psi$ along which $S_{K_nq_{\ell_{n_k}}}(\psi)$ is not close in measure to any constant.

A quick analysis of the construction from the previous section shows that we had $g=g_1+g_2$, where $g_1=\sum_{n=1}^{+\infty}(f_n\circ T-f_n)$ and $g_2=\sum_{n=1}^{+\infty}(h_n\circ T-h_n)$ and we exploited the following:\\
\begin{itemize}
\item $S_{K_nq_{\ell_n}}(g_1)$ was (uniformly) as small as we needed; indeed, for $m\leq n$, to show that $S_{K_nq_{\ell_n}}(f_m\circ T-f_m)$ is small we use the fact that $K_nq_{\ell_n}\alpha$ is as close to~0 as needed, while for $m\geq n+1$, we have
$\|f_m\circ T-f_m\|_{C(\T)}\leq \frac1{q_{k_m}q_{k_m+1}}+\max(12\epsilon_{q_{k_m}},24/q_{k_m})$, so the coboundaries $f_n\circ T-f_n$ are also as small as is needed;
\item $S_{K_nq_{\ell_n}}(g_2)$ was (uniformly) as close to $S_{K_nq_{\ell_n}}(h_n\circ T-h_n)$ as needed (indeed, for $m\leq n-1$ and $m\geq n+1$, we obtain uniform norms of $S_{K_nq_{\ell_n}}(h_n\circ T-h_n)$ as small as we need by the same reason as before).
 \end{itemize}

Now, notice that for each $t,u\geq1$, $c\in\T$ and $j:\T\to\R$, we have
\be\label{eq:zamiana}
S_t(S_u(j)(u\cdot+c))(x)=S_{ut}(j)(ux+c).\ee
It follows that $S_{K_nq_{\ell_n}}(S_r(g_1)(r\cdot))(x)=S_{rK_nq_{\ell_n}}(g_1)(rx)$, so if $\|S_{K_nq_{\ell_n}}(g_1)\|_{C(\T)}\leq\delta$, then
$\|S_{K_nq_{\ell_n}}(g_1(r\cdot))(\cdot)\|_{C(\T)}\leq r\delta$. Using again~\eqref{eq:zamiana} to $S_{K_nq_{\ell_n}}(g_2(s\cdot+c))$, we obtain that $S_{K_nq_{\ell_n}}(\psi)$ is close in measure to
$aS_{K_nq_{\ell_n}}(S_r(h_n\circ T-h_n)(r\cdot))(\cdot)+bS_{K_nq_{\ell_n}}(S_s(h_n\circ T-h_n)(s\cdot+c))(\cdot)$.
In view of \eqref{eq:zamiana}, the result follows whenever
$$
a(h_n(r\cdot+rK_nq_{\ell_n}\alpha)-h_n(r\cdot))+ b(h_n(s\cdot+sK_nq_{\ell_n}\alpha+c)-h_n(s\cdot+c))$$
cannot be close to any constant. This can be achieved by an elementary but a tedious argument.

Since the AOP property of a system implies its orthogonality to any multiplicative function \cite{Ab-Le-Ru}, we obtain the following:

\begin{theorem}\label{p:kontrprzyklady}
Assume that $A\subset \N$ is almost sparse. Then for each irrational $\alpha$ there exists a continuous $g:\T\to\R$ such that the corresponding Anzai skew product $T=T_{\alpha,g}:\T^2\to\T^2$ has the following properties:\\
(i) $T$ is uniquely ergodic.\\
(ii) $T$ is orthogonal to {\bf every} bounded multiplicative function $\boldsymbol{u}:\N\to\C$, that is, $\lim_{N\to\infty}\sum_{n\leq N}f(T^n(x,y))\boldsymbol{u}(n)=0$ for each $f\in C(\T^2)$ of zero mean.\\
(iii) An $A$-equidistribution does not hold for $T$.

In particular, Theorem~\ref{thm:main4} holds.\end{theorem}

\subsection{Proof of iii.\ of Theorem~\ref{thm:main4}}

We only show how to modify parameters in our general construction. We consider $A=\{m^2:\:m\geq1\}$. We take $f(x,y)=e(x)$ and then
$$\frac1N\sum_{m\leq N}f(T_{\alpha,g}^{m^2}(0,0))\mu(m)=\frac1N\sum_{m\leq N}e^{2\pi iS_{m^2}(g)(0)}\mu(m)$$
and will consider $N$ of the form $[M_n^{1/2}]$.
We define $L_{n,w}$ by setting (cf.\ \eqref{eq:lnw1} and~\eqref{eq:lnw2})
$$
f_n(w\alpha)=L_{n,w}\Big(w\alpha-\frac{wp_{k_n}}{q_{k_n}}\Big)=
\frac{7+\mu(w^{1/2})}{4}-r_{w,n},
$$
where $w\in\{w_0,\ldots,w_t\}=A\cap [1,N^2]\setminus B_{q_{k_n}}$. Then,
$$e^{2\pi iS_w(g)(0)}\mu(w^{1/2})= e^{2\pi i \frac{7+\mu(w^{1/2})}{4}}\mu(w^{1/2})e^{2\pi i\sum_{\ell>n}f_\ell(w\alpha)}=\mu^2(w^{1/2})(1+o_n(1)).$$
It follows that, given $\epsilon>0$ and taking $n$ large enough,
$$
\Big|\frac1N\sum_{n\leq N}e^{2\pi iS_{n^2}(g)(0)}\mu(n)-\frac1N\sum_{w\in\{w_0,\ldots,w_t\}}\mu^2(w^{1/2})\Big|\leq o_n(1)+\frac1N t\epsilon$$
which is arbitrarily small as $t\leq N$.

As before, we can make the construction uniquely ergodic, hence minimal. Since the set of square-free numbers has  positive density, iii.\ of Theorem~\ref{thm:main4} follows. Note also that we can adapt the above proof to other multiplicative function, like the Liouville function, so that we obtain the negative answer to the polynomial variant of Sarnak's conjecture in Problem 7.1 \cite{SarWorkshop}.

\vspace{2ex}
\bibliography{PNT6}

\begin{thebibliography}{10}

\bibitem{SarWorkshop}
{A}merican {I}nstitute of {M}athematics, workshop on {S}arnak's conjecture.
\newblock {\em http://aimpl.org/sarnakconjecture/7/}, 2018.

\bibitem{Anzai}
H.~Anzai.
\newblock Ergodic skew product transformations on the torus.
\newblock {\em Osaka Math. J.}, 3:83--99, 1951.

\bibitem{BHP}
R.~C. Baker, G.~Harman, and J.~Pintz.
\newblock The difference between consecutive primes. {II}.
\newblock {\em Proc. London Math. Soc. (3)}, 83(3):532--562, 2001.

\bibitem{Bo1}
J.~Bourgain.
\newblock An approach to pointwise ergodic theorems.
\newblock In {\em Geometric aspects of functional analysis (1986/87)}, volume
  1317 of {\em Lecture Notes in Math.}, pages 204--223. Springer, Berlin, 1988.

\bibitem{Bo}
J.~Bourgain.
\newblock M\"obius-{W}alsh correlation bounds and an estimate of mauduit and
  rivat.
\newblock {\em J. d'Anal.\ Math.}, 119:147--163, 2013.

\bibitem{Bo2}
J.~Bourgain.
\newblock Moebius-{W}alsh correlation bounds and an estimate of {M}auduit and
  {R}ivat.
\newblock {\em J. Anal. Math.}, 119:147--163, 2013.

\bibitem{Bour4}
J.~Bourgain.
\newblock On the correlation of the {M}\"obius function with rank-one systems.
\newblock {\em J. Anal. Math.}, 120:105--130, 2013.

\bibitem{BSZ}
J.~Bourgain, P.~Sarnak, and T.~Ziegler.
\newblock Disjointness of {M}oebius from horocycle flows.
\newblock In {\em From {F}ourier analysis and number theory to {R}adon
  transforms and geometry}, volume~28 of {\em Dev. Math.}, pages 67--83.
  Springer, New York, 2013.

\bibitem{Burgess}
D.~A. Burgess.
\newblock On character sums and {$L$}-series.
\newblock {\em Proc. London Math. Soc. (3)}, 12:193--206, 1962.

\bibitem{Daboussi}
H.~Daboussi.
\newblock Remarques sur les fonctions multiplicatives.
\newblock In {\em S\'{e}minaire {D}elange-{P}isot-{P}oitou, 18e ann\'{e}e:
  1976/77, {T}h\'{e}orie des nombres, {F}asc. 1}, pages Exp. No. 4, 3.
  Secr\'{e}tariat Math., Paris, 1977.

\bibitem{DaboussiDelange}
H.~Daboussi and H.~Delange.
\newblock On multiplicative arithmetical functions whose modulus does not
  exceed one.
\newblock {\em J. London Math. Soc. (2)}, 26(2):245--264, 1982.

\bibitem{deFav}
A.~de~Faveri.
\newblock M\"obius disjointness for ${C}^{1+\epsilon}$ skew products.
\newblock {\em preprint, arXiv:2002.01076}, 2020.

\bibitem{DFI}
W.~Duke, J.~B. Friedlander, and H.~Iwaniec.
\newblock Equidistribution of roots of a quadratic congruence to prime moduli.
\newblock {\em Ann. of Math. (2)}, 141(2):423--441, 1995.

\bibitem{Ei}
T.~Eisner.
\newblock A polynomial version of {S}arnak's conjecture.
\newblock {\em C. R. Math. Acad. Sci.}, 353:569--572, 2015.

\bibitem{Ab-Le-Ru}
H.~El~Abdalaoui, M.~Lema\'nczyk, and T.~de~la Rue.
\newblock Automorphisms with quasi-discrete spectrum, multiplicative functions
  and average orthogonality along short intervals.
\newblock {\em International Math.\ Res.\ Notices}, 14(14):4350--4368, 2017.

\bibitem{SarnakSurvey}
S.~Ferenczi, J.~Ku{\l}aga-Przymus, and M.~Lema\'{n}czyk.
\newblock Sarnak's conjecture: what's new.
\newblock In {\em Ergodic theory and dynamical systems in their interactions
  with arithmetics and combinatorics}, volume 2213 of {\em Lecture Notes in
  Math.}, pages 163--235. Springer, Cham, 2018.

\bibitem{FerMa}
S.~Ferenczi and C.~Mauduit.
\newblock On {S}arnak's conjecture and {V}eech's question for interval
  exchanges.
\newblock {\em J. Anal. Math.}, 134:545--573, 2018.

\bibitem{BrunHooleyFordHalberstam}
K.~Ford and H.~Halberstam.
\newblock The {B}run-{H}ooley sieve.
\newblock {\em J. Number Theory}, 81(2):335--350, 2000.

\bibitem{Fr-Le}
K.~Fr\c{a}czek and M.~Lema{\'n}czyk.
\newblock On the {H}ausdorff dimension of the set of closed orbits for a
  cylindrical transformation.
\newblock {\em Nonlinearity}, 23(10):2393--2422, 2010.

\bibitem{Opera}
John Friedlander and Henryk Iwaniec.
\newblock {\em Opera de cribro}, volume~57 of {\em American Mathematical
  Society Colloquium Publications}.
\newblock American Mathematical Society, Providence, RI, 2010.

\bibitem{Fu1}
H.~Furstenberg.
\newblock Strict ergodicity and transformation of the torus.
\newblock {\em Amer. J. Math.}, 83:573--601, 1961.

\bibitem{Gr}
B.~Green.
\newblock On (not) computing the {M}{\"o}bius function using bounded depth
  circuits.
\newblock {\em Combin.\ Probab.\ Comput.}, 21:942--951, 2012.

\bibitem{Gr-Ta}
B.~Green and T.~Tao.
\newblock The {M}\"{o}bius function is strongly orthogonal to nilsequences.
\newblock {\em Ann. of Math. (2)}, 175(2):541--566, 2012.

\bibitem{Harman}
G.~Harman.
\newblock {\em Prime-detecting sieves}, volume~33 of {\em London Mathematical
  Society Monographs Series}.
\newblock Princeton University Press, Princeton, NJ, 2007.

\bibitem{HeathBrown}
D.~R. Heath-Brown.
\newblock The number of primes in a short interval.
\newblock {\em J. Reine Angew. Math.}, 389:22--63, 1988.

\bibitem{WZX}
W.~Huang, Z.~Wang, and X.~Ye.
\newblock Measure complexity and {M}\"{o}bius disjointness.
\newblock {\em Adv. Math.}, 347:827--858, 2019.

\bibitem{Huxley}
M.~N. Huxley.
\newblock On the difference between consecutive primes.
\newblock {\em Invent. Math.}, 15:164--170, 1972.

\bibitem{IwaniecKowalski}
H.~Iwaniec and E.~Kowalski.
\newblock {\em Analytic number theory}, volume~53 of {\em American Mathematical
  Society Colloquium Publications}.
\newblock American Mathematical Society, Providence, RI, 2004.

\bibitem{KLR}
A.~Kanigowski, M.~Lema{\'n}czyk, and M.~Radziwi{\l\l}.
\newblock Rigidity in dynamics and {M}\"obius disjointness.
\newblock {\em preprint, arXiv:1905.13256}, 2019.

\bibitem{Katai}
I.~K\'{a}tai.
\newblock A remark on a theorem of {H}. {D}aboussi.
\newblock {\em Acta Math. Hungar.}, 47(1-2):223--225, 1986.

\bibitem{Katok}
A.~Katok.
\newblock {\em Combinatorial constructions in ergodic theory and dynamics}.
\newblock vol. 30 of University Lecture Series. American Mathematical Society,
  2003.

\bibitem{Koukoulopoulos}
D.~Koukoulopoulos.
\newblock Primes in short arithmetic progressions.
\newblock {\em Int. J. Number Theory}, 11(5):1499--1521, 2015.

\bibitem{Ku-Le0}
J.~Ku{\l}aga-Przymus and M.~Lema\'nczyk.
\newblock The {M}\"{o}bius function and continuous extensions of rotations.
\newblock {\em Monatshefte Math}, 178:553--582, 2015.

\bibitem{Ku-Le1}
J.~Ku{\l}aga-Przymus and M.~Lema\'nczyk.
\newblock M\"obius disjointness along ergodic sequences for uniquely ergodic
  actions.
\newblock {\em Ergodic Theory Dynam. Systems}, 39:2793--2826, 2019.

\bibitem{Le-spectralsurvey}
M.~Lema\'nczyk.
\newblock {\em Spectral Theory of Dynamical Systems}.
\newblock Encyclopedia of Complexity and System Science. Springer-Verlag, 2009.

\bibitem{SarnakLiu}
J.~Liu and P.~Sarnak.
\newblock The {M}\"{o}bius function and distal flows.
\newblock {\em Duke Math. J.}, 164(7):1353--1399, 2015.

\bibitem{MRShort}
K.~Matom{\"a}ki and M.~Radziwi{\l\l}.
\newblock A note on the liouville function in short intervals.
\newblock {\em arXiv:1502.02374}, 2015.

\bibitem{MatoRad}
K.~Matom\"{a}ki and M.~Radziwi{\l\l}.
\newblock Multiplicative functions in short intervals.
\newblock {\em Ann. of Math. (2)}, 183(3):1015--1056, 2016.

\bibitem{MS}
K.~Matom{\"a}ki and X.~Shao.
\newblock Discorrelation between primes in short intervals and polynomial
  phases.
\newblock {\em arxiv:1902.04708}, 2019.

\bibitem{Ma-Ri}
C.~Mauduit and J.~Rivat.
\newblock Prime numbers along {R}udin-{S}hapiro sequences.
\newblock {\em J. Eur. Math. Soc. (JEMS)}, 17(10):2595--2642, 2015.

\bibitem{MontgomeryTopics}
H.~L. Montgomery.
\newblock {\em Topics in multiplicative number theory}.
\newblock Lecture Notes in Mathematics, Vol. 227. Springer-Verlag, Berlin-New
  York, 1971.

\bibitem{MontgomeryLectures}
H.~L. Montgomery.
\newblock {\em Ten lectures on the interface between analytic number theory and
  harmonic analysis}, volume~84 of {\em CBMS Regional Conference Series in
  Mathematics}.
\newblock Published for the Conference Board of the Mathematical Sciences,
  Washington, DC; by the American Mathematical Society, Providence, RI, 1994.

\bibitem{Exceptional}
H.~L. Montgomery and R.~C. Vaughan.
\newblock The exceptional set in {G}oldbach's problem.
\newblock {\em Acta Arith.}, 27:353--370, 1975.
\newblock Collection of articles in memory of Juri{i} Vladimirovi{c} Linnik.

\bibitem{Mullner}
C.~M\"{u}llner.
\newblock Automatic sequences fulfill the {S}arnak conjecture.
\newblock {\em Duke Math. J.}, 166(17):3219--3290, 2017.

\bibitem{Naz}
F.~L. Nazarov.
\newblock Local estimates for exponential polynomials and their applications to
  inequalities of the uncertainty principle type.
\newblock {\em Algebra i Analiz}, 5(4):3--66, 1993.

\bibitem{Pa}
R.~Pavlov.
\newblock Some counterexamples in topological dynamics.
\newblock {\em Ergodic Theory Dynam. Systems}, 28(4):1291--1322, 2008.

\bibitem{Perelli}
A.~Perelli, J.~Pintz, and S.~Salerno.
\newblock Bombieri's theorem in short intervals. {II}.
\newblock {\em Invent. Math.}, 79(1):1--9, 1985.

\bibitem{Ramachandra}
K.~Ramachandra.
\newblock Some problems of analytic number theory.
\newblock {\em Acta Arith.}, 31(4):313--324, 1976.

\bibitem{SU}
P.~Sarnak and A.~Ubis.
\newblock The horocycle flow at prime times.
\newblock {\em J. Math. Pures Appl. (9)}, 103(2):575--618, 2015.

\bibitem{Chin1}
R.~Shi and Z.~Lian.
\newblock A counter-example for polynomial sarnak conjecture.
\newblock {\em preprint, arXiv:2002.12421}, 2020.

\bibitem{Shiu}
P.~Shiu.
\newblock A {B}run-{T}itchmarsh theorem for multiplicative functions.
\newblock {\em J. Reine Angew. Math.}, 313:161--170, 1980.

\bibitem{Tao}
T.~Tao.
\newblock The logarithmically averaged {C}howla and {E}lliott conjectures for
  two-point correlations.
\newblock {\em Forum Math. Pi}, 4:e8, 36, 2016.

\bibitem{Vinogradov}
I.~M. Vinogradov.
\newblock The method of trigonometrical sums in the theory of numbers.
\newblock {\em Trav. Inst. Math. Stekloff}, 23:109, 1947.

\bibitem{Wang}
Z.~Wang.
\newblock M\"{o}bius disjointness for analytic skew products.
\newblock {\em Invent. Math.}, 209(1):175--196, 2017.

\bibitem{Wi}
M.~Wierdl.
\newblock Pointwise ergodic theorem along the prime numbers.
\newblock {\em Israel J. Math.}, 64(3):315--336 (1989), 1988.

\bibitem{Tang}
T.~Zhan.
\newblock On the representation of large odd integer as a sum of three almost
  equal primes.
\newblock {\em Acta Math. Sinica (N.S.)}, 7(3):259--272, 1991.
\newblock A Chinese summary appears in Acta Math. Sinica {{\bf{3}}5} (1992),
  no. 4, 575.

\end{thebibliography}
\bibliographystyle{plain}

\end{document}